\documentclass[12pt]{amsart}
\usepackage{amsmath, amssymb,latexsym }
\usepackage{verbatim}

\newcommand{\D}{\mathbf D}
\def\ep{\epsilon}

\def\Tr{{\rm Tr}}

\def\e{{\rm e}}
\def\sq2{\sqrt{2}}

\def\t2{{\mathbb T}^2}
\def\s2{{\mathbb S}^2}

\def\N{\mathbb{N}}

\def\R{\mathbb{R}}
\def\Z{\mathbb{Z}}
\def\C{\mathbb{C}}

\def\Lap{\triangle}

\def\ep{\epsilon}

\def\Tr{{\rm Tr}}

\def\e{{\rm e}}
\def\sq2{\sqrt{2}}

\def\t2{{\mathbb T}^2}
\def\s2{{\mathbb S}^2}

\def\N{\mathbb{N}}

\def\R{\mathbb{R}}
\def\Z{\mathbb{Z}}
\def\C{\mathbb{C}}

\def\Lap{\triangle}

\def\hto0{\xrightarrow{\hbar\to 0}}

\def\rto0{\xrightarrow{r\to 0}}

\providecommand{\norm}[1]{\lVert#1\rVert}

\newcommand{\HH}{\mathcal H}

\newcommand{\cE}{\mathcal E}

\newcommand{\cH}{\mathcal H}

\newcommand{\cO}{\mathcal O}
\newcommand{\cP}{\mathcal P}

 \newcommand{\dia}{\varepsilon}

\newcommand{\IR}{\mathbb{R}} %real

\newcommand{\Op}{\operatorname{Op}}
\newcommand{\supp}{\operatorname{supp}}

\newcommand{\X}{{\mathbf X_{\Gamma}}}

\renewcommand{\le}{\leqslant}
\renewcommand{\ge}{\geqslant}

\newcommand{\E}{\mathbb E}
\newcommand{\PP}{\mathbb P}
\newcommand{\Ss}{{\mathbb S}}
\newcommand{\Hh}{{\mathbb H}}

\newcommand{\Mg}{(M,g)}
\newcommand{\pdo}{\psi DO}

\newcommand{\G}{\mathcal{G}}

\newcommand{\dbar}{\bar\partial}
\newcommand{\ddbar}{\partial\dbar}
\newcommand{\cU}{{\mathcal U}}

\renewcommand{\Re}{{\operatorname{Re\,}}}
\renewcommand{\Im}{{\operatorname{Im\,}}}
\newcommand{\half}{{\frac{1}{2}}}
\renewcommand{\phi}{\varphi}

\newcommand{\dcal}{\mathcal{D}}
\newcommand{\ecal}{\mathcal{E}}
\newcommand{\ical}{\mathcal{I}}

\newcommand{\hcal}{\mathcal{H}}
\newcommand{\lcal}{\mathcal{L}}
\newcommand{\mcal}{\mathcal{M}}
\newcommand{\ncal}{\mathcal{N}}
\newcommand{\pcal}{\mathcal{P}}

\newcommand{\scal}{\mathcal{S}}

\newcommand{\oit}{\operatorname{ it}}
\newcommand{\al}{\alpha}

\newcommand{\ga}{\gamma}

\renewcommand{\H}{{\mathbf H}}
\newcommand{\qcal}{\mathcal{Q}}
\newcommand{\sa}{\sigma_A}

\newcommand{\varphijk}{\varphi_{j_k}}

\newcommand{\la}{\lambda}
\newcommand{\de}{\delta}

\newtheorem{theo}{{\sc Theorem}}[section]
\newtheorem{cor}[theo]{{\sc Corollary}}
\newtheorem{conj}[theo]{{\sc Conjecture}}
\newtheorem{lem}[theo]{{\sc Lemma}}
\newtheorem{theoconj}{{\sc Theorem/Conjecture}}
\newtheorem{prop}[theo]{{\sc Proposition}}
\newtheorem{prob}[theo]{{\sc Problem}}

\newenvironment{rem}{\medskip\noindent{\it Remark:\/} }{\medskip}
\newenvironment{defin}{\medskip\noindent{\it Definition:\/} }{\medskip}

\title[Local and Global Analysis of  Eigenfunctions on Riemannian manifolds   ]
{Local and Global Analysis of  Eigenfunctions on Riemannian
manifolds  }

\author{Steve Zelditch }
\address{Department of Mathematics, Johns Hopkins University, Baltimore,
MD 21218, USA} \email{zel@math.jhu.edu}

\thanks{Research partially supported by  NSF grant
 DMS-0603850.}

\date{\today}

\begin{document}

%\tableofcontents

\maketitle

\begin{abstract} This is a survey on eigenfunctions of the Laplacian on Riemannian manifolds (mainly compact
and without boundary). We discuss both local results obtained by analyzing eigenfunctions on small balls, and global
results obtained by  wave equation methods. Among the main topics are nodal sets, quantum limits, and $L^p$ norms of global
eigenfunctions. The emphasis is on the connection between the behavior of eigenfunctions and the dynamics of the geodesic flow,
reflecting the relation between quantum mechanics and the underlying classical mechanics.
We also discuss the analytic continuation of eigenfunctions of real analytic Riemannian manifolds $(M, g)$ to the
complexification of $M$ and its applications to nodal geometry. Besides eigenfunctions, we also consider quasi-modes
and random linear combinations of eigenfunctions with close eigenvalues. Many examples are discussed.

Key Words: Laplacian, eigenvalues and eigenfunctions, quasi-mode,  wave equation, frequency function, doubling estimate,
nodal set, quantum limit, $L^p$ norm, geodesic flow,  quantum complete integrable, ergodic, Anosov, Riemannian random wave.

AMS subject classification: 34L20, 35P20, 35J05, 35L05, 53D25, 58J40, 58J50, 60G60.

\end{abstract}

\tableofcontents

\section*{Introduction}
The aim of this article is to survey both classical and recent
results on the eigenfunctions
\begin{equation}\label{EP}  \Delta_g\; \phi_{\lambda_j} =  \lambda_j^2\; \phi_{\lambda_j} \;\;\; \end{equation}
of the (positive) Laplacian $\Delta_g$ on a (mainly compact) Riemannian
manifold $(M, g)$. We concentrate on the boundaryless case $\partial M = \emptyset$ for simplicity;
 when $\partial M \not=0$ we impose standard
boundary conditions. When $(M, g)$ is compact, the spectrum is
discrete and we arrange the eigenvalues in non-decreasing order
$\lambda_0 < \lambda_1 \leq \lambda_2 \uparrow \infty$. We denote
by $\{\phi_{\lambda_j}\}$ an orthonormal basis of eigenfunctions
with respect to the inner product $\langle \phi_{\lambda_j},
\phi_{\lambda_k} \rangle = \int_M \phi_{\lambda_j}(x)
\phi_{\lambda_k}(x) dV_g$.  The `topography' of an  eigenfunction
ideally encompasses  the  the shapes of pits and peaks in the
graph of $\phi_{\lambda}$, the geometry and connectivity of
`excursion sets' $\{x: \phi_{\lambda_j} (x) > h(\lambda_j)\}$, the
$L^p$ norms and distribution function of $\phi_{\lambda_j}$, the
distribution of its nodal sets and other level sets, the number of
nodal components of different sizes, the number and distribution
of its  critical points, and the concentration, oscillation and
vanishing order properties of $\phi_{\lambda_j}$,  often
encapsulated by the  so-called quantum limits (or microlocal
defect measures), i.e. limits of quantum expectation values
$\langle A \phi_{\lambda_j}, \phi_{\lambda_k}\rangle$ as the eigenvalues tend to
infinity.

Eigenfunctions of Laplacians arise in physics as modes of periodic
vibration of drums and membranes. They also represent stationary
states of a free quantum particle on a Riemannian manifold. More
generally, eigenfunctions of Schr\"odinger operators represent
stationary energy  states of atoms and molecules in quantum
mechanics \cite{Sch}. The topography of modes of vibration and
stationary states began with Ernst Chladni \cite{C,C2} , who
raised the prospect of `visualizing sound' by bowing plates and
observing the patterns of nodal lines (zero sets) of these modes.
Over the last thirty years, chemists and physicists have used
computers rather than bowed plates to visualize the energy states
of atoms. Some computer graphics of eigenfunctions may be found in
such articles as \cite{SHM,H}. In commemoration of the 200th
anniversary of Chladni's diagrams, a recent volume \cite{SS} has
appeared which connects his work with that of contemporary
physicists, chemists and mathematicians.

In mathematics, studies of eigenfunctions tend to fall into two
categories:
\begin{itemize}

\item (i) analyses of ground states, i.e. $\phi_{\lambda_0}$ and $
\phi_{\lambda_1}$;

\item  (ii) analyses of high frequency (or semi-classical)  limits
of eigenfunctions, i.e. the limit as $\lambda_j \to \infty$.

\end{itemize}
Behavior of ground states is very relevant to behavior of highly
excited states,  since an eigenfunction $\phi_{\lambda}$ is always
the ground state Dirichlet eigenfunction in any of its nodal
domains. But our  main emphasis in this survey is on the high
frequency behavior of eigenfunctions rather than on the ground
states.

Studies of high frequency behavior eigenfunctions also fall into
two categories:

\begin{itemize}

\item Local results, which often hold for any solution of (\ref{EP}) on
a (small) ball $B_r(x)$, often $r = O(\lambda^{-1})$,  irrespective of whether the eigenfunction
extends to a global eigenfunction on $M$.   Doubling estimates and
vanishing order estimates in terms of the frequency function, exponential decay bounds and nodal
volume estimates often fall into this category. These methods often apply to large
classes of functions: harmonic functions, polynomials, eigenfunctions, and more general
solutions of elliptic equations.

\item Global results,  which do use this global extension to $M$.
The  typical global assumption is that   $\phi_{\lambda}$ is an
eigenfunction of the wave group $U_t = e^{i t \sqrt{\Delta}}$.
Global properties generally reflect the relation of the wave group
and geodesic flow, particularly the long time behavior of waves
and geodesics on the manifold.

\end{itemize}

We aim to cover both sides of the subject. There already exists a
very well written survey  of the local aspects (the book of Q. Han
and F.H. Lin \cite{H}), so we give more details  on the global
aspects. But one of our purposes in this survey is to state the
main local and global results so that the reader can compare
approaches. As we have recently written a survey on inverse
spectral problems \cite{Z1}, we concentrate on eigenfunctions  and
do not discuss eigenvalue asymptotics very much.

 The  global behavior
of eigenfunctions can only be fully understood by making a phase
space analysis, where the phase space is the co-tangent bundle
$T^*M$ or an energy surface $S^*_g M$. For example, one often
wishes to construct highly localized eigenfunctions or approximate
eigenfunctions (quasi-modes) of $\Delta_g$  or to prove that they
do not exist. Zonal spherical harmonics on the standard sphere (or
any surface of revolution) are examples of eigenfunctions which
are most highly localized at a point $p$ (the poles). But it is
more illuminating to observe that such modes are actually
concentrating on certain Lagrangian submanifolds in phase space
(namely, the one obtained by flowing out $S^*_{p} M$ under the
geodesic flow). With this picture in mind,  one would not expect
to find modes which are highly concentrated at a point unless
there exists a $g^t$ invariant Lagrangian manifold which has a
large singularity over that point,  and  unless the eigenfunction
concentrates in the phase space (microlocal) sense on this
Lagrangian manifold.  One  similarly expect
 modes which are extremal for lower $L^p$ norms to
concentrate along elliptic  closed geodesics. In the case of
integrable systems (see \S \ref{QCI}) one can prove that such
expectations are correct. One further  expects such  special modes
in the integrable case to be extremals for concentration.

To obtain global  phase space results relating the behavior of
eigenfunctions to the behavior of geodesics,  it is necessary to
use microlocal analysis, i.e. the calculus of pseudo-differential
and Fourier integral operators.  Microlocal analysis is in part
the mathematically precise formulation of the semi-classical limit
in quantum mechanics: Pseudo-differential operators are
`quantizations' $Op(a)$ of functions on the phase space $T^* M$
while Fourier integral operators are quantizations of Hamiltonian
flows (and more general canonical transformations). To motivate
the phase space analysis, consider that one often studies the
concentration and oscillation properties of eigenfunctions through
the linear functionals $\rho_{\lambda_j} (A) = \langle A
\phi_{\lambda_j}, \phi_{\lambda_j} \rangle$ on the space of zeroth
order pseudo-differential operators $A$. The possible limits of
the family $\{\rho_{\lambda_j} \}$,   known as quantum limits or
microlocal defect measures, are probability measures on $S^*M$
which are invariant under the geodesic flow. It is difficult but
important to determine or at least constrain the possible limits.
This aim would not even be visible without a phase space
perspective.

We briefly  review some of the basic methods and results of
microlocal analysis in \S \ref{METHODS}. Due to space limitations,
we cannot go over the definitions and methods  in much detail. For
background in microlocal analysis, we  refer  to \cite{EZ,DSj,GSj}
and for systematic treatises to  \cite{HoI-IV} (see especially
Volumes III and IV) and to \cite{SV} (in particular for boundary
problems). Other articles with expositions of   the wave kernel of a Riemannian
manifold are   \cite{Be,Su,Ta2,Z4}. Since
microlocal methods and results  are fundamental to the global
theory, we use them in somewhat detailed arguments in the second
half of this survey. Although the ideas and methods may be
unfamiliar to some geometric analysts interested in
eigenfunctions, they are quite geometric and concrete,   and in a
somewhat less precise form are  common-place among physicists and
chemists.  While writing this survey, the author found it
interesting to compare the range and power of the local and global
methods, and hopes that experts in each side of the subject will
find the other side of the story stimulating.

We would like to thank H. Christianson, Q. Han, H. Hezari,  D. Mangoubi,  S. Nonnenmacher,  M. Sodin,  J. Toth and the
referee  for
corrections on earlier versions of this survey.

\section{\label{BASIC} Basic Definitions and Notations}

 The
Laplacian $\Delta_g$ of a Riemannian Riemannian $(M, g)$ is the
self-adjoint operator associated to the Dirichlet form $Q(f) =
\int_M |df|_g^2 dV_g$ where $dV_g$ is the volume form of $(M, g)$
and $|\cdot|_g$ is the metric on one-forms. It is given in local
coordinates by the expression,
\begin{equation} \label{DELTADEF} \Delta = -  \frac{1}{\sqrt{g}}\sum_{i,j=1}^n
\frac{\partial}{\partial x_i} \left( g^{ij} \sqrt{g}
\frac{\partial}{\partial x_j} \right).  \end{equation} It is a natural
geometric operator in the sense that  $\Delta_g$ commutes with all
isometries $h $ of $(M, g)$ where $h$ acts on functions by
translation $f^h(x) = f(h^{-1} x)$.

The eigenvalue problem (\ref{EP}) is a stationary version of the
homogeneous wave equation
\begin{equation} \Box_g u(t, x) = 0, \;\;\mbox{where} \;\; \Box_g =
\frac{\partial^2}{\partial t^2} + \Delta_g. \end{equation}  The
eigenvalue problem (\ref{EP}) is  equivalent to finding  the
periodic solutions $e^{i \lambda_j t} \phi_{\lambda_j} (x)$ in
time.

The reader should carefully note two conventions we use throughout this article:

\begin{enumerate}

\item We define the Laplacian to be positive in (\ref{DELTADEF}). This is opposite to the convention of many authors,
but saves us from writing many minus signs in wave and heat operators;

\item We denote eigenvalues by $\lambda^2$ in (\ref{EP}),
so that $\lambda$ is the eigenvalue of $\sqrt{\Delta}$. This again saves
many square root signs in wave operators. Physicists often denote $\lambda$ by $k$.

\end{enumerate}

\subsection{Planck's constant and eigenvalue asymptotics}

The eigenvalue problem for  fixed $\lambda$  is a model elliptic
equation. A standard fact  (elliptic regularity) is that
$\phi_{\lambda} \in C^{\infty}(M)$ for any $C^{\infty}$ metric $g$
and that  $\phi_{\lambda} \in C^{\omega}(M)$ (real analytic) when
$(M, g)$ is real analytic.

However, it is more illuminating to regard $\lambda$ as the
inverse of a semi-classical parameter $$\lambda = \hbar^{-1}, $$
where  $\lambda$ is regarded as an `operator' of order $1$. Often
one writes $$\lambda^{-2} \Delta - 1 = h^2 \Delta - 1,$$ in
particular when constructing approximate eigenfunctions and normal
forms.

\subsection{Spectral kernels}

 The individual eigenfunctions are very difficult
to study directly.  One  generally approaches them through various
kernel functions, i.e. Schwartz kernels of functions of
$\Delta_g$. A basic one is  the spectral projections kernel,
\begin{equation} E(\lambda, x, y) = \sum_{j: \lambda_j \leq
\lambda} \phi_j(x) \phi_j(y). \end{equation} Semi-classical
asymptotics is the study of the $\lambda \to \infty$ limit of the
spectral data $\{\phi_j, \lambda_j\}$  or of $E(\lambda, x, y)$.
The (Schwartz) kernel of the wave group $U_t = e^{i t
\sqrt{\Delta}}$  can be represented in terms of the spectral data
by
$$U_t(x,y) = U(t, x, y) :=  \sum_j e^{it \lambda_j} \phi_j(x) \phi_j(y),$$
or equivalently as the Fourier transform $\int_{\R} e^{i t
\lambda} dE(\lambda, x, y)$ of the spectral projections. (Note that we sometimes
write $U_t(x, y)$ and sometimes $U(t, x, y)$ for the wave kernel, whichever
notation is more convenient in its context.)  Hence
spectral asymptotics is often studied through the large time
behavior of the wave group. It is more or less equivalent to study
the resolvent kernel
$$G(\lambda, x,y) = \sum_j  \frac{\phi_j(x) \phi_j(y)}{\lambda_j^2 - \lambda^2}$$
for $\lambda \in \C$ lying on horizontal or logarithmic curves in
$\C$.

To obtain relations between geometry and eigenfunctions, it is necessary to give approximate
formulae for these kernel functions in terms of geometric invariants. Such formulae originate
in works of Hadamard and Riesz and are often called   Hadamard-Riesz parametrices for manifolds
without boundary. Other useful parametrices have been constructed by Lax and H\"ormander.
The  constructions are  reviewed in \S \ref{WAVEKERNEL}, based on the exposition in \cite{D.G,Be}.
Another excellent reference is \cite{HoI-IV}.

We only discuss wave kernel parametrices  for manifolds without boundary.
Parametrices for the wave kernel in the boundary case  are very difficult due to `grazing rays'.
For further discussion we refer to \cite{HoI-IV,SV}.

\subsection{Geodesic flow}

The geodesic flow is the  Hamiltonian flow of $|\xi|_g$
  on the
cotangent bundle $T^*M$ of $M$, equipped with its canonical
symplectic form $\sum_i dx_i \wedge d\xi_i$.
 By definition, $g^t(x, \xi) = (x_t,
\xi_t)$, where $(x_t, \xi_t)$ is the terminal tangent vector at
time $t$ of the unit speed geodesic starting at $x$ in the
direction $\xi$. Here and below, we often identify $T^*M$ with the
tangent bundle $TM$ using the metric to simplify the geometric
description. The geodesic flow preserves the energy surfaces
$\{|\xi|_g = E\}$, i.e.  the co-sphere bundles $S^*_E M$.
 Due to the homogeneity of $H$, the
flow on any energy surface $\{|\xi|_g = E\}$ is equivalent to that
on the co-sphere bundle $S^*M$.

This definition of the geodesic flow makes $g^t$ homogeneous of
degree one, and is slightly different from the geometer's
definition as the Hamiltonian flow of $|\xi|_g^2$. The homogeneous
geodesic flow is not well-defined on the zero section of $T^*M$,
and so it is punctured out. The punctured cotangent bundle is
denoted $T^*M \backslash 0$.

\subsection{Closed geodesics}

By a closed geodesic $\gamma$  one means a periodic orbit of $g^t$
on $S^*_g M$. The period is the length $L_{\gamma}$  of $\gamma$
viewed as a curve in $M$ (i.e. projected to $M$), thus
$g^{L_{\gamma}}(x, \xi) = (x, \xi)$ for all $(x, \xi) \in \gamma$.
The set of all periodic  points $(x, \xi)$, of all possible
periods, is often called `the set of closed geodesics'. A basic
dichotomy in spectral theory is as follows:

\begin{enumerate}

\item  Aperiodic;  The Liouville measure of the closed
 orbits of $g^t$  is zero; or

\item  Positive measure of periodic orbits: the Liouville measure
of the closed orbits is positive.

\item  Periodic : If the entire geodesic flow is periodic,  $g^T =
id$ for some $T>0$, the metric is said to be Zoll.  The common
Morse index of the $T$-periodic geodesics will be denoted by
$\beta$.

\end{enumerate}

In the real analytic case, $(M, g)$ is automatically of type either (1) or (3),
 since a positive measure of closed geodesics implies
that all geodesics are closed.  In the $C^{\infty}$
case, it is simple to construct examples with a positive but not
full  measure of closed geodesics (e.g. a pimpled sphere).

\subsection{Jacobi fields and linear Poincar\'e map along a closed
geodesic}

We recall that Jacobi's equation for a normal vector field $Y$
along $\gamma$ is $Y'' + R(T, Y) T = 0$ where $T = \dot{\gamma}$
and where $R$ is the curvature tensor.  We denote the space of
normal Jacobi fields along a closed geodesic $\gamma$ by
${\mathcal J}_{\gamma}^{\bot}$. It is a symplectic vector space of
dimension $2(\dim M-1)$  with respect to the Wronskian
$$\omega(X,Y) = g(X, \frac{D}{ds}Y) - g(\frac{D}{ds}X, Y).$$
The linear Poincare map $P_{\gamma}$ is  the (real) linear
symplectic monodromy map on $({\mathcal
J}_{\gamma}^{\bot},\omega)$ defined by $P_{\gamma} Y(t) = Y(t +
L_{\gamma}).$

 To diagonalize it, we
 complexify $P_{\gamma}$  to obtain a linear complex symplectic map on
the space
 ${\mathcal J}_{\gamma}^{\bot}\otimes \C$ of complex normal Jacobi fields.
 Since $P_{\gamma}^{\C} \in Sp({\mathcal J}_{\gamma}^{\bot}\otimes \C,\omega)$ (the
 symplectic group),
its spectrum $\sigma (P_{\gamma}^{\C})$ is stable under inverse
and complex conjugation: thus, if $\rho \in \sigma
(P_{\gamma}^{\C})$, then also $ \rho^{-1}, \bar{\rho},
\bar{\rho}^{-1} \in \sigma (P_{\gamma}^{\C})$. The closed geodesic
is  {\it non-degenerate} if
$$\rho_1^{m_1} \dots \rho_n^{m_n} = 1 \Rightarrow m_i = 0\;\;\;\; (\forall i, m_i \in
\N ).$$ In particular, the eigenvalues are simple and $\pm 1
\notin \sigma (P_{\gamma}^{\C}).$ It is called   {\it elliptic} if
all eigenvalues of $P_{\gamma}$ are of modulus one. We then denote
them by  $\{ e^{\pm i \alpha_j}, j=1,...,n\}$. Thus,  $\{\alpha_1,
...,\alpha_n\}$, together with $\pi$, are independent over ${\bf
Q}$. The closed geodesic is called {\it hyperbolic} if all of the
eigenvalues are real. They then come in inverse pairs and we
denote them by  $\{ e^{\pm \mu_j}, j=1,...,n\}$. When $ \dim M >
2$, there are mixed hyperbolic and elliptic geodesics, and more
general ones   where the eigenvalues are complex and not of
modulus one. For simplicity of exposition, we only consider
elliptic and hyperbolic geodesics.

 The associated normalized eigenvectors
will be denoted $\{ Y_j, \overline{Y_j}, j=1,...,n \}$,
$$P{\gamma} Y_j = e^{i \alpha_j}Y_j \;\;\;\;\;\;P_{\gamma}\overline{Y}_j =
e^{-i\alpha_j} \overline{Y}_j \;\;\;\; \omega(Y_j, \overline{Y}_k)
= \delta_{jk}$$
 and relative to a fixed parallel
normal frame $e(s):= (e_1(s),...,e_n(s))$ along $\gamma$ they will
be written in the form $Y_j(s)= \sum_{k=1}^n y_{jk}(s)e_k(s).$

\subsection{Geodesic flow as a unitary operator}

Since the Hamiltonian flow of $|\xi|_g$ preserves the canonical
symplectic form  $\omega = \sum_j dx_j \wedge d\xi_j$ of $T^*M$,
it preserves the volume form $\omega^m$ where $m = \dim M$. It
also preserves the one form $d |\xi|$ and hence it preserves the
Liouville form $d\mu_L : = \frac{\omega^m}{d |\xi|_g}$ on $S^*M$.
By definition, the quotient is the unique $2m - 1$ form whose
wedge product with $d |\xi|_g$ equals $\omega^m$.

We define the unitary operator $V_t$ on $L^2(S^*M, d\mu_L)$
\begin{equation}\label{VT}  V_t (a) : = a \circ g^t. \end{equation}
 It is sometimes called
the Koopman operator associated to the geodesic flow.

\subsection{Spectrum and geodesic flow}

One of the principal emphases in this survey is on the relations between
the global dynamics of the geodesic flow and the eigenfunctions of the wave group.
These relations have a long tradition in quantum mechanics, and are referred
to as semi-classical analysis or study of the classical limit. There exists a large
speculative physics literature on relations between eigenfunctions and eigenvalues
and the underlying classical Hamiltonian system.
Very little is understood, even conjecturally, outside  of the following model types of geodesic flow:

\begin{itemize}

\item Ergodic or  weak mixing geodesic flows, reviewed in \S \ref{QuE}.  More quantitative results
hold if  the geodesic flow is Anosov (e.g. rates of quantum ergodicity \S \ref{QuE}; entropy
of quantum limits \S \ref{ENTROPY}).

\item Integrable systems, reviewed in \S \ref{QCI}-\ref{LPQCI}.

\item KAM systems, i.e. small perturbations of integrable systems. There are few if any results
on eigenfunctions, but there exist results on approximate eigenfunctions or quasi-modes \cite{L,CV2,Pop}.

\item Special systems  with a `divided phase space' possessing  an open set of periodic orbits \cite{MOZ}.

\item Special cases where group theory is available (quotients of reductive, or solvable, or
nilpotent groups by discrete subgroups).

\end{itemize}

\subsection{Ergodic, weak mixing and Anosov geodesic flows}

Ergodicity and weak mixing are spectral conditions on the geodesic flow.
The geodesic flow is called {\it ergodic} if the only invariant
functions, $V_t f = f$ with $f \in L^2(S^*M, d\mu_L)$ are the
constant functions. Equivalently, $1$ is an eigenvalue of
multiplicity one. It is called weak mixing if the spectrum of
$V_t$ is continuous on the orthogonal complement of the constant
functions. Mixing systems add a `smoothness' assumption on the spectral
measures and is also a spectral condition.

\subsubsection{Examples}

\begin{itemize}

\item The most familiar examples are $(M, g)$ of strictly negative sectional curvatures.
Their geodesic flows are Anosov. Model examples are compact quotients of hyperbolic space.

\item In \cite{BD}, many embedded surfaces in $\R^3$ with ergodic geodesic flow are described.
They may have any genus. There exist real analytic metrics on $S^2$ with ergodic geodesic flow.

\item Although we do not discuss manifolds with boundary in detail, there are many examples of
domains in $\R^m$ or other Riemannian manifolds with ergodic geodesic flow. The most famous
are the Bunimovich stadium (a rectangle with semi-circular ends)  and a Sinai billiard table (a rectangle
or torus with a disc removed). References to the literature are given in \cite{HZ,ZZw,GL}, where ergodicity
on the quantum level is studied.

\end{itemize}

We recall that   a geodesic flow $g^t$ is Anosov if,  for each
$\rho\in S^*_g M$, the tangent space $T_\rho S^*_g M$
splits into $g^t$ invariant sub-bundles
$$
T_\rho S^*_g M=E^u(\rho)\oplus E^s(\rho) \oplus
\IR\,X_H(\rho)\,
$$
where $E^u$ is the unstable subspace and $E^s$ the stable
subspace and where $X_H$ is the Hamiltonian vector field of the function
$H(x, \xi) = |\xi|_g$ on $T^*M$.  The unstable Jacobian $J^u(\rho)$ at the point $\rho$ is
defined as the Jacobian of the map $g^{-1}$, restricted to the
unstable subspace at the point $g^1\rho$:
$J^u(\rho)=\det\big(dg^{-1}_{|E^u(g^1\rho)}\big)$ (the unstable
spaces at $\rho$ and $g^1\rho$ are equipped with the induced
Riemannian metric)

\subsection{Completely integrable geodesic flow}

By a (classical) completely integrable system on $T^{*}M$ with $\dim M = n$, we mean
a set of $n$ independent, $C^{\infty}$ functions  $p_1, \dots,
p_n$,
 on $T^{*}M$ satisfying:

\medskip

\begin{tabular}{l} $\bullet \,\, \{ p_{i}, p_{j} \}=0$ \,\, for all $1 \leq i,j \leq n$;\\
$\bullet \,\, dp_{1} \wedge dp_{2} \wedge \cdot \cdot \cdot \wedge
dp_{n} \neq 0$ \,\, on an open dense subset of $T^{*}M.$
\end{tabular}
\medskip

The associated moment map is defined by
\begin{equation} \label{MM} \pcal  = (p_1, \dots, p_n): T^*M \rightarrow
B \subset\R^n. \end{equation}  We  refer to to the set $B$ as the
`image of the moment map,' and denote the set of regular values of $\pcal$ by
$B_{reg}$.
When   $\pcal$ is proper and $b$ is a regular value of $\pcal,$
\begin{equation} \label{CI1}
\pcal^{-1}(b) = \Lambda^{(1)}(b) \cup \cdot \cdot \cdot \cup
 \Lambda^{(m_{cl})}(b) , \;\;\;(b \in B_{reg})
\end{equation}
\noindent where each $\Lambda^{(l)}(b) \simeq T^n$ is an
$n$-dimensional Lagrangian torus.  The number  $m_{cl} (b) = \# \pcal^{-1} (b)$ of  orbits on the level set $\pcal^{-1}(b)$
 is constant on
connected components of $B_{reg}$ and the moment map (\ref{MM}) is
a fibration over each component with fiber (\ref{CI1}).

 The Hamiltonians $p_j$ generate an action of
$\R^n$ defined by
\begin{equation} \label{PHIT}  \Phi_t = \exp t_1 \Xi_{p_1} \circ \exp t_2 \Xi_{p_2} \dots \circ \exp t_n
\Xi_{p_n}. \end{equation} We  denote the  $\Phi_t$-orbits by   $\R^n \cdot (x,
\xi) $, and the isotropy group of $(x, \xi)$ by
 ${\mathcal I}_{(x, \xi)}.$  When  $\R^n \cdot (x, \xi) $ is a compact Lagrangian orbit, then ${\mathcal I}_{(x, \xi)}$
is a lattice of full rank in $\R^n$, and is  known as the `period
lattice', since it consists of the `times' $T \in \R^n$ such that
$\Phi_T |_{\Lambda^{(\ell)}(b)} = Id.$ By the Liouville-Arnold theorem, the orbits of
$\Phi_{t}$ are diffeomorphic to $\R^k \times T^m$ for some $(k,m),
k + m \leq n.$
 In  sufficiently small neighbourhoods $\Omega^{(l)}(b)$ of each component torus,
$\Lambda^{(l)}(b)$, the Liouville-Arnold theorem also gives the
existence of local action-angle variables
$(I^{(l)}_{1},...,I^{(l)}_{n},
\theta^{(l)}_{1},...,\theta^{(l)}_{n})$ in terms of which the
joint flow of $\Xi_{p_{1}},...,\Xi_{p_{n}}$ is linearized.
For convenience, we henceforth normalize the action variables
$I^{(l)}_{1},...,I^{(l)}_{n}$ so that $I^{(l)}_{j} = 0; \,
j=1,...,n$ on the torus $\Lambda^{(l)}(b)$.

Some examples of integrable systems are as follows.

\subsubsection{$T^n = \R^n/\Z^n$}  In the case of a flat torus, the $\xi_j$
variables are action variables and the moment map is the
projection $\pcal: T^* T^n \to \R^n$. The joint eigenvalues of the
action operators $\hat{I}_j = \frac{\partial }{\partial x_j}$ is
the standard integer lattice $\Z^n$ and rays are multiples of
lattice points.

\subsubsection{$S^2$} The classical action variables are $I_1 =
p_{\theta}(x, \xi) = \langle \xi, \frac{\partial}{\partial
\theta} \rangle$, known as the Clairaut integral, and $I_2 = |\xi|_g$. The
moment map $\ical = (I_1, I_2)$ maps $T^* S^2 $ to the triangular
cone $\{(x, y): y > 0, \; |x| \leq y\} \subset \R^2$.

\subsubsection{Simple surfaces of revolution}

These are metrics on $S^2$ for which the geodesic flow and
eigenfunctions  are almost as simple as for the standard sphere.
The key feature is that they are `toric integrable'.

To explain this notion, we assume that the  $S^1$ action is given
by rotations around the $x_3$-axis. We are not primarily
interested in this $S^1$ action but rather the $S^1 \times \R$
action on $T^* M \backslash 0$ generated by rotations and by the
geodesic flow.  The basic assumption defining simple surfaces of
revolution  is that the $\R \times S^1$ action simplifies to a
$T^2$ ($2$-torus) action  on $T^* S^2 \backslash 0$. That is, can
define two global action variables $I_1, I_2$ with $2 \pi
$-periodic Hamiltonian flows such that $|\xi|_g = H(I_1, I_2). $
If we  write the metric as $dr^2 + a(r)^2 d\theta^2$ in geodesic
polar coordinates centered at a fixed point, which we visualize as
the north pole, then a sufficient condition that the geodesic flow
be toric is that  the distance function to the axis of revolution
possesses precisely one local maximum. For instance, a convex
surface of revolution is simple. A simple surface of revolution
possesses precisely one $S^1$ invariant closed geodesic, which we
refer to as the equator.

The first action function is the Clairaut integral generating the
$S^1$ action, defined by $p_{\theta}(v):= \langle v,
\frac{\partial}{\partial \theta}\rangle$.  To define the second
action variable, we need to consider the moment map
$$\pcal = (|\xi|_g, p_{\theta}): T^*S^2 \rightarrow B:= \{(b_1, b_2) :
|b_2| \leq a(r_o)b_1\} \subset \R \times\R^+ $$ of the Hamiltonian
$ \R \times S^1$-action defined by the geodesic flow and by
rotation. The singular set of $\pcal$ is the closed conic set $Z:=
\{(r_o, \theta, 0, p_{\theta}): \theta \in [0, 2 \pi), p_{\theta}
\in \R\}$, i.e. $Z$ is the cone through the equatorial geodesic
(in either orientation). The map $\pcal |_{T^*S_gS^2 - Z}$ is a
trivial $S^1 \times S^1$ bundle over the open convex cone $B_o$
(the interior of $B$), so that
 $\pcal^{-1}(b)$ is an invariant torus for the geodesic flow and
 for rotation. The second action function is the function  $I_2
 (\pcal)$ given by
$$I_2 (b_1, b_2)  =\frac{1}{\pi}
\int_{r_{-}(b)}^{r_{+}(b)} \sqrt{b_1^2 - \frac{b_2^2}{a(r)^2}} dr
+ |b_2|$$ where $r_{\pm}(b)$ (with $b = (b_1, b_2)$  are the
extremal values of $r$ on the annulus $\pi \circ \pcal^{-1}(b)$
(with $\pi : S^*_g S^2 \rightarrow S^2$ the standard projection).

These action variables are best thought of in the following
classical way (Liouville, Jacobi, Arnold): For each $b\in B_o$ ,
let $H_1(F_b, \Z)$ denote the homology of the fiber $ \pcal^{-1}
(b).$ This lattice bundle is trivial since $B$ is contractible, so
there exists a smoothly varying homology basis $\{\gamma_1(b),
\gamma_2(b)\} \in H_1(\pcal^{-1} (b), \Z)$ where $\gamma_1$ are
the orbits of the $S^1$ action, and $\gamma_2$  is a fixed closed
meridian $\gamma_M$ when $b$ is on the center line $\R^+ \cdot
(1,0)$. Then the actions are the integrals of the `action form'
over this moving homology basis:
$$I_1 (b) = \int_{\gamma_1(b)} \xi dx = p_{\theta},\;\;\;\;\;\;I_2 (b) =
\int_{\gamma_2(b)} \xi dx =\frac{1}{\pi}
\int_{r_{-}(b)}^{r_{+}(b)} \sqrt{b_1^2 - \frac{b_2^2}{a(r)^2}} dr
+ |b_2|.$$  On the torus of meridians in $S^*_g S^2$, the value of
$I_2$ equals $\frac{L}{\pi}$ and it equals one on the equatorial
geodesic. So extended, $I_1, I_2$ are smooth homogeneous functions
of degree 1 on $T^*S^2$, and   generate $2\pi$-periodic Hamilton
flows.

The pair $\ical:=(I_1, I_2)$ generate a global Hamiltonian torus
($S^1 \times S^1$)-action commuting with the geodesic flow. The
singular set of $\ical$ equals ${\mathcal Z}:= \{I_2 =\pm
p_{\theta}\}$, corresponding to the equatorial geodesics.
 The map
$$\ical: T^*S^2 - {\mathcal Z} \rightarrow \Gamma_o:= \{(x,y) \in \R \times \R^+ :
|x| < y\}$$ is a trivial torus fibration.

\subsection{Quantum mechanics: Wave group and pseudo-differential operators}

The quantization of the Hamiltonian $|\xi|_g$ is the square root
$\sqrt{\Delta}$ of the positive Laplacian,    of $(M,g)$. It
generates the wave
$$U_t =
e^{i t \sqrt{\Delta}}, $$ which is a group of Fourier integral
operators which propagates singularities along geodesics. In
particular, for fixed $t, x$, the singular support of the wave
kernel $e^{i t \sqrt{\Delta}}(t, x, \cdot)$ lies on the distance
sphere centered at $x$ of radius $t$. This is the wave front of a
spherical  wave launched at $x$. Moreover, the wave front
propagates outward along geodesics normal to the distance spheres.
Thus,  the  `singular directions' lie along the geodesic rays
emanating from $x$.

Above, $\sqrt{\Delta}$, resp.  $e^{i t \sqrt{\Delta}}$ are defined
by the spectral theorem: i.e. they have  the same eigenfunctions
$\phi_{\lambda}$ as $\Delta$ on $(M, g)$ and with eigenvalues
$e^{i t \lambda_j}$ resp. $\lambda_j$.

\subsubsection{Pseudo-differential operators and expected values}

The relations between the wave group and geodesic flow are
fundamental to the global analysis and will be discussed at
greater  length in \S \ref{METHODS}. They originated in Dirac's
Principles of Quantum Mechanics \cite{Dir}  as the relations
between a classical Hamiltonian flow and its quantization as a
unitary group on a Hilbert space. The theory of
pseudo-differential and Fourier integral operators is the rigorous
mathematical framework for the somewhat heuristic ideas of Dirac's
book. We refer the reader to textbook treatments \cite{DSj,EZ,GSj}
and to the comprehensive treatise \cite{HoI-IV} for background.

We denote the class of pseudo-differential operators of order $m$
on $M$ by $\Psi^m$. We fix a quantization $a \to Op(a)$ of
pseudo-differential operators to functions $a(x, \xi) \in
C^{\infty}(T^*M \backslash 0)$ which are symbols of order $m$ in
the sense of admitting poly-homogeneous expansions, $a \sim
\sum_{j = 0}^{\infty} a_{m - j}$ where $a_{m - j}$ is homogeneous
of order $m - j$ on $T^*M \backslash 0$. It is the behavior at
infinity and not at $\xi = 0$ of a symbol which is important and
one usually cuts off the homogeneous functions in a small ball
around $0$. The principal symbol $\sigma_A$ of $A = Op(a)$ is the leading
homogeneous term $\sigma_A = a_m$.

On $\R^m$, $Op(a)$ is given by the formula
$$Op(a) f(x) = \int_{\R^m} a(x,  \xi) e^{i \langle x - y, \xi
\rangle} f(y) dy d\xi. $$ Equivalently, $Op(a)$ is defined by its
actions on exponentials, \begin{equation} \label{OPa} Op(a) e^{i
\langle x, \xi \rangle} = a(x, \xi) e^{i \langle x, \xi \rangle}.
\end{equation} Since $a$ is assumed to be a symbol, the right side is a
perturbation expansion with leading term $a_0 (x, \xi) e^{i
\langle x, \xi \rangle}$. When $a$ is independent of $\xi$ one has
a multiplication operator and when it is independent of $x$ one
has a convolution operator.

We also follow the notation of Dimassi-Sjoestrand \cite{DSj} for
operator classes:  Given an open $U \subset {\Bbb R}^{n}$, we say
that $a(x,\xi;\hbar) \in C^{\infty}(U \times {\Bbb R}^{n})$ is  in
the symbol class $S^{m,k}(U \times {\Bbb R}^{n})$, provided
$$ |\partial_{x}^{\alpha} \partial_{\xi}^{\beta} a(x,\xi;\hbar)| \leq
C_{\alpha \beta} \hbar^{-m} (1+|\xi|)^{k-|\beta|}.$$ \noindent We
say that $ a \in S^{m,k}_{cl}(U \times {\Bbb R}^{n})$ provided
there exists an asymptotic expansion:
$$ a(x,\xi;\hbar) \sim \hbar^{-m} \sum_{j=0}^{\infty} a_{j}(x,\xi)
\hbar^{j},$$ \noindent valid for $|\xi| \geq \frac{1}{C} >0$  with
$a_{j}(x,\xi) \in S^{0,k-j}(U \times {\Bbb R}^{n})$ on this set.
The  associated $\hbar$ Kohn-Nirenberg quantization is given by
$$Op_{\hbar}(a)(x,y) = (2\pi \hbar)^{-n} \int_{{\Bbb R}^{n}}
e^{i(x-y)\xi/\hbar} \,a(x,\xi;\hbar) \,d\xi.$$ As is well-known,
the  definition can be globalized to $M$ using a partition of
unity. We denote this class by  $Op_{\hbar}(S^{m,k})(T^*M) $.
 The symbol of the composition is given by the usual
formula: Given $a \in S^{m_{1},k_{1}}$ and $b \in
S^{m_{2},k_{2}}$, the composition $Op_{\hbar}(a) \circ
Op_{\hbar}(b) = Op_{\hbar}(c) + {\mathcal O}(\hbar^{\infty})$ in
$L^{2}(M)$ where locally,
$$ c(x,\xi;\hbar) \sim \hbar^{-(m_{1}+m_{2})} \sum_{|\alpha|=0}^{\infty}
\frac{{(-i\hbar)}^{|\alpha|} }{\alpha !} (\partial_{\xi}^{\alpha}
a)\cdot (
\partial_{x}^{\alpha} b).$$
For further details, we refer to \cite{DSj}.

On a manifold one patches together such local expressions using a
partition of unity. There is no unique way to do this, and in fact
there is no unique definition of $Op(a)$ on $\R^m$. The one above
is the `Kohn-Nirenberg' definition; other natural choices are the
Weyl definition and the Friedrichs positive quantization.

On special manifolds, such as symmetric spaces, one may define
 $Op(a)$ using the Fourier transform of the symmetric space in
 place of the Euclidean Fourier transform in local coordinates.
 We will discuss the definition in the case of hyperbolic space.

Pseudo-differential operators greatly enlarge the class of
operators which can be used to test eigenfunctions for their
properties. One often uses matrix elements $\langle A
\phi_{\lambda_j}, \phi_{\lambda_j} \rangle$ for this purpose. The
diagonal matrix element is viewed as the `expected value of the
observable $A$ in the state $\phi$'  in quantum mechanics. Such
quadratic forms in $\phi_{\lambda}$ often arise in the study of
eigenfunctions, usually with $Op(a) $ being just multiplication by
a function, or a gradient operator. It is useful to analyze
quadratic expressions of a more general kind.

\subsection{Modes and quasi-modes}

  Quasi-modes are approximate eigenfunctions. It is important to
discuss quasi-modes along with modes (true eigenfunctions) for the
following reasons:

\begin{itemize}

\item Most theorems and proofs  concerning  modes also apply to
quasi-modes. There are few techniques that distinguish modes from
quasi-modes.

\item It is often possible to construct quasi-modes with special
properties. In many applications, they are just as useful as
modes, and often have clearer geometric properties.

\end{itemize}

 There are several
ways to define `approximate eigenfunction'. The classical
definition of  Keller \cite{K},  Babich \cite{B} (see also \cite{BB}), Lazutkin
\cite{L1}, Arnold \cite{Ar}, Ralston \cite{Ra,Ra2},
Guillemin-Weinstein \cite{GW}, Colin de Verdi\`ere \cite{CV2} (see also Popov \cite{Pop})
  and others is that a quasi-mode
$\{\psi_k\}$ of order zero is a sequence of $L^2$-normalized functions
satisfying \begin{equation} \label{QM0} ||(\Delta - \mu_k)
\psi_k||_{L^2} = O(1),
\end{equation}  for a sequence of quasi-eigenvalues $\mu_k$. By
the spectral theorem it follows that there must exist true
eigenvalues in the interval $[\mu_k - \delta, \mu_k + \delta]$ for
some $\delta > 0$. Moreover, if $E_{k, \delta}$ denotes the
spectral projection for the Laplacian corresponding to this
interval, then
$$ ||E_{k, \delta} \psi_k - \psi_k||_{L^2} = O(k^{-1}). $$

One can refine the definition by demanding that the remainder in
(\ref{QM0}) is of order $O(\mu_k^{-s})$ and define a quasi-mode of order $s$
by
 \begin{equation} \label{QM0s} ||(\Delta - \mu_k)
\psi_k||_{L^2} = O(\mu_k^{-s}).
\end{equation}
We refer to \cite{CV2,Z9} for other modifications.

 In the `classical'
work on quasi-modes of Babich, Lazutkin, Arnol'd,  Ralston and
others, quasi-modes are often constructed as oscillatory integrals
(or Lagrangian quasi-modes)
$$\psi_{\lambda}(x) = \int_{\R^k} e^{i \lambda S(x, \xi)} a(x,
\xi) d \xi$$ with special phases and amplitudes designed so that
$|| (-\Delta - \lambda^2) \psi_{\lambda}||$ is small. An important
example is the construction of a quasi-mode associated to a stable
elliptic orbit, reviewed in \S \ref{QMGEO}. Another important
quasi-mode of this type is the quasi-mode associated to the
central rectangle of a Bunimovich stadium, or more precisely to
the Lagrangian cylinder  with boundary consisting of unit tangent
vectors to the `bouncing ball orbits' in the vertical direction in
the rectangle. A recent study of such quasi-modes and of the
possible behavior of actual modes is given in the articles
\cite{BZ,BZ2} of Burq-Zworski. However, the definition (\ref{QM0})
does not force quasi-modes to be Lagrangian and  includes `random'
combinations of eigenfunctions with eigenvalues in a small
interval around $\lambda^2$, which need not have such an
oscillatory integral structure.

 An important  example of such a
quasi-mode is a sequence of ``shrinking spectral projections",
i.e. the $L^2$-normalized projection kernels
$$\Phi_j^z(x) = \frac{\chi_{[\lambda_j, \lambda_j + \epsilon_j]}(x,
z)}{\sqrt{\chi_{[\lambda_j, \lambda_j + \epsilon_j]}(z, z)}}$$
with second point frozen at a point $z \in M$ and with width
$\epsilon_j \to 0$. Here, $\chi_{[\lambda_j, \lambda_j +
\epsilon_j]}(x, z)$ is the orthogonal projection onto the sum of
the eigenspaces $V_{\lambda}$ with $\lambda \in [\lambda_j,
\lambda_j + \epsilon_j]$.  The zonal eigenfunctions of a surface of
revolution are examples of such shrinking spectral projections for
a sufficiently small $\epsilon_j$, and when $z$ is a partial focus
such $\Phi_j^z(x)$ are generalizations of zonal eigenfunctions. On
a general Zoll manifold, shrinking spectral projections of widths
$\epsilon_j = O(\lambda_j^{-1})$ are the direct analogues of zonal
spherical harmonics, and are quasi-modes of order $1$.

\subsection{Heuristics and intutions}

There are several key intuitions to keep in mind from the outset:

\begin{itemize}

\item Local intuition: Eigenfunctions of eigenvalue $\lambda^2$
resemble polynomials of degree $\sim C  \lambda$ in terms of their
local  complexity and growth, e.g. vanishing order at zeros,
volumes of nodal hypersurfaces, growth rates on small balls.

\item Global intuition: Eigenfunctions are stationary states of
the quantization $U(t) = \exp it \sqrt{\Delta}$ of the geodesic
flow. Their high-frequency limits $\lambda \to \infty$ should
reflect the dynamics of the classical geodesic flow. When the
geodesic flow is integrable, eigenfunctions should localize on the
invariant tori (or more correctly, on level sets of the moment
map). When the geodesic flow is ergodic, eigenfunctions should be
diffuse (i.e. not localize).

\item Modes versus quasi-modes and random waves: Most results
about eigenfunctions apply to quasi-modes, i.e. linear combination
of eigenfunctions with very close by eigenvalues. More precisely,
when $|\lambda_j - \lambda_k| \leq \frac{C}{\log \lambda}$. In
integrable cases, one can spectrally separate out true
eigenfunctions from such `random waves' but in general one cannot.
In ergodic cases, eigenfunctions in many respects resemble random
waves.

\end{itemize}

The key difficulty  in relating classical to quantum mechanics (e.g. in  quantum chaos)
is that it involves a
comparison between long-time dynamical properties of $g^t$
 and $U_t$ through the symbol map and similar classical limits.
 The classical dynamics defines the `principal symbol' behavior of
 $U_t$ and the `error' $U_t A U_t^* - Op(\sigma_A \circ g^t)$
 typically grows exponentially in time. This illustrates the
 ubiquitous `exponential barrier' in the subject. The classical approximation is not
 clearly valid after the so-called `Heisenberg time' (see \S \ref{ENTROPY}).   The articles
 \cite{A,AN,ANK} have excellent discussions of these problems.

\subsection{Notational Index}

\begin{itemize}

\item $g_{ij} = g(\frac{\partial}{\partial
x_i},\frac{\partial}{\partial x_j}) $, $[g^{ij}]$ is the inverse
matrix to $[g_{ij}]$.

\item   $r = r(x,y)$  denotes the distance function of $(M, g)$.

\item $B(x_0, r) \subset M $ denotes the geodesic ball of radius
$r$ centered at $x_0$;

\item $dV_g$ denotes the volume density of $(M, g)$;

\item    $\Theta(x, y)$  denotes the volume density in normal
coordinates  at $x$, i.e. $dV_g(y) = \Theta(x, y) dy. $

\item $T^*M$ is the cotangent bundle of $M$, and $T^* M \backslash
\{0\}$ is the puncture cotangent bundle where the zero section is
deleted.

\item
  $|\xi|_g = \sqrt{\sum_{ij = 1}^n
g^{ij}(x) \xi_i \xi_j} : T^* M \backslash
\{0\} \to \R^+$  denotes the length of a
(co)-vector.

\item  The unit (co-) ball bundle is denoted $B^*M = \{(x, \xi):
|\xi|_g \leq 1\}$. Its boundary $S^*M = \{|\xi|_g = 1\}$ is  the
unit cosphere bundle;

\item $\mu$ is the {\it Liouville measure} on $S^*M$, i.e. the
surface measure $d\mu = \frac{dx d\xi}{d H}$  induced by the
Hamiltonian   $H = |\xi|_g$ and by the symplectic volume measure
$dx d\xi$
 on $T^*M$.

 \item $\omega$ is the linear functional on $C(S^*M)$ defined by
 $\omega(\sigma) = \frac{1}{\mu(S^*M)} \int_{S^*M} \sigma d\mu$. The same
 notation is used for the functional (state) on the algebra $\Psi^0(M)$
 of zeroth order pseudo-differential operators defined by $\omega(A)
 = \omega(\sigma_A)$.

\item $g^t: T^* M \backslash
\{0\} \to T^*M \backslash \{0\}$  denotes the geodesic flow,
i.e. the Hamilton flow of $|\xi|_g.$

\item
  $\gamma$
 denotes a closed geodesic, i.e. closed orbit of $g^t$ in $S^*M$.
Thus,  $\gamma(t) = g^t(x_0, \xi_0) \in S^*M$ where $g^L (x_0,
\xi_0) = (x_0, \xi_0).$ $L = L_{\gamma}$ is the period of the
closed geodesic.
 By abuse of notation, we sometimes also  use $\gamma$ to denote
its projection to $M$, where $L_{\gamma}$ is the length of
$\gamma$.

\item Geodesic loops versus closed geodesics: Viewed as curves on
$M$ both satisfy $\alpha(0) = \alpha(L)$, but the latter also
satisfy $\alpha'(0) = \alpha'(L)$.

\item  ${\rm inj}(M,g)$  denote the injectivity radius.

\item    $\Lambda (M)$  denote the $H^1$ loopspace of $M$;

\item   $\G (M)$  denote the subset of closed geodesics in
$\Lambda(M)$;

\item $\G_{[\gamma]}$  denote the set of closed geodesics in
$\G(M)$ whose free homotopy class is $[\gamma]$;

\end{itemize}

\bigskip

\section{\label{EXPLICIT} Explicitly solvable eigenfunctions}

There are only a few Riemannian manifolds $(M, g)$  where one has
explicit formulae for eigenfunctions.  In this section, we briefly
review these examples. What they have in common is that in each
case $\Delta_g$ is {\it completely integrable}, i.e. commutes with
a maximal family of (pseudo-differential operators). The joint
eigenfunctions of such quantum integrable Laplacians have very
special properties reflecting the complete integrability of the
geodesic flow of $(M, g)$.  Our intuition is that such integrable
eigenfunctions should be models of {\it extremal} behavior: for
instance,  extremal growth and  concentration. Quantum integrable
eigenfunctions will be discussed in  depth in \S \ref{QCI}. Here,
we only wish to go through the simplest examples.

\subsection{\label{EUCLIDEAN} $\R^n$}

The eigenspaces of the Laplacian on $\R^n$ are defined by
$$\ecal_{\lambda} = \{\phi_{\lambda} \in \scal'(\R^n): \;  \Delta \phi_{\lambda}  =
\lambda^2\; \phi_{\lambda} \}, $$ where $\scal'(\R^n)$ is the space of tempered
distributions. Since eigenfunctions are $C^{\infty}$ the
temperedness only constrains the growth to be polynomial, i.e.
rules out exponentially growing eigenfunctions such as $e^{\langle
\lambda x, \xi \rangle}$.

Since the Euclidean motion group  $E_{n}$ commutes with the flat
Laplacian  $\Delta = \Delta_{\R^{n}}$ it preserves the
eigenspaces. Hence the infinitesimal translations
$\frac{\partial}{\partial x_j}$ and the infinitesimal rotations
commute with $\Delta$.   The joint complex valued eigenfunctions of
the translations  are the Euclidean plane waves $e^{i \langle \xi,
x \rangle}$ with $\xi \in \R^n$. The eigenspaces $\ecal_{\lambda}
$ are infinite dimensional but are spanned in the following sense
by the  $e^{i \langle \xi, x \rangle}$ with $|\xi| = \lambda$
(i.e. of frequency $\lambda$): There is a Poisson type integral
formula for eigenfunctions: for any $\phi_{\lambda} \in \ecal_{\lambda}$ there
exists a distribution $d\mu \in \dcal'(S^{n-1})$ such that
\begin{equation} \phi_{\lambda} (x) = \int_{S^{n-1}} e^{i \lambda \langle \xi, x
\rangle} d\mu(\xi). \end{equation}    We refer to \cite{Hel} for
further background.

 Let $\phi_{\lambda} \in \ecal_{\lambda}$. Since $SO(n)$ acts on
$\ecal_{\lambda}$ we may decompose it into isotypic subspaces,
$$\ecal_{\lambda} = \bigoplus_{N = 0}^{\infty} \ecal_{\lambda}(N),
$$
where $\ecal_{\lambda}(N)$ is the subspace of eigenfunctions
transforming by the $N$th  irreducible representation of $SO(n)$,
realized by the space $\hcal_N$ of spherical harmonics of degree
$N$ on $S^{n-1}$ (see \S \ref{SPHERE}).  As this implies, the
elements of $\ecal_{\lambda}(N)$ may be expressed as sums of of
separation of variable  eigenfunctions $J_{N, n}(\lambda r)
\phi_N(\omega)$ where $\phi_N \in \hcal_N(S^{n-1})$. The radial
factor is a temperate solution of the $n$ dimensional spherical
Bessel equation
$$\left( r^2 \frac{d^2}{dr^2} + (n - 1) r \frac{d}{dr} + (r^2  -
N(N + n-1)) \right) J_{N, n}(\lambda r) = 0,
$$
and is given explicitly by
$$J_{N, n}(\lambda r) = (\lambda r)^{- \frac{n-2}{2}}
J_{|N| + \frac{n-2}{2}} (\lambda r) = C_{N + \frac{n-2}{2}}
(\lambda r)^N   \int_0^\pi \cos (\lambda r \cos \theta) \sin^{2 N
+ n-2} \theta d \theta,  $$ where  $C_{N + \frac{n-2}{2}}  =
\frac{1}{ 2^{\frac{n-2}{2}} \Gamma (\frac{n-2}{2} + \frac{1}{2})
\Gamma(\frac{1}{2})}$ and where  $J_{\nu}$ is the Bessel function
of order $\nu$,
$$J_{N + \frac{n-2}{2}}(r \lambda)  = C_{N + \frac{n-2}{2}} (\lambda r)^{N +  \frac{n-2}{2}} \int_0^\pi \cos (\lambda r
\cos \theta) \sin^{2 N + n-2} \theta d \theta. $$

The space $\ecal_{\lambda}$ carries an inner product which is
invariant under $E_{n + 1}$. An orthonormal basis is given by
\begin{equation} u_{\alpha} (x) = \int_{S^{n-1}} e^{i \lambda \langle \xi, x
\rangle} \chi_{\alpha} (\xi) dV_0(\xi), \end{equation} where
$\{\chi_{\alpha}\}$ is an orthonormal basis of spherical harmonics
of $L^2(S^{n+1})$. The notation is simplest when $n = 2$: Then an
orthonormal basis of $\ecal_{\lambda}$ is given by
$\{e_{\lambda}^{\ell}: = J_{\ell} ( \lambda r) e^{i \ell
\theta}\}_{\ell = - \infty}^{\infty}.$ Equivalently, the norm of
$u$ in $\ecal_{\lambda}$ is the norm of the distribution $d\mu$ in
$L^2(S^{n - 1})$. We denote by $\ecal^{(2)}_{\lambda}$ the
subspace of elements of finite norm.

In view of the infinite dimensionality of $\ecal_{\lambda}$, it is
possible to construct eigenfunctions with very special properties.
First, we consider the orthogonal projection onto
$\ecal^{(2)}_{\lambda}$, which is given by the Bessel kernel
\begin{equation} E_{\lambda} f (x) = \int_{\R^n}
J_{\frac{n-2}{2}}(\lambda |x - y| ) f(y) dy. \end{equation} In the
notation of \cite{HoI-IV}  (vol III, Chapter XVII), the spectral
projection for $\Delta_{\R^n}$ for the spectral interval $[0,
\lambda^2]$ is given by
\begin{equation} e_0(x, \lambda^2) = (2 \pi)^{-n} \int_{|\xi| <
\lambda} e^{i \langle x, \xi \rangle} d \xi. \end{equation} Hence,
the spectral projection onto $\ecal_{\lambda}$ is given by
\begin{equation} \frac{d}{d \lambda} \; e_0(x, \lambda^2) = (2 \pi)^{-n} \int_{|\xi|
= \lambda} e^{i \langle x, \xi \rangle} d S, \end{equation} where
$S$ is the standard surface measure.

If we fix $y$ we obtain a normalized
 Euclidean coherent state
$\phi_{\lambda}^y = J_{\frac{n-2}{2}}(\lambda |x - y| )$.  By the
standard asymptotics of Bessel's function, $|\phi_{\lambda}^y  |
\sim (r(\cdot, y) \lambda)^{-1/2}$ as $r(\cdot, y) \lambda \to
\infty$. Thus, $\phi_{\lambda}^y$ peaks at $x = y$ and decays at a
rate $r(x, y)^{-1/2}$ away from the peak point. Also, in the high
frequency limit it decays like $\lambda^{-1/2}$ for fixed $x \not=
y$.

%Another type of concentrated eigenfunction is a Gaussian beam
%along a geodesic. For simplicity, we assume the geodesic is the
%line in the $e_n = (0,  0, \dots, 1)$ direction, and seek an
%eigenfunction for the form $e^{i \lambda x_n} D(\lambda x')$ where
%$x' = (x_1, \dots, x_{n-1})$.

\subsection{Flat tori}

A flat torus is a
 compact quotients $\R^n/ L$ where $L$ is
lattice such as $\Z^n.$  The Laplacian $\Delta$ of the flat metric
again commutes with the $n$ vector fields
$\frac{\partial}{\partial x_j}$. In the compact case, an
orthonormal basis of joint eigenfunctions $\phi_{\lambda}$   is
provided by  the exponentials $ e^{i \langle \lambda, x \rangle}$
 where $\lambda \in L^*$, the
dual lattice. The corresponding Laplace eigenvalue is
$|\lambda|^2$.  Usually (for instance when studying nodal sets),
we prefer the real orthonormal basis $\sin \langle \lambda, x
\rangle, \cos \langle \lambda, x \rangle. $ A key feature of these
eigenfunctions is that they are  linear combinations of a finite
number (two) functions of the form  $a(x) e^{i S(x)/h}$ with $h
=|\lambda |^{-1}$. Such functions are known as WKB modes or
Lagrangian states. The  heuristic
 scaling (\ref{SCALING}) is exactly true, with
 $$e^{i \langle \lambda, x_0 + \frac{u}{|\lambda|} \rangle} = e^{i \langle \lambda, x_0
 \rangle}
e^{i \langle \frac{\lambda}{|\lambda|}, u \rangle}, $$ i.e. with
$dT^{x_0}_{\lambda} = e^{i \langle \lambda, x_0 \rangle} \delta_{
 \frac{\lambda}{|\lambda|}}. $ The `phase function' $S(x) = \langle x,
\frac{\lambda}{|\lambda|} \rangle$ is the generating function of a
Lagrangian submanifold $\Lambda = \{(x, \xi =
\frac{\lambda}{|\lambda|}) \}$, which is a Lagrangian torus in
 $T^*_g (\R^n/ L)$. The eigenfunctions have norm one, $|e^{i
 \langle x, \lambda \rangle}| = 1$. We also note that
 $$\langle Op(a)  e_{\lambda}, e_{\lambda} \rangle = \int_{\R^n/L} a(x,
 \lambda) dx = \int_{T_{\lambda}} a d\mu_L. $$

\subsection{\label{SPHERE} Standard Sphere}

Eigenfunctions of the Laplacian  $\Delta_{S^n}$ on the standard
sphere  $S^n$  are restrictions of harmonic homogeneous
polynomials on $\R^{n+1}$.

Let $\Delta_{\R^{n+1}} = - ( \frac{\partial^2}{\partial x_1^2} +
\cdots + \frac{\partial^2}{\partial x_{n+1}^2})$ denote the
Euclidean Laplacian. In  polar coordinates $(r, \omega)$ on
$\R^{n+1}$, we have  $\Delta_{\R^{n+1}} = - \left( \frac{\partial^2}{\partial r^2} +
\frac{n}{r} \frac{\partial}{\partial r}\right) + \frac{1}{r^2}
\Delta_{S^{n}}. $ A polynomial $P(x) = P(x_1, \dots, x_{n+1})$ on
$\R^{n+1}$  is called:

\begin{itemize}

\item  homogeneous of degree $k$ if $P(r x ) = r^k P(x).$ We
denote the space of such polynomials by $\pcal_k$. A basis is
given by the monomials $$x^{\alpha} = x_1^{\alpha_1} \cdots
x_{n+1}^{\alpha_{n+1}}, \;\;\; |\alpha| = \alpha_1 + \cdots +
\alpha_{n+1} = k.$$

\item Harmonic if $\Delta_{\R^{n+1}} P(x) = 0.$ We denote the
space of harmonic homogeneous polynomials of degree $k$ by
$\hcal_k$.

\end{itemize}

Suppose that $P(x)$ is a homogeneous harmonic polynomial of degree
$k$ on $\R^{n+1}$. Then,
$$\begin{array}{l} 0 = \Delta_{\R^{n+1}} P  = - \{\frac{\partial^2}{\partial r^2} + \frac{n}{r}
\frac{\partial}{\partial r} \}  r^k P(\omega)   + \frac{1}{r^2}
\Delta_{S^{n}} P (\omega)  \\ \\
\implies \Delta_{S^{n}} P (\omega) = (k (k - 1) + n k) P(\omega).
\end{array}
$$
Thus, if we restrict $P(x)$ to the unit sphere $S^n$ we obtain an
eigenfunction of eigenvalue $k (n + k - 1). $
 Let $\hcal_k \subset L^2(S^n)$ denote the space of spherical harmonics
of degree $k$.  Then:
\begin{itemize}

\item $L^2(S^n) = \bigoplus_{k = 0}^{\infty} \hcal_k$. The sum is
orthogonal.

\item $Sp(\Delta_{S^n}) = \{ \lambda_k^2 = k ( n + k - 1) \}. $

\item $\dim \hcal_k$ is given by $$d_k = { n + k - 1  \choose k}
  - { n + k - 3 \choose k -
2}$$

\end{itemize}

The Laplacian $\Delta_{S^n}$ is quantum integrable. For
simplicity, we restrict to $S^2$. Then the  group $SO(2) \subset
SO(3)$ of rotations around the $x_3$-axis commutes with the
Laplacian. We denote its infinitesimal generator by $L_3 =
\frac{\partial }{i
\partial \theta}$. The standard basis of spherical harmonics
is given by the joint eigenfunctions $(|m| \leq k$)
$$\left\{ \begin{array}{l} \Delta_{S^2} Y^k_m = k ( k + 1) Y^k_m; \\ \\
 \frac{\partial }{i \partial \theta} Y^k_m  = m Y^k_m. \end{array} \right.$$

Two basic spherical harmonics are:

\begin{itemize}

\item The highest weight spherical harmonic $Y^k_k$. As a
homogeneous polynomial it is given up to a normalizing constant by
$ (x_1 + i x_2)^k$ in $\R^3$ with coordinates $(x_1, x_2, x_3)$.
It is a  `Gaussian beam' along the equator $\{x_3 = 0\}$, and is
also a quasi-mode associated to this stable elliptic orbit. These
general notions will be discussed in \S \ref{QMGEO}.
\medskip

\item The zonal spherical harmonic $Y^k_0$. It may be expressed in
terms of the orthogonal projection $\Pi_k: L^2(S^2) \to \hcal_k. $

\end{itemize}

We now explain the last statement:  For any $n$, the  kernel  $\Pi_k(x, y)$ of
$\Pi_k$ is defined by
$$\Pi_k f(x) = \int_{S^n} \Pi_k(x, y) f(y) dS(y), $$
where $dS$ is the standard surface measure.  If $\{Y^k_m\}$ is an
orthonormal basis of $\hcal_k$ then
$$\Pi_k(x, y) = \sum_{m = 1}^{d_k} Y^k_m(x) Y^k_m(y). $$
Thus for each $y$, $\Pi_k(x, y) \in \hcal_k.$ We can $L^2$
normalize this function by dividing by the square root of
$$||\Pi_k(\cdot, y)||_{L^2}^2 = \int_{S^n} \Pi_k(x, y) \Pi_k(y, x)
dS(x) = \Pi_k(y, y). $$

We note that $\Pi_k(y, y) = C_k$ since it is rotationally
invariant and $O(n + 1)$ acts transitively on $S^n$. Its integral
is $\dim \hcal_k$, hence, $\Pi_k(y, y) = \frac{1}{Vol(S^n)} \dim
\hcal_k.$ Hence the normalized projection kernel with `peak' at
$y_0$ is
$$Y^k_0(x) = \frac{\Pi_k(x, y_0) \sqrt{Vol(S^n)}}{\sqrt{\dim
\hcal_k}}.$$ Here, we put $y_0$ equal to the north pole $(0, 0
\cdots, 1).$  The resulting function is called a zonal spherical
harmonic since it is invariant under the group $O(n + 1)$ of
rotations fixing $y_0$.

One can rotate $Y^k_0(x) $ to $Y^k_0(g \cdot x)$ with $g \in
O(n+1)$ to place the `pole' or `peak point' at any point in $S^2$.

\subsection{\label{SR} Surface of revolution} By a surface of revolution is
meant a surface, necessarily  $M = S^2$ or $M = \R^2/\Z^2$, whose
metric $g$  is  invariant under an $S^1$ action by isometries. In
this case, one can separate variables to analyze eigenfunctions,
i.e. the joint eigenfunctions of $\Delta_g$ and
$\frac{\partial}{\partial \theta}$ have the form  $e^{in \theta}
\phi_{n, j}(r)$. But this is not necessarily the best way to
analyze eigenfunctions. In \S \ref{QCISIMPLE}, we will discuss how
to express $\Delta$ in terms of action operators and how to obtain
results on eigenfunctions that would be awkward if one used
separation of variables. In the case of simple surfaces of
revolution, the existence of a `global Birkhoff normal form' for
$\Delta_g$ gives much more control than separating variables.

\subsection{$\H^{n}$}

Much of the discussion of \S \ref{EUCLIDEAN} has an analogue on
hyperbolic space. For the sake of simplicity we assume $n = 2$, so
that the group of isometries is $SL(2, \R)$.  Then the analogue of
Euclidean plane waves are the horocyclic (or hyperbolic) plane
waves $e^{(i \lambda + 1) \langle z, b \rangle }$ where $\langle
z, b \rangle$ is the function on $\H^2 \times B$ (with $B =
S^{1}$ the ideal boundary of $\H^2$) equal to the signed
distance from $0$ to the horocycle passing through $z $ and $b$.
 Equivalently,
$$e^{\langle z, b \rangle} = \frac{1 - |z|^2}{|z - b|^2} = P_\D(z, b),$$
 where  $P_\D(z, b)$ is the Poisson kernel of the unit disc. (We
 caution again that $e^{\langle z, b \rangle}$ is written $e^{2\langle z, b
 \rangle}$ in \cite{H, Z2}).
Helgason defines a non-Euclidean Fourier transform by means of
these plane waves, and there exist analogues of the objects and
results of \S \ref{EUCLIDEAN}. Further one can define a covariant
calculus of pseudo-differential operators by using the
non-Euclidean Fourier transform by replacing (\ref{OPa}) by
\begin{equation} \label{HYPOPa} Op(a)e^{(i \lambda + 1) \langle z, b \rangle
}= a(z, b, \lambda) e^{(i \lambda + 1) \langle z, b \rangle }.
\end{equation}

The most interesting aspect  of $\H^2$ lies in its quotients by
discrete groups $\Gamma \subset SL(2, \R)$. Eigenfunctions on the
quotient $\H^2 \backslash \Gamma$ are the same as automorphic
eigenfunctions on $\H^2$ satisfying $\phi(\gamma z) = \phi(z)$.
Helgason has introduced a generalized Poisson formula for
eigenfunctions
 of exponential growth in the sense that  there
exists $C>0$ such that $|\phi(z)|\leq C e^{Cd_\D(0, z)}$ for all
$z$. We assume  $\Lap \phi= \lambda^2 \phi$ and following a
traditional notation put $\lambda^2=\frac14+r^2$.  Then in
(\cite{H}, Theorems 4.3 and 4.29, it is proved that there exists a
distribution $ T_{ir, \phi_{ir}} \in {\mathcal D}'(B)$ such that
\begin{equation} \label{HELGAFOR} \phi_{ir}(z) = \int_B e^{(\frac12 + ir) \langle z,
b \rangle } T_{ir, \phi_{ir}} (db), \end{equation}  for all $z\in\D$. The
distribution is unique if $\frac{1}{2} + ir \not= 0, -1,
-2,\cdots$.

 The distribution $T_{ir,
\phi_{ir}}$ is called the  boundary value of $\phi_{ir}$ and may
be obtained from $\phi_{ir}$ in several explicit ways. One is to
expand the eigenfunction into the ``Fourier series",
\begin{equation} \label{GENSPH}  \phi_{ir}(z)  = \sum_{ n \in \Z} a_{n}
\Phi_{r, n}(z), \end{equation} in the disc model  in terms of the
generalized spherical functions $\Phi_{r, n}$ defined by
(\cite{H}, Theorem 4.16)
\begin{equation}  e^{(\frac12+i r) \langle z, b \rangle } = \sum_{ n \in \Z}
\Phi_{r, n}(z) b^n,\;\;\; b \in B.
\end{equation} Then (cf. \cite{H}, p. 113)
\begin{equation} \label{HELBV}  T_{ir, \phi_{ir}}(db) =  \sum_{n \in \Z} a_{n} b^n |db|.
\end{equation}
A second way is that, at least when $\Re (ir) > 0$, the boundary
value is given by the limit (\cite{H}, Theorem 4.27)
$$\lim_{d(0, z) \to \infty} e^{(\frac{1}{2} + ir) d(0, z)} \phi_{ir}(z) =
c(ir) T_{ir, \phi_{ir}}, $$ where $c$ is the Harish-Chandra
$c$-function and $d(0, z)$ is the hyperbolic distance.

In particular, eigenfunctions on the quotient $\H^2 \backslash
\Gamma$ are generalized  Poisson transforms of boundary values. We
fix an orthonormal basis $\{\phi_{ir_j}\}$ of eigenfunctions on
$\H^2 \backslash \Gamma$ and denote their boundary values by
$T_{ir_j}$. As observed in \cite{Z2}, when $\phi_{ir_j}$ is a
$\Gamma$-invariant eigenfunction, the boundary values
$T_{ir_j}(db)$ have the following invariance property:
\begin{equation}\label{CONFORMAL} \begin{array}{ll} \phi_{ir_j}(\gamma z) = \phi_{ir_j} (z) & \implies
e^{(\frac12+ir_j)\langle \gamma z, \gamma b \rangle} T_{ir_j}(d
\gamma b) =
e^{(\frac12+ir_j)\langle z, b \rangle} T_{ir_j} (d b)\\ &  \\
& \implies T_{ir_j}( d\gamma b) = e^{- (\frac12+ir_j) \langle
\gamma \cdot 0, \gamma \cdot b \rangle} T_{ir_j} (d b) \end{array}
\end{equation}
This follows from the  uniqueness of the Helgason representation
 and by the identity $\langle \gamma z, \gamma b \rangle = \langle z, b \rangle + \langle \gamma 0, \gamma b \rangle$.
 Conversely, any distribution with such a  $\Gamma$-invariance
 property defines an eigenfunction on $\H^2 \backslash
\Gamma$. It follows that the study of eigenfunctions in this
setting is equivalent to the study of such boundary values. It is
proved in \cite{GO} that the boundary values are derivatives of
$C^{\frac{1}{2}}$ functions in the co-compact case. Similar
results are proved in \cite{MS} with  graphs of the boundary
values. The boundary values will be discussed again in \S \ref{PS}
in relation to quantum chaos and in \S \ref{ANALYTIC} in relation
to analytic continuation.

\subsection{\label{DISC} The Euclidean unit disc $D$} Although our emphasis is on
manifolds without boundary, we review the standard orthonormal
basis of eigenfunctions for the disc.

The standard   orthonormal basis  of real valued Neumann
eigefunctions is given in polar coordinates by $\phi_{m, n}(r,
\theta) = C_{m, n} \sin m \theta J_m(j'_{m, n} r), $ (resp. $
C_{m, n} \cos m \theta J_m(j'_{m, n} r)$) where $j'_{m,n }$ is the
$n$th critical point of the Bessel function $J_m$ and where
$C_{m,n}$ is the normalizing constant. The $\Delta$-eigenvalue is
$\lambda_{m, n}^2 = (j'_{m, n})^2$. The parameter $m$ is referred
to as the angular momentum. Dirichlet eigenfunctions have a
similar form with $j'_{m,n}$ replaced by the $n$th zero $j_{m, n}$
of $J_m$. Nodal loops  correspond to zeros of the radial factor
while open nodal lines correspond to zeros of the angular factor.

If we fix $m$ and let $\lambda_{m, n} \to \infty$ we obtain a
sequence of eigenfunctions of bounded angular momentum but high
energy. At the opposite extreme are the whispering gallery modes
which concentrate along the boundary. These are eigenfunctions of
maximal angular momentum (with given energy), and $\lambda_m \sim
m$. As discussed in \cite{BB}, they are asymptotically given by
the real and imaginary parts of  $e^{i \lambda_m s}
Ai_p(\rho^{-1/3} \lambda_m^{2/3} y)$. Here, $Ai_p(y):= Ai(- t_p +
y)$ where  $Ai$ is  the Airy function and $\{- t_p\}$ are its
negative zeros. Also, $s$ is arc-length along $\partial D$, $\rho$
is a normalizing constant and   $y = 1 - r$.

\subsection{\label{ELLIPSE} An ellipse}

Ellipses have several special sequences of eigenfunctions. One is
a sequence of eigenfunctions concentrating in a Gaussian fashion
along the minor axis. Such eigenfunctions are known as Gaussian
beams, or as bouncing ball modes, and quite some effort has gone
into their generalizations to more general domains and to
manifolds without boundary.  Hence, we briefly review the
existence of exact eigenfunction of this kind. We will discuss the
construction of approximate eigenfunctions of such Gaussian beams
in  \S \ref{QMGEO}.

One can separate variables in special coordinates on the ellipse
and obtain exact, if rather esoteric, formulae for eigenfunctions
of the Dirichlet or Neumann problem. We express an ellipse in the
form $x^2 + \frac{y^2}{1 - a^2} = 1, \;\;\; 0 \leq a < 1, $ with
foci  at $(x, y) = (\pm a, 0)$. We define elliptical coordinates
$(\phi,\rho)$ by $(x, y) = (a \cos \phi \cosh \rho, a \sin \phi
\sinh \rho). $ Here, $0 \leq \rho \leq \rho_{\max} = \cosh^{-1}
a^{-1}, \;\; 0 \leq \phi \leq 2 \pi. $ The lines $\rho = const$
are confocal ellipses and the lines $\phi = const$ are confocal
hyperbolae. The foci occur at $\phi = 0, \pi$ while the origin
occurs at $\rho = 0, \phi = \frac{\pi}{2}. $ The eigenvalue
problem  separates into a pair of Mathieu equations,
\begin{equation} \label{mathieu}
\left\{ \begin{array}{l} \partial_{\phi}^2 G_{m,n} - c^2 \cos^2
\phi G_{m,n}= -
\lambda_{m,n}^{2} G_{m,n} \\ \\
\partial_{\rho}^2 F_{m,n} - c^2 \cosh^2 \rho F_{m,n} =
\lambda_{m,n}^{2} F_{m,n} \end{array} \right.
\end{equation}
where $c$ is a certain parameter. The eigenfunctions have the form
$\Psi_{m, n}(\phi, \rho) = C_{m,n} F_{m,n}(\rho) \cdot
G_{m,n}(\phi)$ where, $F_{m,n}(\rho) = \,Ce_m(\rho, \frac{k_n
c}{2})$ and $ G_{m,n}(\phi) = ce_m(\phi, \frac{k_n c}{2})$ (and
their sin analogues). Here,   $ce_m, Ce_m$ are special Mathieu
functions (cf. \cite{C} (3.10)-(3.2)).  The Neumann or Dirichlet
boundary conditions determine the eigenvalue parameters $k_n c$.

There exists a  special sequence of eigenfunctions which are like  Gaussian
beams along the minor axis. For this special sequence,  $G =
G_{m,n}$ are asymptotic to ground state Hermite functions. More
precisely,
\begin{equation} \label{asymptotics1}
G_{m,n}(\phi;\lambda_{m,n}) = c_{m,n}(\lambda_{m,n})e^{-
\lambda_{m,n} \cos^{2} \phi }  ( 1 + {\mathcal
O}(\lambda_{m,n}^{-1})),
\end{equation}
while
\begin{equation} \label{asymptotics2}
F_{m,n}(\rho;\lambda_{m,n}) = e^{i \lambda_{m,n} \int_{0}^{\rho}
\sqrt{ \cosh^{2} x + 1} dx} a_{+}(\rho;\lambda_{m,n}) + e^{-i
\lambda_{m,n} \int_{0}^{\rho} \sqrt{ \cosh^{2} x + 1} dx}
a_{-}(\rho;\lambda_{m,n})
\end{equation}
where $a_{\pm}(\rho;\lambda_{m,n}) \sim \sum_{j=0}^{\infty}
a_{\pm,j}(\rho) \lambda_{m,n}^{-j}$ are determined by the
Dirichlet or Neumann boundary conditions. Moreover, from the
$L^{2}$-normalization condition  $\int_{I}
|\Psi_{m,n}(\rho,\frac{\pi}{2})|^{2} d\rho = 1$  it follows that
$c_{m,n}(\lambda_{m,n}) \sim \lambda_{m,n}^{1/4}.$

 From (\ref{asymptotics1}) and (\ref{asymptotics2}), the Gaussian beams  are roughly  asymptotic to superpositions of $e^{\pm i
k s} e^{- \lambda_{m,n} y^2}$ (cf. \cite{BB}),  where  $s$ denotes
arc-length along the bouncing ball orbit and $y$ denotes the Fermi
normal coordinate. It follows that outside a tube of any given
radius $\epsilon > 0$, the Gaussian beam decays on the order
$O(e^{- \lambda_{m,n} \epsilon^2})$.

Before leaving this example, we should point out two further
interesting sequences of eigenfunctions. As in the case of the
disc, there exists a `whispering gallery' sequence which
concentrates on the boundary of the domain. As with Gaussian
beams, there exists a generalization of this sequence to any
convex smooth domain in the form of `quasi-modes'. A second
interesting sequence becomes highly enhanced at the two foci.

\section{Local behavior of eigenfunctions}

By the local behavior of eigenfunctions, we often mean methods and
results which pertain to all solutions of $\Delta u =  \lambda^2
u$ on a ball $B(x_0, R)$, not just to solutions which extend to
global eigenfunctions on $(M, g)$. More generally, we consider results
which are obtained from a small ball analysis and which use   covering arguments  to draw
global conclusions from local arguments. But we emphasize that
 some aspects of eigenfunctions discussed in  this section are truly global, for instance the estimate
of the frequency function in terms of the eigenvalue. Thus, we do not aim  to  segregate local from global results in this
section, but do try to indicate when a result assumes that the eigenfunction is global.

Local properties include:
\begin{itemize}

\item Frequency function estimates; we also include Carleman estimates,
although they often require integration of global eigenfunctions over all of $M$;

\item Vanishing order estimates at points;

\item Doubling estimates;

\item Bernstein estimates;

\item Lower bounds on masses in small balls.

\item Local structure of nodal sets.

\end{itemize}

For background and references on general elliptic estimates, see
\cite{GiTr,HL}.

\subsection{\label{HARMCONE} Eigenfuntions and  harmonic functions on a cone}

One can easily convert eigenfunctions $\phi_{\lambda}$  to harmonic functions in a
space of one higher dimension.   Two closely related options are:

\begin{itemize}

\item Form the cone $\R_+ \times M $ and consider the metric
$\hat{g} = dr^2 + r^2 g$. Let $\hat{\phi_{\lambda}} = r^{\alpha} u$ where
$$\alpha = \frac{1}{2} \left( \sqrt{4 \lambda + (n-1)^2} - (n-1)
\right). $$ Let $\hat{\Delta}$ be the Laplacian on the cone. Then,
$$\hat{\Delta} \hat{\phi_{\lambda}} = 0. $$

\item Form $\R_+ \times M$ and consider $e^{\lambda t} \phi_{\lambda} $.
Then $(\partial_t^2 + \Delta) (e^{- \lambda  t} \phi_{\lambda} ) = 0. $

\end{itemize}

This approach was  first  taken in \cite{GaL,Lin} and used further
in \cite{Ku,CM}, among other places. It is a useful approach when
the frequency function is employed, since the latter has its best
properties for harmonic functions. In \cite{CM} it allowed for the
use of Harnack inequalities on balls of $\R_+ \times M$ where
$\phi_{\lambda}(x) $ is positive.

\subsection{\label{FREQ} Frequency function}

The frequency function $N(a, r)$ of a function $u$  is a local
measure of its `degree'   as a polynomial like function in $B_r(a)$. More precisely, it controls the
local growth rate of $u$.
In the case of harmonic functions, it is given by
\begin{equation} \label{FF} N(a, r) = \frac{r D(a,r)}{H(a,r)},
\end{equation} where
$$\;\; H(a,r) = \int_{\partial
B_r(a)} u^2 d \sigma, \;\; D(a,r) = \int_{B_r(a)} |\nabla u|^2 dx.
$$
A well-written detailed treatment of the frequency function and
its applications can be found in \cite{H,Ku}, following the
original treatments in \cite{GaL,GaL2,Lin}.

To motivate the frequency function, let us calculate it in the
special case of a global harmonic function $u$ on $\R^n.$
 Then $u$  may be  decomposed  into a
sum of  homogeneous components of non-negative integral order:
$$u = \sum_{N = 0}^{\infty} u_N, \;\;\; u_N (r \omega) = a_N  r^N
\phi_N(\omega). $$ Then  $\phi_N$ is a spherical harmonic of
degree $N$ on $S^{n-1}$, and we choose $a_N$ so that its  $L^2$
norm is  equal to one. From the fact that the $\phi_N$ are
orthogonal on $S^{n-1}$, one easily calculates that
\begin{equation} \label{ALM2} N(r) = \frac{\sum_{N = 0}^{\infty} N |a_N|^2 r^{2N}
}{\sum_{N = 0}^{\infty} |a_N|^2 r^{2N} }= \frac{d}{2 \; d \log r}
\log \sum_{N = 0}^{\infty} |a_N|^2 r^{2N},
\end{equation}
i.e. \begin{equation} \label{ALM} \frac{d}{dr} \left( \log
\frac{H(r)}{r^{n-1}} \right) = 2 \frac{D(r)}{H(r)}. \end{equation}
 Here, one also uses that $\int_{B_r} |\nabla u|^2 =
\int_{\partial B_r} u \partial_{\nu} u $, where $\partial_{\nu}$
is the unit normal. From the formula it is immediate that if $u$
is homogeneous of degree $N$, i.e. has only one component $a_N r^N
\phi_N$, then $N(r)$ is the constant function $N(r) \equiv N$
equal to its degree.

We note that  $ N(r)$ is analogous to the average energy of a statistical ensemble with
  partition function $Z(r) =  \sum_{N
= 0}^{\infty} |a_N|^2 r^{N}$. We recall a partition function for the
canonical ensemble at temperature $\beta$ is given by
$Z(\beta) = \int e^{- \beta E} d \omega(E)$ where $\omega(E)$
is the density of states. The average energy is then $\langle E \rangle
= - \frac{d}{d\beta} \log Z(\beta) $.   Monotonicity of the average energy
 follows from the  fact   its
derivative $\langle (E - \langle E \rangle)^2 \rangle
= - \frac{d^2}{d\beta^2} \log Z(\beta)$  is the
variance of the energy.

Frequency functions may also be defined for eigenfunctions. At
least two variations have been studied: (i) where the
eigenfunctions are converted  into harmonic
functions on the cone $\R^+ \times M$ as in \S \ref{HARMCONE};
(ii) where a frequency function adapted to eigenfunctions on $M$
is defined.

We first consider method (i) in the case of an  eigenfunction
$\phi_{\lambda}$ on $S^{n}$. The associated harmonic function on
the cone is precisely the homogeneous harmonic  polynomial on
$\R^{n+1}$ which restricted to $S^n$  gives $\phi_{\lambda}$. By
the previous calculation, $N(0, r) \equiv N$ where $\lambda = N(N
+ n - 1)$. We note that on the cone $\R^{n + 1}$, the ball of
radius $r$ has the form $[0, r] \times S^n$, i.e the frequency
function is global on $S^{n}$. On a general manifold, the
analogous global calculation is cleanest if we define $N(r)$  with
$B_r(0)$ everywhere replaced by $[0, r] \times M$. If we
`harmonize' an eigenfunction $\phi_{\lambda} \to r^{\alpha}
\phi_{\lambda}$ as in \S \ref{HARMCONE},  we obtain $N(r) \equiv \alpha$.

The second method is to define a frequency function on balls of
$M$ itself.
  The generalization to eigenfunctions
(\ref{EP}) is as follows (see
 \cite{GaL,GaL2,Ku}). Fix a point $a \in M$ and choose geodesic normal
coordinates centered at $a$ so that $a = 0$. Put
$$\mu(x) = \frac{g_{ij} x_i x_j}{|x|^2}, $$
 and put
\begin{equation}\label{HAR}  D(a, r) :=   \int_{B_r} \left( g^{ij} \frac{\partial
\phi_{\lambda}}{\partial x_i} \frac{\partial \phi_{\lambda}}{\partial x_j} + \lambda^2 \phi_{\lambda}^2
\right) dV, \;\;\;\;\mbox{resp.}\;\; H(a, r) : = \int_{\partial
B_r} \mu \phi_{\lambda}^2,
\end{equation}
By the divergence theorem, one has
\begin{equation} D(a, r) = \int_{\partial B_r} \phi_{\lambda} \frac{\partial
\phi_{\lambda}}{\partial \nu}. \end{equation}  Define the frequency function of $\phi_{\lambda}$  by
\begin{equation} \label{GENFREQ} N(a, r): = \frac{r D (r)}{H(r)}.
\end{equation} As in the case of harmonic
functions, the  main properties of the frequency function of an
eigenfunction are a certain monotonicity in $r$ in small balls of radius $O(\frac{1}{\lambda})$  (see Theorem
\ref{MONO})  and the fact that $N(a, r)$ is commensurate with
$N(b, r)$ when $a$ and $b$ are close.

Simple examples show that, despite its name,  the frequency function measures local growth but not frequency of
oscillations of eigenfunctions, and therefore is not necessarily comparable to $\lambda$.
 For instance, the frequency function of $\sin n x$ in a ball or radius $\leq \frac{C}{|n|}$ is
bounded. An example considered in \cite{DF} are the global eigenfunctions $e^{s x} \sin t y$ on $\R^2$ of
$\Delta$-eigenvalue  $s^2 - t^2$ and  frequency function  of size $s$. One could let $s, t \to \infty$
with $s^2 - t^2$ bounded and obtain a high frequency function but a low eigenvalue.  However, as discussed in \cite{L} (p. 291),
if one forms the harmonic function from a global eigenfunction as in \S \ref{HARMCONE}, then one has $N(0, 2) \leq C \lambda$
where $C$ depends only on the metric. This is a global estimate since a ball centered at $0$ in the cone will cover all of $M$.
An application is given  in Theorem \ref{DOUBLE}.

Let us work out the frequency function for a global eigenfunction
on $\R^n$, parallel to the discussion above for harmonic
functions. We use the notation of \S \ref{EUCLIDEAN}. Since the
tempered solution is unique up to constant multiples, we have
$\ecal_{\lambda}(N) \simeq \hcal_N(S^{n-1}), $ as with harmonic
functions. In the notation \S \ref{EUCLIDEAN}  of may therefore
write
$$\phi_{\lambda}(r \omega) = \sum_{N = 0}^{\infty} a_N J_{N, n}(\lambda r)
\phi_N(\omega), $$ where as before $||\phi_N||_{L^2((S^{n-1})} =
1$. Then,
$$r D(0, r) = \lambda r^n \sum_{N = 0}^{\infty} |a_N |^2  J_{N, n}' J_{N,n} (\lambda r)
, \;\; H(0, r) = r^{n -1}  \sum_{N = 0}^{\infty} |a_N |^2 ( J_{N,
n})^2(\lambda r),$$ and $$\begin{array}{lll} N(r, 0) = \lambda \frac{\sum_{N =
0}^{\infty} |a_N |^2 r J_{N, n}' J_{N, n} (\lambda r)}{\sum_{N =
0}^{\infty} |a_N |^2 ( J_{N, n})^2(\lambda r)} & = & \frac{ r}{2} \frac{d}{dr} \log  \sum_{N =
0}^{\infty} |a_N |^2 ( J_{N, n})^2(\lambda r). \end{array}$$
Thus, the frequency function is again analogous to an average energy, but
with respect to a partition function $ Z(r) = \sum_{N =
0}^{\infty} |a_N |^2 ( J_{N, n})^2(\lambda r)$ which no longer has the canonical form (i.e. it
involves Bessel functions rather than powers). Unlike $r^n$,  Bessel functions oscillate on the
scale $\frac{1}{\lambda}$  and
hence the frequency function is generally not
monotone non-decreasing in  $r$;  but it is monotone for $r$ in a small
interval $[0, r_0(\lambda)]$  with $r_0(\lambda) = \frac{a}{\lambda}$ for some $a > 0$. The following
weaker monotonicity result is however good enough for applications
to vanishing order and doubling estimates.

\begin{theo} \label{MONO}   Theorem 2.3 of \cite{GaL} (see also \cite{GaL2,Lin,H} and
 \cite{Ku} (Th. 2.3, 2.4);.  There exists $C > 0$ such that
$e^{C r} (N(r) + \lambda^2 + 1)$ is a non-decreasing function of
$r$ in some interval $[0, r_0(\lambda)]$.
\end{theo}

Another basic fact is that the frequency of $\phi_{\lambda}$ in $B_r(a)$ is
comparable to its frequency in $B_R(b)$ if $a, b$ are close and
$r, R$ are close. More precisely, there exists $N_0(R) << 1$ such
that if $N(0, 1) \leq N_0(R)$, then $\phi_{\lambda}$ does not vanish in $B_R$,
while if $N(0, 1) \geq N_0(R)$, then
\begin{equation} \label{NCOMP} N(p, \frac{1}{2}(1 - R)) \leq C \; N(0, 1), \;\;
\forall p \in B_R.
\end{equation}

\subsection{Doubling estimate, vanishing order estimate and lower bound estimate}

Doubling  estimates and vanishing order estimates give
quantitative versions of unique continuation theorems.  Given a
partial differential operator $L$ and a solution $u$ of $L u = 0$,
the unique continuation problems asks whether  $u$ is uniquely
determined in a ball  $B$ by its values in a smaller set $E
\subset B$? That is, if  $u \equiv 0$ in $E$, must $u \equiv
0 $ in $B$? In the limiting case where  $K = \{x_0\}$ is  a point,
if $u$ vanishes to high enough order $k$ at $x_0$ must $u \equiv
0$? A related question is the `proximity to zero' of $u$ in the sense of Nevanlinna theory,
i.e. how small can $\sup_{x \in B} |u(x)|$ be in a ball $B$ (see \S \ref{SCL})?
 These questions are answered by frequency function estimates
and by Carleman estimates.

An early doubling inequalities was proved by Bernstein
 for polynomials of one variable:
\begin{equation} \label{REALBERN} \max_{- R \leq x \leq R} |p_N(x)| \leq R^N \; \max_{-1 \leq x \leq 1} |p_N(x)| \end{equation}
for any polynomial of degree $N$.  A generalization known as Remez's inequality allows one to compare
the growth on an interval  to growth on any measurable set in the interior.
Let $\sigma \in \R_+$ and  denote by  $\pcal_N(\sigma)$  to be those polynomials $p_N$ of degree $N$
such that $||p_N||_{E} \leq 1$ for some subset $E \subset [-1, 1 + \sigma]$. Then
$$\sup_{p_N \in \pcal_N(\sigma)} ||p||_{[-1, 1 + \sigma]} \leq ||T_N||_{[-1, 1 + \sigma]}. $$
Here, $||f||_E = \sup_E |f(x)|$ and $T_N$ is Tchebychev's polynomial.

In the case of harmonic functions on $\R^n$ one may set
$$\bar{H}(a, r) = \frac{H(a,r)}{r^{n-1}}. $$
The monotonicity of $N(a, r)$ immediately implies the doubling
formula
\begin{equation} \bar{H}(a, 2 R) = \bar{H}(a, R) \exp \left(
\int_R^{2 R} \frac{2 N(a, r)}{r} d r \right) \leq 4^{N(a, 1 - a)}
\bar{H}(a, R). \end{equation} Integrating in $R$ and using
(\ref{NCOMP}) gives
\begin{equation} \frac{1}{Vol(B_{\frac{3}{4} }(a))} \int_{B_{\frac{3}{4}} (a)} u^2 dx
\leq C(n) 4^{2 N(0, 1)} \frac{1}{Vol B_{\frac{1}{2}} (a)}
\int_{B_{\frac{1}{2}} (a)} u^2.
\end{equation}

As mentioned in \S \ref{FREQ}, the frequency function of a global eigenfunction
may be estimated in terms of the eigenvalue.

\begin{theo} \label{DOUBLE}\cite{DF,Lin} and \cite{H} (Lemma 6.1.1) Let $\phi_{\lambda}$ be a global
eigenfunction of  a $C^{\infty}$ $(M, g)$
there exists $C = C(M, g)$ and $r_0$  such that for $0 < r < r_0$,
$$\frac{1}{Vol(B_{2r}(a))} \int_{B_{2r}(a)} |\phi_{\lambda}|^2
dV_g \leq e^{C \lambda} \frac{1}{Vol(B_{r}(a))} \int_{B_{r}(a)}
|\phi_{\lambda}|^2 dV_g. $$

Further,
\begin{equation} \max_{B(p, r)} |\phi_{\lambda}(x)| \leq
\left(\frac{r}{r'} \right)^{C \lambda} \max_{x \in B(p, r')}
|\phi_{\lambda}(x)|, \;\; (0 < r' < r). \end{equation}

\end{theo}

The doubling estimates imply  the vanishing order estimates. Let
$a \in M$ and suppose that $u(a) = 0$. By the vanishing order
$\nu(u, a)$ of $u$ at $a$ is meant the largest positive integer
such that $D^{\alpha} u(a) = 0$ for all $|\alpha| \leq \nu$. The vanishing order
of an eigenfunction at each zero is
of course finite since eigenfunctions cannot vanish to infinite
order without being identically zero. The following estimate is a quantitative
version of this fact.

\begin{theo} \label{VO} (see \cite{DF}; \cite{Lin} Proposition 1.2 and Corollary 1.4; and \cite{H} Theorem 2.1.8.)
 Suppose that $M$ is compact and of dimension $n$. Then there exist constants $C(n), C_2(n)$ depending only on the dimension such that
the  the vanishing order $\nu(u, a)$ of $u$ at $a \in M$ satisfies
$\nu(u, a) \leq C(n) \; N(0, 1) + C_2(n)$ for all $a \in
B_{1/4}(0)$. In the case of a global  eigenfunction, $\nu(\phi_{\lambda},
a) \leq C(M, g) \lambda.$
\end{theo}

In the case of harmonic functions, one may write $u = P_{\nu} +
\psi_{\nu}$ where $P_{\nu}$ is a homogeneous harmonic polynomial
of degree $\nu$ and where $\psi_{\nu}$ vanishes to order $\nu + 1$ at
$a$. We note that highest weight spherical harmonics $C_n (x_1 + i
x_2)^N$ on $S^2$ are examples which vanish at the maximal order of
vanishing at the poles $x_1 = x_2 = 0, x_3 = \pm 1$.

\subsection{\label{SCL} Semi-classical Lacunas}

 For the purposes of this article, we
define a `semi-classical lacuna' to be an open subset $U \subset M$ for which there exist
 a sequence
$\{\phi_{\lambda_{j_k}}\}$ of $L^2$-normalized eigenfunctions $(M,
g)$ and constants  $C,
a
> 0$ so that
$$\int_U |\phi_{\lambda_{j_k}}|^2 dV_g \leq C e^{- a
\lambda_{j_k}}. $$ The sup-norm could be used in place of the $L^2$ norm. Another descriptive term is
`exponential trough'.

 Lacunae often arise as `classically forbidden regions' of quantum
 mechanical systems, and in fact we do not know of any other examples.  Consider for instance a semi-classical Schr\"odinger operators $h^2 \Delta +
V$ for which the  classical energy level $\xi^2 + V(x) = E$
projects over a compact subset $K_E = \{x: V(x) \leq E\} \subset
M$. Then eigenfunctions of $h^2 \Delta + V$ with eigenvalues
$E_j(h) \in [E - O(h), E + O(h)]$ decay exponentially outside
$K_E$ at a rate given by $O(e^{- \frac{1}{h} d(x, K_E)})$ where
$d(x, K_E)$ is the distance from $x$ to $K_E$ in the Agmon metric,
i.e. the metric $(E - V(x)) dx^2$.

Classically forbidden regions also occur on compact Riemannian manifolds with integrable
geodesic flow.  For instance,
on the round sphere,  we consider sequences $\{Y^N_m\}$  of joint
eigenfunctions for which $m/N \to E$. As will be discussed below,
this sequence concentrates on the Lagrangian torus $\Lambda_E
\subset S^*S^2$ for which $\langle \frac{\xi}{|\xi|},
\frac{\partial}{\partial \theta} \rangle = E. $  This torus
projects to an $S^1$ invariant  annulus $K_E$ on $S^2$. This annulus is
the `classically allowed region' and its complement is the
classically forbidden region. It is not hard to show that
$|Y^N_m(x)| \leq e^{- N d_A(x, K_E)}$ in this example, where $d_A$
is a suitable Agmon distance. We will discuss this and related
examples in more detail in \S \ref{QCI}.

To our knowledge, there are no converse results characterizing
lacunae in terms of classically allowed or forbidden regions. For
instance, can lacunae occur if the geodesic flow of $(M, g)$ is
ergodic? Theorem \ref{QE} shows that lacunae cannot occur in the full density sequence of
`ergodic eigenfunctions', but might occur in a possible sparse subsequence.

A more refined notion is that of  microlocal lacunae, i.e. open
subsets $U \subset S^*_g M$ such that $\langle \chi(x, D)
\phi_{\lambda_{j_k}}, \phi_{\lambda_{j_k}} \rangle \leq C_R
\lambda_{j_k}^{- R}$ for all $R > 0$ and all $\chi(x, D) \in
\Psi^0(M)$ which are microsupported in $U$. In the real analytic
case, one can insist on exponential decay, but in the general
$C^{\infty}$ case one could always add a smoothing operator to
$\chi(x, D)$ to ruin exponential decay.

One could also define lacunae depending on $\lambda$. Namely, one could term a sublevel set of
the form $\{x: |\phi_{\lambda}(x)| \leq C e^{- A \lambda}\}$ a lacunae if its
 volume  $Vol\{x: |\phi_{\lambda}(x)| \leq C e^{- A \lambda}\}$  is bounded below by some constant $\epsilon > 0$
  for some $C, A$.  Eigenfunctions with a semi-classical lacuna clearly have this property, but
   there could exist other examples. It would be interesting to know if
 ergodic eigenfunctions  can have lacunae of this form, or even if their  sublevel sets can be  larger than a $\frac{1}{\lambda}$ tube
around the nodal set.

The doubling estimates and Carleman estimates  give quantitative lower bounds on the exponential
decay rate of eigenfunctions in balls as the eigenvalue tends to infinity and show that the rate
is never faster than in the definition above of semi-classical lacunae:

\begin{cor}\cite{DF} \label{DFED}   Suppose that $M$ is compact and that $\phi_{\lambda}$ is a
global eigenfunction,  $\Delta \phi_{\lambda} = \lambda^2 \phi_{\lambda}$.
 Then $$ \max_{x \in B(p, h)} |\phi_{\lambda}(x)| \geq C'
e^{- C \lambda}.$$ \end{cor}

As an illustration, Gaussian beams such as highest weight
spherical harmonics decay at a rate $e^{- C \lambda d(x, \gamma)}$
away from a stable elliptic orbit $\gamma$. Hence if the closure
of an open set is disjoint from $\gamma$, one has a uniform
exponential decay rate which saturate the lower bounds. To our knowledge, it is unknown whether
semi-classical lacunae can occur in more general situations than `classically forbidden regions'.
It even seems to be unknown whether semi-classical lacunae can occur on  $(M, g)$ with classically chaotic
(i.e. highly ergodic) geodesic flows.

\subsection{Three ball inequalities and propagation of smallness}

These inequalities generalize  Hadamard's three circle theorem and  Nevanlinna's two constants theorem.
We only briefly mention the results   and refer to \cite{M,Ku} for discussion of the
 work of Nadirashvili, Korevaar-Meyers and others
and for further background.

Let $\Omega$ be a domain in $\bold R^n$ $(n\geq 2)$, let
$\Omega_0\subset\Omega$ be a non-empty subdomain, and let
$E\subset \Omega$ be a non-empty compact set. Define
$\|u\|_A=\sup_{x\in A}|u(x)|$. Then
there exists a constant $\alpha=\alpha(E,\Omega_0,\Omega)\in(0,1]$
such that $\|u\|_E\leq
\|u\|^\alpha_{\Omega_0}\|u\|^{1-\alpha}_\Omega$ for all
complex-valued harmonic functions $u$ on $\Omega$. If
$\|u\|_{\Omega} \leq C$ and if $\|u\|_{\Omega_0} << 1$, the
smallness of $u$ on the subdomain $\Omega_0$ propagates to any
compact subset $E \subset \Omega$.

In particular, one has a three spheres inequality (see (\ref{HAR}) for notation)

\begin{theo} (see \cite{Ku})  Let $0 < r_1 < r_2 < r_3$. Then,
$$H(r_2) \leq C_1 \; \left(\frac{r_1}{r_2} \right)^{C_2 \lambda}
\; H(r_1)^{\frac{\alpha_0}{\alpha_0 + \beta_0}} \;
H(r_3)^{\frac{\beta_0}{\alpha_0 + \beta_0}}, $$ where
$$\alpha_0 = \log \frac{r_3}{r_2}, \;\; \beta_0 = C \log
\frac{r_2}{r_1}. $$
\end{theo}

\subsection{Bernstein inequalities}

There are a variety of types of inequalities known as Bernstein inequalities. The original inequalities
were proved for polynomials of one variable;
  a survey is given in \cite{RY}. Further
results in this direction (for analytic functions) are in \cite{Bru1,Bru2}. We will discuss these further
in \S \ref{ANALYTIC}.

Among the classic Bernstein inequalities are the gradient estimates on polynomials of degree $N$:

\begin{enumerate}

\item $|p_N'(x)| \leq N (1 - x^2)^{-1/2} ||p_N||_{[-1,1]} $ for $x \in [-1,1]$ (the Bernstein-Markov inequality);
it  implies that $\int_{-1}^1 |p_N'(x)| dx \leq \pi   N  ||p_N||_{[-1,1].} $

\item  $||p_N'(x)||_{[-1,1]} \leq N^2  ||p_N||_{[-1,1]} $ for $x \in [-1,1]$  (Markov's inequality);

\item $|p_N'(x)| \leq N (1 - x^2)^{-1/2} \left( ||p_N||_{[-1,1]}^2 - p_N(x)^2 \right) $ for $x \in [-1,1]$.

\end{enumerate}

The following result of Donnelly-Fefferman generalizes Bernstein's gradient inequality to eigenfunctions:

\begin{theo} \cite{DF3} Local eigenfunctions of a Riemannian manifold satisfy:

\begin{enumerate}

\item $L^{2}$ Bernstein estimate:
\begin{equation} \left( \int_{B(p,r)} |\nabla \phi_{\lambda}|^2 dV
\right)^{1/2} \leq \frac{C \lambda}{r} \left( \int_{B(p,r)} |
\phi_{\lambda}|^2 dV \right)^{1/2}. \end{equation}

\item $L^{\infty}$ Bernstein estimate: There exists $K > 0$ so that
\begin{equation}  \max_{x \in B(p, r)} |\nabla \phi_{\lambda}(x)|
\leq \frac{C \lambda^K}{r}  \max_{x \in B(p, r)}
|\phi_{\lambda}(x)|. \end{equation}

\item Dong's improved bound:
$$\max_{B_r(p)}|\nabla \phi_{\lambda}|\leq{C_1\sqrt{\lambda}\over r}\max_{B_r(p)}|\phi_{\lambda}| $$ for $ r\leq
C_2\lambda^{-1/4}.$

\end{enumerate}

\end{theo}

\subsection{Carleman inequalities}

Carleman inequalities are weighted  integral inequalities which
are an alternative to frequency function estimates in giving
quantitative unique continuation results. We only indicate some results
in this section and refer to the original articles \cite{Ar,DF,Ta} and more expository
articles \cite{Ta2,I,JL,EZ} for further discussion. As discussed in
\cite{JL}, the  idea of Carleman estimates is to use weights
$\psi$ which are largest on the set from which one wants
uniqueness (or other features) to propagate. On the other hand,
there is a constraint on $\psi$ in order that the Carleman
estimate be true: it needs to be convex in a suitable sense.

We first follow \cite{Ta2,I}. Let
$$A(x, \partial) = \sum_{j, k} a_{jk}(x) \partial_j \partial_k + \sum_j a_j
\partial_j + a(x). $$
One searches for   weights $\phi = e^{\lambda \psi}$ so that the
weighted $L^2$ Carleman estimate holds: \begin{equation}
\label{CARLE} \int_{\Omega} e^{2 \tau \phi} \left(\tau^3 |u|^2 +
\tau |u|^2 \right) dx \leq C_1 \left( \int_{\Omega} e^{2 \tau
\phi} |A u|^2 + \int_{\partial \Omega} e^{2 \tau \phi} \left(
\tau^3 |u|^2 + \tau |\nabla u|^2 \right) dx\right), \end{equation}
for all $\tau > C_1$ and all $u \in H_{(2)}(\Omega)$.

\begin{theo} \cite{I} If $\psi$ is pseudo-convex with respect to $A$ on
$\overline{\Omega}$, then there exist constants $C_1(\lambda),
C_2$ so that for $\lambda > C_2$ and $\tau > C_1(\lambda)$, the
Carleman estimate (\ref{CARLE}) holds. \end{theo}

Here, a function $\psi$ is called pseudo-convex with respect to
$A$ on $\overline{\Omega}$ if
$$A(x, \xi) = 0; \; \sum_j \frac{\partial A}{\partial \xi_j} (x,
\xi) \partial_j \psi(x) = 0, \;\; \mbox{for}\;\; x \in
\overline{\Omega}, \xi \in \R^n \backslash \{0\}, $$ imply that
$$\sum \partial_j \partial_k \psi \frac{\partial A}{\partial \xi_j}
\frac{\partial A}{\partial \xi_k} + \sum \left( ( \partial_k
\frac{\partial A}{\partial \xi_j}) \frac{\partial A}{\partial
\xi_k} - \partial_j A \frac{\partial^2 A}{\partial \xi_j \partial
\xi_k} \right) \partial_j \psi > 0$$ and if
$$A(x, \nabla \psi(x)) \not= 0, \;\; \forall x \in
\overline{\Omega}.$$

\begin{enumerate}

\item If $A$ is elliptic then any $\psi \in
C^2(\overline{\Omega})$ with $\nabla \psi \not= 0$ in
$\overline{\Omega}$ is pseudo-convex with respect to $A$ on
$\overline{\Omega}$.

\end{enumerate}

We are particularly interested in the  Helmholtz operator $A =
\Delta - \lambda^2$, and wish to  estimate the dependence of the
constants on $\lambda$. In this case one can put $\phi(x) = r(x,
b)^2$.

\begin{theo} \cite{I}, Theorem 3.1: Let $\phi(x) = |x - b|^2$.
Then
$$\int_{\Omega} e^{2 \tau \phi} \left(\tau^3 |u|^2 + \tau |u|^2
\right) dx \leq C \left( \int_{\Omega} e^{2 \tau \phi} |(\Delta -
k^2) u|^2 dx + \int_{\partial \Omega} e^{2 \tau \phi} \left(\tau^3
|u|^2 + \tau |\nabla u|^2 dS \right)  \right)$$ for $\tau > C$ and
all real $k$ and $u \in H_{(2)}(\Omega)$. \end{theo}

There is a microlocal interpretation of the convexity condition.
Let $P(x, D)$ be a pseudo-differential operator, and  let $h \in
[0, h_0)$ be a small parameter, and put
\begin{equation} P_{\phi} = e^{\frac{\phi}{h}} P e^{-
\frac{\phi}{h}}. \end{equation} Then,
$$P_{\phi} = P(x, D + \frac{i}{h} \nabla \phi). $$
We denote the principal symbol of $P_{\phi}$ by $p_{\phi}$. The
Carleman weight condition (in the  H\"ormander formulation) is
\begin{equation} \label{HC} p_{\phi}(x, \xi) = 0 \implies \{\Re
p_{\phi}, \Im p_{\phi} \}(x, \xi) > 0, \end{equation} and in the
boundary case,
\begin{equation} \partial_{\nu} \phi (x) \not= 0. \end{equation}
$$||P_{\phi} f||^2 = ||Op(\Re p_{\phi} )  f||^2 + ||Op(\Im p_{\phi}) f||^2 + (f, [Op(\Re p_{\phi} ), Op(\Im p_{\phi} )] f).$$
 One  has, $$[Op(\Re p_{\phi} ), Op(\Im p_{\phi} ) ] >>   0\;\;\; \mbox{
modulo terms depending on}\;\; Op(\Re p_{\phi} ), Op(\Im p_{\phi}
),
$$ and this is the condition $\{\Re
p_{\phi}, \Im p_{\phi} \}(x, \xi) > 0$.

In \cite{DF,DF2} (see also \cite{JL}), Donnelly-Fefferman use
singular weight functions to prove vanishing order estimates and
doubling estimates (Theorem \ref{DFED}). Their Carleman inequalities involve a function $\bar{r}$
which is comparable to the distance $r$. Let $u \in C_c(B(p, h))$ and let $r(x) = d(x, p)$.
Put $\bar{r}(x) = \int_0^{r(x)} e^{- \nu s^2} ds = r(x) + O( r(x)^3))$ for some constant $\nu >> 0$.

\begin{theo} \cite{Ar} \cite{DF3} (Lemma A) There exist constants $C, B, \delta$ such that,
 if $\alpha > C (1 + \lambda)$ and if $u$ vanishes in a ball $B(p, \delta) \subset B(p, h)$,
then

\begin{equation} \int_{B(p, h)}  \bar{r}^{- 2 \alpha} |(\Delta -
\lambda^2) u|^2 dV \geq B \alpha^4  \int_{B(p, \delta + \frac{c
\delta}{\alpha})}  \bar{r}^{-2 \alpha - 4 } u^2  dV.
\end{equation}
\end{theo}

\subsection{\label{KROGER} Geometric comparision inequalities}

We now state a theorem due to P. Kr\"oger and Bacry-Qian which compares eigenfunctions of $(M, g)$
to eigenfunctions of model operators $L_{R, m}$ on an interval $(a, b)$.  We follow the exposition in \cite{BQ}.  The
model operators  are essentially the radial
parts of the Laplacian in geodesic polar coordinates on spaces of constant curvature . To motivate them, we note that the radial part equals

\begin{itemize}

 \item $\frac{d^2}{dr^2} + \frac{m - 1}{r} \frac{d}{dr}$ on $[0, \infty]$ on  $\R^m$;
\item $\frac{d^2}{dr^2} + (m - 1) \cot r \frac{d}{dr}$ on $[0, \pi]$ on the standard sphere $S^m$. A better comparison
interval is to shift the pole to the center of the interval to obtain
$\frac{d^2}{dr^2} - (m - 1) \tan r \frac{d}{dr}$ on $[- \frac{\pi}{2} ,\frac{ \pi}{2}]$;

\item $\frac{d^2}{dr^2} + (m - 1) \coth r \frac{d}{dr}$ on $[0, \infty]$ on hyperbolic space $\hcal^m$.

\end{itemize}

We put
 $$I_1 \cup I_2 \cup I_3 = (0, \infty) \cup (- \infty, \infty) \cup  (- \infty, 0) . $$

 The model operators  are defined as follows:
\begin{enumerate}

\item If $R > 0, m > 1$,
$$L_{r, m}(v) (x) = v''(x) - \sqrt{R(m-1)} \tan \left(\sqrt{\frac{R}{m-1}} x \right) v'(x)$$

on $[ - \frac{\pi}{2} \sqrt{\frac{R}{m-1}}, \frac{\pi}{2} \sqrt{\frac{R}{m-1}} ]. $

\item If $R < 0, m > 1$ then
$$\begin{array}{ll} L_{r, m}(v) (x) = v''(x) + \sqrt{- R(m-1)} \coth \left(\sqrt{\frac{- R}{m-1}} x \right) v'(x), \; \;&  \mbox{on} \; I_1 \cup I_3\\ & \\ L_{r, m}(v) (x) = v''(x) + \sqrt{- R(m-1)} \tanh \left(\sqrt{\frac{- R}{m-1}} x \right) v'(x), \; \;&  \mbox{on} \; I_2,
\end{array}
$$

\item If $R = 0, m > 1$ then,
$$\begin{array}{ll}L_{0, m} v (x)= v''(x) + \frac{m-1}{x} & \; \mbox{on}\; I_1 \cup I_3 \\ & \\
L_{0, m} v (x)= v''(x)  & \; \mbox{on}\; I_2, \end{array}$$

\end{enumerate}

\begin{theo} Let $(M, g)$ be a compact Riemannian manifold of dimension $m$ with Ricci curvature bounded below
by $R$. Let $\phi_{\lambda}$ be an eigenfunction of eigenvalue $\lambda^2$, let $(a,b) \subset \R$ be a finite
interval,  and let $v_{\lambda}$ be a solution
of $$L_{R, m} v  = - \lambda^2 v, \;\; \mbox{on}\; (a,b)$$ satisfying:

\begin{itemize}

\item $ v'(a) = v'(b) = 0. $

\item $v' \not= 0$ on $(a,b)$.

\item $[\min \phi_{\lambda}, \max \phi_{\lambda} ] \subset [\min v_{\lambda}, \max v_{\lambda}]. $
\end{itemize}

Then $$|\nabla (v_{\lambda}^{-1} \circ \phi_{\lambda}) | \leq 1. $$
\end{theo}

In other words, consider an interval $(a, b)$ so that  $\lambda^2$ is the lowest eigenvalue of the Neumann problem
for $L_{R, m}$ on $(a,b)$ and let $v_{\lambda}$ be the corresponding ground state eigenfunction. Then if the range
of $\phi_{\lambda}$ is contained in the range of $v_{\lambda}$ then at any points $x \in M$ and $y \in (a, b)$
such that $\phi_{\lambda}(x) = v_{\lambda}(y)$, $|\nabla \phi_{\lambda}(x)| \leq |\nabla v_{\lambda}(y)|.$

\subsection{Symmetry of positive and negative sets}

Let $M$ be a smooth compact manifold and let $\phi_{\lambda}$ be a
 real nonconstant eigenfunction of the Laplacian on $M$. Let $\phi_{\lambda}^+$, resp.  $\phi_{\lambda}^-$
  denote the positive and the negative part of $\phi_{\lambda}$. Since $\int_M \phi_{\lambda} dV_g = 0$, it
  is obvious that $\int_M \phi^+_{\lambda} dV_g = \int_M \phi_{\lambda}^- dV_g$. This represents a symmetry
  in $L^1$ between the positive and negative parts.  On the other hand, eigenfunctions are not necessarily
  symmetric in this sense for higher $L^p$ norms.
  The article \cite{JN2} of Jakobson-Nadriashvili discusses the extent of such symmetries and proves:

\begin{theo}\cite{JN2}   Then for  $p\ge1$ there exists a positive constant $C$, depending only on $p$ and $M$ such that $1/C\le\|\phi_{\lambda}^+\|_{L^p}/\|\phi_{\lambda}^-\|_{L^p}\le C$.  \end{theo}

\subsection{Alexandroff-Bakelman-Pucci-Cabr\'e inequality}

The following is Cabr\'e's improvement of the Alexandroff-Bakelman-Pucci estimate. In the following, $||u ||$ denotes the $L^2$ norm in $\Omega$.

\begin{theo} (see \cite{Cab}, Theorem 1.4)
There exists a constant $C = C(M, g)$ independent of $\Omega$ so that, for  any subdomain $\Omega \subset M$,
$$||u||_{L^{\infty(}\Omega)} \leq \limsup_{x \to \partial \Omega} |u| + C   |\Omega|^{\frac{1}{2}}  ||\Delta u||. $$
Hence for any smooth function vanishing on $\partial \Omega$, $||u||_{L^{\infty(}\Omega)} \leq  C   |\Omega|^{\frac{1}{2}}  ||\Delta u||$.
\end{theo}

\begin{cor} If $\phi_{\lambda}$ satisfies
$\Delta u_{\lambda} =  \lambda^2 u$ and  $\phi_{\lambda} |_{\partial \Omega} = 0$, then
$||\phi_{\lambda}||_{L^{\infty(}\Omega)} \leq   C \lambda^2   |\Omega|^{\frac{1}{2}}  || \phi_{\lambda}||. $
\end{cor}

An interesting application of this inequality
to extrema of eigenfunctions on  nodal domains is given in  \cite{PS}.

\subsection{Bers scaling near zeros}

By Theorem \ref{VO},  we know that at each $x_0 \in
\ncal_{\phi_{\lambda}}$, the order of vanishing $k$ of
$\phi_{\lambda}$ is finite.

\begin{prop} \cite{Bers,HW2} Assume that $\phi_{\lambda}$ vanishes to order $k$ at
$x_0$. Let $\phi_{\lambda}(x) = \phi_k^{x_0} (x) + \phi^{x_0}_{k +
1} + \cdots$ denote the $C^{\infty}$ Taylor expansion of
$\phi_{\lambda}$ into homogeneous terms in normal coordinates $x$
centered at $x_0$.  Then $\phi_k^{x_0}(x)$ is a Euclidean harmonic
homogeneous polynomial of degree $k$.
\end{prop}

To prove this, one substitutes  the homogeneous expansion into the
equation  $\Delta \phi_{\lambda} = \lambda^2 \phi_{\lambda}$ and
rescales $x \to \lambda x$, i.e. one  applies the dilation operator
\begin{equation}D_{\lambda}^{x_0} \phi_{\lambda}(u ) = \phi(x_0 + \frac{u}{\lambda}). \end{equation}
The rescaled eigenfunction is an eigenfunction of the locally
rescaled Laplacian $$ \Delta^{x_0}_{\lambda} : = \lambda^{-2}  D_{\lambda}^{x_0}
\Delta_g (D_{\lambda}^{x_0} )^{-1} = \sum_{j = 1}^n
\frac{\partial^2}{\partial u_j^2} + \cdots $$  in Riemannian
normal coordinates $u$ at $x_0$ but now with eigenvalue $1$,
\begin{equation}\begin{array}{l}  D_{\lambda}^{x_0}  \Delta_g (D_{\lambda}^{x_0} )^{-1} \phi(x_0 + \frac{u}{\lambda}) =
\lambda^2 \phi(x_0 + \frac{u}{\lambda}) \\ \\
\implies \Delta^{x_0}_{\lambda} \phi(x_0 + \frac{u}{\lambda}) =
\phi(x_0 + \frac{u}{\lambda}).  \end{array} \end{equation}
Since
$\phi(x_0 + \frac{u}{\lambda})$ is, modulo lower order terms, an
eigenfunction of a standard flat Laplacian on $\R^n$, it behaves near a zero as
a sum of homogeneous Euclidean harmonic polynomials.

 The Bers scaling is used by S.Y. Cheng (see also earlier results of Hartman-Wintner  \cite{HW,Ch1,Ch2})
to prove that at a singular point of $\phi_{\lambda}$ in dimension two, the nodal line
branches in $k$ curves at $x_0$ with equal angles between the curves.  For further
applications, see \cite{Bes}. Unfortunately, no analogue of the Bers  scaling
seems to exist  at critical points

\subsection{Heuristic scaling at non-zero points} One may ask what becomes of the scaling
of eigenfunctions on
 the length scale  of `Planck's constant' $h = \lambda_j^{-1}$ around  a point where the
 eigenfunction does not vanish. Since
$\phi(x_0 + \frac{u}{\lambda})$ is, modulo lower order terms, an
eigenfunction of a standard flat Laplacian on $\R^n$, it is
reasonable to think  that asymptotically  $\phi_{\lambda}$ behaves
like  a sum of plane waves, i.e. that there exists a distribution $dT_{\lambda}$ on the unit sphere in
momentum space such that
\begin{equation} \label{SCALING} \phi_{\lambda}(x_0 + \frac{u}{\lambda}) \sim \int_{|\xi| = 1} e^{i \langle \xi, u \rangle} d T_{\lambda}^{x_0} (\xi).
\end{equation}
 This is not a  rigorous definition of $dT_{\lambda}$ since $\sim$ has  not been precisely defined.

However, examples suggest that there do exist rigorous plane wave approximations.
For instance, if we rescale the Helgason representation (\ref{HELGAFOR}) for an eigenfunction on a hyperbolic quotient, we obtain
$$\phi_{\lambda}(z_0 + \frac{u}{\lambda}) = \int_B e^{(i \lambda + 1) \langle z_0 + \frac{u}{\lambda}, b \rangle} dT_{\lambda}. $$
We have  $\langle z_0 + \frac{u}{\lambda}, b \rangle = \langle z_0, b \rangle  + \frac{u}{\lambda} \cdot b + O(\lambda^{-1})$
where $u \cdot b $ denotes the Euclidean inner product of $u$ with the element of $S^1$ represented by $b \in B = S^1$. Hence,
 $e^{(i \lambda + 1) \langle z_0, b \rangle} dT_{\lambda}$ is a plausible candidate for the  distribution. In this heuristic discussion,
  we neglect the remainder estimate which
 requires an estimate of the $\lambda$ dependence of $dT_{\lambda}$.

\section{\label{SMOOTHNODAL} Nodal sets on $C^{\infty}$ Riemannian manifolds}

In this section we review results on the nodal, critical and
singular sets of eigenfunctions in the smooth case. In \S \ref{ANALYTIC}  we
review  the much stronger results in the real analytic case.

The nodal set of an eigenfunction $\phi_{\lambda}$  is the zero set
\begin{equation} Z_{\phi_{\lambda}} = \{x \in M: \phi_{\lambda}(x) = 0\}. \end{equation}
The global structure of the nodal set is determined by integrals $\int_{Z_{\phi_{\lambda}}} f d\hcal^{n-1}$
of continuous functions (or characteristic functions of nice sets) over $Z_{\phi_{\lambda}}$. This seems very
difficult, so we study first the local structure of the set, e.g. its Hausdorf dimension and local Hausdorf measure.
In \S \ref{ANALYTIC} we present some global results on the nodal set when the geodesic flow is ergodic.

The following theorem,  due to Br\"uning after an observation of   R. Courant
  (see  \cite{Br})), is used to obtain lower bounds
on volumes of nodal sets:

\begin{theo}\label{COURANT}  If $(M^n, g)$ is a $C^{\infty}$ compact Riemannian
manifold without boundary, then there exists $C(M, g) > 0$ so
that, in each ball of radius $\geq \frac{C}{\lambda}$ there exists
a point where $\phi_{\lambda}$ vanishes. \end{theo}

\begin{proof} Fix  $x_0, r$ and consider $B(x_0, r)$. If
$\phi_{\lambda}$ has no zeros in $B(x_0,r)$, then $B(x_0, r) \subset D_{j; \lambda}$
must be contained in the interior of a nodal domain $ D_{j; \lambda}$ of $\phi_{\lambda}$.
Now $\lambda^2 = \lambda_1^2 (D_{j; \lambda})$ where $ \lambda_1^2 (D_{j; \lambda})$ is the smallest
Dirichlet eigenvalue for the nodal domain. By domain monotonicity of the lowest Dirichlet eigenvalue (i.e.
$\lambda_1(\Omega)$ decreases as $\Omega$ increases),
$\lambda^2 \leq \lambda_1^2 (D_{j; \lambda}) \leq \lambda_1^2(B(x_0, r)).$ To complete the proof
we show that  $\lambda_1^2(B(x_0, r)) \leq \frac{C}{r^2}$ where $C$ depends only on the  metric. This is proved
by comparing $\lambda_1^2(B(x_0, r)) $   for the metric
$g$ with the lowest Dirichlet Eigenvalue  $\lambda_1^2(B(x_0, c r ); g_0)$ for the Euclidean ball $B(x_0, c r; g_0)$
 centered at $x_0$
of radius $ c r$ with Euclidean metric $g_0 $ equal to $g$ with coefficients  frozen at $x_0$; $c$ is chosen
so that $ B(x_0, c r; g_0) \subset B(x_0, r, g)$. Again by domain monotonicity, $\lambda_1^2 (B(x_0, r, g))
\leq \lambda_1^2 ( B(x_0, c r; g)) $ for $c < 1$. By comparing
 Rayleigh quotients $\frac{\int_{\Omega} |df|^2 dV_g}{\int_{\Omega} f^2 dV_g}$ one easily sees that
 $\lambda_1^2 ( B(x_0, c r; g)) \leq C \lambda_1^2 ( B(x_0, c r; g_0))$ for some $C$ depending only on the metric.
 But by explicit calculation with Bessel functions, $\lambda_1^2 ( B(x_0, c r; g_0)) \leq \frac{C}{r^2}. $
 Thus, $\lambda^2 \leq \frac{C}{r^2}$.

 \end{proof}

For background we refer to \cite{Ch}. A nice variation on the proof is given in \cite{CM}, where eigenfunctions
are converted to harmonic functions as \S \ref{HARMCONE}, and Harnack's inequality on positive harmonic functions
 is used to prove Theorem \ref{COURANT}. Another use of positivity is given in \cite{H} to prove
   $\lambda_1(B(x_0, r)) \geq \lambda$:  Let $u_{r}$ denote the ground
state Dirichlet eigenfunction for $B(x_0, r)$.   Then  $u_{r} > 0$ on the interior of $B(x_0, r)$. If $B(x_0, r) \subset D_{j; \lambda}$
then also $\phi_{\lambda} > 0 $ in $ B(x_0, r)$.
Hence the ratio  $\frac{u_{r}}{\phi_{\lambda}}$
is smooth and non-negative, vanishes only on $\partial B(x_0,r)$, and must have its maximum at a point
$y$ in the interior of $B(x_0, r)$. At this point (recalling that our $\Delta$ is minus the sum of squares),
$$\nabla \left(\frac{u_{r}}{\phi_{\lambda}} \right) (y) = 0, \;\; - \Delta \left(\frac{u_{r}}{\phi_{\lambda}} \right) (y) \leq 0, $$
so at $y$,
$$0 \geq - \Delta \left( \frac{u_{r}}{\phi_{\lambda}} \right) = - \frac{\phi_{\lambda} \Delta u_{r} - u_r \Delta \phi_{\lambda}}{\phi_{\lambda}^2}
= -  \frac{( \lambda_1^2(B(x_0, r)) - \lambda^2)\phi_{\lambda}  u_{r}}{\phi_{\lambda}^2}. $$
Since $\frac{ \phi_{\lambda}  u_{r}}{\phi_{\lambda}^2} >0,$ this is possible  only if
$\lambda_1(B(x_0, r)) \geq \lambda$.

We recall that the nodal set of an eigenfunction $\phi_{\lambda}$  is its zero set. When zero is a regular value
of $\phi_{\lambda}$ the nodal set is a smooth hypersurface. This is a generic property of eigenfunctions \cite{U}.
It is pointed out in \cite{Bae} that eigenfunctions can always be locally represented  in the form
$$\phi_{\lambda} (x) = v(x) \left(x_1^k + \sum_{j = 0}^{k - 1} x_1^j u_j(x') \right),$$
in suitable coordinates $(x_1, x')$ near $p$,  where $\phi_{\lambda}$ vanishes to order $k$ at $p$, where
$u_j(x')$ vanishes to order $k - j$ at $x' = 0$,  and where
$v(x) \not= 0$ in a ball around $p$. It follows that the nodal set is always countably $n-1$ rectifiable when
$\dim M = n$.

For a general $C^{\infty} (M, g)$ of dimesion $n$,  S.T. Yau \cite{Y1,Y2} has conjectured that
\begin{equation} \label{YAUCON} c_{M,g}  \lambda \leq \hcal^{n-1}(Z_{\phi_{\lambda}}) \leq
C_{M, g}  \lambda. \end{equation}  Here, $\hcal^k$ is the $k$ dimensional Haussdorf measure. The conjecture was proved in \cite{DF} in the real analytic case, which will
be discussed in detail in \S \ref{ANALYTIC}; see Theorem \ref{DFNODAL}.  In the $C^{\infty}$ case, the lower bound was proved in dimension $2$
by J. Br\"uning \cite{Br} and S.T. Yau. Because of this result, and lower bounds in some other calculable cases,  Yau conjectured
the same lower bound in all dimensions. It remains an open problem in dimensions $\geq 3$.

Regarding upper bounds, in dimension two one has
\begin{theo} \cite{DF4,Dong} Suppose $(M, g)$ is $C^{\infty}$ and that  $\dim M = 2$. Then,
$$\hcal^1(Z_{\phi_{\lambda}}) \leq C_{M, g} \lambda^{3/2}. $$

\end{theo}

Dong's proof was based on the following integral formula:

\begin{theo} \cite{Dong} Let $q = |\nabla u|^2 + \frac{\lambda^2}{n}
u^2$. Then
$$\hcal^{n-1}(Z_u) = \frac{1}{2} \int_M \frac{(\Delta +
\lambda^2) |u|}{q} dV. $$
\end{theo}

In higher dimensions, the  best estimate to date is the following:
\begin{theo}  \cite{HS} For any $C^{\infty}$ $(M, g)$ of dimension $n$  one has  $$\hcal^{n-1}(Z_{\phi_{\lambda}}) \leq C_{M, g}
e^{c_{M,g} \lambda \log \lambda}. $$
\end{theo}

\subsection{Courant and Pleijel bounds on nodal domains}

Another well-known result is Courant's nodal domain theorem:

\begin{theo} \label{COURANTS} The number $n_k$ of nodal domains
 of the $k$ satisfies $n_k \leq k$.  \end{theo} Here, a nodal
domain is a component of $M \backslash Z_{\phi_{\lambda_k}}$. This estimate is not sharp.  Pleijel \cite{P} proved:
\begin{theo}\label{PLEIJEL}
For any plane domain with Dirichlet boundary conditions, $\limsup_{k \to \infty} \frac{n_k}{k} \leq \frac{4}{j_1^2} \simeq 0. 691...$,
where $j_1$ is the first zero of the $J_0$ Bessel function. \end{theo} He conjectured that the same result should be true for a free
membrane, i.e. for Neumann boundary conditions. This was recently proved in the real analytic case
 by I. Polterovich \cite{Po}, using a result \cite{TZ3}
counting the number of nodal lines which touch the boundary. Another recent result is the proof that there exists
an asymptotic mean number of nodal domains of random spherical harmonics \cite{NS}; see \S \ref{NS}.

\subsection{Critical and singular sets of eigenfunctions on
$C^{\infty}$ Riemannian manifolds}

Critical points are points where $\nabla \phi_{\lambda}(x) = 0$.
We denote the critical set by $ \Sigma = \nabla
\phi_{\lambda}^{-1}(0) $.  Singular points are critical nodal points, i.e. critical points
lying on the nodal hypersurface, $\Sigma_0 = \Sigma \cap \ncal$. The reader is warned that the
term `critical point' is sometimes used for singular points.
In general, the singular set is simpler to study than the critical
set.

There exist simple examples such as surfaces of revolution where  the critical point sets of
eigenfunctions are   of codimension one. For instance, the rotationally invariant eigenfunctions on a surface
of revolution have $S^1$ invariant critical point sets, and not all consist of fixed points. Thus, zonal spherical
harmonics have critical point sets consisting of unions of lattitude circles.

For generic metrics, all of the eigenfunctions are Morse functions and consequently their critical point
sets are discrete \cite{U}. However, there are no known bounds on the number of critical points. In the analogy
of eigenfunctions of eigenvalue $\lambda^2$ to polynomials of degree $\lambda$, a very naive application
of  B\'ezout's theorem   suggests  that
the number of critical points should be bounded above by $C \lambda^n $, since the critical point
equation is a system of   $n$ equations
of degree $\lambda - 1$ in $n$ unknowns.  Even for real analytic metrics, no rigorous results in this direction
are known. Any such bounds would have to reflect the non-degeneracy of the critical points; a simple count could
be unstable if a small perturbation of the metric had  a sequence of eigenfunctions with   codimension one critical point sets.
On the other hand, no lower bound on the number of critical points exists: it is proved in \cite{JN}
that there exists a Riemannian surface  possessing a sequence of eigenfunctions with a fixed finite number of critical points,
answering a question of S.T. Yau \cite{Y3},
which asks if the number of critical point for the eigenfunction grows
when the eigenvalue grows (Yau believes that the answer is positive for most 
metrics,  and that it is interesting to understand it and  to give some estimate of the growth.)
In \cite{Y3}, Yau  also proved the existence of at least one nontrivial critical point  for surfaces
that can not be guaranteed by standard Morse theory.  Up to now,
this is the only result of proving existence of nontrivial critical points.

Singular sets have been studied for harmonic functions, eigenfunctions, and more general solutions of elliptic equations.
The singular set is somewhat simpler than the critical point set.
In the real analytic case, the singular set is known to be at most
of codimension $2$ and $\hcal^{m-2}(\Sigma) < \infty$ (see
\cite{Fed}, 3.4.8). This was extended to $C^{\infty}$ metrics in \cite{HHL}. One has:

\begin{theo} \cite{HNOO,HHL} Let $\dim M = n$. There exists a constant $C_{N_R}$ depending only on the
frequency function $N_R$ of $\phi$ in $B_R(0)$ such that the  singular  set $\Sigma_0$ satisfies:
$$\hcal^{n-2} (\Sigma_0 \cap B_{R/2}) \leq C_{ N_R}. $$
\end{theo}

A key point in the proof is that the complex critical set of a
homogeneous harmonic polynomial in dimension $n$ is of dimension
$\leq n - 2$.

Lin has conjectured:
\begin{conj} If $\dim M = n, $ the singular  set $\Sigma$ satisfies:
$$\hcal^{n-2} (\Sigma \cap B_{R/2}) \leq C N_R^2. $$
\end{conj}

\begin{theo} \cite{Dong} Suppose that $\dim M = 2$ and let $m(\lambda_j)$ denote the number of
singular points of $\phi_{\lambda_j}$ counted with multiplicity
(i.e. $m(\lambda_j) $ equals the sum over all singular points of
the  order of vanishing minus one). Then
$$m(\lambda_j) \leq \frac{1}{4 \pi} \left( \lambda^2 Vol(M, g) - 2 \int_M \min(K, 0) dV_g \right), $$
where $K$ is the Gaussian curvature.
\end{theo}

\section{\label{WAVEKERNEL} The wave kernel of a compact Riemannian manifold}

Global properties of  eigenfunctions often arise from the fact
that global eigenfunction are eigenfunction of the wave group
$U_t = e^{i t \sqrt{\Delta}}$. We begin by reviewing some basic theory
of the wave equation on a compact Riemannian manifold.

The wave group of a Riemannian manifold is the unitary group $U_t
= e^{i t \sqrt{\Delta}}$ is defined by the spectral theorem,
\begin{equation} U(t, x,y) = \sum_j e^{it \lambda_j} \phi_j(x) \phi_j(y). \end{equation}
Closely related but  simpler wave kernels are   the even part of
the wave kernel, $\cos t\sqrt{\Delta}$ which solves the initial
value problem
\begin{equation} \left\{ \begin{array}{ll} (\frac{\partial}{\partial t}^2 - \Delta) u = 0& \\
u|_{t=0} = f & \frac{\partial}{\partial t} u |_{t=0} = 0
\end{array}\right .\end{equation} Similar, the odd part of the
wave kernel, $\frac{\sin t\sqrt{\Delta}}{\sqrt{\Delta}}$ is the
operator solving
\begin{equation} \left\{ \begin{array}{ll} (\frac{\partial}{\partial t}^2 - \Delta) u = 0& \\
u|_{t=0} = 0 & \frac{\partial}{\partial t} u |_{t=0} = g
\end{array}\right .\end{equation}

To employ wave kernels in spectral geometry it is indispensible to
have approximations defined in terms of geometric data. They were
first constructed by Hadamard and Riesz and are usually referred
to as Hadamard(-Riesz) parametrices. There are alternative
parametrices due to Lax and  H\"ormander.

We begin with the small-time  Hadamard Riesz parametrices for
$\cos t \sqrt{\Delta}, \frac{\sin t \sqrt{\Delta}}{\sqrt{\Delta}}.
$ These kernels only  involve $\Delta$ and their kernels can be
constructed in the form
\begin{equation}   \int_{0}^{\infty} e^{i \theta (r^2-t^2)}
\sum_{j=0}^{\infty} W_j(x,y) \theta_{reg}^{\frac{n-1}{2} - j}
d\theta
 \;\;\;\mbox{mod}\;\;C^{\infty}  \end{equation}
where $W_j$ are the Hadamard-Riesz coefficients determined
inductively by the transport equations
\begin{equation}\begin{array}{l}
 \frac{\Theta'}{2 \Theta} W_0 + \frac{\partial W_0}{\partial r} = 0\\ \\
4 i r(x,y) \{(\frac{k+1}{r(x,y)} +  \frac{\Theta'}{2 \Theta})
W_{k+1} + \frac{\partial W_{k + 1}}{\partial r}\} = \Delta_y W_k.
\end{array}\end{equation}
Here, $r = r(x,y)$ is the geodesic distance and $\theta_{reg}^s$ is a regularization of $\theta^s$
at $\theta = 0$;  $t^{-n}$ is the distribution defined by $t^{-n} =
Re (t + i0)^{-n}$ (see \cite{Be}).  We recall
that $(t+i0)^{-n} =  e^{-i\pi \frac{n}{2}}\frac{1}{\Gamma(n)}
\int_0^{\infty} e^{itx} x^{n-1} dx$ and also that
$\int_0^{\infty}e^{itx} x^{n-1} dx$ has precisely the same
singularity at $t=0$ as the sum $\sum_{k=0}^{\infty}
e^{it(k+\frac{\beta}{4})} (k+\frac{\beta}{4})^{n-1}$.

The solutions are given by:
\begin{equation}\label{HR} \begin{array}{l} W_0(x,y) = \Theta^{-\half}(x,y) \\ \\
W_{j+1}(x,y) =  \Theta^{-\half}(x,y) \int_0^1 s^k \Theta(x,
x_s)^{\half} \Delta_2 W_j(x, x_s) ds
\end{array} \end{equation}
where $x_s$ is the geodesic from $x$ to $y$ parametrized
proportionately to arc-length and where $\Delta_2$ operates in the
second variable.

Performing the integrals, one finds that
\begin{equation} \cos t\sqrt{\Delta}(x,y) \sim
 C_o |t| \sum_{j=0}^{\infty}(-1)^j w_j(x,y)\frac{
(r^2-t^2)_{-}^{j-\frac{d -3}{2} - 2}}{4^j \Gamma(j - \frac{d-3}{2}
- 1)}
 \;\;\;\mbox{mod}\;\;C^{\infty} \end{equation}
where $C_o$ is a universal constant and where $W_j=\tilde{C}_o
e^{-ij\frac{\pi}{2}} 4^{-j}w_j(x,y),$. Similarly
\begin{equation} \frac{\sin t\sqrt{\Delta}}{\sqrt{\Delta}}(x,y)
\sim
 C_o sgn(t) \sum_{j=0}^{\infty}(-1)^j w_j(x,y)\frac{
(r^2-t^2)_{-}^{j-\frac{d - 3}{2} - 1}}{4^j \Gamma(j -
\frac{d-3}{2})}
 \;\;\;\mbox{mod}\;\;C^{\infty} \end{equation}
 Here, $\sim $ means that the difference of the two sides is a
 $C^{\infty}$ function on $M \times M$, or more precisely, that if
 one truncates the sum after a number $n_R$ of terms, the difference lies in
 $C^{R}(M \times M)$. The formulae are  only valid for times $t < inj(M,g)$ and for this
 reason are called small-time parametrices. When the metric is real
analytic, the series for the ampltiude converges for $t$
sufficiently small and $(x, y)$ sufficiently near the diagonal
\cite{Be}.

  To obtain truly global results on eigenfunctions,
 one actually needs large time parametrices. These are very complicated and in many respects are
 not understood.  The simplest way to
 obtain one is to use
the group property of $U(t) = U(t/N)^N$ to  determine the wave
kernel for all times from the wave kernel at a small time.  It
shows that for fixed $(x,t)$ the kernel $U(t)(x,y)$ is singular
along the distance sphere $S_t(x)$ of radius $t$ centered at $x$,
with singularities propagating along geodesics.

By a similar but more complicated calculation (using the action of
$\sqrt{\Delta}$ on the oscillatory integral for $\frac{\sin t
\sqrt{\Delta}}{\sqrt{\Delta}}$, one has
\begin{equation} U(t, x,y) = \int_0^{\infty} e^{ i \theta (r^2(x,y) - t^2)} \sum_{k =
0}^{\infty} W_k (x,y) \theta^{\frac{d-3}{2} - k}
d\theta\;\;\;\;\;\;\;\;\;(t < {\rm inj}(M,g)) \end{equation} where
$U_0(x,y) = \Theta^{-\half}(x,y)$ is the volume 1/2-density, where
the higher coefficients are determined by transport equations, and
where  again $\theta^r$ is regularized at $0$.

An alternative parametrix has  the form
\begin{equation} \label{PARAONE} U(t, x, y) = \int_{T^*_y M} e^{i
t |\xi|_{g_y} } e^{- i \langle \xi, \exp_y^{-1} (x) \rangle} A(t, x,
y, \xi) d\xi
\end{equation} where $|\xi|_{g_x} $ is the metric norm function at
$x$, and where $A(t, x, y, \xi)$ is a polyhomogeneous amplitude of
order $0$. The expression $\exp_x^{-1}(y)$ is again only defined
in a sufficiently small neighborhood of the diagonal.  For
background, we refer to \cite{HoI-IV,D.G}.

The existence of these parametrices is sufficient to prove:

\begin{theo} $U(t, x, y) \in \dcal'(\R \times M \times M)$ is a
a Fourier intetgral operator of order $-\frac{1}{4}$ associated to
the canonical relation,
$$\Gamma = \{(t, \tau, x, \xi, y, \eta): \tau + |\xi|_g = 0, g^t(x,
\xi) = (y, \eta)\} \subset T^*(\R \times M \times M). $$
\end{theo}

This means that $U(t, x, y)$ can  be locally written as a
finite sum of oscillatory integrals $\int_{\R^M} e^{i \phi(t, x,
y, \xi)} a(t, x, y, \xi) d \xi$ whose phases $\phi$ locally
parametrize $\Gamma$ in the following sense: Define the critical
set along the fibers by
$$C_{\phi} = \{(t, x, y, \xi) \in \R \times M \times M \times \R^M : d_{\xi} \phi = 0 \}$$
and define the immersion $i_{\phi}: C_{\phi} \to \Gamma_{\phi}
\subset  T^*(\R \times M \times M)$ by
$$i_{\phi}(t, x, y, \xi) = (t, d_{t} \phi, x, d_x \phi, y, -
d_y \phi). $$ For instance, when $$\phi(t, x, y, \xi) =  - t
|\xi|_{g_y} +  \langle \xi, \exp_y^{-1} (x) \rangle$$ one has
$$C_{\phi} = \{(t, x, y, \xi) : \exp_y ( t
\frac{\xi}{|\xi|_g}) = x\}, $$ so that $x, y$ are linked by a
geodesic segment of length $t$ and
$$\Gamma_{\phi} = \{(t, - |\xi|_{g_y}, x, \xi, g^t(x, \xi) ) \}. $$

\subsubsection{\label{MWOCP} Manifolds without conjugate points}

The Riemannian manifolds with the simplest wave groups are those
without conjugate points. A clear exposition is given in \cite{Be}. We recall that a  Riemannian manifold
$(M, g)$ is without conjugate points if the exponential maps
$\exp_x: T_x M \to M$ have no singular points. In this case, the
universal Riemannian cover $\tilde{M}$, the total space of covering map $\pi: (\tilde{M}, \tilde{g}) \to (M,
g)$, is diffeomorphic to $\R^n$ ($n = \dim M$) and the exponential maps $exp_x :
T_x \tilde{M} \to \tilde{M}$ are diffeomorphisms for all $x$. The
distance function $\tilde{r}$ is globally well defined and real
analytic away from the diagonal $\Delta_{\tilde{M}}$ of $\tilde{M}
\times \tilde{M}$.

 On a manifold  without conjugate
points, the Hadamard and Lax-H\"ormander parametrices for the wave
kernel $\tilde{U}(t, x, y)$ on $\tilde{M} \times \tilde{M}$ are
well-defined for all $t$.
 The wave kernel on the
quotient $M$  can be expressed in the form
\begin{equation} \label{GLOBALHD} U(t, x, y) = \sum_{\gamma \in
\Gamma} \tilde{U}(t, x, \gamma y)
\end{equation}  where $\Gamma$ is the deck transformation group of $\tilde{M}
\to M$ and where we implicitly identify $ x, y$ with one of their
lifts to $\tilde{M}$. The series converges when one takes for
$\tilde{U}$ the $\cos t \sqrt{\Delta} $ or $\frac{\sin t
\sqrt{\Delta}}{\sqrt{\Delta}}$: by finite propagation speed,
there are only a finite number of terms for each $t$ namely terms
where $\tilde{r}(x, \gamma y) \leq t$. The unitary kernel $e^{i t
\sqrt{\Delta}}(t, x, y)$ does not have finite propagation speed,
but outside the light cone the kernel is smooth and the same
finite number of terms of the sum determine the singularity
completely.

\section{\label{METHODS} Methods for global analysis}

The global analysis is based on properties of eigenfunctions which
derive from the fact that they extend globally to $M$ and hence
solve the wave equation eigenvalue problem $e^{i t \sqrt{\Delta}}
\phi_{\lambda} = e^{i \lambda t} \phi_{\lambda}. $ The wave operator is very
useful because it propagates singularities along geodesics.
Studying singularities of various expressions involving
eigenfunctions and eigenvalues gives further information that the
local methods do not provide.

To take full advantage of the wave group it is important to use
pseudo-differential operators, which transform well under
conjugation by the wave group. We briefly review some of the key
properties of  pseudo-differential operators in this section. For
a detailed treatment, we refer to \cite{HoI-IV,EZ,DSj,GSj,T2}.

We denote by $\Psi^m(M)$ the space of pseudo-differential operators of order $m$ on $M$.
The principal symbol of $A$ is denoted $\sigma_A$.

\subsection{Egorov's theorem}

Egorov's theorem for the wave group  concerns the conjugations
\begin{equation} \label{ALPHAT} \alpha_t (A) : = U_t A U_t^*, \;\;\; A \in
\Psi^m(M). \end{equation} Such a conjugation defines the quantum
evolution  of observables  in the Heisenberg picture.  Egorov's
theorem  is the following:

\begin{theo}  $\alpha_t$ defines an order-preserving automorphism of
$\Psi^*(M)$, i.e. $\alpha_t(A) \in \Psi^m(M)$ if $A \in
\Psi^m(M)$, and that
\begin{equation} \sigma_{U_t A U_t^*} (x, \xi)  =  \sigma_A
(\Phi^t(x, \xi)) : = V_t (\sigma_A), \;\;\; (x, \xi) \in T^*M
\backslash 0,
\end{equation}
where $V_t$ is the unitary operator (\ref{VT}).

\end{theo}

\subsection{Sharp Garding inequality}

The Garding inequality addresses the question: To what extent is
the quantization $Op(a)$ of a positive symbol a positive
pseudo-differential operator? Can one give lower bounds for
expressions $\langle Op(a) \phi_j, \phi_j \rangle$ in terms of
$a$?

The answer depends on the precise definition of $Op(a)$. There is
a somewhat complicated quantization $Op^F(a)$ due to Friedrichs in
the case of $\R^m$ such that $Op^F(a) \geq 0$ if $a \geq 0$.  But
for a general definition of $Op(a)$ one can only expect the
following sharp Garding  inequality: \begin{theo} For any $f \in
C^{\infty}(M)$, we have  $$\langle Op(a) f, f \rangle \geq (\inf a
) ||f||^2 - C ||\Delta^{-1/2} f||^2, $$ where $||f|| $ is the
$L^2$-norm of $f$.
\end{theo}

This immediately implies the
\begin{cor} If $a \geq 0$, then
$$\langle Op(a)  \phi_{\lambda},  \phi_{\lambda} \rangle \geq  - C \lambda^{-1}. $$
\end{cor}

\subsection{Operator norm and symbol norm}

Another natural question comparing the properties of the classical
observable $a$ to its quantization $Op(a)$ concerns traces and
norms. For instance,

\begin{theo}

For any $A \in \Psi^0$, $||\sigma_A||_{L^{\infty}} = \inf_K ||A +
K||$ where the infimum is taken over the set of compact operators $K$.

\end{theo}

\subsection{\label{QULIMITS} Quantum Limits (Microlocal defect measures)}

One of the principal problems in the global analysis of
eigenfunctions is to determine the possible weak limits of the
Wigner measures, or microlocal defect measures, which arise from
matrix elements.

To define them, we regard diagonal matrix elements as  linear
functionals on $\Psi^0$
\begin{equation} \label{RHOJ} \rho_k(A) = \langle  A \phi_k, \phi_k
\rangle. \end{equation} We observe that $\rho_k(I) = 1$, that
$\rho_k(A) \geq 0$ if $A = \geq 0$ and that
\begin{equation}\label{INVAR}  \rho_k(U_t A U_t^*) = \rho_k(A). \end{equation}
Indeed, if $A \geq 0$ then $A = B^*B$ for some $B \in \Psi^0$ and
we can move $B^*$ to the right side. Similarly (\ref{INVAR}) is
proved by moving $U_t$ to the right side and using the fact that
the eigenvalues of $U_t$ are of modulus one. In quantum
statistical mechanics, these properties are summarized by saying
that $\rho_j$ is an {\it invariant state} on the algebra $\Psi^0$,
or more precisely, on its  closure  in the operator norm. An
invariant state is the analogue in quantum statistical mechanics
of an invariant probability measure.

We denote by  $\mcal_I$ the convex set of invariant probability
measures for the geodesic flow. Further, we say that a measure is
time-reversal invariant if it is invariant under the
anti-symplectic involution  $(x, \xi) \to (x, - \xi)$ on $T^*M$.
We denote the time-reversal invariant elements of $\mcal_I$ by
$\mcal_I^+$.

\begin{prop}
Any weak limit of the sequence $\{\rho_k\}$ on $\Psi^0$ is a
time-reversal invariant, $g^t$ invariant probability  measure on
$S^*M$, i.e. is an element  of $\mcal_I^+$.
\end{prop}

\begin{proof}

For any compact operator $K$,  $\langle K \phi_j, \phi_j \rangle
\to 0$. Hence,  any limit of $\langle A \phi_k, \phi_k \rangle$ is
equally a limit of  $\langle (A + K) \phi_k, \phi_k \rangle$. By
the norm estimate, the limit  is bounded by $\inf_K ||A + K||$
(the infimum taken over compact operators).   Hence any weak limit
is bounded by a constant times $||\sigma_A||_{L^{\infty}} $ and is
therefore continuous on $C(S^*M)$. It is a positive functional
since each $\rho_j$ is and hence any limit is a probability
measure. By Egorov's theorem and the invariance of the $\rho_k$,
any limit of $\rho_k(A)$ is a limit of $\rho_k(Op (\sigma_A \circ
\Phi^t))$ and hence the limit measure is invariant. It is also
time-reversal since the eigenfunctions are real-valued, i.e.
complex conjugation invariant.
\end{proof}

\begin{prob}\label{Q} Determine the set  ${\mathcal Q}$  of `quantum limits', i.e.
weak* limit points of the sequence $\{\Phi_k\}$ of distributions
on the classical phase space $S^*M$, defined by
$$\int_X a d\Phi_k := \langle Op(a) \phi_{\lambda_k}, \phi_{\lambda_k} \rangle$$
where $a \in C^{\infty}(S^*M)$.\end{prob}

The set $\qcal$ is independent of the definition of $Op$. The
simplest examples are the exponentials on  a flat torus
$\R^m/\Z^m$. By definition of pseudodifferential operator, $A e^{2
\pi i \langle k, x \rangle} = a(x, k) e^{2 \pi i \langle k, x
\rangle}$ where $a(x, k)$ is the complete symbol. Thus, \begin{equation}
\label{MDMTORUS} \langle
A e^{2 \pi i \langle k, x \rangle}, e^{2 \pi i \langle k, x
\rangle} \rangle = \int_{\R^n/\Z^n} a(x, k) dx \sim
\int_{\R^n/\Z^n} \sigma_A(x, \frac{k}{|k|}) dx .
\end{equation}
A subsequence  $e^{2 \pi i \langle k_j, x \rangle}$ of
eigenfunctions has a weak limit if and only if $\frac{k_j}{|k_j|}$
tends to a limit vector $\xi_0$ in the unit sphere in $\R^n$. In
this case, the associated weak* limit is $\int_{\R^n/\Z^n}
\sigma_A(x, \xi_0) dx $, i.e. the delta-function on the invariant
torus  $T_{\xi_0} \subset S^*M$ for $g^t$,  defined by the
constant momentum condition $\xi = \xi_0.$ The eigenfunctions are
said to localize on this invariant torus. Given $\xi_0$, we can
always define a sequence $k_j$ so that $\frac{k_j}{|k_j|} \to
\xi_0,$ and thus, every invariant torus measure arises as a
quantum limit.

 In general, there are many possible limit measures. The most
 important are:

\begin{enumerate}

\item Normalized Liouville measure. In fact, the functional
$\omega$ of (\ref{OMEGA}) is also a state on $\Psi^0$ for the
reason explained above. A subsequence $\{\varphijk\}$ of
eigenfunctions is considered diffuse if $\rho_{j_k} \to \omega$.

\item A periodic orbit measure $\mu_{\gamma}$ defined by
$\mu_{\gamma}(A) = \frac{1}{L_{\gamma}} \int_{\gamma} \sigma_A ds$
where $L_{\gamma}$ is the length of $\gamma$. A sequence of
eigenfunctions for which $\rho_{k_j} \to \mu_{\gamma}$ obviously
concentrates (or strongly `scars') on the closed geodesic.

\item A finite sum of periodic orbit measures.

\item A delta-function along an invariant Lagrangian manifold
$\Lambda \subset S^*M$. The associated eigenfunctions are viewed
as {\it localizing} along $\Lambda$.

\item A more general measure which is singular with respect to
$d\mu$.

\end{enumerate}

All of these possibilities arise as $(M, g)$ varies among
Riemannian manifolds. Indeed, the standard sphere provides an extreme example:

\begin{theo} \cite{JZ} For
the standard round sphere $S^n$, $\qcal = \mcal_I^+$. \end{theo}

An important case is when  $\rho_{k_j} \to \omega$, i.e. the limit
measure is Liouville measure. The corresponding eigenfunctions
become uniformly distributed on the energy surface $S^*_g M$ and
are sometimes called 'diffuse'.  By testing against multplication
operators, one gets
$$\frac{1}{Vol(M)} \int_E |\phi_{k_j}(x)|^2 dVol \to
\frac{Vol(E)}{Vol(M)} $$ for any measurable set $E$ whose boundary
has measure zero. In the interpretation of $|\phi_{k_j}(x)|^2
dVol$ as the probability density of finding a particle of energy
$\lambda_k^2$ at $x$, this says that the sequence of probabilities
tends to uniform measure. However,  $\rho_{k_j} \to \omega$ is
much stronger since it says that the eigenfunctions become
uniformly distributed on  $S^*M$ and not just on the configuration
space $M$. For instance, on   the flat torus $\R^n/\Z^n$,  the
standard exponentials $e^{2 \pi i \langle k, x \rangle}$ satisfy
$|e^{2 \pi i \langle k, x \rangle}|^2 = 1$, and are thus uniformly
distributed in configuration space. On the other hand, as seen
above, in phase space  they localize on invariant Lagrange tori in
$S^*M$.

The flat torus is a model of a completely integrable system, on
both the classical and quantum levels. On the other hand, if the
geodesic flow is ergodic one would expect the eigenfunctions to be
diffuse in phase space. The statement that the all eigenfunctions
are diffuse, i.e.  $\qcal = \{\omega\}$,  is known as {\it quantum
unique ergodicity}. It will be discussed in \S \ref{QuE}.

Off-diagonal matrix elements
\begin{equation} \rho_{jk} (A) = \langle A \phi_{\lambda_j}, \phi_{\lambda_k} \rangle \end{equation} are also important
as  transition amplitudes between states. They  no longer define
states since $\rho_{jk}(I) = 0$, are no longer positive, and are
no longer invariant. Indeed, $\rho_{jk}(U_t A U_t^*) = e^{i t
(\lambda_j - \lambda_k)} \rho_{jk}(A),  $ so they are eigenvectors
of the automorphism $\alpha_t$ of (\ref{ALPHAT}). A sequence of
such matrix elements cannot have a weak limit unless the spectral
gap $\lambda_j - \lambda_k$ tends to a limit $\tau \in \R$. In
this case, by the same discussion as above, any weak limit of the
functionals $\rho_{jk}$ will be a time-reversal invariant
eigenmeasure of the geodesic flow which transforms by $e^{i \tau
t}$ under the action of $g^t$. Examples of such eigenmeasures are
orbital Fourier coefficients $\frac{1}{L_{\gamma}}
\int_0^{L_{\gamma}} e^{- i \tau t} \sigma_A(g^t(x, \xi)) dt$ along
a periodic orbit. Here $\tau \in \frac{2 \pi}{L_{\gamma}} \Z.$ We
denote by $\qcal_{\tau}$ such eigenmeasures of the geodesic flow.
Problem $1$ has the following extension to off-diagonal elements:
\medskip

 \begin{prob} Determine the set
${\mathcal Q}_{\tau} $ of `quantum limits', i.e.  weak* limit
points of the sequence $\{\rho_{kj}\}$  on the classical phase
space $S^*M$. \end{prob}

As will be discussed in \S \ref{QWMS}, the asymptotics of
off-diagonal elements depends on the weak mixing properties of the
geodesic flow and not just its ergodicity.

\section{Singularities pre-trace formulae}

One of the principal methods for relating eigenfunctions and
geodesic flow are the {\it local Weyl laws}. One form of the local
Weyl law is to find the asymptotics and remainder for the
pointwise sums,
$$N(\lambda; x) = \sum_{j: \lambda_j \leq \lambda}
|\phi_{\lambda_j}(x)|^2. $$ Another is to find, for $A \in
\Psi^0$, the asymptotics and remainder for
$$N_A (\lambda) = \sum_{j: \lambda_j \leq \lambda}
\langle A \phi_{\lambda_j}, \phi_{\lambda_j} \rangle.  $$ In both
cases, the asymptotics are determined by a Fourier Tauberian
method, by studying the singularities of the dual trace,
$$S(t,  x) = \sum_{j} e^{i t \lambda_j}
|\phi_{\lambda_j}(x)|^2,$$ resp.
$$S_A(t) = \sum_{j} e^{i t \lambda_j}
\langle A \phi_{\lambda_j}, \phi_{\lambda_j} \rangle. $$

\subsection{Duistermaat-Guillemin short time pre-trace formula}

By a pre-trace formula one means a `geometric' singularities
formula for the wave kernels $$U(t, x, x) = \sum_j e^{i t
\lambda_j} \phi_{\lambda_j}(x)^2 $$ on the diagonal. The actual trace
formula is the integral over $M$ of the pre-trace formula. The
trace formula clearly gives only spectral information, while the
pre-trace formula gives information on eigenfunctions.

\begin{prop} \label{DGNX} \cite{D.G} Let $(M, g)$ be a $C^{\infty}$ compact Riemannian manifold
of dimension $n$.  Then there exists a sequence $\omega_1,
\omega_2, \dots$ of real valued smooth densities on $M$ such that,
for every $\rho \in \scal(\R)$ with supp $\hat{\rho}$ contained in
a sufficiently small neighborhood of $0$ and $\hat{\rho} \equiv 1$
in a small neighborhood of $0$,
$$\sum_j \rho(\lambda - \lambda_j) |\phi_{\lambda_j}(x)|^2 \sim
\sum_{k = 0}^{\infty} \omega_k \lambda^{n - k - 1}$$ as $\lambda
\to \infty$ (and rapidly decaying as $\lambda \to - \infty$) with
$$\omega_0(x) = Vol(S^*_x M), \;\; \omega_1 = 0 = \omega_n;
\omega_k = 0 \;\; \mbox{for odd}\; k. $$
\end{prop}

\begin{proof} One uses a short-time parametrix,
$$U(t, x, y) = (2\pi)^{-n}\int_{\R^{n}}e^{i\varphi(t, x,
y, \eta)} a(t, x, y, \eta) d\eta$$ where $\alpha$ is a classical
symbol of order $m$ and where $\varphi(t, x, y, \eta) = \psi(x, y,
\eta) - t|\eta|$, with $\psi(x, y, \eta) = 0$ if $\langle x-y,
\eta\rangle = 0$. Hence,
$$\begin{array}{lll}
\rho*dN(\lambda ,x)&=&\displaystyle{\int_\R\; e^{i\lambda
t}\hat \rho
(t)\:U(t, x, x)dt}\\[12pt]
&=&(2\pi)^{-n}\displaystyle{\;\int_{\R}\;\int_{\R^{n}}e^{i\lambda
t} \hat \rho (t) e^{-it|\eta|} a(t, x, x, \eta)
\:d\eta\:dt\;\;.}\end{array}$$ One now changes variables $\eta
\to \lambda \eta, $ puts the $d\xi$ integral into polar
coordinates $\xi = r \omega, |\omega| = 1$ and carries out the $dt
dr $ integral by the method of stationary phase.

\end{proof}

The same kind of argument applies to $N_A(\lambda)$:

\begin{prop}\label{DGNA}  For $A\in \Psi^m(M)$, let $N_A(\lambda) = \sum_{\lambda _j\leq \lambda }(A\varphi_j,
\varphi_j)$.  Then for any $\rho \in \scal(\R)$ with $\hat \rho
\in C^\infty_0(\R)$, $\supp\hat\rho \cap \:\mbox{Lsp}\Mg = \{0\}$
and with $\hat \rho\equiv 1$ in some interval around 0, we have:
$$\rho*dN_A(\lambda)\sim\sum^\infty_{k=0} \alpha_k\lambda
^{n+m-k-1}\;\;(\lambda \rightarrow +\infty)$$ where:
$$\begin{array}{l}
n= \dim M, \;\;
\alpha_0 =  \displaystyle{\int_{S^{*}M}\sigma_Ad\mu}, \;\;
\alpha_k =  \displaystyle{\int_{S^{*}M}\omega_kd\mu}
\end{array}$$ where $\omega_k$ is determined from the $k$-jet of
the complete symbol $a$ of $A$. \end{prop}

\noindent {\it Proof.} The only new step is to apply $A$ to the
parametrix for $U_t$. Applying a  $\pdo$ to $a e^{i \phi}$
produces an expression $\alpha e^{i \phi}$ with the same phase and
only a change in the amplitude.  Hence,
$$AU(t, x, y) = (2\pi)^{-n}\int_{\R^{n}}\alpha(t, x, y, \eta)e^{i\varphi(t, x,
y, \eta)}d\eta$$ where $\alpha$ is a classical symbol of order
$m$. Now proceed as before.

\subsection{Long time pre-trace formulae} We now discuss pre-trace
formulae for $U(t, x, x)$ for long times $t$. The main difficulty
is that there is no better  parametrix for $U(t, x, x)$ than to
use the group formula $U(t) = U(t/N)^N$ to reduce to times $<
inj(M, g)$. But this will require $N \times m$ integral signs and
amplitudes that are difficult to control. Moreover, the
singularities become very difficult to control.

One can get a good idea of the difficulties by studying the
simplest case, that  of manifolds without conjugate points. This
case is a rather straightforward generalization of the
 Selberg pre-trace formula
for  compact hyperbolic manifolds.  The key feature  is that there
exists a global in time parametrix on the universal cover,  so it
is not necessary to assume that the suppport of $\hat{\rho}$ is
small. As in the classical trace formula, one organizes  the
elements $\gamma \in \Gamma$ (the deck transformation group)  into
conjugacy classes $\hat{\gamma}$. One then has
$$U(t, x, x) = \sum_{\hat{\gamma}} U_{\hat{\gamma}}(t, x, x), $$ where $$ U_{\hat{\gamma}}(t, x, y) =
\sum_{\alpha \in \Gamma/\Gamma_{\gamma}} \tilde{U}(t, x,
\alpha^{-1} \gamma \alpha x) = \sum_{\alpha \in \Gamma/\Gamma_{\gamma}} \tilde{U}(t, \alpha  x,  \gamma
\alpha x),
$$
since $\alpha$ is an isometry. Here, $\Gamma_{\gamma}$ is the
stabilizer in $\Gamma$ of $\gamma$.

In this case, it is best to use the Hadamard parametrix. Since we
are interested in the long time singularities, we can use the
phase function $r - t$ instead of $r^2 - t^2$. Either one
parameterizes the graph of the geodesic flow away from $r = 0$.
One then has,
$$\begin{array}{lll}
\rho*dN(\lambda ,x)&=& \sum_{\gamma \in \Gamma}
\displaystyle{\int_\R\;e^{i\lambda t}\hat \rho
(t)\:\tilde{U}(t, x, \gamma x)dt}\\[12pt]
&=&(2\pi)^{-n} \sum_{\hat{\gamma}}
 \sum_{\alpha \in
\Gamma/\Gamma_{\gamma}}
 \displaystyle{ \int_{\R}\;\int_{\R_+}e^{i\lambda t} \hat
\rho (t) e^{i \theta (r (\alpha x,  \gamma \alpha x) - t))}  a(
\alpha x, \gamma \alpha x, \theta) \:d\theta \:dt}\\
[12pt] &=&(2\pi)^{-n}  \sum_{\hat{\gamma}}
 \sum_{\alpha \in
\Gamma/\Gamma_{\gamma}}
 \displaystyle{ \int_{\R}\;\int_{\R_+} \hat
\rho (t) e^{i \lambda (t +  \theta (r (\alpha x,  \gamma \alpha x)
- t))} a( \alpha x, \gamma \alpha x, \theta) \:d\theta
\:dt\;\;.}\end{array}$$ To determine the asymptotics as $\lambda
\to \infty$ one again applies stationary phase. The phase function
is critical when $\theta = 1, t = r (\alpha x,  \gamma \alpha x)$.
This corresponds to the times $t$ when there exists a geodesic
loop at $x \in M$ (i.e. at $\alpha x \in \tilde{M}$). The geodesic
loops are in one-one correspondence with conjugacy classes in
$\Gamma$ and hence form a countable set. However, the growth rate
of this set as $t \to \infty$ is often exponentially large. Thus,
if supp $\hat{\rho} \subset [- T, T]$, then there are often $e^{C
T}$ terms in the sum over $\gamma$. This makes it difficult to
control the remainder terms.

\subsection{Safarov trace formula}

Safarov added some precision to the Duistermaat-Guillemin
singularities pre-trace formula. For fixed $x$,  Given $x\in M$, we let $\lcal_x$ denote the set
of loop directions at $x$:
\begin{equation}\label{LCALX}  \lcal_x = \{\xi \in S^*_x M : \exists T: \exp_x T
\xi = x \}.
\end{equation}
 We  let $T_x: S^*_x M \to \R_+ \cup
\{\infty\}$ denote the return time function to $x$,
$$T_x (\xi) = \left\{ \begin{array}{ll} \inf \{ t > 0: \exp_x t \xi
= x\}
, & \; \mbox{if}\; \xi \in \lcal_x; \\ & \\
+ \infty, & \mbox{if no such} \;t \; \mbox{exists}. \end{array}
\right. $$
We then define the  first return map by
$\Phi_x = g^{T_x}_x : \lcal_x \to S^*_x M. $ We also define $T^{(k)}(\xi)$ to be the time
of $k$th return for directions which loop back at least $k$ times.

We then consider the
positive partially unitary operator (the Perron-Frobenius
operator)
$$U_x: L^2(\lcal_x, |d\omega|) \to L^2(S^*_x, |d\omega|),
\;\;\;U_x f(\xi)  = \left\{ \begin{array}{ll} f((\Phi_x)(\xi))
\sqrt{J_x(\xi)}, &
\xi  \in  \lcal_x, \;\;\\ & \\
0, & \xi \notin \lcal_x.
\end{array} \right.$$

Here, $J_x$ is the Jacobian of the map $\Phi_x$, i.e.
$\Phi_x^*|d\xi| = J_x(\xi) |d\xi|$. We have:
$$\ker U_x = \{f \in L^2(S^*_x): \;\supp f \cap  \Phi_x ( \lcal_x) =
\emptyset\} ;\;\; \Im U_x =  \{f \in L^2(S^*_x): \supp f \subset
\lcal_x \}. $$ We  further define
$$U_x^{\pm} (\lambda) = e^{i \lambda T_x^{\pm}} U_x^{\pm}. $$

Let $\rho_T$ be the dilated test function satisfying $\widehat{\rho_T}(\tau) = \hat{\rho}(\frac{\tau}{T})$.
The pre-trace formula then has the form,

\begin{equation} \label{REPLACE2} \rho_T * d N(\lambda, x)  = a_0(x) \lambda^{n-1} + \lambda^{n-1} \int_{\lcal_x} \sum_{k =
1}^{\infty} \hat{\rho} (\frac{T_x^{(k)}(\xi)}{T}) U_x(\lambda)^k
|d\xi| + o_{T, x}(\lambda^{n - 1}). \end{equation}

A key point in the proof is that  the phase of the oscillatory integral for the left side  only has a
stationary phase point at $\xi \in \lcal_x$. That  reduces the integral  to
one over $\lcal_x$ modulo an  error
$o_T(\lambda^{n-1})$. For the details, we refer to    Proposition
4.1.16 of \cite{SV}.

\section{Weyl law and local Weyl law}

The classical  Weyl law asymptotically counts the number of
eigenvalues less than $\lambda,$
\begin{equation}\label{WL} N(\lambda ) = \#\{j:\lambda _j\leq \lambda \}= \frac{|B_n|}{(2\pi)^n} Vol(M, g) \lambda ^n
+O(\lambda ^{n-1}). \end{equation} Here, $|B_n|$ is the Euclidean
volume of the unit ball and $Vol(M, g)$ is the volume of $M$ with
respect to the metric $g$. Equivalently,
\begin{equation} Tr E_{\lambda} = \frac{Vol(|\xi|_g
\leq \lambda)}{(2\pi)^n} +O(\lambda ^{n-1}) , \end{equation} where
$Vol$ is the symplectic volume measure relative to the natural
symplectic form $\sum_{j=1}^n dx_j \wedge d\xi_j$ on $T^*M$. Thus,
the dimension of the space where $H = \sqrt{\Delta} $ is $\leq
\lambda$ is asymptotically the volume where its symbol $|\xi|_g
\leq \lambda$.

The Weyl law with remainder is proved by  using the integrated
version of  Proposition \ref{DGNX} together with the Fourier
Tauberian  estimate $N(\lambda) - \rho *  N(\lambda) =
O(\lambda^{n-1})$, valid when $\rho * dN(\lambda) =
O(\lambda^{n-1})$. The latter is visibly true from Proposition
\ref{DGNX}.  See the Appendix \S \ref{APPENDIX}, Theorem \ref{ET}.

Proposition \ref{DGNX} assumes that $\hat{\rho}$ is supported in a
small interval around $t = 0$. Thus it only takes into account the
generic singularity at $t = 0$. An improved, two-term Weyl law has
been proved which takes into account the singularities of $Tr \cos
t \sqrt{\Delta}$ for larger values of $t$. The singular $t \not=
0$ are the lengths of the closed geodesics $\gamma$ of $g^t$. The
size of the remainder reflects the measure of closed geodesics.

\begin{enumerate}

\item  In the {\it aperiodic} case, Ivrii's two term Weyl law
states
$$N(\lambda ) = \#\{j:\lambda _j\leq \lambda \}=c_m \;
Vol(M, g) \; \lambda^m +o(\lambda ^{m-1})$$
 where $m=\dim M$ and where $c_m$ is a universal constant.

\item  In the {\it periodic} case,
 the spectrum of $\sqrt{\Delta}$ is a union of eigenvalue clusters $C_N$ of the form
$$C_N=\{(\frac{2\pi}{T})(N+\frac{\beta}{4}) +
 \mu_{Ni}, \; i=1\dots d_N\}$$
with $\mu_{Ni} = 0(N^{-1})$.   The number $d_N$ of eigenvalues in
$C_N$ is a polynomial of degree $m-1$.
\end{enumerate}

We refer to \cite{D.G,HoI-IV,SV,Z1} for background and further
discussion.

The  integrated Weyl law is relevent to spectral theory, while the
pointwise local Weyl law is relevant to eigenfunctions. By the
same Fourier Tauberian theorem, one has (\cite{D.G}),

\begin{equation} \label{PLWL} \sum_{\lambda _{j} \leq
\lambda } |\varphi_j(x)|^2 = \frac{1}{(2 \pi)^n} |B^n| \lambda^n +
R(\lambda, x),
\end{equation}
where $R(\lambda, x) = O(\lambda^{n-1})$ uniformly in $x$. In this
case one obtains a $o(\lambda^{n-1})$ remainder if the set of
geodesic loops at $x$ has measure zero. Such refinements will be
discussed in \S \ref{LP}.

The other  {\it local Weyl law} concerns the traces $Tr A
E(\lambda)$ where $A \in \Psi^m(M).$ It asserts that
\begin{equation} \label{LWL} \sum_{\lambda _{j} \leq
\lambda }\langle A\varphi_j, \varphi_j \rangle = \frac{1}{(2
\pi)^n}  \left(\int_{B^*M} \sigma_A dx d\xi \right)  \lambda^n +
O(\lambda^{n-1}).
\end{equation}
When the periodic geodesics form a set of measure zero in $S^*M$,
one could average over the shorter interval $[\lambda, \lambda +
1].$ Combining the Weyl and local Weyl law, we find the surface
average of $\sigma_A$ is a limit of traces:
\begin{equation} \label{OMEGA} \begin{array}{lll} \omega(A)
&: =&\displaystyle{ \frac{1}{\mu (S^{*}M)}\int_{S^*M}\sa d\mu} \\
& & \\
 &=& \displaystyle{ \lim_{\lambda
\rightarrow\infty}\frac{1}{N(\lambda )}\sum_{\lambda _{j} \leq
\lambda }\langle A\varphi_j, \varphi_j \rangle}
\end{array} \end{equation}
Here,  $\mu$ is the {\it Liouville measure} on $S^*M$.

\section{\label{LP} Local and global $L^p$ estimates of eigenfunctions}

One of the applications of pointwise and local Weyl laws is to
obtain bounds on $L^p$ norms of $L^2$-normalized eigenfunctions, when $p$ is
sufficiently large.
 Estimates of  the $L^{\infty}$ norms
can be obtained from  the local Weyl law (\ref{LWL}). We use the notation
$f(\lambda) = O(g(\lambda))$ if there exists a constant $C$ such that
$f(\lambda) \leq C g(\lambda)$, and $f(\lambda) = \Omega (g(\lambda))$ if there exists a constant $C$ such that
$f(\lambda_j) \geq C g(\lambda_j)$ for some sequence $\lambda_j \to \infty$ (in other words, the negation
of $f = o(g(\lambda))$.

\begin{prop} Let $(M, g)$ be a compact $m$-dimensional $C^{\infty}$ Riemannian
manifold. Then $||\varphi_{\lambda_j}||_{L^{\infty} } = O(\lambda_j^{\frac{m-1}{2}})$.

\end{prop}

\begin{proof}

Since the jump in the the left hand side of (\ref{PLWL})  at
$\lambda$ is $\sum_{j: \lambda_j = \lambda}
|\varphi_{\lambda_j}(x)|^2$ and the jump in the right hand side is
the jump of $R(\lambda, x)$, we have
\begin{equation} \label{SUP} \sum_{j: \lambda _{j} =
\lambda } |\varphi_{\lambda_j}(x)|^2 = O(\lambda_j^{m-1}) \implies \;
||\varphi_{\lambda_j}||_{L^{\infty} } = O(\lambda_j^{\frac{m-1}{2}}).
\end{equation}

\end{proof}

We note that this estimate is stronger by the factor $\lambda^{-\frac{1}{2}}$ than the estimate obtained by
Sobolev estimates of eigenfunctions: If $\dim M = m$, then
$$||\phi_{\lambda}||_{L^p} \leq \lambda^{m (\frac{1}{q} - \frac{1}{p})} ||\phi_{\lambda}||_{L^q}. $$
Thus, when $p = \infty$, the Sobolev estimate
gives  $||\phi_{\lambda}||_{L^{\infty}} \leq \lambda^{ \frac{m}{2}} ||\phi_{\lambda}||_{L^2}$.

For general $L^p$-norms, the following  bounds hold on any compact
Riemannian manifold

\begin{theo}  \cite{Sog}:
\begin{equation}
\frac{\|\phi_{\lambda}\|_p}{\|\phi_{\lambda}\|_2}=O(\lambda^{\delta(p)}), \quad 2\le
p\le \infty.
\end{equation}
where
\begin{equation} \delta(p)=
\begin{cases}
n(\tfrac12-\tfrac1p)-\tfrac12, \quad \tfrac{2(n+1)}{n-1}\le p\le
\infty
\\
\tfrac{n-1}2(\tfrac12-\tfrac1p),\quad 2\le p\le
\tfrac{2(n+1)}{n-1}.
\end{cases}
\end{equation}
\end{theo}

These estimates are sharp for some $(M, g)$. For instance,  on the
unit sphere $S^{n} \subset \R^{n+1}$, the zonal spherical
harmonics achieve the maximal $L^{\infty}$ norm. Moreover,  the
highest weight spherical harmonics on $S^2$ saturate the bounds
for small $L^p$ norms: using Gaussian integrals, one easily finds
that $||(x + i y)^k||_{L^2(S^2)} \sim k^{-1/4}. $ For instance,
$$\begin{array}{lll} \int_{\R^3} (x^2 + y^2)^{k} e^{-(x^2 + y^2 +
z^2)} dx dy dz & = & ||(x + i y)^k||_{L^2(S^2)}^2  \int_0^{\infty}
r^{2k} e^{-
r^2} r^2 dr, \\&&  \\
 &
\implies & ||(x + i y)^k||_{L^2(S^2)}^2 = \frac{\Gamma(k +
1)}{\Gamma(k + \frac{3}{2})} \sim k^{-1/2}. \end{array}$$

Thus, the $L^2$ normalized highest weight vector has the form
$k^{1/4} (x + i y)^k$.  It achieves its $L^{\infty}$ norm at
$(1,0, 0)$ where it has size $k^{1/4}. $ More importantly,  it is
an extremal for $L^p$ for $2 \leq p \leq 6$. For instance,
$$\begin{array}{lll} \int_{\R^3} (x^2 + y^2)^{3k} e^{-(x^2 + y^2 +
z^2)} & = & ||(x + i y)^k||_{L^6(S^2)}^6 \int_0^{\infty} r^{6k}
e^{- r^2} r^2 dr, \\ && \\ & \implies &  ||(x + i
y)^k||_{L^6(S^2)}^6 = \frac{\Gamma(6k + 1)}{\Gamma(6k +
\frac{3}{2})} \sim k^{-1/2}.
\end{array}$$
Hence, the $L^6$ norm of $k^{1/4} (x + i y)^k$ equals
$$k^{1/4} k^{-1/12} = k^{1/6}. $$
Since $\lambda_k \sim k$ and $\delta(6) = \frac{1}{6}$ in
dimension $2$, we see that it is an extremal.

It is natural to ask if extremals (in order of magnitude) for the $L^p$ norms
necessarily resemble these examples. This motivates the problems:

\begin{enumerate}

\item {\it What is the structure
 of eigenfunctions which are maximal for the $L^p$ norms?}

\item {\it Determine the
$(M,g)$ for which $L^{\infty}(\lambda,g) =
\Omega(\lambda^{\frac{n-1}{2}}).$}

\item   {\it At the other extreme, determine the structure of eigenfunctions
which minimize the $L^p$ norms, and determine the  $(M, g)$
with a uniformly bounded orthonormal basis of eigenfunctions. Irrational flat tori are
examples.}

\end{enumerate}

\subsection{Sketch of proof of the Sogge $L^{p}$ estimate }

We now sketch the proof of the Sogge estimates, following the appraoch in  \cite{KTZ}.

The estimates are proved by interpolation from three estimates: $p = \infty, \tfrac{2(n+1)}{n-1}$ (for $n \not= 1$).
The $L^2$ estimate is of course trivial. The $L^{\infty}$ estimate already followed from the un-integrated local Weyl
law with remainder estimated by the Fourier Tauberian method. We will nevertheless sketch an alternative proof which
has more in common with the  $L^{\tfrac{2(n+1)}{n-1}}$
estimate.

\subsubsection{The $L^{\infty}$ estimate}

\begin{proof} The first observation is that the estimates are local in phase space $T^*  M$.
That is,  the estimate holds for $\phi_{\lambda}$ if it holds for $\chi(x, \lambda^{-1} D) \phi_{\lambda}$
for any   microlocal cutoff $\chi(x, \lambda^{-1} D) = Op_{\lambda^{-1}} (\chi)$, where $\chi$ is a smooth function on $T^*M\backslash 0$
 which is supported in a  small conic subset (i.e. the cone through
a small open subset of $S^*_g M$.)

The semi-classical symbol of $\lambda^{-2}\Delta - 1$ equals $p(x, \xi) = |\xi|_g^2 -1$. We may choose coordinates
in the support of $\chi$ so that $\partial_{\xi_1} (|\xi|_g^2 - 1) \not= 0 $.
Locally write $p = p(x, \xi', \xi_1), $ where $\xi' = (\xi_2, \dots, \xi_m)$.  By the implicit function theorem, there exists
 a local function $a(x, \xi')$ such that $\xi_1 = a(x, \xi')$ when $p = 0$, and there exists a symbol $q(x, \xi) > 0, $
on the support of $\chi$
so that \begin{equation} \label{FACTOR} p(x, \xi) = q(x, \xi) (\xi_1 -  a(x, \xi')). \end{equation}  Quantizing the right
factor and using the support properties of the symbol, we get
\begin{equation} \label{RIGHTEST} \begin{array}{lll} \left(q(x, \lambda^{-1}D)  (\lambda^{-1}D_{x_1}) - a(x, \lambda^{-1} D') \right) (\chi(x, \lambda^{-1}D) \phi_{\lambda} )
 &=  &(\lambda^{-2} \Delta - 1) \chi(x, \lambda^{-1} D) \phi_{\lambda} \\ && \\
&  =&
[\lambda^{-2} \Delta, \chi(x, \lambda^{-1} D) ] \phi_{\lambda} = O(\lambda^{-1}). \end{array} \end{equation}
Since $q(x, \lambda^{-1}D)$ is elliptic we may invert it to obtain,
\begin{equation} \label{f} f_{\lambda}(x_1, x') : = ( \lambda^{-1}D_{x_1} - a(x, \lambda^{-1}D')) \chi(x, \lambda^{-1} D) \phi_{\lambda}   = O(\lambda^{-1}). \end{equation}
Let $E(t)$ be the unitary evolution operator which solves
\begin{equation} \label{ODE} ( \lambda^{-1}D_{t} - a(t, x', \lambda^{-1}D')) \chi(t, x', \lambda^{-1} D) u = f \end{equation}
with $E(0) = Id$.
The $L^{\infty}$ estimate then follows from the Sobolev bound in dimension $m-1$ as long as
\begin{equation} \label{O(1)} ||\chi(x, \lambda^{-1} D) \phi_{\lambda}(x_1, \cdot)  ||_{L^2(\R^{m-1})}  =  O(1). \end{equation}
This follows by solving the ODE (\ref{ODE}) in the $x_1$ variable. If the right side were equal to zero,
it would be solved by a unitary evolution operator $E(t) (\chi(x, D) \phi_{\lambda} )  \phi_{\lambda}(0, x')$.
 One then solves the inhomogeneous equation
by the Duhamel formula,
$$(\chi(x, \lambda^{-1} D) \phi_{\lambda} ) (x_1, x') = E(x_1) (\chi(x, \lambda^{-1} D) \phi_{\lambda} ) (0, x') + i \lambda \int_0^{x_1} E(t - s) O(\lambda^{-1}) ds,$$
proving (\ref{O(1)}) and completing the proof of the Theorem.

\end{proof}

\subsubsection{The $L^{\tfrac{2(n+1)}{n-1}}$ estimate}

The estimate to be proved is that
\begin{equation} \begin{array}{lll} || \phi_{\lambda}||_{L^{\frac{2(n + 1)}{n - 1}}} & \leq &
C \lambda^{\frac{n-1}{2(n + 1)}} . \end{array} \end{equation}

We start again with the symbol factorization (\ref{FACTOR}) and the estimate (\ref{RIGHTEST}), which implies
\begin{equation} \int_{\R} ||f(x_1, \cdot)||_{L^2(\R^{n-1})} dx_1 \leq C ||f||_{L^2(\R^n)} = O(\lambda^{-1}).  \end{equation}
We now use the evolution operator $E(t)$ of (\ref{ODE}). The equation is time dependent, so the solution operator solves an operator
ODE of the form
\begin{equation} \label{ET}  \lambda^{-1} D_t E(t, s) + A(t) E(t, s) = 0, \;\; E(t,t) = Id. \end{equation}
Then there exists a parametrix for $E(t)$ of the type (\ref{PARAONE}). More generally, suppose that $F(t, s)$ solves the equation
(\ref{ET}) with initial condition $F(s,s) = G(s)(x, \lambda^{-1} D)$, a smoothing operator. Then for small times $t \in (0, t_0)$
there exists an oscillatory integral parametrix
$$F(t, s) u(x) = \lambda^{-k} \int_{\R^k} e^{i \lambda (\phi(t, s, x, \eta) - \langle y, \eta \rangle)} b(t, x, \eta; \lambda) dy d\eta
+ R(t, s) u(x), $$
where $R(t, x)$ is a smoothing operator.

Let $U(t, s) = \psi(t) F(t, s) \chi(x, \lambda^{-1}D )$ where $\chi \in C_c^{\infty}(\R)$. Then the main point is to prove,
\begin{equation} \label{MAINST}
\sup_{s \in I} \left( \int_{\R} ||U(t, s) f||^p_{L^q((\R^n)} dt \right)^{\frac{1}{p}} \leq C \lambda^{- \frac{1}{p}} ||f||_{L^2(\R^n)}.
\end{equation}
For the application to the $L^{\frac{2 (n + 1)}{n-2}}$ estimate one takes $ p = q = \frac{2(n + 1)}{n - 2}$, and we assume this
from now on.
The idea of the proof is to use the oscillatory integral parametrix for $F$ and $U(t, s)$ and the stationary phase method to prove,
\begin{equation} \label{U(t,r)} ||U(t, r) U(s, r)^* f||_{L^{\infty}} \leq C \lambda^{(n-1)/2} (\lambda^{-1} + |t - s|)^{- (n-1)/2}. \end{equation}
Then (\ref{MAINST}) follows by an abstract Strichartz estimate (see \cite{KTZ} for references).

One can choose $\psi$ and $\chi$ so that $$\chi(x, \lambda^{-1} D) \phi_{\lambda}(x_1, x')
= i \lambda \int_0^{x_1} U(x_1, s) f_{\lambda}(s, x') ds + O(\lambda^{-\infty}).$$
Then,
\begin{equation} \begin{array}{lll} ||\chi(x, \lambda^{-1} D) \phi_{\lambda}(x_1, x') ||_{L^{\frac{2(n + 1)}{n - 1}}} & \leq &
\lambda \lambda^{\frac{n-1}{2(n + 1)}} \int_{\R} ||f_{\lambda}(x, \cdot)||_{L^2(\R^{n-1})} ds + O(\lambda^{-\infty})
\\ && \\
& \leq & C \lambda^{\frac{n-1}{2(n + 1)}} . \end{array} \end{equation}

\subsection{Generic non-sharpness of Sogge estimates}

As the proof indicates, the Sogge bounds  are  `(micro-) local results', and
do not take the global dynamics of the geodesic
flow into account. The dynamics of the
geodesic flow has a strong impact on the growth of $L^p$ norms of
eigenfunctions.
In \cite{SoZ}, it   is shown that the Sogge $L^p$  estimates for large $p$ are  very rarely
sharp. To state the result, we need some notation.
Let $V_{\lambda} := \{\phi: \Delta \phi_\lambda = \lambda \phi_\lambda\}$
denote the $\lambda$-eigenspace for $\lambda \in Sp (\Delta)$  and
define
 \begin{equation} \label{L} L^{\infty}(\lambda, g) = \sup_{\phi\in V_{\lambda} \atop ||\phi||_{L^2}
= 1 }
 ||\phi||_{L^{\infty}},\;\;\;\;\;\;\;\ell^{\infty}(\lambda, g) = \inf_{ONB \{\phi_j\}
\in V_{\lambda}} (\sup_{j = 1, \dots, \dim V_{\lambda}}
 ||\phi_j||_{L^{\infty}}). \end{equation}
Thus, $L^{\infty}(\lambda, g) = O(\lambda^{\frac{n-1}{2}})$ for
any $(M, g)$.

 We define
the  measure $|\lcal_x|$ of loops at $x$ (see (\ref{LCALX}) for the definition)  as the surface measure on $S^*_g M$
induced by the metric $g_x$ on $T^*_xM$.  For instance, the poles
$x_N, x_S$ of a surface of revolution $(S^2, g)$ satisfy
$|\lcal_x| = 2 \pi$.

\begin{theo} \label{SOGZ}If $L^{\infty}(\lambda, g) = \Omega
(\lambda^{\frac{n-1}{2}})$, then there exists a point $x_0$ such
that $|\lcal_{x_0}| > 0$. In particular, if $g$ is analytic, then
all geodesics leaving $x_0$ must return to $x_0$ at the same time.
\end{theo}

The proof is based on a study of the remainder term in the local Weyl law. The main
step is:

\begin{lem}\label{maintheorem} Let $R(\lambda,x)$ denote
the remainder for the local Weyl law at $x$.  Then
\begin{equation}\label{M20}
R(\lambda,x)= o_x(\lambda^{n-1}) \, \, \text{if } \, \, |{\mathcal
L}_x|=0.
\end{equation}
Additionally, if $|{\mathcal L}_x|=0$ then, given $\varepsilon>0$,
there exists  a neighborhood ${\mathcal N}$ of $x$ and  $\Lambda
<\infty$, both depending on $\varepsilon$ so that
\begin{equation}\label{M3}
|R(\lambda,y)|\le \varepsilon \lambda^{n-1}, \, \, y\in {\mathcal
N}, \, \, \lambda\ge \Lambda.
\end{equation}
\end{lem}

Thus, there are topological conditions on $M$ which are necessary
for $M$ to possess a real analytic metric such that some sequence
of eigenfunctions has the maximally growing $L^{\infty}$ norms. In
dimension two, only the sphere possesses such a metric.  Moreover,  the maximal growth rate of
$L^{\infty}$ exhibited by zonal spherical harmonics can never
occur for a metric on $S^2$  with ergodic geodesic flow.

Theorem \ref{SOGZ} is certainly not the end of the story. In all known cases where `maximal eigenfunction
growth' occurs, all geodesics leaving the `pole' $x_0$ return to it at the same time $T$, not just a set of positive measure.
Further, the `first return map' $G^T: S_{x_0}^* M \to S_{x_0}^* M$ is the identity map.  One would hope to prove that
the second property holds at least on a set of positive measure.

When the geodesic flow is `chaotic' (i.e. highly mixing), it is
expected that the eigenfunctions resemble Gaussian random
functions. The random wave model (see \S \ref{RWONB}) then predicts
that eigenfunctions of Riemannian manifolds with chaotic geodesic
flow should have the bounds $||\phi_{\lambda}||_{L^p} = O(1)$ for
$p < \infty$ and that $||\phi_{\lambda}||_{L^{\infty}} <
C \sqrt{\log \lambda}.$ But no rigorous PDE methods to date have done
better than  $ O(\frac{\lambda^{n-1}}{\log
\lambda})$. There also exist counterexamples to the logarithmic estimate on  special arithmetic hyperbolic quotients (see \cite{RS,IS,Don2}).
As  in the case of zonal spherical harmonics on surfaces of revolution (but using Hecke operators in place of the rotations), there exist special eigenfunctions
which take large values at `fixed points' of the Hecke action or theta-corresopondence.
 In general, the  exponential growth of the
geodesic flow is a huge obstacle to improving  the estimate
beyond the logarithm.  Further discussion of $L^{\infty}$-norms, as well as
zeros,  will be given at the end of \S \ref{QuE} for ergodic
systems.

\section{\label{QMGEO} Gaussian beams and quasi-modes associated to stable closed geodesics}

We noted above that highest weight spherical harmonics maximized
$L^p$ norms for $2 \leq p \leq \frac{2(n + 1)}{n-1} $, and we
noted that they had the shape of an oscillating bump along the
equator with Gaussian decay in the transverse direction. Such
eigenfunctions are only known in a few cases, all of them quantum
completely integrable: convex surfaces of revolution, ellipsoids
and ellipses. What they have in common is the existence of a
stable elliptic closed geodesic along which to construct a
Gaussian beam.

The  construction of Gaussian beams  is possible on any compact
Riemannian manifold with a stable elliptic closed geodesic.
However, in general it only produces an approximate eigenfunction
or quasi-mode. We review them in this section since they are a
good introduction to quantum Birkhoff normal forms and local
models. In the following section we will consider quantum
integrable systems where the construction produces true
eigenfunctions. References for Gaussian beams and more general
quasi-modes are  \cite{BB,Ra,Ra2,TZ2,W} among many other places.
 We follow \cite{Z8} and \cite{BB} Chapter 9.

Quasi-modes along stable elliptic orbits are approximate
eigenfunctions (sometimes exact) which have the form of Gaussian
beams along geodesics. A Gaussian beam is a simple oscillatory
function $e^{i k s}$ along the geodesic times a transverse
Gaussian or (more generally) Hermite function. Hence, the
quasi-modes are very localized along closed geodesics and in a
sense are the best localized approximate eigenfunctions.

Let $(M, g)$ be a Riemannian manifold of dimension $m = n + 1$ with a
stable elliptic closed geodesic $\gamma $ of length $L$.  We assume $\gamma$ is an
embedded (non self-intersecting) curve with an orientable normal
bundle, and use Fermi normal coordinates $(s, y)$ along $\gamma$.
That is, we fix an origin $\gamma(0)$ and let $s$ denote
arc-length on $\gamma$. The exponential map $\exp: N_{\gamma} \to
T_{\epsilon}(\gamma)$ from the $\epsilon$-ball in the normal
bundle along $\gamma$ to a tubular neighborhood of radius
$\epsilon$ is a diffeomorphism and any choice of linear
coordinates on $N_{\gamma}$ endows $T_{\epsilon}(\gamma)$ with
Fermi normal coordinates.

We denote the  eigenvalues of the linear Poincar\'e map by
$\{e^{i \alpha_j}\}$. For $q \in \N^m$ we put
$$r_{kq} = \frac{1}{L} (2 \pi k + \sum_{j=1}^n (q_j + \frac{1}{2}) \alpha_j) $$
In this section, Planck's constant takes the form,
$$h = r_{kq}^{-1}. $$

\subsection{Local model}

The model space is the normal bundle $N_{\gamma}$ along $\gamma$,
which may be identified with $\R^{n} \times S^1$. On the
quantum level, the model Hilbert space is isomorphic to
${\mathcal H} = H^2(S^1_L) \otimes L^2(\R^n)$,  where $H^2(S^1_L)$
is the Hardy space of the circle of length $L = L_{\gamma}$.  An
orthonormal basis of $L^2(\R^n)$ of joint eigenfunctions of
Harmonic oscillators $$\hat{I}_j= \hat{I}_j(y,D_y) := \frac{1}{2} (D_{y_j}^2 +
y_j^2) $$ is
 provided by the Hermite
functions $D_{ q}, q \in \N^n.$ Here, $D_0$  is the Gaussian
$D_0(y)= e^{-\half |y|^2}$, while
   $ D_{ q}:= C_q A_1^{* q_1}...A_n^{*q_n}
D_0 (q \in {\bf N}^n),$ with $C_q = (2 \pi)^{-n/2} (q!)^{-1/2},
q!=q_1!...q_n!$ and where
$$A_j := y_j + iD_{y_j}  \;\;\;\;\;\;\;\; A_j^* = y_j - i D_{y_j}$$
are the creation/annihiliation operators. An orthonormal basis of
$H^2(S^1_L) \otimes L^2(\R^n)$ of  joint eigenfunctions of
$\frac{d}{ds}$ and the $\hat{I}_j$  is then furnished by
$$\phi_{kq}^o(s,y):= e_k(s)\otimes D_{ q}(y),\;\;\;\;\;\;e_k(s):= e^{\frac{2 \pi}{L}iks}. $$

This model needs to be adapted to $\gamma$ in the sense that $\R^n
\times S^1$ needs to be converted to the normal bundle
$N_{\gamma}$. Without explaining this in detail, it may be
illustrated by the fact that the adapted model transverse ground
state Hermite function is
$$U_o(s,u) = (det Y(s))^{-1/2} e^{i\half \langle \Gamma(s)
u,u\rangle}, $$ where $Y = \left( y_{j k} \right)$ is the matrix
whose columns are a basis of vertical Jacobi fields along $\gamma$
and where $\Gamma(s) := \frac{dY}{ds} Y^{-1}.$ Higher transverse
Hermite functions are obtained by applying standard
creation/annihilation operators
$$\begin{array}{ll}
\Lambda_j =\sum_{k=1}^n (i y_{jk} D_{u_k} -\frac{dy_{jk}}{ds} u_k)
& \Lambda_j^*=\sum_{k=1}^{n} (-i \overline{y}_{jk}D_{u_k} -
\overline{\frac{dy_{jk}}{ds}}u_k) \end{array} $$  (adapted to
$\gamma$):
$$U_q = \Lambda_1^{q_1}...\Lambda_n^{q_n} U_o. $$
Roughly speaking, the  metaplectic representation $\mu$  maps the
model
 eigenfunctions on $\R^n \times S^1$  to the normal bundle $N_{\gamma}$.
  Introducing the  symplectic matrix
$$a_s:= \left( \begin{array}{ll} Im\dot{Y}(s)^* \;\;\;&ImY(s)^*\\
Re\dot{Y}(s)^*\;\;\;&Re Y(s)^*  \end{array} \right).
$$
with $s \in S^1_L$, we introduce the unitary metaplectic  operator
(depending on the parameter $s$)
$\mu(a):=\int_{\gamma}^{\oplus}\mu(a_s)ds $ on
$\int_{S^1_L}^{\oplus} L^2(\R^n) ds$.
 Then $\mu(a)$
conjugates standard creation/annihilation operators to the adapted
ones, and thus explains why the full quasi-mode is obtained by
applying adapted creation/annihilation operators. We refer to
\cite{Z8,BB} for further details.

\subsection{WKB ansatz for a Gaussian beam}

The most direct way to  construct the Gaussian beam  is to try to
construct an approximate solution  $u = e^{i s \lambda} U(s, y_1,
\dots, y_n)$ of $(\Delta - \lambda^2)u = 0$. Assuming that $U$ is
localized in a $\lambda^{-1/2}$ neighborhood of $0$, one obtains a
parabolic transport equation for $U$,
\begin{equation} \label{LCAL} \lcal U : = 2 i \lambda \frac{\partial U}{\partial s} +
\sum_{j = 1}^n \frac{\partial^2 U}{\partial y_j^2} - \lambda
\sum_{j, h = 1}^n K_{j h} (s) y_j y_h U = 0, \end{equation}  where
$K_{jk}$ are components of $R(T, Y) T$ where $R$ is the curvature
tensor and where $T = \dot{\gamma}$.

 The detailed ansatz
$\{\Phi_{kq}(s,\sqrt{r_{kq}}y)\}$ of \cite{BB} is defined by
\begin{equation} \label{PHI} \Phi_{kq}(s,\sqrt{r_{kq}}y) =e^{ir_{kq}s} \sum_{j=0}^{\infty}
r_{kq}^{-\frac{j}{2}} U_q^{\frac{j}{2}}(s,
\sqrt{r_{kq}}y,r_{kq}^{-1}) \end{equation}  with $U_q^o = U_q$ and
where $U^j_q$ are constructed inductively so that
\begin{equation} \label{QEQUATION} \Delta  \;  \Phi_{kq}(s,
\sqrt{r_{kq}}y,r_{kq}^{-1}) \sim \lambda_{kq}^2
 \Phi_{kq}(s, \sqrt{r_{kq}}y), \end{equation}
where \begin{equation} \label{LAMBDA} \lambda_{kq} \equiv r_{kq} +
\frac{p_1(q)}{r_{kq}} + \frac{p_2(q)}{r_{kq}^2} +
....\end{equation} The numerators $p_j(q)$ are polynomials of
degrees $j + 1$.  Thus one simultaneously constructs approximate
eigenfunctions and eigenvalues around a stable elliptic geodesic.
For each $q$ and large $k$, the right side solves the eigenvalue
problem to high approximation with a sequence of functions
concentrated near $\gamma$.

Instead of scaling the WKB quasi-mode, one can scale the
Laplacian.
 More precisely (see \cite{Z8}), one
rescales the Laplacian around $\gamma$ in $T_{\epsilon}(\gamma)$
and conjugates by $e^{i r_{qk} s}$ using the operators (with $h =
r_{kq}^{-1})$,
$$T_h (f(s,u)):= h^{-n/2} f(s, h^{-\half}u), \;\;
M_h(f(s,u)):= e^{\frac{i}{hL}s} f(s,y).$$ We  define the scaling of
an operator $A$ as $A_h := T_h^*M_h^*AT_hM_h$. Then a simple
computation gives
$$-\Delta_{h} =  -(hL)^{-2} g^{oo}_{[h]} + 2i(hL)^{-1}g^{oo}_{[h]}\partial_s + i(hL)^{-1}\Gamma^o_{[h]}+
 h^{-1}( \sum_{ij=1}^n g^{ij}_{[h]}\partial_{u_i}\partial_{u_j}) + h^{-\half}(\sum_{i=1}^{n} \Gamma^{i}_{[h]}
\partial_{u_i}) + (\sigma)_{[h]},$$ where
 the subscript $[h]$ indicates to dilate the coefficients of the operator in the form,
$f_h(s, u):=f(s, h^{\half} u).$ One then has,
$$\Delta_h \sim \sum_{m=0}^{\infty} h^{(-2 +m/2)}{\mathcal L}_{2-m/2}$$
where ${\mathcal L}_2 = L^{-2},$ ${\mathcal L}_{3/2}=0$ and where
$${\mathcal  L}_1 = 2 L^{-1}[i  \frac{\partial}{\partial s} + \half \{\sum_{j=1}^{n}
\partial_{u_j}^2 -
\sum_{ij=1}^{n} K_{ij}(s) u_i u_j\}].$$ One then plugs into the
ansatz (\ref{PHI}) to get
$$\left(\sum_{m=0}^{\infty} h^{(-2 +m/2)}{\mathcal L}_{2-m/2}
\right) \sum_{j = 0}^{\infty}  h^{j/2} U_q^{j/2} \sim (h L)^{-2}
\sum_{j = 0}^{\infty}  h^{j/2} U_q^{j/2} $$ and solves recursively
for $U_q^{j/2}. $ A key point is that the metaplectic operator
$\mu$ intertwines $D_s$ to the operator
  $${\mathcal L}:= \mu(a)^{*} D_s \mu(a) = D_s - \frac{1}{2}(\sum_{j=1}^n D_{u_j}^2 + \sum_{ij=1}^n K_{ij}(s)
u_iu_j)$$ of (\ref{LCAL}).

\subsection{\label{INTOP} Quantum Birkhoff normal form: intertwining to the
model}

A second approach,  which perhaps originates in \cite{W} and which
has been developed in many articles including \cite{Sj,Z8,TZ,TZ2},
is the method of interwining operators.  In this approach, the
construction of quasi-modes follows from the dual  construction of
an interwining operator $W_{\gamma}$ to a normal form of $\Delta$
around $\gamma$. The intertwining operator  maps the model
eigenfunctions to the detailed quasi-mode,
$$W_{\gamma} \phi_{kq}(s,y) = \Phi_{kq}(s, \sqrt{r_{kq}}y)$$ where
$\Phi_{kq}$ solve (\ref{QEQUATION}). Rather than constructing the
quasi-mode, one instead constructs $W_{\gamma}$ so that it
conjugates $\sqrt{\Delta}$ to a function of the model operators on
the model space. We are assuming that $\gamma$ is elliptic, so the model operators
are harmonic oscillators.

 This function is the quantum Birkhoff normal
form.  The normal form  is implicit in (\ref{QEQUATION}), i.e. the
normal form is the same as  the quasi-eigenvalues (\ref{LAMBDA})
interpreted on the operator level.  The precise statement is:

\begin{theo} \label{QBNFTHM}  There exists a microlocally unitary Fourier integral
operator $W_{\gamma}$ defined near the cone  $R^+\gamma $
generated by $\gamma$  in $T^*(S^1_L \times \R^n)$ such that
   $$W_{\gamma}^{-1} \sqrt{\Delta}  W_{\gamma} \equiv
 P_{1}({\mathcal L},I_{\gamma1},...,I_{\gamma n})
+ P_{o}({\mathcal L},I_{\gamma1},...,I_{\gamma n}) + \dots,$$
where
$$P_1({\mathcal L}, I_{\gamma1},...,I_{\gamma n})\equiv {\mathcal L} +
\frac{ p_{1}^{[2]}(I_{\gamma1},...,I_{\gamma n})}{L {\mathcal L}}
+ \frac{p_{2}^{[3]}(I_{\gamma1},...,I_{\gamma n})}{(L{\mathcal
L})^2} + \dots$$
$$P_{-m}({\mathcal L}, I_{\gamma1},...,I_{\gamma n})\equiv \sum_{k=m}^{\infty}
 \frac{p_{k}^{[k-m]}(I_{\gamma1},...,I_{\gamma n})}{(L{\mathcal L})^j}$$
with $p_{k}^{[k-m]}$, for m=-1,0,1,..., homogenous of degree k-m
in the variables $(I_{\gamma1},. ..,I_{\gamma n}) $ and of weight
-1. The  kth remainder term lies in the space
$\bigoplus_{j=o}^{k+2} O_{2(k+2-j)}\Psi^{1-j}$
\medskip

  \end{theo}

  Here, $O_n \Psi^r$ is the space of pseudodifferential operators of order
r whose complete symbols vanish to order n at $(y,\eta)=(0,0).$  Thus, the
remainder terms are `small' in that they combine in some mixture  a low
pseudodifferential order or a high vanishing  order along $\gamma$.

If one plugs in the eigenvalues $ \in \Z^n$ for the harmonic
oscillators $(I_{\gamma }$, the operator expansion in Theorem
(\ref{QBNFTHM}) becomes the eigenvalue expansion in
(\ref{QEQUATION}).

\section{\label{BNF} Birkhoff normal forms around closed geodesics}

The classical Birkhoff normal form of a Hamiltonian around a closed geodesic (or an invariant torus) is
a local approximation around $\gamma$ of the Hamiltonian by  completely integrable Hamiltonians modulo errors which
vanish to higher order around the orbit. One constructs local action variables $I_j$ (Hamiltonians with
$2 \pi$-perodic flows)  and approximates the
Hamiltonian by a function $H_k(I)$ up to an error $I^k$.

On the quantum level one has two notions of order: `order in $\hbar$' and vanishing order at $\gamma$. As in
Theorem \ref{QBNFTHM}, one constructs local `action operators' and expresses $\hat{H} = \sqrt{\Delta}$ as
a polyhomogeneous function of the action operators modulo small remainders.

  In the general non-degenerate case, the normal
form  involves a greater variety of quadratic normal
forms or `action operators' than in Theorem \ref{QBNFTHM}.  In addition to the elliptic action operator $\hat{I}^e_j$
there  can also occur the real hyperbolic action operators $\hat{I}^h_j$ and
complex hyperbolic (or loxodromic) action operators $\hat{I}_j^{ch, Re},
\hat{I}_j^{ch, Im}.$

 In the elliptic
case  $P_{\gamma}$  was a direct sum of rotations, and the quantum normal form
of ${\mathcal L}$ had the form
$${\mathcal R}^e = D_s + \frac{1}{L} H_{\alpha},\;\;\;\;\;\;\; H_{\alpha} =
\sum_{j=1}^n \alpha_j \hat{I}_j^e$$
where the spectrum $\sigma(P_{\gamma}) = \{e^{\pm i \alpha_j}\}$. In the general
non-degenerate case the normal form will similarly depend on   the spectral
decomposition of
$P_{\gamma}$. Recall that,
 since $P_{\gamma}$ is
symplectic, its eigenvalues $\rho_j$  come in three types: (i) pairs
$\rho, \bar{\rho}$ of conjugate eigenvalues of modulus 1; (ii) pairs $\rho,
\rho^{-1}$ of inverse real eigenvalues; and (iii) 4-tuplets $\rho, \bar{\rho},
\rho^{-1} \bar{\rho}^{-1}$ of
complex eigenvalues. We will often write them in the forms: (i) $e^{ \pm i\alpha_j} $,
(ii)$e^{\pm \lambda_j}$, (iii)  $e^{\pm  \mu_j \pm i \nu_j}$ respectively
 (with $\alpha_j, \lambda_j, \mu_j, \nu_j \in \R$), although
a pair of inverse real eigenvalues $\{-e^{\pm \lambda}\}$ could  be negative.  Here,
and throughout, we make the assumption that $P_{\gamma}$ is {\it non-degenerate}
in the sense that
$$\Pi_{i=1}^{2n} \rho_i^{m_i} \not= 1,\;\;\;\;\;\;\;\;\;(\forall \rho_i \in
\sigma(P_{\gamma}),\;\;\;\;\; (m_1,\dots,m_{2n}) \in {\bf N}^{2n}).$$
Each type of eigenvalue then determines a different type of quadratic
action, both on the classical and quantum levels (cf. [Ho, Theorem 3.1],[Ar]):
\medskip

\begin{tabular}{l|l|l} Eigenvalue type & Classical Normal form & Quantum normal form
\\
\hline   (i) Elliptic type  & $I^e = \half \alpha (\eta^2 + y^2)$
&
$\hat{I}^e:=
\half \alpha (D_y^2 +y^2)$
\\
$\{e^{\pm i \alpha}\}$ & & \\
\hline (ii) Real hyperbolic type & $I^{h} = 2 \lambda y \eta$ & $\hat{I}^{h}:=
\lambda (y D_y + D_y y)$
\\
$\{e^{\pm \lambda}\}$ & & \\
\hline(iii) Complex hyperbolic   & $I^{ch,Re} = 2 \mu (y_1 \eta_1 + y_2 \eta_2)$ &
 $\hat{I}^{ch, Re} = \mu  (y_1D_{y_1} + D_{y_1} y_1 + y_2 D_{y_2} + D_{y_2} y_2),$\\
(or loxodromic type) & $I^{ch,Im} = \nu ( y_1 \eta_2 - y_2 \eta_1)$ & $\hat{I}^{ch,Im}
=   \nu (y_1D_{y_2}  - y_2 D_{y_1} )$
\\
$\{e^{\pm \mu + \pm i \nu}\}$ & & \\
\end{tabular}
\medskip

  In the case where  the Poincare map $P_{\gamma}$ has p
pairs of complex conjugate eigenvalues of modulus 1, q pairs of inverse real
eigenvalues and c quadruplets of complex hyperbolic eigenvalues,
the linearized $\sqrt{\Delta}$ will have the form:
$${\mathcal R} =  D_s +  \frac{1}{L}[\sum_{j=1}^p \alpha_j
\hat{I}_j^e + \sum_{j=1}^q \lambda_j \hat{I}_j^h + \sum_{j=1}^c \mu_j \hat{I}^{ch,
Re}_j +
\nu_j \hat{I}_j^{ch, Im}]. $$
The full
quantum Birkhoff normal form is then given by the analogue of Theorem B
of [Z.1]:
\medskip

\noindent{\bf Theorem B} ~~~~{\it Assuming $\gamma$ non-degenerate, there exists a
microlocally elliptic Fourier integral operator $W$ from the conic neighborhood of
$\R^+ \gamma$ in $T^*(N_{\gamma})$ to  the corresponding cone in $T_+^*S^1$ in
$T^*(S^1
\times
\R^n)$ such that
$$W \sqrt{\Delta_{\psi}} W^{-1} \equiv D_s + \frac{1}{L}[\sum_{j=1}^p \alpha_j
\hat{I}_j^e + \sum_{j=1}^q \lambda_j \hat{I}_j^h + \sum_{j=1}^c \mu_j \hat{I}^{ch,
Re}_j +
\nu_j \hat{I}_j^{ch, Im}] +$$
$$ +  \frac{p_1(\hat{I}_1^e,\dots,
\hat{I}_p^e, \hat{I}_1^h, \dots, \hat{I}_q^h, \hat{I}_1^{ch,
Re},\hat{I}_1^{ch,Im},\dots, \hat{I}_c^{ch,Re}, \hat{I}_c^{ch, Im})}{D_s} +\dots$$
$$+ \frac{p_{k+1}(\hat{I}_1^e,\dots, \hat{I}_c^{ch, Im})}{D_s^k} +\dots $$
where the numerators $p_j(\hat{I}_1^e,\dots, \hat{I}_p^e, \hat{I}_1^h, \dots,
\hat{I}_c^{ch, Im})$ are polynomials of degree j+1 in the variables
$(\hat{I}_1^e,\dots,  \hat{I}_c^{ch, Im})$ and where the kth remainder term lies in
the space $\oplus_{j=o}^{k+2} O_{2(k+2-j)}\Psi^{1-j}$}
\medskip

Here, $O_n \Psi^r$ is the space of pseudodifferential operators of order
r whose complete symbols vanish to order n at $(y,\eta)=(0,0).$  Thus, the
remainder terms are `small' in that they combine in some mixture  a low
pseudodifferential order or a high vanishing  order along $\gamma$.

\subsection{Local quantum Birkhoff normal forms}

Readers  unfamiliar with microlocal Birkhoff normal forms might
keep in mind a comparison to the local models used in \S
\ref{KROGER} to study gradient estimates and ranges of
eigenfunctions. The local models were spaces of constant
curvature, a classical kind of comparison in geometric analysis.
The quantum Birkhoff normal forms are of a different nature: they
are more global, since they are local around closed orbits rather
than points. Moreover, they are dynamical and geodesic-based
rather than curvature-based. The local models around regular
orbits will be used in \S \ref{LOCALIZENF} to obtain localization
results for quantum limits and in \S \ref{LPQCI} normal forms
around singular closed orbits will be used to deduce $L^p$
estimates on quantum integrable eigenfunctions. The local models
are enumerated and described in \S \ref{MODELS}. Note that they
model operators are often Schr\"odinger operators with quadratic
potentials rather than Laplacians for special metrics.

Throughout this section,
 the parameter
$h = \lambda^{-1}$  plays the role of Planck's constant.

\subsection{\label{MODELS}Model eigenfunctions around closed geodesics}

 Model quantum
completely integrable systems are direct sums of
 the quadratic Hamiltonians:

\begin{itemize}

\item $\hat{I}^{h} := \hbar (D_{y} y + y D_{y}) $ on
$L^2(\R_+)$\,\,(hyperbolic \,\, Hamiltonian),

\item $\hat{I}^{e} := \hbar^{2}D_{y}^{2} + y^{2}$ on $L^2(\R)$
\,\,(elliptic \,\, harmonic oscillator Hamiltonian),

\item $ \hat{I}^{ch} := \hbar  \, [ ( y_{1} D_{y_{1}} + y_{2}
D_{y_{2}}) + \sqrt{-1} (y_{1} D_{2}- y_{2}D_{y_{1}}) ] $\,\,on
$L^2(\R^2)$  (complex \,\, hyperbolic \,\, Hamiltonian),

\item $ \hat{I} := \hbar \,  D_{\theta } $, \,\,on $S^1$  (regular
\,\, Hamiltonian).

\end{itemize}

\vspace{2mm}

\noindent The corresponding  model eigenfunctions are:

\begin{itemize}

\item  $u_{h}(y; \lambda, \hbar) =  |\log \hbar|^{-1/2} \, [
c_{+}(\hbar)  Y(y)\,|y|^{-1/2 + i\lambda(\hbar)/\hbar} +
c_{-}(\hbar)
 Y(-y) \, |y|^{-1/2 + i\lambda(\hbar)/\hbar} ];\, |c_{-}(\hbar)|^{2} + |c_{+}(\hbar)|^{2} =1; \, \, \lambda(\hbar) \in {\Bbb R}$.

\item $u_{e}(y;n, \hbar) = \hbar^{-1/4} \exp ( - y^{2}/\hbar )
\,\,\Phi_{n}(\hbar^{-1/2}y) ; \,\, n \in {\Bbb N}.$

\item $u_{ch}(r, \theta ; t_{1}, t_{2}, \hbar) =  |\log
\hbar|^{-1/2} r^{(-1 + it_{1}(\hbar) )/\hbar} \,e^{it_{2}(\hbar)
\theta } ; \,\, t_{1}(\hbar), t_{2}(\hbar) \in {\Bbb R}.$

\item $u_{reg}( \theta ;m,  \hbar) = e^{i m \theta}; \,\, m \in
{\Bbb Z}.$

\end{itemize}

\noindent Here, $Y(x)$ denotes the Heaviside function,
$\Phi_{n}(y)$ the $n$-th Hermite polynomial  and $(r,\theta )$
polar variables in the $(y_{1}, y_{2})$ complex hyperbolic plane.

The important part of a model eigenfunctions is its
microlocalization to a neighborhood of $x = \xi = 0$, so we put:
$$ \psi(x;\hbar):= Op_{\hbar}( \, \chi(x) \chi(y) \chi(\xi) \, )  \cdot  u(y;\hbar),$$
where $\epsilon >0$ and  $\chi \in C^{\infty}_{0}([-\epsilon,
\epsilon])$. In the hyperbolic, complex hyperbolic, elliptic  and
regular cases, we write  $\psi_{h}(y;\hbar),  \psi_{ch}(y ;\hbar),
\psi_{e}(y ;\hbar)$ and $ \psi_{reg}(y ;\hbar)$ respectively. A
straightforward computation \cite{T2} shows that when
$t_{1}(\hbar), t_{2}(\hbar), n \hbar , m \hbar = {\mathcal
O}(\hbar)$  the model quasimodes are $L^{2}$-normalized; that is
\begin{equation} \label{OP} \| Op_{\hbar}( \chi(x) \chi(y) \chi(\xi) \, ) \, u(y;\hbar)
\|_{L^{2}} \sim 1 \end{equation}  as $\hbar \rightarrow 0$. Note
that, although the model eigenfunctions above are not in general
smooth functions, the microlocalizations are $C^{\infty}$ and
supported near the origin.

\section{\label{QCI} Quantum integrable Laplacians}

We now go into more detail on (globally) integrable Laplacians.
As discussed in \S \ref{EXPLICIT}, quantum integrable Laplacians
are the only examples where one has explicit formulae for
eigenfunctions. They are also examples where one can explicitly
determine the $L^p$ norms of the eigenfunctions. Moreover, there
are geometric properties of the geodesic flow which account for
the behavior of the $L^p$ norms.

Quantum integrable Laplacians are fundamental for two reasons
besides the relatively explicit computability of their
eigenfunctions.
 The first reason,   which is rigorous and now well-understood, is that
quantum integrable Laplacians provide local models for all
Laplacians around closed geodesics. The sense in which they are
local models is made precise by quantum Birkhoff normal forms. The
second reason why they are fundamental is heuristic: to date, all
extremals for $L^p$ norms occur in integrable systems and are
explained by the geometry of the geodesic flow. It is possible
that eigenfunction sequences which achieve maximal $L^p$ growth
rates must resemble sequences of eigenfunctions of integrable
systems. At least, one hopes that they must do so locally, e.g. in
the vicinity of a closed geodesic.

In this section, we define quantum complete integrability and give
some examples. In succeeding sections, we survey results on the
norms and concentration properties of the joint eigenfunctions.

\subsection{Quantum integrability and ladders of eigenfunctions}

\begin{defin} \label{SCQCI}
We say that the operators $P_{j} \in \Psi^1(M); \,\,j=1,...,n$,
generate a semiclassical quantum completely integrable system if
$$ [P_{i}, P_{j}]=0; \,\,\, \forall{1 \leq i,j \leq n},$$
\noindent and the respective semiclassical principal symbols
$p_{1},...,p_{n}$ generate a classical integrable system with
$dp_{1} \wedge dp_{2} \wedge \cdot \cdot \cdot \wedge dp_{n} \neq
0$ on a dense open set $\Omega \subset T^{*}M \backslash 0$.
\end{defin}

Semiclassical limits are taken along ladders in the joint
spectrum. For fixed $b=(b_{1},b_{2},...,b_{n}) \in {\Bbb R}^{n}$,
we define a ladder of  joint eigenvalues of the original
homogeneous problem  $P_{1} = \sqrt{\Delta},P_{2},...,P_{n}$ by:
\begin{equation} \label{QCI1} \Sigma_b: =
\{ (\lambda_{1k},...,\lambda_{nk}) \in Spec(P_{1},...,P_{n});
\,\forall j=1,..,n, \,  \lim_{k\rightarrow
\infty}\frac{\lambda_{jk}}{|\lambda_{k}|}  \, = \, b_{j}\},
\end{equation}
\noindent where $|\lambda_{k}|:= \sqrt{ \lambda_{1k}^{2} + ... +
\lambda_{nk}^{2} }$.

We define the joint eigenspace corresponding to $\Sigma_{b}$ as
follows: For $b \in B: = {\mathcal P}(T^{*}M \backslash 0),$ where $\pcal$ is the classical moment map, define
\begin{equation} \label{efn}
 V_{b} := \{ \phi_{\mu}; \|\phi_{\mu} \|_{L^{2}} =1 \, \mbox{with} \, \mu \in \Sigma_{b} \}.
\end{equation}

\subsection{Geometric examples}

\subsubsection{The round sphere}

On the quantum level we define $\hat{I}_1 =
\frac{\partial}{\partial \theta}$ and $\hat{I}_2 = \sqrt{\Delta +
\frac{1}{4}}$, so that $\hat{I}_2 Y^k_m = k + \frac{1}{2}.$ The
joint spectrum of $\hat{I}_1, \hat{I}_2$ equals $\{(m, k +
\frac{1}{2}): |m| \leq k\}$.

Let us compare classical and quantum values. The image of the
classical moment map $(p_{\theta}(x, \xi), |\xi|)$ is the triangular region  of points $(x,
y)$  satisfying $|x| \leq y, y \geq 0$.  It is
easy to see that the inverse image of any point in the interior is
invariant under the $x_3$-axis rotations and under the geodesic
flow; this is a Lagrangian torus. The boundary of the image
corresponds to singular points of the moment map, namely the
equatorial geodesic, traversed in either of its two orientations.

The joint eigenfunction $Y^k_m$, with joint eigenvalue $(m, k +
\frac{1}{2})$,  corresponds to the Lagrangian torus with
$p_{\theta}(x, \xi) = m $ and $|\xi| = k.$ If we rescale back to
$S^*S^2$ we obtain the Lagrangian torus  $p_{\theta} = m/k$.  The
Lagrangian corresponding  to  zonal spherical harmonics is the
meridian torus  $p_{\theta} = 0$, i.e. to longitudinal great
circles which depart from the north pole, converge at the south
pole and then return to the north pole.  The highest weight
spherical harmonics  $Y^k_k$ correspond  to the boundary points of
the triangular image of the moment map, hence to the equatorial
great circle. Indeed, $Y_k^k$ is the restriction of the harmonic
polynomial $(x_1 + i x_2)^k$ (up to normalization). This
polynomial is independent of $x_3$ and is a holomorphic function
of $x_1, x_2$ so it is certainly a harmonic homogeneous
polynomial. Also, by its form it clearly is largest on the unit
circle in the $(x_1,x_2)$ plane and tends to zero as $(x_1, x_2) \to
(0,0)$. Hence on the sphere it defines a spherical harmonic which
is large on the equator and tends to zero at the poles.

\subsubsection{\label{QCISIMPLE} Simple surfaces of revolution \cite{CV3}}

The classical action variables can be quantized to produce quantum
action variables  $\hat{I}_j$, which  are first order
pseudodifferential operators with the property that $e^{2\pi i
\hat{I}_j} = C_j Id$ for some constant $C_j$ of modulus one. From
the fact that $e^{2 \pi i \hat{I}_j} = C_j Id$ for a quantum torus
action, it follows that the joint spectrum of the quantum moment
map
$$Sp({\mathcal I}) \subset \Z^2 \cap \Gamma + \{\mu\}$$
is the set of integral   lattice points, translated by $\mu$, in
the closed convex conic subset $\Gamma \subset \R^2$. The vector
$\mu = (\mu_1, \mu_2)$ can be identified with the Maslov indices
of the homology basis of the invariant torii.  In the case of
$\sqrt{\Delta_g}$, $\mu = (0, 1/2)$ \cite{CV3}.

The operator $\hat{I}_1  = \frac{1}{i} \frac{\partial }{\partial
\theta}$ is the standard infinitesimal generator of the $S^1$
action, but the second action operator $\hat{I}_2$ is a rather
unexpected operator that does not appear in the classical
`separation of variables' approach to eigenfunction. We may define
spectral projectors $\Pi_N$ projecting onto the eigenspace of
$\hat{I}_2$ of eigenvalue $N$. In the case of the round sphere it
is the projection onto the eigenspace $\hcal_N$ of $\Delta$ of
spherical harmonics of degree $N$, but in general it does not
project onto eigenspaces of the Laplacian and is a rather novel
feature of the pseudo-differential approach.

One may express $\sqrt{\Delta_g}$ as a function
$\hat{H}(\hat{I}_1, \hat{I}_2)$ where $\hat{H}(\xi_1, \xi_2)$ is a
polyhomogeneous function on $\Gamma$. Its principal symbol is the
function  $|\xi|_g = H(I)$ expressing the metric norm function in
terms of   the action variables. These expressions are known as
the Birkhoff normal forms of the classical, resp. quantum,
Hamiltonian. In generic situations, the action variables are only
locally defined and there exists only a local Birkhoff normal
form. The interesting feature of simple surfaces of revolution is
that the Birkhoff normal forms are global.

One can then characterize the the packet of eigenvalues of
$\sqrt{\Delta}$ which corresponds to the eigenfunctions with
eigenvalue $N$ for $\hat{I}_2$: the eigenvalues form the set
$\{\hat{H}(m, N): \; |m| \leq N\}$.

\subsubsection{Hyperbolic cylinders with boundary \cite{CVP}}

A hyperbolic cylinder with boundary is a metric tube $X$  around
the unique closed geodesic $\gamma$ of a hyperbolic cylinder $\H
/\Z$ where $\Z$ acts by the cyclic group generated by a hyperbolic
element. Its axis then projects to the closed geodesic $\gamma$.

As with any surface of revolution, the Laplacian commutes with the
generator $\frac{\partial}{\partial \theta}$ of the $S^1$ action.
Here, one puts either Dirichlet or Neumann boundary conditions on
the boundary of the tube $X$.  The novel feature is that the
closed geodesic is hyperbolic and lies on a singular level of the
moment map and the normal form of $\sqrt{\Delta}$ around $\gamma$
is given by $\hat{I}^{h} := \hbar (D_{y} y + y D_{y}) $. The (not
very close) analogue of the Gaussian beam or highest weight
spherical harmonic is the sequence of eigenfunctions where the
weight, i.e. $\frac{\partial}{i \partial \theta}$ -eigenvalue,  is
asymptotically the same as the $\sqrt{\Delta}$ eigenvalue.
Equivalently, if one separates variables to obtain Sturm-Liouville
operators $L_m$ for each weight, then the sequence is obtained by
taking the smallest $L_m$-eigenvalue for each weight.

\subsection{Localization of integrable eigenfunctions }

In this section, we consider the construction of highly localized
or concentrated  eigenfunctions (and quasi-modes).

\subsubsection{Toric integrable systems}

Let $A \in \Psi^o(M)$ denote any zeroth order   pseudodifferential
operator  and $d\mu_{\lambda}$ denote Lebesgue measure on the
Lagrangian torus $T_{\lambda}$.  In the toric case we have the
following localization theorem:

\begin{prop} \cite{Z1} \label{TORASYMP} For any ladder $\{k \lambda + \nu: k = 0, 1, 2, \dots\} $ of joint eigenvalues, we have:
$$ (A \phi_{k \lambda}, \phi_{k \lambda})   =
\int_{T_{\lambda}} \sigma_A d\mu_{\lambda} +
O(k^{-1}).$$\end{prop}

We thus have:

\begin{cor} \label{LOC} For any invariant torus $T_{\lambda} \subset S^*M$, there exists
a ladder $\{\phi_{k \lambda}, k = 0, 1, 2, \dots\}$ of
eigenfunctions localizing on $T_{\lambda}.$ \end{cor}

We illustrate the localization result on $S^2$. The proof is the
same in all toric quantum integrable cases, but is harder for
general integrable systems for reasons explained in the next
section.

The image of $T^*S^2 - 0$ under the moment map $\mu(x, \xi) =
(p_{\theta}(x, \xi), |\xi|)$ is a vertical triangular wedge. It is
a cone, reflecting that $\mu(x, r \xi) = r \mu(x, \xi)$ is
homogeneous. We can break the homogeneity by taking a base for the
cone with $|\xi| = 1$, i.e. by considering points $(x, 1)$. This
corresponds to looking at $p_{\theta}: S^*S^2 \to \R$.

Thus, we consider pairs $(m_j, k_j)$ in the joint spectrum of
$D_{\theta}, A = \sqrt{\Delta + 1/2} - 1/2$ whose projection to
the base of the cone has a limit $(c, 1)$.

\begin{theo} Suppose that $m_j/k_j \to c$. Then
$$\langle Op(a) Y^k_m, Y^k_m \rangle \to \int_{\mu^{-1}(c, 1)} a_0
dx, $$
where $a_0 $ is the principal symbol of $Op(a)$.
\end{theo}

Thus, the eigenfunctions in this ray localize on the invariant
torus $p_{\theta}^{-1}(c)$.

We define $U(t_1, t_2) = e^{i(t_1 D_{\theta} + t_2 A)}$ and note
that it is a unitary representation of the 2-torus $T^2$ on
$L^2(S^2)$. Further
$$\langle Op(a) Y^k_m, Y^k_m \rangle = \langle U(t_1, t_2)^* Op(a) U(t_1, t_2) Y^k_m, Y^k_m
\rangle. $$ Indeed, the eigenvalues cancel out. Average this
formula over $T^2$. We note that
$$\langle A \rangle := \int_{T^2} U(t_1, t_2)^* Op(a) U(t_1, t_2) dt_1 dt_2 $$
commutes with both $D_{\theta}$ and $A$. Indeed, the commutator
with $A$ gives $\frac{d}{dt_2}$ under the integral sign, and the
integral of this derivative equals zero.

But $D_{\theta}, A$ have a simple joint spectrum: the dimension of
the joint eigenspace equals one. Hence, any operator which
commutes with them is a function of them. Thus,
$$\langle A \rangle = F(D_{\theta}, A). $$
The function $F$ must be homogeneous of degree zero. Also, the
right side is a $\Psi$D0. Its symbol is
$$\langle a_0 \rangle : \int_{T^2} a_0 (\Phi^{t_1, t_2}(x, \xi))  dt_1 dt_2. $$

It follows first that
$$\begin{array}{lll}\langle Op(a) Y^k_m, Y^k_m \rangle& = &\langle \langle Op(a) \rangle Y^k_m, Y^k_m
\rangle\\ & & \\
& = & F(m, k). \end{array}$$

Secondly, as $(m_j, k_j) \to \infty$ with $m_j/k_j \to c$, we have
$ F(m_j, k_j) \to F(c, 1). $
But also, the limit is the integral of $a_0$ against an invariant
measure. The principal symbol of $F$ is $\langle a_0 \rangle$,
which is a function on the image of the moment map. Its value at
$(c, 1)$ is by definition $\int_{\mu^{-1}(c, 1)} a_0 dx,$
concluding the proof.

\subsubsection{\label{LOCAL} {\bf $\R^n$-integrable systems} }

It is a surprisingly large step  from toric integrable systems such as the torus, sphere or simple surface
of revolution to a general $\R^n$ integrable system, even such as a surface of revolution with a hyperbolic
orbit to  a `peanut' shaped surface in $\R^3$ or an undulating surface of revolution with several
 `waists'. From the dynamical point of view, there now exists a hyperbolic
closed orbit (the waist of the peanut), and a cylinder of geodesics which asymptotically spiral toward the hyperbolic
closed geodesic. Together they form a singular level set with several
components. Eigenfunctions on this singular level Proposition (\ref{TORASYMP})  localize on level sets of the
moment map rather than on individual tori or components. We prove this and further  analyze the
 degree of concentration by putting the Laplacian into a quantum normal form around each component. We follow
 \cite{TZ2} and use its semi-classical notation where $h = \lambda^{-1}$.

Let
 $b$ be a regular value of the moment map $\pcal$,  let
$$ \pcal^{-1}(b) = \Lambda^{(1)}(b) \cup \dots \cup  \Lambda^{(m_{cl})}(b),$$
\noindent where the $\Lambda^{(l)}(b); l=1,...,m$ are
$n$-dimensional Lagrangian tori, and  $d\mu_{\Lambda^{(j)}(b)}$
denote  the normalized Lebesgue measure on the torus
$\Lambda^{(j)}(b)$.
Define the  {\it semiclassical ladders}
\begin{equation} \label{LAD}
 {\bf L}_{b; \delta}(\hbar):= \{  b_{j}(\hbar):= (b_{j}^{(1)}(\hbar), b_{j}^{(2)}(\hbar), \, ... , \, b_{j}^{(n)}(\hbar)) \in \mbox{Spec}(P_{1},...,P_{n}); \,\, | b_{j}(\hbar) - b | \leq C \hbar^{1-\delta} \, \}.
\end{equation}
Taking a sequence  $\hbar \to 0$, the joint eigenvalues in ${\bf L}_{b; \delta}(\hbar)$ form a  sequence tending to $b$ which is the analogue of a homogeneous ladder.

 Define
\begin{equation} \label{c_j}
c_{l}(\hbar; b_j(\hbar)):= \langle Op_{\hbar}(\chi_{l})
\phi_{b_j(\hbar)}, \phi_{b_j(\hbar)} \rangle ;
\,\,\,\,l=1,...,m_{cl}(b).
\end{equation}
\noindent Here,   $\chi_{l}$ is  a cutoff function which is
equal to 1 in a neighbourhood $\Omega^{(l)}(b)$ of the torus
$\Lambda^{(l)}(b)$ and vanishes on $\cup_{k\neq l}
\Omega^{(k)}(b)$.
\begin{prop} \label{LWL} Let $b \in B_{reg}$, and  let $\{ \phi_{b_{j}(\hbar)} \}$ be a sequence of $L^{2}$-normalizeed joint eigenfunctions of $P_{1},...,P_{n}$ with joint eigenvalues in
the ladder  ${\bf L}_{b, \delta}(\hbar)$.   Then, for any $a \in
S^{0,-\infty }$,  we have that as $\hbar \to 0$:
$$\langle Op_{\hbar}(a) \phi_{b_{j}(\hbar)}, \phi_{b_{j}(\hbar)} \rangle = \left( \sum_{l= 1}^m c_l(\hbar; b_{j}(\hbar)) \right)
\int_{\Lambda^{(j)}(b)} a \,\, d\mu_{\Lambda^{(j)}(b)} + {\mathcal
O}(\hbar^{1-\delta}). $$ \noindent Here, $d\mu_{\Lambda^{(j)}(b)}$
denotes Lebesgue measure on $\Lambda^{(j)}(b)$.
\end{prop}

This result says that a sequence of eigenfunctions of a quantum integrable system localizes on the level
set of the moment map corresponding to the limit point in the image of the clsssical
moment map,  but it is not very precise about the weight attached to each component.

\subsection{\label{LOCALIZENF} Conjugation to normal form around torus orbits}

We sketch the proof of Proposition \ref{LWL} as an example of the use of normal forms.

Let  $\Omega^{(l)}$ be a  small neighbourhood of $\Lambda^{(l)}$
on which there exist action-angle variables $(\theta^{(l)},
I^{(l)})$. Then there exists a `microlocal' conjugation to quantum
Birkhoff normal form. That is, for $\ell=1,...,k$ and $j=1,...,n$,
there exist microlocally unitary  $\hbar$-Fourier integral operators,
$U^{(\ell)}_{b,\hbar}: C^{\infty}(M) \rightarrow C^{\infty}({\Bbb
T}^{n}; {\mathcal L}^{(l\ell)})$, together with  $C^{\infty}$ symbols,
$f_{j}^{(l)}(x;\hbar) \sim \sum_{k=0}^{\infty} f^{(l)}_{jk}(x)
\hbar^{k}$, with $f_{j0}(0)=0$ such that:
\begin{equation} \label{QBNF}
U^{(\ell)*}_{b,\hbar}
f_{j}^{(l)}(P_{1}-b^{(1)},...,P_{n}-b^{(n)};\hbar)U^{(\ell)}_{b,\hbar}
=_{\Omega_{0}^{(l)}} \frac{\hbar}{i} \frac{\partial}{\partial
\theta_{j}}.
\end{equation}
The simple operators on the right side are the normal forms.

Conjugation to  normal form shows that the space of (microlocal)  solutions of equations
\begin{equation}\label{MICRO}
 P_{k} \phi_{b_{j}(\hbar)} =_{\Omega^{(l)}(b)}  b^{(k)}_{j}(\hbar) \phi_{b_j(\hbar)}
\end{equation}
is one-dimensional. Indeed, such solutions are the same as
solutions of
$$f_{k}^{(l)}(P_{1}-b^{(1)},...,P_{n}-b^{(n)};\hbar) \phi_{j} =_{\Omega^{(l)}(b)} f_{k}^{(l)}(b^{(1)}_{j}-b^{(1)},...,b^{(n)}_{j}-b^{(n)};\hbar)  \phi_{j}. $$
After conjugation to  Birkhoff normal form (\ref{QBNF})
the equation becomes
$$ \frac{\hbar}{i} \frac{\partial}{\partial \theta_j} u_{j} = m_{j} \,u_{j}$$
\noindent and the solutions  are just multiples of $\exp [ i ( n + \pi \gamma/4)
\theta ]$, where $\gamma$ is the Maslov index and $n \in \Bbb Z$.
Thus, the joint eigenfunctions $\phi_{b_{j}(\hbar)}$ are given
microlocally by
\begin{equation} \label{CHAR}
\phi_{b_{j}(\hbar)} =_{\Omega^{(l)}(b)} \,
\sqrt{c_{l}(\hbar;b_{j}(\hbar))} \,\, U^{(l)}_{b;\hbar} (
e^{i(n_{j}+\pi \gamma/4) \theta} ).
\end{equation}
 The right sides  of (\ref{CHAR}) are the usual quasimodes or semiclassical Lagrangian distributions [CV2]

We further use the normal form to prove the localization statement. Let
  $\chi_{l}(x,\xi) \in C^{\infty}_{0}(T^{*}M);
l=1,....,m_{cl}(b)$ be a cutoff function which is identically
equal to one on the neighborhood $\Omega^{(l)}(b)$ and vanishes
on $\Omega^{(k)}(b)$ for $k \neq l$. For  $\hbar$ sufficiently
small, we then have
$$ \langle Op_{\hbar}(a) \phi_{b_{j}(\hbar)}, \phi_{b_{j}(\hbar)} \rangle
 = \sum_{l=1}^{m_{cl}(b)} \langle Op_{\hbar}(a) \circ Op_{\hbar}(\chi_{l})
 \phi_{b_{j}(\hbar)}, \phi_{b_{j}(\hbar)} \rangle + {\mathcal O}(\hbar^{\infty}).$$
In each term, we now conjugate to normal form. Recalling the definition (\ref{c_j}) of
the weights $c_j(\hbar, b_j(\hbar))$, we have
\begin{equation}\begin{array}{lll} \label{sec452}
 \langle Op_{\hbar}(a) \circ Op_{\hbar}(\chi_{l}) \phi_{b_{j}(\hbar)},  \phi_{b_{j}(\hbar)} \rangle &
= & c_l(\hbar;b_{j}(\hbar))
 \langle Op_{\hbar}(a) \circ Op_{\hbar}(\chi_{l}) U^{(l)}_{b;\hbar} ( e^{i(n_{j}+\pi \gamma/4) \theta} ) ,
 U^{(l)}_{b;\hbar} ( e^{i(n_{j}+\pi \gamma/4) \theta} ) \rangle \\ & & \\
& = &  c_l(\hbar;b_{j}(\hbar)) \langle U^{(l) * }_{b;\hbar}
Op_{\hbar}(a) \circ Op_{\hbar}(\chi_{l}) U^{(l)}_{b;\hbar}
e^{i(n_{j}+\pi \gamma/4) \theta}  ,    e^{i(n_{j}+\pi \gamma/4)
\theta}  \rangle \\ &  & \\ &= & (2\pi)^{-n} \,
c_{l}(\hbar;b_{j}(\hbar)) \, \left( \int_{\Lambda^{(l)}} a \,
\,d\mu_{l} \right) + {\mathcal O}(\hbar^{1-\delta}),
\end{array} \end{equation}
   In the last step we used the Egorov theorem to recognize $ U^{(l) * }_{b;\hbar}
Op_{\hbar}(a) \circ Op_{\hbar}(\chi_{l}) U^{(l)}_{b;\hbar}$ as a pseudo-differential operator on the standard
torus and then used the calculation in (\ref{MDMTORUS}) to complete the proof. For the details of the remainder
estimate we refer to \cite{TZ2}.

\section{\label{CNC} Concentration and non-concentration for
general $(M, g)$}

In this section we go over results about concentration on, and $L^p$ norms
along,  submanifolds of general Riemannian manifolds. Then we again compare
the results to the quantum integrable case.

\subsection{\label{BGT} $L^p$ norms for restrictions to  submanifolds}

A measure of concentration near a submanifold is given by the  $L^p$ norms of the restrictions
of eigenfunctions to the submanifold. Early results  were obtained by
Reznikov \cite{R} for curves on  hyperbolic surfaces. The most
general results are due to  Burq-Gerard-Tzvetkov \cite{BGT} and
pertain to  submanifolds of general Riemannian manifolds.
 On curves, the estimates in \cite{BGT} are as follows:
 \begin{theo}
  \begin{itemize} \label{bgt} \cite{BGT} Let $||\phi_{\lambda_j}||_{L^2} = 1$. Then
 \item (i) If $\gamma$ is a unit-length geodesic, then
 $$ \int_{\gamma} |\phi_{\lambda_j}(s)|^{p} ds = {\mathcal O}(\lambda_j^{\tilde{\delta}(p)}),$$
 where
\begin{equation} \tilde{\delta}(p)=
\begin{cases}
\tfrac12-\tfrac1p, \quad 4 \le p\le \infty
\\
\frac{1}{4},\quad 2\le p\le 4.
\end{cases}
\end{equation}

 \item(ii) If $\gamma$ is a curve with strictly-positive geodesic curvature,
 $$ \int_{\gamma} |\phi_{\lambda_j}(s)|^{p} ds =
 {\mathcal O}(\lambda_j^{\frac{p}{3} - \frac{1}{3 } }) .$$
  \end{itemize}
  \end{theo}

\subsection{Non-concentration in tubes around hyperbolic closed geodesics}

We now consider results which give upper bounds on the concentration of eigenfunctions
in tubes around hyperbolic closed orbits and other invariant sets. The article \cite{CVP} gave such
an  bound and a general principle was then proved in  \cite{BZ} (see also \cite{Chr}) . It is stated in terms of lower bounds
of eigenfunction mass outside of a tube around a hyperbolic closed geodesic.

\begin{theo} \cite{BZ,Chr} Let $(M, g)$ be a compact Riemannian manifold, and let
$\gamma$ be a hyperbolic closed geodesic. Let $U$ be any tubular
neighborhood of $\gamma$ in $M$.  Then for any eigenfunction
$\phi_{\lambda}$, there exists a constant $C$ depending only on
$U$ such that
$$\int_{M \backslash U} |\phi_{\lambda}|^2 dV_g  \geq \frac{C}{\log \lambda}
||\phi_{\lambda}||^2_{L^2}. $$

More generally, let  $A \in \Psi^0(M)$
be a pseudo-differential orbit whose symbol equals one in a
neighborhood of $\gamma$ in $S^*_g M$ and equals zero outside
another neighborhood. Then for any eigenfunction $\phi_{\lambda}$
$$||(I - A) \phi_{\lambda}||_{L^2} \geq \frac{C}{\sqrt{\log \lambda}}
||\phi_{\lambda}||_{L^2}. $$
\end{theo}

These theorems are proved in \cite{BZ} using a rather abstract {\it observability estimate}:

\begin{theo} Let $P(h)$ be a family of self-adjoint operators on a Hilbert space $\hcal$, with
a fixed domain $\dcal$. Let $\hcal_1$ be a second Hilbert space and suppose there exists a family
of bounded operators $A(h): \dcal \to \hcal_1$ satisfying
$$||u||_{\hcal} \leq \frac{G(h)}{h} \; ||(P(h) + \tau) u ||_{\hcal} + g(h) ||A(h) u||_{\hcal_1}, $$
for all $\tau \in (- b, -a)$ and with $1 \leq G(h) \leq C h^{-N_0}$ for some $N_0$. Let $\chi \in C_c^{\infty}(-b, -a)$.
Then there exist constants $c_0, C_0$ and $h_0 > 0$ so that for any function $T(h)$ satifying $\frac{G(h)}{T(h)} < c_0$,
$$||\chi(P) u||_{\hcal}^2 \leq C_0 \frac{g(h)^2}{T(h)} \int_0^{T(h)} ||A(h) e^{- i t P(h)/h} \chi(P(h)) u||^2_{\hcal_1} dt. $$
\end{theo}

This formulation has applications to non-concentration of eigenfunctions in the Bunimovich stadium as well.

\subsection{Non-concentration around closed geodesics on compact
hyperbolic surfaces}

 We sketch the proof in \cite{CVP} in the
case of hyperbolic closed geodesics on a compact hyperbolic
surface $X = \H \backslash \Gamma$.

Suppose $\gamma$ is a simple closed geodesic whose normal bundle
is trivial. Then it possesses a tubular (collar) neighborhood
$X_0$  which is isometric to a collar around a closed geodesic
$\gamma$  in a hyperbolic surface of revolution $\H \backslash
\Gamma_{\gamma} \simeq S^1_{\theta} \times \R_y$. We noted above
that such a collar is quantum integrable and that the Birkhoff
normal form of $\sqrt{\Delta}$ around $\gamma$ is the hyperbolic
operator $y D_y + D_y y$. Using the explicit formulae for the
eigenfunctions in the model domain, one can calculate the amount
of mass of the model eigenfunction near $\gamma$, and finds that
the mass increases logarithmically in $\lambda$ (cf. \cite{CVP},
Proposition 16). The mass for the singular energy level $E$ is
explicitly given by
\begin{equation} \label{modelmassa}\begin{array}{l}
\int_0^{\frac{1}{h}} \frac{d\eta}{\eta} \left| \int_0^{\infty}
e^{- i (y - \frac{E}{h} \ln y)} \chi(\frac{y}{\eta}) \frac{d
y}{\sqrt{y}} \right|^2 \sim \left| \Gamma(\frac{1}{2} + i
\frac{E}{h}) \right|^2 e^{\frac{E\pi}{h}} |\ln h|.  \end{array}
\end{equation}
The mass of the sequence of eigenmodes of the hyperbolic surface
of revolution with boundary grows at the same logarithmic rate,
since it is microlocally unitarily equivalent to the model
eigenfunctions.

Now consider the eigenfunctions of the compact hyperbolic surface
without boundary. Unlike the hyperbolic cylinder with boundary,
its geodesic flow is ergodic and far from completely integrable.
Let $\phi_{\lambda_j}$ denote a sequence of eigenfunctions, and
let
$$\phi_{\lambda_j}(y, \theta) = \sum_k a_{k; \lambda_j}(y) e^{i k \theta}$$
be the Fourier series expansion of $\phi_{\lambda_j}$ for each $y$. Let $||a_{k; \lambda_j}||_{[-1, 1]}$ denote the $L^2$
norm of the $k$ component in $[-1, 1]$ and let $||\phi_j||_{X_0}$ denote the $L^2$ norm in the collar. Then we have
$$||\phi_{\lambda_j}||_{X_0}^2 = \sum_k ||a_{k; \lambda_j}||_{[-1, 1]}^2. $$

Suppose now that the sequence $\{\phi_{\lambda_j}\}$ concentrates in $X_0$ in the sense that
$$||\phi_{\lambda_j}||_{X_0}^2 \geq C_0 > 0, \;\; ||\phi_{\lambda_j} ||_{K}^2 \to 0, \;\; \forall K \subset \subset X \backslash \gamma. $$
Then,
\begin{equation} \label{SERIESLB} \sum_k ||a_{k; \lambda_j}||_{[-1, 1]}^2 \geq C_0 > 0 \end{equation}
uniformly in $j$. Then the only possible quantum limit measures of the sequence must have the form $C \mu_{\gamma} + \nu$
 where $\nu = 0$ in the tube around $\gamma$ and $C > 0$.  For a time-reversal symmetric symbol $a$ with essential
 support in the collar,  we must have
 $$\langle Op(a)  \phi_{\lambda_j}, \phi_{\lambda_j} \rangle \to C \int_{\gamma} a_0 ds.$$
 Now let $P = \Delta^{-1} \frac{\partial^2}{\partial \theta^2} + 1$, a pseudo-differential operator of order $0$ whose
principal symbol equals $1 - \frac{p_{\theta}^2}{|\xi|_g^2}$. The symbol vanishes along $\gamma$.
 Then
 \begin{equation} \label{SERIES0} \int_{X_0}  \phi_{\lambda_j} P \phi_{\lambda_j} dV =  \sum_k (1 - \frac{k^2}{\lambda_j^2})  ||a_{k; \lambda_j}||_{[-1, 1]}^2 \to 0. \end{equation}
 Let $Z_j = \{k :  \; |\frac{k^2}{\lambda_j^2} - 1| \leq \frac{1}{2} \}.$ Combining (\ref{SERIESLB}) and (\ref{SERIES0}), we have
  \begin{equation} \label{SERIESB} \sum_{k \in Z_j} ||a_{k; \lambda_j}||^2 \geq C >
  0.\end{equation}
 Thus,  the Fourier coefficients of a
  sequence of eigenfunctions concentrating near $\gamma$  concentrate around  the joint eigenvalues of the
  modes on the hyperbolic surface of revolution which concentrate
  on the central geodesic. The mass around $\gamma$ of the associated modes can
  only grow logarithmically must charge  the complement of any
  collar around $\gamma$ with mass at least of order
  $\frac{1}{\log \lambda}$. It follows from (\ref{SERIESB}) that
  any sequence of eigenfunctions on the compact surface without
  boundary must also charge the complement of the collar by at
  least  $\frac{1}{\log \lambda}$.

\section{\label{LPQCI} $L^p$ norms and concentration  in the Quantum integrable case}

$L^p$ norms of eigenfunctions in the quantum integrable case have a much more geometric theory than in general.
Explicit examples show that the they are often extremals for $L^p$ norm and concentration inequalities.

One extremal problem raised in \cite{Y1,Y2} is to determine the
 Riemannian manifolds which possess
orthonormal bases of eigenfunctions with uniformly bounded
$L^{\infty}$ norms. An obvious example is a flat torus. The
question arises whether any others exist. The following result
shows that flat tori are the unique minimizers in the class of
quantum integrable Laplacians.

\begin{theo} \label{RM} \cite{TZ}  Suppose that $\Delta$ is a quantum completely
integrable Laplacian on a compact Riemannian manifold $(M, g)$.
Then
\medskip

(a) If $L^{\infty}(\lambda, g) = O(1)$  then $(M,g)$ is flat.

(b) If $\ell^{\infty}(\lambda, g) = O(1)$,  then $(M,g)$ is flat.

\end{theo}

It is possible that flat tori are the only
compact Riemannian manifolds with a uniformly bounded orthonormal basis of
eignfunctions. But at this time, it is not even known if the standard sphere possesses such
an orthonormal basis. The idea of the proof of Theorem \ref{RM} is that the joint eigenfunctions
concentrate on level sets of the moment map, and therefore develop singularities at points where the
projection of the level set to the base manifold is singular. The only case where no singularities occur
is when $(M, g)$ is a torus  without conjugate points, and in this case  Burago-Ivanov proved (the Hopf conjecture) that
$(M, g)$ must be flat.

There also exists a quantitative improvement of Theorem \ref{RM} which gives blow-up rates for $L^p$ norms
for quantum integrable eigenfunctions concentrating on singular level sets, i.e. level sets which
are not regular in the sense of (\ref{CI1}). These eigenfunctions are the extremals
for $L^p$ blow-up and mass concentration.  In
the following there is an additional technical assumption
(Eliasson non-degeneracy) which we omit for simplicity.

\begin{theo} \label{LP} \cite{TZ2} Suppose that $(M, g)$ is a compact  Riemannian manifold
whose  Laplacian $\Delta$ is quantum completely integrable as in
(\ref{QCI}).  Then, unless  $(M, g)$ is a flat torus, this action
must have a singular orbit of dimension $<n$. If the minimal
dimension of the singular orbits is $ \ell$, then
    for every $\epsilon >0$, there exists  a sequence of
eigenfunctions satisfying:
$$\left\{ \begin{array}{l}   \| \phi_{k} \|_{L^{\infty}} \geq C(\epsilon) \lambda_{k}^{\frac{n - \ell}{4} -
\epsilon}. \\ \\
 \| \phi_{k} \|_{L^{p}} \geq C(\epsilon) \lambda_{k}^{ \frac{(n -
\ell) ( p - 2)}{ 4p } - \epsilon }, \,\,\, 2 < p < \infty.
\end{array} \right.$$
   \end{theo}

Here,

\begin{itemize}

\item A point $(x, \xi)$ is called a singular point of the moment map ${\mathcal
P}$ (\ref{MM})  if  $dp_{1} \wedge \cdot \cdot \cdot \wedge dp_{n}(x,\xi) =
0$.

\item  A level set  ${\mathcal P}^{-1}(c)$ of the moment map is called  a singular level if
it contains a singular point $(x, \xi) \in {\mathcal P}^{-1}(c)$.

\item An orbit $\R^n \cdot (x, \xi) $ of $\Phi_{t}$ (\ref{PHIT}) is  singular if it is non-Lagrangean, i.e.   has dimension $<n$;

\end{itemize}

The  idea in the proof is to consider singular orbits and to  conjugate to a quantum Birkhoff normal form around
the orbit. Hence, one calculates the mass   in the normal form space.  Conjugation to normal form does not preserve
$L^p$ norms, and what  is really calculated are not $L^p$
norms but rather $L^2$ norms in shrinking tubes. Interestingly, this method produces optimal results.
The proof  does not determine the minimal  dimension $\ell$. By
taking products of lower dimensional manifolds,  it is easy to
construct examples with any value of $\ell = 1, \dots, n - 1$.

\subsection{Mass concentration on  small length scales}

We sketch the proof  of  Theorem \ref{LP} as an illustration of mass estimates in shrinking tubes.
We follow the semi-classical notation $\hbar = \lambda^{-1}$ of \cite{TZ2}.

 Let $\Lambda := {\Bbb R}^{n} \cdot v$ be a compact, $k < m$-dimensional singular orbit of the
 Hamiltonian
${\Bbb R}^{n}$-action generated by $(p_{1},...,p_{n}).$ In this
section, we study mass concentration of modes in shrinking tubes
of radius  $\sim \hbar^{\delta}$ for $0< \delta < 1/2$ around
$\pi(\Lambda)$ in $M$, where $\pi: T^{*}M \longrightarrow M$
denotes the canonical projection map.

We denote by $T_{\epsilon}(\pi(\Lambda))$ the set of points of
distance $< \epsilon$ from $\pi(\Lambda)$. For $0 < \delta < 1/2$,
we introduce a  cutoff $\chi_{1}^{\delta} (x;\hbar) \in
C^{\infty}_{0}(M)$ with $0 \leq \chi_{1}^{\delta} \leq 1,$
satisfying
\begin{itemize} \label{tube3}
\item (i) supp $\chi_{1}^{\delta} \subset T_{\hbar^{\delta} }
(\pi(\Lambda))$ \item (ii)
 $\chi_{1}^{\delta} = 1$ on $ T_{3/4 \hbar^{\delta}           }  (\pi(\Lambda))$.
\item (iii) $|\partial_x^{\alpha} \chi_1^{\delta}(x; \hbar)|\leq
 C_{\alpha} h^{- \delta |\alpha|}. $
\end{itemize}
Under the assumption that  $\Lambda$ is an embedded submanifold of
$M$, the functions
\begin{equation} \label{cutoff1}
\chi_{1}^{\delta}(x;\hbar) = \zeta_{1}( \hbar^{-2\delta}
d^{2}(x,\pi(\Lambda)) )
\end{equation}
are smooth on  $ T_{\epsilon } (\pi(\Lambda)) $ and satisfy the
conditions. Here,  $d(.,.)$ is the Riemannian distance function.
Also, $\zeta_{1} \in C^{\infty}_{0}({\Bbb R})$ with $0 \leq
\zeta_{1} \leq 1, \zeta_{1}(x)=1$ for $|x| \leq 3/4$ and supp
$\zeta_{1} \subset (-1, 1).$

\begin{theo} \label{SSM} Let $\phi_{\mu} \in V_{c}(\hbar)$ satisfy the bounds in Lemma \ref{LAMBDA}. Then for any   $ 0 \leq \delta <1/2$,
$ (Op_{\hbar}(\chi_{1}^{\delta}) \phi_{\mu}, \phi_{\mu}) \gg |\log
\hbar|^{-m}.$
\end{theo}

We briefly sketch the proof.
 Let  $\chi_{2}^{\delta}(x,\xi;\hbar)
\in C^{\infty}_{0}(T^{*}M ; [0,1])$ be a
 second cutoff supported in a radius $\hbar^{\delta}$ tube,  $\Omega (\hbar)$, around $\Lambda$ with
$
\Omega(\hbar) \subset supp \chi_{1}^{\delta}$
 and such that
$
\chi_{1}^{\delta}  = 1 \,\,\mbox{on} \,\, supp \chi_{2}^{\delta}.
$
Then, clearly
\begin{equation} \label{dom1}
\chi_{1}^{\delta}(x,\xi) \geq \chi_{2}^{\delta}(x,\xi),
\end{equation}
for any $(x,\xi) \in T^{*}M$. By Garding's    inequality, (\ref{dom1})
 implies
\begin{equation} \label{dom2}
 ( \, Op_{\hbar}(\chi^{\delta}_{1}) \phi_{\mu}, \phi_{\mu} \, ) \gg ( \, Op_{\hbar}(\chi^{\delta}_{2}) \phi_{\mu}, \phi_{\mu} \, ).
\end{equation}

We now conjugate the right side to the model by the $\hbar$-
Fourier integral operator $F$ of Lemma (\ref{QBNF}).  Since $F$ is a microlocally elliptic
$\hbar$-Fourier integral operator associated to a canonical
transformation $\kappa$, it follows by Egorov's theorem
\begin{equation} \label{bound}
(Op_{\hbar}(\chi_{2}^{\delta}) \phi_{\mu}, \phi_{\mu}) =
|c(\hbar)|^2  (Op_{\hbar}(\chi_{2}^{\delta} \circ \kappa) u_{\mu},
u_{\mu}) - C_{3} \hbar^{1-2\delta}
\end{equation}
\noindent where $c(\hbar) u_{\mu} (y,\theta;\hbar)$ is the
microlocal normal form  for the eigenfunction
$\phi_{\mu}$.  Since $\phi_{\mu} \in V_{c}(\hbar)$ satisfies the
bounds in Lemma (\ref{LAMBDA}) it follows that $|c(\hbar)| ^{2} \gg |\log \hbar|^{-m}$ and from (\ref{bound}) we are left with estimating the matrix elements $(Op_{\hbar}(\chi_{2}^{\delta} \circ \kappa) u_{\mu}, u_{\mu})$ from below.
  As in the localization proof, the matrix elements are now in terms of elementary model eigenfunctions
  and the calculation has become easy. The normal form eigenfunctions separate into a product of factors
  and one only has to calculate one (or two) dimensional integrals. As an example, in  the hyperbolic case
the integral has the form
\begin{equation} \label{modelmass1}\begin{array}{l}
M_{h} = \frac{ 1} { \log
\hbar} \, \left( \int_{0}^{\infty} \chi(\hbar \xi/\hbar^{\delta})
\left| \int_{0}^{\infty} e^{-ix} x^{-1/2 + i \lambda/\hbar}
\chi(x/\hbar^{\delta}\xi) dx \right|^{2} \frac{d\xi}{\xi} \right) \\ \\
\geq
\frac{1}{C_{0}} \, (\log \hbar)^{-1} \, \int_{0}^{\hbar^{\delta
-1} } \frac{d\xi}{\xi} \left| \int_{0}^{\hbar^{\delta} \xi}
e^{-ix} x^{-1/2 + i \lambda/\hbar} dx \right|^{2} + {\mathcal
O}(|\log \hbar|^{-1})\\ \\ \gg \,  |\Gamma (1/2 + i \lambda/\hbar)|^{2}\, (1-2\delta) +
{\mathcal O}(|\log \hbar|^{-1})  \geq C(\epsilon) >0\end{array} \end{equation}
 uniformly
for $\hbar \in (0,\hbar_{0}(\epsilon)]$.

\subsubsection{Completion of the proof of Theorem \ref{LP}}

The small scale mass estimates immediately imply lower bounds on $L^{\infty}$ norms and $L^p$
norms due to the shrinking volumes of the tubes. For instance,
\begin{equation} \label{IES} \begin{array}{lll}
\int_{M}  | \phi_{\mu}(x)  |^{2} \chi^{\delta}_{1} (x;\hbar ) \,
\, d vol(x) & \leq &  \; \sup_{x \in T_{h^{2
\delta}}(\pi(\Lambda))} | \phi_{\mu}(x)  |^{2} \int_{M}
\chi^{\delta}_{1} (x;\hbar ) \, \, d vol(x) \\ & & \\ & \leq &  \|
\phi_{\mu} \|^{2}_{L^{\infty}} \cdot  \int_{M}
\chi^{\delta}_{1} (x;\hbar ) \, \, d vol(x)
\end{array} \end{equation} and  it follows  from Lemma \ref{SSM} that
\begin{equation} \label{lower2}   \| \phi_{\mu} \|^{2}_{L^{\infty}} \cdot \left(
\int_{M} \chi^{\delta}_{1} (x;\hbar ) \, \, d vol(x) \right) \geq
C(\epsilon) |\log \hbar|^{-m} ,\end{equation}
 \noindent uniformly for  $\hbar \in (0,\hbar_{0}(\epsilon)]$.
Since \begin{equation} \label{TUBEINT} \int_{M}
\chi_{1}^{\delta}(x;\hbar) \,\, d vol(x) = {\mathcal
O}(\hbar^{\delta(n-\ell)}), \end{equation}
 (\ref{lower2}) implies
$$ \| \phi_{\mu} \|^{2}_{L^{\infty}} \geq C(\epsilon) \hbar^{-\frac{1}{2}(n- \ell) + \epsilon} |\log \hbar|^{-m}.$$
Recalling that  $\hbar^{-1} \in \{ \lambda_{j}; \lambda_{j} \in Spec -\sqrt{\Delta} \} $, this gives:
$$\| \phi_{\lambda_{j} } \|_{L^{\infty}} \geq C(\epsilon)
\lambda_{j}^{\frac{n-\ell}{4} - \epsilon}.$$

\subsubsection{Concentration of quantum integrable eigenfunctions on submanifolds}

Similar methods were used in   \cite{To2} to obtain   sharp  bounds on $L^2$ norms
for restrictions to submanifolds in the quantum integrable case, making more precise the
results of \cite{BGT} in this special case.  For simplicity, let us consider
curves on surfaces.  First is the generic upper bound:

\begin{theo} \cite{To2} \label{corollary}
Let $\phi_{\lambda_j}; j=1,2,3,...$ be the $L^{2}$-normalized joint
Laplace eigenfunctions of the commuting operators $P_{1}= -\Delta
$ and $P_{2}$ on a Riemannian surface $(M^{2},g)$. Then for a generic curve $\gamma$ such that
$\iota^{*}p_{2}|_{S^*_{\gamma}M}$ is Morse, we have
$$ \int_{\gamma} |\phi_{\lambda_j}|^{2} ds = {\mathcal O}_{|\gamma|} \left( \log \lambda_j \right). $$
\end{theo}

When  the curve is a geodesic, the bounds depend on the type of level set the geodesic lies on:

\begin{theo} \label{mainthm2} \cite{To2}
Let  $P_j(\hbar);  j=1,2$ be a non-degenerate quantum integrable system
system on a surface, $(M,g)$. Then,
\begin{itemize}
\item (i) \, When $\gamma$ is the projection of a geodesic
segment  contained in $ {\mathcal P}^{-1}({\mathcal
B}_{reg}),$
$$ \int_{\gamma}|\phi_{\lambda_j}(s)|^{2} ds = {\mathcal O}_{|\gamma|}(1),$$
\item (ii) \, When $\gamma$ is the projection of a singular joint
orbit in ${\mathcal P}^{-1}({\mathcal B}_{sing}),$
$$ \int_{\gamma} |\phi_{\lambda_j}(s)|^{2} ds = {\mathcal O}_{|\gamma|}(\lambda_j^{1/2} ).  $$
\noindent Moreover,
 there exists a constant $c_{\gamma} >0$ depending only on the curve $\gamma,$ and a subsequence of joint eigenfunctions, $\phi_{\lambda_{j_k}=}; k=1,2,...$ such that
 $$ \int_{\gamma} |\phi_{\lambda_{j_k}}(s)|^{2} ds \geq  c_{\gamma} \lambda_{j_k}^{1/2}  \,\,\, \mbox{when} \, \gamma \,\mbox{ is stable},$$
$$ \int_{\gamma} |\phi_{\lambda_{j_k}}(s)|^{2} ds \geq c_{\gamma} \lambda_{j_k}^{1/2} |\log \lambda_{j_k}|^{-1}\,\,\, \mbox{when} \, \gamma \, \mbox{is unstable}.$$
\end{itemize}
\end{theo}

Thus the exact bound depends on the nature
 of the geodesic.
In the general quantum integrable  case, most geodesics lie on regular Lagrangian tori in ${\mathcal P}^{-1}(B_{reg})$ and these
 geodesics do not support large $L^{2}$-bounds. But as in Theorems \ref{RM} and \ref{LP},
  there always exists  a subsequence of joint
eigenfunctions of $P_1$ and $P_{2}$ with mass concentrated along
(singular) orbits  contained in ${\mathcal P}^{-1}({\mathcal B}_{sing})$, and the
 associated  eigenfunctions saturate the upper  bounds. For instance
 in the case of a simple  surface of revolution,  the equator is the projection
of a singular orbit of the $\R^2$ action generated by geodesic flow and rotation.
 The  corresponding joint eigenfunctions (the analogs of highest weight spherical harmonics) satisfy
$\int_{\gamma} |\phi_{\lambda_j}|^{2} ds
 \sim \lambda_j^{1/2}$ along the equator, $\gamma$.  The  equatorial geodesic is
singular and the $L^2$ norms along it had singular blowup. In the case of the
 meridian great circles, the closed
geodesic lies in the base space projection of a maximal Lagrangian
torus. The zonal harmonics have $\hbar$-microsupport on this torus
and have $L^2$-restriction bound $\sim \log \lambda$
along
 any meridian great circle.

\section{\label{QuE} Delocalization in quantum ergodic systems, I}

In this section, we discuss general  results on eigenfunctions
when the geodesic flow of  $(M, g)$ is assumed to be ergodic (see  \S \ref{BASIC} and \S \ref{QULIMITS}
for  definition and notation). The study of eigenfunctions and eigenvalues
of Laplacians on manifolds with ergodic (or more highly mixing) geodesic flows
is generally known as `quantum chaos'.  The basic question is, what impact do dynamical
properties of the geodesic flow $g^t$ have on eigenvalues and eigenfunctions of its quantization
$U_t$?  This question has been studied over the last three decades  by a large collection of mathematicians
and physicists, using both theoretical and computational methods.  In this section, we largely follow
our recent survey \cite{Z3}. Another exposition with emphasis on arithmetic hyperbolic quotients is \cite{Sar2}.
For recent computational results, we refer to \cite{Bar}.

One of the basic and most studied problem is Problem \ref{Q} for quantizations of classically ergodic systems.
The main result is   that there exists a subsequence $\{\phi_{j_k}\}$ of
eigenfunctions whose indices $j_k$ have counting density one for
which $\rho_{j_k}(A): = \langle A \phi_{j_k}, \phi_{j_k}\rangle \to \omega(A)$ (where as above
$\omega(A) = \frac{1}{\mu(S^*M)} \int_{S^*M} \sigma_A d\mu $ is the normalized Liouville average of $\sigma_A$).  Such a
 sequence of  eigenfunctions is called a sequence of  `ergodic
eigenfunctions'. The key quantities to study are the quantum variances
\begin{equation} \label{diag} V_A(\lambda) : =
\frac{1}{N(\lambda)} \sum_{j:  \lambda_j \leq \lambda} |\langle A
\phi_j, \phi_j \rangle - \omega(A)|^2.
\end{equation}

\begin{theo} \label{QE} \cite{Sh.1,Sh.2,Z2,CV,Su,ZZw,GL,Z3}   Let $(M,g)$ be a compact
Riemannian manifold (possibly with boundary), and let
$\{\lambda_j, \phi_j\}$ be the spectral data of its Laplacian
$\Delta.$ Then the geodesic flow
 $G^t$ is ergodic  on $(S^*M,d\mu)$ if and only if, for every
$A \in \Psi^o(M)$,  we have:
\medskip

\begin{enumerate}

 \item $\lim_{\lambda \rightarrow \infty} V_A(\lambda) =0.$
\medskip

 \item $(\forall \epsilon)(\exists \delta)
\limsup_{\lambda \rightarrow \infty} \frac{1} {N(\lambda)}
\sum_{{j \not= k: \lambda_j, \lambda_k \leq \lambda}\atop {
|\lambda_j - \lambda_k| < \delta}} |( A \phi_j, \phi_k )|^2 <
\epsilon $
\end{enumerate}

\end{theo}

Since all the terms in (1)
are positive, no cancellation is possible, hence  (1)  is
equivalent to the existence of a subset ${\mathcal S} \subset \N$
of density one such that ${\mathcal Q}_{{\mathcal S}} := \{ d
\Phi_k : k \in {\mathcal S}\}$ has only $\omega$ as a weak* limit
point.

As explained in  \cite{Z9},
this ergodicity of eigenfunctions may be viewed as  a convexity theorem:   By the Banach-Alaoglu theorem,
 the set of invariant probability measures for the geodesic
flow is a compact convex set $\mcal_I$.
An invariant measure is ergodic if it is   an extreme point of
the compact convex set. The same is true on the quantum level: The set of  ${\mathcal E}_{\R}$ of invariant states for
$\alpha_t$ is a convex set. A classical invariant measure is an invariant state, and if it is ergodic
classically it is also ergodic quantum mechanically, i.e. it is an extreme point of ${\mathcal E}_{\R}$.
Hence Liouville measure $\omega$  is an extreme point of this convex set.
 But the local Weyl law says that $\omega$ is  the limit of the
convex combination $\frac{1}{N(E)} \sum_{\lambda_j \leq E}
\rho_j.$  An extreme point cannot be written as a convex
combination of other states unless all the states in the
combination are equal to it. In our case, $\omega$ is only a limit
of an infinite sequence of convex combinations, and the result is that   almost all terms in the sequence  tend to
$\omega$, and that is equivalent to (1).
\medskip

\noindent{\bf Sketch of Proof of (1)}~~~

Let  \begin{equation} \langle A \rangle_T : = \frac{1}{2T} \int_{-T}^T
U_t^* A U_t dt. \end{equation}
Then,
\begin{equation} \sum_{\lambda_j \leq \lambda} |\langle A \phi_{\lambda_j} ,  \phi_{\lambda_j} \rangle
- \omega(A)|^2  = \sum_{\lambda_j \leq \lambda} |\langle \langle A
\rangle_T - \omega(A)  \phi_{\lambda_j},  \phi_{\lambda_j}  \rangle |^2. \end{equation}  Apply the Schwartz inequality
for states,
$$\sum_{j: \lambda_j \leq \lambda}  |(B \phi_{\lambda_j}, \phi_{\lambda_j}))|^2  \leq {\rm Tr\,} \Pi_{[0, \lambda]} B^*B,$$
where  $\Pi_{[0, \lambda]}$ is the
spectral projection for $\sqrt{\Delta}$ corresponding to the interval
$[0, \lambda]$, to the operator
 $B = \Pi_{[0, \lambda]}
 [\langle A \rangle_T - \omega(A)] \Pi_{[0, \lambda]} $. We then have
\begin{equation}\begin{array}{lll}  \sum_{\lambda_j \leq \lambda} |\langle \langle A \rangle_T -
\omega(A)  \phi_{\lambda_j},  \phi_{\lambda_j}\rangle  |^2 &  \leq & {\rm Tr\,} (\Pi_{[0, \lambda]}
[\langle A \rangle_T - \omega(A) \Pi_{[0, \lambda]} )^*
[\Pi_{[0, \lambda]} \langle A \rangle_T - \omega(A)] \Pi_{[0, \lambda]} ) \\ && \\
 &  \leq & {\rm Tr\,} (\Pi_{[0, \lambda]}
[\langle A \rangle_T - \omega(A))^*
[\langle A \rangle_T - \omega(A)] \Pi_{[0, \lambda]} )\\ && \\
& = &   \omega((\langle A
\rangle_T - \omega(A))^*(\langle A \rangle_T - \omega(A))).
 \end{array} \end{equation}
Here,  we used the Jensen  inequality,
$$\frac{1}{N(\lambda)} {\rm Tr\,} \phi (\Pi_{[0, \lambda]} [\langle A \rangle_T - \omega(A)] \Pi_{[0, \lambda]} ) \leq
\frac{1}{N(\lambda)}  {\rm Tr\,} \Pi_{[0, \lambda]} \phi ([\langle A \rangle_T -
\omega(A)]) \Pi_{[0, \lambda]}, $$
 valid for any convex function $\phi$ in the case $\phi(x) = x^2$.
  By the local Weyl law,  we get
$$\lim_{\lambda \rightarrow \infty}\frac{1}{N(\lambda)} \sum_{\lambda_j \leq \lambda} |
(\langle A  \phi_{\lambda_j},  \phi_{\lambda_j}  \rangle  - \omega(A))^2  \leq \int_{S^*M } |\langle \sigma_A \rangle_T - \omega(A))|^2
d\mu.$$ As $T \rightarrow \infty$ the right side
approaches $0$ by the dominated convergence theorem and by
Birkhoff's ergodic theorem.  Since the left hand side is
independent of $T$, this implies the stated theorem. \qed

\subsection{Quantum ergodicity in terms of operator time and space
averages}

To explain the term `quantum ergodicity', we reformulate the result in terms of space
and time averages. We assume for simplicity the generic condition that all eigenvalues
are of multiplicity one (cf. \cite{U}). The space and time averages are defined as follows:

\noindent{\bf Definition}~~~{\it Let $A \in \Psi^0$ be an
observable and define its time average to be:
 $$\langle A \rangle := w- \lim_{T \rightarrow \infty} \frac{1}{2T} \int_{-T}^T
U_t^* A U_t dt$$ and its space average to be scalar operator
$$\omega (A) \cdot I$$ }
The  limit is taken in the weak operator topology, i.e. in the sense of matrix elements.
 To verify it we observe that
$$( \frac{1}{2T} \int_{-T}^T U_t^* A U_t dt \phi_i, \phi_j) =
\frac{\sin  T(\lambda_i - \lambda_j)}{ T(\lambda_i - \lambda_j)}
(A \phi_i, \phi_j),$$ hence the
matrix element tends to zero as $T \rightarrow \infty$ unless
$\lambda_i = \lambda_j$.  The limit only occurs in the weak sense since the rate   is clearly not uniform,
as the spacings $\lambda_i -
\lambda_j$ could be arbitrarily small.

Quantum ergodicity can thus be reformulated as the condition,
\begin{equation} \langle A \rangle = \omega(A) I +
K,\;\;\;\;\;\;\mbox{where}\;\;\;\;\; \lim_{\lambda \rightarrow
\infty} \omega_{\lambda}(K^*K) \rightarrow 0, \end{equation} where
$\omega_{\lambda}(A) = Tr E(\lambda) A. $ Thus, the time average
equals the space average moduli a term $K$ whose Hilbert-Schmidt norm
in the range of $\Pi_{\lambda}$ is $o(N(\lambda)).$
Note that  $\langle A \rangle$ commutes with
$\sqrt{\Delta}$,  hence is diagonal in the basis $\{\phi_j\}$ of
joint eigenfunctions of $\langle A \rangle$ and of $U_t$. Also,
$K$ is the diagonal matrix with entries $\langle A \phi_k,
\phi_k\rangle - \omega(A)$. The condition is therefore equivalent
to
$V_A(\lambda) \to 0$

\subsection{Quantum unique ergodicity and converse quantum ergodicity}

A Laplacian is said to be QUE (quantum uniquely ergodic) if Liouville measure
is the only weak * limit point of the semi-classical Wigner measures.  The terminology was
introduced in \cite{RS}.

The condition
may be reformulated in terms of $K$: Namely, QUE is equivalent to the compactness of $K$.
Indeed, this  would imply that $\langle K \phi_k,  \phi_k \rangle \to
0$, hence $\langle A \phi_k, \phi_k \rangle \to \omega(A)$ along
the entire sequence.  A key  difficulty of settling the question whether $K$ is compact is
that the time averaged operator $\langle A \rangle$ no longer belongs to the class of pseudo-differential
operators, due to the very weak nature of the weak operator limit.

 It is widely conjectured there exist non QUE $(M, g)$ with ergodic geodesic flow. We refer
 to \cite{Bar,BZ2,Z9} for some recent discussions and references to the literature.
  The simplest example is that of the Bunimovich stadium, which possesses
quasi-modes concentrating   on
the invariant Lagrangian cylinder (with boundary) formed by
bouncing ball orbits in the middle rectangle. An analogue among Riemannian manifolds without
boundary is that of a non-positively curved surface with a flat cylindrical part and ergodic geodesic flow  \cite{Don1},
which also carries product quasi-modes of the same kind.    The existence of such quasi-mode
suggests that there are nearby modes which are not quantum ergodic (see \cite{Z9}
for the precise suggestion).    Faure-Nonnenmacher-de
Bi\`evre \cite{FNB} have shown that QUE does not hold for the hyperbolic
system defined by a quantum cat map on the torus, and since
the  methods available for studying  eigenfunctions of  quantum maps and of
Laplacians are very similar,
 this negative result shows that there cannot exist a universal structural proof of
 QUE. A QUE result has been proved by
E. Lindenstrauss, namely the QUE property for the orthonormal basis
of Laplace-Hecke eigenfunctions eigenfunctions on arithmetic
hyperbolic surfaces (see \cite{LIND}).

To the author, an interesting and almost completely open problem is the converse of quantum ergodicity: does quantum ergodicity
imply classical ergodicity? More precisely, are there natural sufficient conditions for this? For instance, if
$\langle A \rangle
= \omega(A) + K$ where $K$ is compact, is the geodesic flow ergodic?
Very  little is known on this converse problem at present. One may imagine a system which is
non-ergodic but in which `tunnelling' causes eigenfunctions to become uniformly distributed in phase space.
A simple model is the one-dimensional  Schr\"odinger operator with the  even double well potential $V(x) = (1 - x^2)^2$.
Level sets  $\xi^2 + V(x) = E$ with $E$ near zero have two components and hence the Hamiltonian flow is not ergodic on
this level; yet the eigenfunctions are either even or odd and hence they are quantum ergodic. No such example is known
in the case of Laplacians  $\Delta$ of compact Riemannian manifolds. However, the example shows that no general `abstract' proof
of $QE \implies CE$ is possible.

 In \cite{Z5,MOZ} it is shown
that if there exists an  open set in $S^*M$ filled by
periodic orbits, then the Laplacian cannot be quantum ergodic. But it has not even
been proved that
 KAM systems, which have Cantor-like positive measure  invariant sets, are not
 quantum ergodic although there exists a positive density of quasi-modes concentrating
 on invariant tori \cite{Pop}. The problem is to prove that eigenfunctions are linear combinations
 of not too many quasi-modes.

\subsection{\label{QWMS} Quantum weak mixing}

Quantum weak mixing concerns the off-diagonal matrix elements.

\begin{theo} \label{QWM} (see \cite{Z3} for references)  The geodesic flow  $\Phi^t$ of $(M, g)$
is weak mixing if and only if the conditions (1)-(2) of Theorem
\ref{QE}  hold and additionally, for any $A \in \Psi^o(M)$,
$$(\forall \epsilon)(\exists \delta)
\limsup_{\lambda \rightarrow \infty} \frac{1} {N(\lambda)}
\sum_{{j\not= k: \lambda_j, \lambda_k \leq \lambda}\atop {
|\lambda_j - \lambda_k-\tau| < \delta}} |( A \phi_j, \phi_k )|^2 <
\epsilon \;\;\;\;\;\;\; (\forall \tau  \in \R )$$
\end{theo}

The restriction $j\not =k$ is of course redundant unless $\tau =
0$, in which case the statement coincides with quantum ergodicity.
This result follows from the general asymptotic formula, valid for
any compact Riemannian manifold $(M, g)$, that \begin{equation}
\label{QMF} \begin{array}{l}  \frac{1}{N(\lambda)}  \sum_{i \not=
j, \lambda_i, \lambda_j \leq \lambda} |\langle A \phi_i, \phi_j
\rangle|^2 \left|\frac{\sin T(\lambda_i -\lambda_j -\tau)}
{T(\lambda_i -\lambda_j -\tau)}\right|^2 \\ \\
 \sim ||\frac{1}{2T}
\int_{- T}^T e^{i t \tau} V_t(\sigma_A) ||_2^2 - |\frac{\sin T
\tau}{T \tau}|^2 \omega(A)^2. \end{array} \end{equation}  In the
case of weak-mixing geodesic flows, the right hand side $\to 0$ as
$T \to \infty$.

\subsection{Spectral measures and matrix elements}

Theorem \ref{QWM}  is based on expressing the spectral measures of
the geodesic flow in terms of matrix elements. The main limit
formula is:

\begin{equation} \label{SPECMEAS} \int^{\tau +\varepsilon }_{\tau-\varepsilon
} d\mu_{\sigma_A}:=\lim_{\lambda \rightarrow
\infty}\frac{1}{N(\lambda )}\sum_{i,j: \;\;\lambda_i,  \lambda _j\leq
\lambda,  \;\; |\lambda _i-\lambda _j-\tau|<\varepsilon
\\}\;
 |\langle A\varphi_i,
\varphi_j \rangle|^2\;\;,  \end{equation} where $d\mu_{\sigma_A}$
is the spectral measure for the geodesic flow corresponding to the
principal symbol of $A$,  $\sigma_A \in C^{\infty} (S^*M, d\mu)$.
Recall that the spectral measure of $V_t$ corresponding to $f\in
L^2$ is the measure $d\mu_f$ defined by
$$\langle V_tf,f \rangle_{L^2(S^*M)}  = \int_{\R} e^{\oit\tau} d\mu_f(\tau)\;.$$

The limit formula (\ref{SPECMEAS})  is equivalent to the dual
formula (under the Fourier transform)
\begin{equation} \label{QM2}\lim_{\lambda \to \infty} \frac{1}{N(\lambda)}   \sum_{i, j: \lambda_j \leq \lambda}
e^{i t (\lambda_i - \lambda_j)} |\langle A \phi_i, \phi_j
\rangle|^2 = \langle V_t \sigma_A, \sigma_A \rangle_{L^2(S^*M)}.
\end{equation}
The proof of (\ref{QM2}) is to consider, for $A\in \Psi^\circ$,
 the operator $A^*_tA\in \Psi^\circ$ with $A_t =
U^*_tAU_t$. By the local Weyl law,
$$\lim_{\lambda \rightarrow \infty}\frac{1}{N(\lambda )}
\Tr E(\lambda) A^*_tA = \langle V_t \sigma_A,\sigma_A
\rangle_{L^2(S^*M)}\;.$$ The right side of (\ref{SPECMEAS})
defines a measure $dm_A$ on $\R$ and (\ref{QM2}) says
$$\int_\R e^{it \tau}dm_A(\tau) = \langle V_t \sigma_A,\sigma_A
\rangle_{L^2(S^*M)}\;\;=\int_\R e^{it
\tau}d\mu_{\sigma_A}(\tau).$$

Since weak mixing systems are ergodic, it is not necessary to
average in both indices along an ergodic subsequence:

\begin{equation} \label{SPECMEASII} \lim_{\lambda_j \to \infty} \langle A_t^* A \phi_j, \phi_j \rangle =
 \sum_{j}
e^{i t (\lambda_i - \lambda_j)} |\langle A \phi_i, \phi_j
\rangle|^2 = \langle V_t \sigma_A, \sigma_A \rangle_{L^2(S^*M)}.
\end{equation}
Dually, one has

\begin{equation} \label{SPECMEASI} \lim_{\lambda_j \to \infty}
\sum_{i\; : \;  |\lambda _i-\lambda _j-\tau|<\varepsilon
\\}\;
 |\langle A\varphi_i,
\varphi_j \rangle|^2\;\; = \int^{\tau +\varepsilon
}_{\tau-\varepsilon } d\mu_{\sigma_A}.  \end{equation} For QUE
systems, these limit formulae are valid for the full sequence of
eigenfunctions.

\subsection{Rate of quantum ergodicity and mixing}

A quantitative refinement of quantum ergodicity is to  ask at what
rate the sums in Theorem \ref{QE}(1) tend to zero, i.e. to
establish a rate of quantum ergodicity.
In the off-diagonal case one may view $|\langle A\varphi_i,
\varphi_j \rangle|^2$ as analogous to $|\langle A \phi_j, \phi_j)
- \omega(A)|^2$. However, the sums in (\ref{SPECMEAS}) are double
sums while those of (\ref{diag}) are single. One may also average
over the shorter intervals $[\lambda, \lambda + 1].$

The only  rigorous result valid on general
Riemannian manifolds with hyperbolic geodesic flow  is the
logarithmic decay:

\begin{theo} \cite{Z5} (see also \cite{Schu2} for $p = 2$)   For any $(M, g)$ with hyperbolic geodesic flow,
$$ \frac{1}{N(\lambda)}
 \sum_{\lambda_j \leq \lambda}
|(A\phi_j,\phi_j)-\omega(A)|^{2p} = O \left( \frac{1}{(\log \lambda)^p} \right). $$
\end{theo}
The proof uses the central limit theorem for geodesic flows of M. Ratner.
The logarithm  reflects the exponential blow up in time of
remainder estimates for traces involving the wave group associated
to hyperbolic flows. It would be surprising if the logarithmic
decay is sharp for Laplacians. It was shown by  R.
Schubert \cite{Schu}  that the estimate is sharp in the case of
two-dimensional hyperbolic quantum cat maps. Hence the estimate
cannot be improved by semi-classical arguments that hold in both
settings.

A stunning asymptotic formula for $V_A(\lambda)$ was proved by Luo-Sarnak for holomorphic
forms of arithemetic hyperbolic quotients. Before stating the result, we review some conjectures
in the physics literature.

\subsection{Quantum chaos conjectures}

First, consider off-diagonal matrix elements. One conjecture is
that it is not necessary to sum in $j$ in (\ref{SPECMEASI}): each
individual term has the asymptotics consistent with
(\ref{SPECMEASI}). This is implicitly conjectured by
Feingold-Peres  in \cite{FP} (11) in the form
\begin{equation} \label{FPCONJ} |\langle A \phi_i, \phi_j
\rangle|^2 \simeq \frac{C_A (\frac{E_i - E_j)}{\hbar})}{2 \pi
\rho(E)},
\end{equation} where $C_A(\tau) = \int_{- \infty}^{\infty} e^{- i
\tau t} \langle V_t \sigma_A, \sigma_A \rangle dt. $ In our
notation, $\lambda_j = \hbar^{-1} E_j$ and $\rho(E) dE \sim d
N(\lambda)$. There are  $\sim C \lambda^{n-1}$ eigenvalues
$\lambda_i$ in the interval $[\lambda_j - \tau - \epsilon,
\lambda_j - \tau + \epsilon],$ so (\ref{FPCONJ}) says that
individual terms have the asymptotics of (\ref{SPECMEASI}).

On the basis of the analogy between $|\langle A\varphi_i,
\varphi_j \rangle|^2$ and  $|\langle A \phi_j, \phi_j\rangle -
\omega(A)|^2$, it is  conjectured in \cite{FP}  that
$$\ V_A(\lambda)  \sim \frac{ C_{A - \omega(A) I}(0) }{\lambda^{n-1} vol(\Omega)}.
$$
The idea is that $\phi_{\pm} = \frac{1}{\sqrt{2}} (\phi_i \pm
\phi_j)$ have the same  matrix element asymptotics as
eigenfunctions when $\lambda_i - \lambda_j$ is sufficiently small.
But then $2 \langle A \phi_+, \phi_- \rangle = \langle A \phi_i,
\phi_i \rangle - \langle A \phi_j, \phi_j \rangle$ when $A^* = A$.
Since we are taking a difference, we  may replace each matrix
element by $\langle A \phi_i, \phi_i \rangle $ by $\langle A
\phi_i, \phi_i \rangle - \omega(A)$ (and also for $\phi_j$). The
conjecture then assumes that $ \langle A \phi_i, \phi_i \rangle -
\omega(A)$ has the same order of magnitude as $ \langle A \phi_i,
\phi_i \rangle - \langle A \phi_j, \phi_j \rangle$.  The order
of magnitude is predicted by some natural random wave models, as
discussed below in \S \ref{RWONB}.

\subsection{Rigorous results}

At this time, the strongest variance result is an  asymptotic
formula for the diagonal variance proved by Luo-Sarnak for special
Hecke eigenfunctions on the quotient $\H^2/SL(2, \Z)$ of the upper
half plane by the modular group \cite{L.S.2}. What they prove is an
asymptotic variance formula for
 holomorphic Hecke eigenforms.  One expects that their proof, suitably
 modified, extends to  smooth Maass-Hecke eigenfunctions. Therefore,
 as in \cite{Z3}, we describe the statement   for smooth eigenfunctions  as a
Theorem/Conjecture -- i.e. it is a Theorem for holomorphic forms,
but still a conjecture for non-holomorphic forms.  Note that $\H^2/SL(2, \Z)$ is a non-compact
finite area surface whose Laplacian $\Delta$ has both a discrete
and a continuous spectrum. The discrete Hecke eigenfunctions are
joint eigenfunctions of $\Delta$ and the Hecke operators $T_p$.

\begin{theoconj} \cite{L.S.2} \label{LS} Let $\{\phi_k\}$ denote the orthonormal basis of
Hecke eigenfunctions for $\H^2/SL(2, \Z)$. Then there exists a
quadratic form $B(f) $ on $C_0^{\infty}(\H^2/SL(2, \Z))$ such that
$$ \frac{1}{N(\lambda)}
 \sum_{\lambda_j \leq \lambda}
|\int_X f |\phi_j|^{2} dvol - \frac{1}{Vol(X)} \int_X f dVol|^2 =
\frac{B(f, f) }{\lambda} + o(\frac{1}{\lambda}).
$$
\end{theoconj}

 When the multiplier $f =
\phi_{\lambda}$ is itself an eigenfunction, Luo-Sarnak have shown
that
$$B(\phi_{\lambda}, \phi_{\lambda}) = C_{\phi_{\lambda}}(0)
L(\frac{1}{2}, \phi_{\lambda})$$ where $L(\frac{1}{2},
\phi_{\lambda})$ is a certain $L$-function. Thus, the conjectured
classical variance is multiplied by an arithmetic factor depending
on the multiplier which has no dynamical significance. At this time, it is unknown
whether variance asymptotics exist in the non-arithmetic case. From numerical experiments,
it is believed that arithmetic settings behave somewhat differently from non-arithmetic
ones (see \cite{Sar2}), and this could be another example of the non-generic behavior of
arithmetic quantum chaos.

\subsection{\label{PS} Quantum limits on a hyperbolic surface and Patterson-Sullivan distributions}

In this section, we mention
a curious link between quantum limits and classical dynamics
 on a hyperbolic quotient that was  observed in \cite{AZ}. It is related to the invariant triple
 products studied in \cite{Sar,BR,MS}.

  We write $G =
 PSU(1,1) := SU(1,1)/\pm I \equiv PSL(2, \R), K = PSO(2)$
and identify the quotient $G/K$ with the hyperbolic disc $\D$.
 We let $\Gamma \subset G$ denote a co-compact discrete group and
 let
  $\X = \Gamma \backslash \D$ denote the associated hyperbolic
 surface.  In this context it is standard to denote eigenvalues
by $\lambda_j^2=s_j(1-s_j)=\frac14+r_j^2$ ($s_j=\frac12+ir_j$) and
eigenfunctions  by $\{\phi_{ir_j}\}_{j = 0, 1, 2, \dots}$.

We recall that on a  hyperbolic quotients there exists a quantization  $a \to Op(a)$
 of symbols which is adapted to the non-Euclidean
Fourier transform of Helgason (see (\ref{HYPOPa}) and \cite{AZ} for references and details). We then define the Wigner
distributions by
\begin{equation} \label{WIGDEF} \langle a, W_{ir_j} \rangle  = \int_{S^*\X} a(g)W_{ir_j}(dg) :=\langle Op(a)\phi_{ir_j},
\phi_{ir_j}\rangle_{L^2(\X)},\;\;\; a \in C^{\infty}(S^* \X)
\end{equation}
On the other hand, one can define a second sequence of phase space
distributions, the Patterson-Sullivan distributions
$\{PS_{ir_j}\}$ associated to the eigenfunctions
$\{\phi_{ir_j}\}$,  by the expression
\begin{equation} \label{PATSULDEF} PS_{ir_j}(dg)= PS_{ir_j}(db', db, dt) := \frac{T_{ir_j}(db)
T_{ir_j}(db')}{|b - b'|^{1 +  2i r_j}} \otimes |dt|.
\end{equation} In this definition, $T_{ir_j}$ is the boundary values
of $\phi_{ir_j}$ in the sense of Helgason (see (\ref{HELBV})).
The parameters $(b',b)$ ($b\not=b'$) vary in $B\times B$, where $B
=
\partial \D$ is the boundary of the hyperbolic disc, and $t$
varies in $\R$; $(b', b)$ parametrize the space of oriented
geodesics, $t$ is the time parameter along geodesics, and the
three parameters $(b', b, t)$ are used to parametrize the unit
tangent bundle $S\D$. The Patterson-Sullivan distributions
${PS}_{ir_j}$  are by construction  invariant under the geodesic
flow $(g^t)$ on $S \D$, i.e.
\begin{equation} \label{CLINV} (g^t)_* {PS}_{ir_j} = {PS}_{ir_j}, \end{equation}
 and by using (\ref{CONFORMAL}) they can be shown to be  $\Gamma$-invariant. Hence to each eigenfunction
 one obtains a  geodesic-flow
invariant distributions on $S \X$. We also introduce  normalized
Patterson-Sullivan distributions
\begin{equation} \label{NORMPATSULDEF} \widehat{PS}_{ir_j} : = \frac{1}{\langle 1,
PS_{ir_j} \rangle_{S\X}} \; PS_{ir_j}, \end{equation} which
satisfy the same normalization condition $\langle 1,
\widehat{PS}_{ir_j} \rangle = 1$ as $W_{ir_j}$ on the quotient
$S\X$. In \cite{AZ} the following is proved:

\begin{theo} \label{mainintro} For any $a \in C^{\infty}(\Gamma
\backslash G)$,
$$\int_{S\X} a(g)W_{ir_j}(dg)= \int_{S\X} a(g) \widehat{PS}_{ir_j}(dg)+ O(r_j^{-1}).$$
\end{theo}

It follows  that the Wigner distributions are equivalent to the
Patterson-Sullivan distributions in  the study of quantum
ergodicity. Yet, the Patterson-Sullivan distributions have a
purely classical dynamical definition: Define the classical
dynamical zeta functions,
\begin{equation} \label{LFUNDEFa} \left\{ \begin{array}{ll} (i) & \lcal_2(a,s)=
\sum_\gamma  \frac{e^{ -(s-1)L_{\gamma}}}{|\sinh(L_{\gamma}/2)
|^2} \left(\int_{\gamma_0}a \right), \\ &\\(ii) &\lcal(s; a): =
\sum_\gamma \frac{e^{-s L_{\gamma}}}{1-e^{-L_{\gamma}}}
\left(\int_{\gamma_0} a \right), \;\;\; (\Re e\;s > 1)
\end{array} \right. \end{equation}  where the sum runs over all
closed orbits, and $\gamma_0$ is the primitive closed orbit traced
out by $\gamma$. The sum converges absolutely for $\Re e\;s
> 1$.

\begin{theo} \label{main2}Let $a$ be a real analytic function on the unit tangent
bundle. Then $\lcal(s; a)$ and
   $\lcal_2(s; a)$  admit  meromorphic extensions to $\C$. The poles in the critical strip $0 <\Re e\;s < 1$, appear at
$s=1/2+ir$, where as above $1/4+r^2$ is an  eigenvalue of $\Lap$.
For each zeta function, the residue is
$$ \sum_{j: r_j^2 = r^2} \langle a, \widehat{PS}_{ir_j}\rangle_{S\X},$$
where $\{\widehat{PS}_{ir_j}\}$ are the normalized
Patterson-Sullivan distributions associated to an orthonormal
eigenbasis $\{\phi_{ir_j}\}$.

\end{theo}
Thus, the quantum limit problem is the same (for compact
hyperbolic surfaces) as the problem of finding the limiting
behavior of the residues of  classical dynamical zeta functions as
the pole moves up the critical line.

\section{\label{ENTROPY} Delocalization of eigenfunctions: II: Entropy of quantum limits on manifolds with
Anosov geodesic flow}

We now describe the recent results of  Anantharaman \cite{A},
 Anantharaman- Nonnenmacher \cite{AN} (see also \cite{ANK}) on entropy of quantum
limits for $(M, g)$ with Anosov geodesic flow. These articles give lower
bounds on entropies of the quantum limit measures that arise from
sequences of eigenfunctions. We closely follow
the presentation in \cite{AN,ANK}, and refer in particular to the partly expository
article \cite{ANK} and to the recent Bourbaki seminar report \cite{CV4} for an exposition of the results.

As discussed above, it is known that quantum limits are invariant
probability measures for $g^t$, but many such measures exist for
any geodesic flow. To pin down the possible quantum limits, one
needs to add constraints on the possible limit measures. The
entropy bounds of Anantharaman et al provide almost the only
additional constraints known at this time. The lower bound on
entropy rules out such possible limit measures as periodic orbit
measures $\mu_{\gamma}$ (which have entropy zero) or finite sums of such
measures. However, the entropy results leave open the possibility that a
sequence of eigenfunctions could tend to a limit of the form $a \mu_{\gamma} + (1 - a) d\mu$
if $a$ is small enough (here, $\mu$ is Liouville measure).

We recall (see \S \ref{BASIC} that   a geodesic flow $g^t$ is Anosov on on $S^*_g M$
if
 the tangent bundle $T S^*_g M$
splits into $g^t$ invariant sub-bundles
$E^u(\rho)\oplus E^s(\rho) \oplus
\IR\,X_H(\rho)\,
$
where $E^u$ is the unstable subspace and $E^s$ the stable
subspace. The unstable Jacobian $J^u(\rho)$ at $\rho$ is
defined by
$J^u(\rho)=\det\big(dg^{-1}_{|E^u(g^1\rho)}\big)$.

The goal is to give a lower bound for entropy of quantum limits.
 Entropy is complicated to define, and we only provide a brief
 sketch here. Classically, entropies
are defined for an  invariant probability measure $\mu$ for the
geodesic flow and measures the average complexity of $\mu$-typical
orbits. In the Kolmogorov-Sinai entropy, one starts with a
partition $\pcal = (E_1, \dots, E_k)$ of $S^*_g M$ and defines the
Shannon entropy of the partition by $h_{\pcal}(\mu) =  \sum_{j =
1}^k \mu(E_j) \log \mu(E_j). $ Under iterates of  the time one map
$g$ of the geodesic flow, one refines the partition to
$$\pcal^{v n} = \{ E_{\alpha_0} \cap g^{-1} E_{\alpha_1} \cap
\cdots \cap g^{-n + 1}  E_{\alpha_{n-1}} \}. $$ One defines $h_n(\pcal,
\mu)$ to be the Shannon entropy of this partition and then defines
$h_{KS}(\mu, \pcal) = \lim_{n \to \infty} \frac{1}{n} h_n(\mu, \pcal)$. Then
$h_{KS}(\mu) = \sup_{\pcal} h_{KS}(\mu, \pcal)$.

The main result of \cite{AN,ANK} is the following
\begin{theo}\label{thethm}
Let $\mu$ be a semiclassical measure associated to the
eigenfunctions of the Laplacian on $M$. Then its metric entropy
satisfies
\begin{equation}\label{e:main1}
h_{KS}(\mu)\geq  \left|\int_{S^*M} \log J^u(\rho)d\mu(\rho)
\right|- \frac{(d-1)}{2} \lambda_{\max}\,,
\end{equation}
where $d=\dim M$ and $\lambda_{\max}=\lim_{|t| \to
\infty}\frac{1}t \log \sup_{\rho\in \cE} |dg^t_\rho|$ is the
maximal expansion rate of the geodesic flow on $\cE$.

In particular, if $M$ has constant sectional curvature $-1$, this
means that
\begin{equation}\label{e:main2}
h_{KS}(\mu)\geq \frac{d-1}2.
\end{equation}
\end{theo}

The proof is based on a quantum analogue of the metric entropy,  and
in particular on a development of  the  `entropic
uncertainty principle' of Maassen- Uffink. There are several
notions of quantum or non-commutative entropy, but for
applications to eigenfunctions it is important to find one with
good semi-classical properties.

 Let $(\hcal, \langle.,.\rangle)$ be a complex
Hilbert space, and let
$\norm{\psi}=\sqrt{\langle\psi,\psi\rangle}$ denote the associated
norm. The quantum notion of partition is a family
$\pi=(\pi_k)_{k=1,\ldots,\ncal}$  of operators on $\hcal$ such
that $ \sum_{k=1}^{\hcal}\pi_{k}\pi_{k}^{*}=Id.$ If
$\norm{\psi}=1$,  the entropy of $\psi$ with respect to the
partition $\pi$ is define by
$$
h_{\pi}(\psi)=-\sum_{k=1}^{\ncal} \norm{\pi_k^*
\psi}^2\log\norm{\pi_k^*\psi}^2\,.
$$
We note that the quantum analogue of an invariant probability
measure $\mu$ is an invariant state $\rho$, and the direct
analogue of the entropy of the partition would be $\sum
\rho(\pi_{k}\pi_{k}^{*}) \log \rho(\pi_{k}\pi_{k}^{*}). $ If the state is $\rho(A) = \langle
A \psi, \psi \rangle$ then $\rho(\pi_k \pi_k^*) = ||\pi^*_k
\psi||^2. $

 The dynamics is generated by a unitary
operator $\cU$ on $\cH$. We now state a simple version of the  entropy uncertainty inequality of Maasen-Uffink.
A more elaborate version in \cite{A,AN,ANK} gives a lower bound for a certain `pressure'.

\begin{theo}
\label{t:WEUP}
 For any $\epsilon \geq 0$, for any normalized $\psi\in\cH$,
$$
h_{\pi}\big({\mathcal U} \psi \big) +
h_{\pi}\big(\psi\big) \geq - 2 \log c({\mathcal U})\,,
$$
where
$$ c({\mathcal U}) = \sup_{j, k} |\langle e_k, \; {\mathcal U} e_j \rangle|$$
is the supremum of all matrix elements in the orthonormal basis $\{e_j\}$. In particular,
$h_{\pi}(\psi) \geq - \log c({\mathcal U})$ if $\psi$ is an eigenfunction of ${\mathcal U}$.
\end{theo}

In the application to eigenfunctions, one fixes a partition $\{M_k\}$ of $M$ and a corresponding
partition $T^* M_k$ of $T^*M$. One then defines a smooth quantum partition of unity $P_k$ by smoothing
out the characteristic functions of $M_k$. The partition is refined by
\begin{equation} \label{PALPHA} P_{\alpha} = P_{\alpha_{n-1}} (n-1) P_{\alpha_n-2}(n-2) \cdots P_{\alpha_0}, \end{equation}
where $\alpha = (\alpha_0, \dots, \alpha_{n-1})$ and where $P(k) = {\mathcal U}^{*k} P {\mathcal U}^{k}$.
One then specifies:

\begin{enumerate}

\item ${\mathcal U} = e^{i T_E \sqrt{\Delta}}$ is the wave operator at the
`Ehrenfest time' $T_E = \frac{\log \lambda}{\lambda_{\max}}$.  Or from a semi-classical
(where $h = \frac{1}{\lambda}$, where the Hamiltonian is $H = h^2 \Delta$ and where the time
evolution is $e^{i \frac{t}{h} H}$),  ${\mathcal U} =  e^{i n_E(h)  h \Delta}$ with $n_E(h)
= \frac{\log \frac{1}{h}}{\lambda_{\max}}. $

\item $\psi_h$ is an eigenfunction of ${\mathcal U}$;

\item $h_{\pi}(\psi_h) = \sum_{|\alpha| = n_E} ||P_{\alpha}^* \psi_h||^2 \log ||P_{\alpha}^* \psi_h||^2. $

\end{enumerate}

With these specifications,
\begin{equation} \label{c} c({\mathcal U}) = \max_{|\alpha| = |\alpha'| = n_E} ||P_{\alpha'} U^{n_E} P_{\alpha} Op(\chi^{(n_E)})|| \end{equation}
where $\chi^{(n_E)}$ is a very sharp energy cutoff supported in a tubular neighborhood $\ecal^{\epsilon} := H^{-1}(1 - \epsilon, 1 + \epsilon)$
  of $\ecal = S^*M$ of width
$2 h^{1 - \delta} e^{n \delta}$
for a given $\delta> 0$.

We now give a very sketchy outline of how the entropy uncertainty inequality (Theorem \ref{t:WEUP}) is used
to prove the lower bound on the entropy of the limit measure (Theorem \ref{thethm}). The  argument is
 technical and difficult and the outline only gives the flavor of the estimates; the interested  reader should consult \cite{ANK} for
a complete exposition.

 The next  step is to link $c({\mathcal U})$ to the classical dynamics. The authors introduce a discrete `coarse-grained'
 unstable Jacobian
 $$J_1^u(\alpha_0, \alpha_1) : = \sup \{J^u(\rho): \rho \in T^* \Omega_{\alpha_0} \cap \ecal^{\epsilon}: \; g^t \rho \in T^* \Omega_{\alpha_1} \}, $$
 for $\alpha_0, \alpha_1 = 1, \dots, K$. Here, $\Omega_{j}$ are small open neighborhoods of the partition sets $M_j$. For a sequence
 ${\bf \alpha} = (\alpha_0, \dots, \alpha_{n-1})$ of symbols of length $n$, one defines
 $$J_n^u({\bf \alpha}) : = J_1^u(\alpha_0, \alpha_1) \cdots J_1^u(\alpha_{n-2}, \alpha_{n-1}). $$
 \begin{theo} \label{t:main}
Given a partition $\pcal^{(0)}$ and $\delta, \delta'>0$ small
enough, there exists $\hbar_{\cP^{(0)},\delta,\delta'}$ such that,
for any $\hbar\leq \hbar_{\pcal^{(0)},\delta,\delta'}$, for any
positive integer $n\leq n_E(\hbar)$, and any pair of sequences
${\bf \alpha}$, ${\bf \alpha'}$ of length $n$,
\begin{equation}\label{e:main}
\norm{P_{{\bf \alpha}'}^*\, U^n\, P_{{\bf \alpha}} \Op(\chi^{(n)})}
 \leq C\,\hbar^{-(d-1+c\delta)}\, \sqrt{J^u_n({\bf \alpha}) J^u_n({\bf \alpha}')}\,.
\end{equation}
Here, $d = \dim M$ and the  constants $c$, $C$ only depend on
$(M,g)$.
\end{theo}

To prove this, one shows  that any state of the form
$\Op(\chi^{(*)})\Psi$  can be decomposed as a superposition of
essentially $\hbar^{-\frac{(d-1)}2}$ normalized Lagrangian states,
supported on Lagrangian manifolds transverse to the stable leaves
of the flow.  The action of the operator
$P_{{\bf \alpha}}$ on such
Lagrangian states is intuitively as follows: each application of
$U$ stretches the Lagrangian in the unstable direction (the rate
of elongation being described by the unstable Jacobian) whereas
each multiplication by $P_{\alpha_j}$ projects onto a small piece of the Lagrangian.
This iteration of stretching and cutting accounts for the
exponential decay.

Combined with the entropy uncertainty inequality, one obtains

\begin{prop}\label{p:WEUP}
Let $d = \dim M$ and let  $(\psi_\hbar)_{\hbar\to 0}$ be a sequence of eigenfunctions.
Then there exist $\delta, \delta'$ so that,   at
time $n=n_E(\hbar)$,
\begin{equation}\label{e:ineg}
h_{n}(\psi_\hbar)\geq
2(d-1+c\delta)\log \hbar+\cO(1) \geq
-2\frac{(d-1+c\delta)\lambda_{\max}}{(1-\delta')}\; n +\cO(1)\,.
\end{equation}
\end{prop}

Now suppose  that  the Wigner measures  $W_{\psi_{\hbar}}$ of a
subsequence $(\psi_{\hbar})_{\hbar\to 0}$ of eigenfunctions
converges to the semiclassical measure $\mu$ on $\cE$.
Consider the limit $\hbar\to 0$ of  $h_{\pi}(\psi_h)$(so that $n_E(h)\to \infty$).  For
any sequence $\alpha$ of length $n_E$,  each
$\norm{P^*_{\alpha}\,\psi_{\hbar}}^2$ converges to
$\mu(\{{\bf \alpha}\})$, where $\{{\bf \alpha} \}$ is the function
$P^2_{\alpha_0}\,(P^2_{\alpha_1}\circ g^1)\ldots (P^2_{\ep_{n_E}}\circ
g^{n_E})$ on $T^*M$. Then for any $n_0 \leq n_E$,  $h_{n_0}(\psi_{\hbar})$ semiclassically
converges to the classical entropy
$$
h_{n_0}(\mu, \pcal_{sm} )=h_{n_0}(\mu,(P^2_k))=-\sum_{|{\bf \alpha}|=n_0}\mu(\{{{\bf 1}^{sm}_{M_{\bf \alpha}}}\})^2 \log
J_n^u({\bf \alpha})\,,
$$
where $${{\bf 1}^{sm}_{M_{\bf \alpha}}} = ({\bf 1}^{sm}_{M_{\alpha_{n_0} - 1}} \circ g^{n_0 - 1} ) \cdots
({\bf 1}^{sm}_{M_{\alpha_{1}}} \circ g ) {\bf 1}^{sm}_{M_{\alpha_0}}.$$
Here, ${\bf 1}^{sm}_{M_{\alpha_0}}$ is a smoothing of the characteristic function of the indicated set.

Using Proposition \ref{p:WEUP},  one obtains the lower bound
\begin{equation}\label{e:classicentropy}
\frac{ h_{n_o}(\mu, \pcal_{sm})}{n_o} \geq -\frac{n_0 - 1}{n_0}  -\sum_{\alpha_0, \alpha_1}\mu(\{{{\bf 1}^{sm}_{M_{\bf \alpha}}}\})^2 \log
J_1^u(\alpha_0, \alpha_1)
-\frac{(d-1+c\delta)\lambda_{\max}}{(1-\delta')}-2\frac{R}{n_o}\,.
\end{equation}
Here, $\delta$ and $ \delta'$ could be taken arbitrarily small, and at
this stage they can be set equal to zero.

The Kolmogorov--Sinai entropy of $\mu$ is by definition the limit
of the left side of (\ref{e:classicentropy}) when $n_o \to \infty$.
 Then let  $n_o \to \infty$, and let the diameter
$\dia/2$ of the partition tend to $0$. Then  the first term in the
right hand side of \eqref{e:classicentropy} converges to the
integral $-\int_{\cE} \log J^u(\rho)d\mu(\rho)$ as $\dia\to
0$, proving \eqref{e:main1}.

$\hfill\square$

\medskip

\section{\label{ANALYTIC} Real analytic manifolds and their complexifications}

In this section, we consider eigenfunctions on a real analytic
compact Riemannian manifold $(M, g)$. The advantage of real
analyticity is that one can complexify  the manifold and
analytically continue  the eigenfunctions to the complexification
of $M$. This allows one to use methods of holomorphic and
pluri-subharmonic function theory to  obtain sharper results on
volumes and distribution of nodal hypersurfaces than are possible
for $C^{\infty} (M, g)$. The gain in simplicity is two fold,
reflecting the relative simplicity of real polynomials over smooth
functions, and of complex zeros of polynomials over real zeros.
This point of view has been taken in \cite{DF,Lin,Z6} among other
articles.

A real analytic manifold $M$ always possesses a
unique complexification $M_{\C}$ generalizing the complexification
of $\R^m$ as $\C^m$. The complexification is  an open  complex
manifold in  which $M$ embeds $\iota: M \to M_{\C}$  as a totally
real submanifold (Bruhat-Whitney). As examples, we have:

\begin{itemize}

\item  $M = \R^m/\Z^m$ is $M_{\C} = \C^m/\Z^m$.

\item   The unit sphere $S^n$ defined by  $x_1^2 + \cdots +
x_{n+1}^2 = 1$ in $\R^{n+1}$ is complexified as the complex
quadric $S^2_{\C} = \{(z_1, \dots, z_n) \in \C^{n + 1}: z_1^2 +
\cdots + z_{n+1}^2 = 1\}. $

\item  The hyperboloid model of hyperbolic space is the
hypersurface in $\R^{n+1}$ defined by
$$\Hh^n = \{ x_1^2 + \cdots
x_n^2 - x_{n+1}^2 = -1, \;\; x_n > 0\}. $$ Then,
$$H^n_{\C} = \{(z_1, \dots, z_{n+1}) \in \C^{n+1}:  z_1^2 + \cdots
z_n^2 - z_{n+1}^2 = -1\}. $$

\item Any real algebraic subvariety of $\R^m$ has a similar
complexification.

\item Any  Lie group $G$ (or symmetric space) admits a
complexification $G_{\C}$.
\end{itemize}

The Riemannian metric determines a special kind of distance
function on $M_{\C}$ \cite{GS1, GS2, LS1, LS2, GLS}. The  metric
$g$ determines a  plurisubharmonic function $\sqrt{\rho} =
\sqrt{\rho}_g$ on $M_{\C}$ as the unique solution of the
Monge-Amp\`ere equation
$$(\ddbar \sqrt{\rho})^m = \delta_{M_{\R}, dV_g}, \;\; \iota^*
(i \ddbar \rho) = g. $$ Here, $\delta_{M_{\R}, dV_g}$ is the
delta-function on the real $M$ with respect to the volume form
$dV_g$, i.e. $f \to \int_M f dV_g$. In fact, $\sqrt{\rho}(\zeta) =
i \sqrt{ r^2(\zeta, \bar{\zeta})}$ where $r^2(x, y)$ is the
squared distance function in a neighborhood of the diagonal in $M
\times M$.

One defines the Grauert tubes $M_{\tau} = \{\zeta \in M_{\C}:
\sqrt{\rho}(\zeta) \leq \tau\}$. There exists a maximal $\tau_0$
for which $\sqrt{\rho}$ is well defined, known as the Grauert tube
radius. For $\tau \leq \tau_0$, $M_{\tau}$ is a strictly
pseudo-convex domain in $M_{\C}$.

The complexified exponential map $(x, \xi) \to exp_x i \xi$
defines a diffeomorphism from $B_{\tau}^* M$ to $M_{\tau}$ and
pulls back $\sqrt{\rho}$ to $|\xi|_g$. The one-complex dimensional
null foliation of $\ddbar \sqrt{\rho}$, known as the
`Monge-Amp\`ere' or Riemann foliation,  are the complex curves $t
+ i \tau \to \tau \dot{\gamma}(t)$, where $\gamma$ is a geodesic,
where $\tau > 0$ and where  $\tau \dot{\gamma}(t)$ denotes
multiplication of the tangent vector to $\gamma$ by $\tau$. We
refer to \cite{LS1} for further discussion.

\subsection{Analytic Continuation of eigenfunctions}

Let $A(\tau)$ denote the operator of analytic continuation of a
function on $M$ to the Grauert tube $M_{\tau}$. It is simple to
see that $A(\tau) = U_{\C}(i \tau) e^{\tau \sqrt{\Delta}}$ where
  $U(i\tau , x, y) = e^{-
\tau \sqrt{\Delta}}(x, y)$ is the Poisson operator of $(M, g)$,
i.e. the wave operator at positive imaginary time,  and $U_{\C}(i
\tau, \zeta, y)$ is its analytic continuation in $x$ to
$M_{\tau}$. In terms of the eigenfunction expansion, one has
\begin{equation} \label{UI} U(i \tau, \zeta, y) = \sum_{j = 0}^{\infty} e^{-
\tau  \lambda_j} \phi_{ j}^{\C} (\zeta) \phi_j(y),\;\;\; (\zeta,
y) \in M_{\epsilon}  \times M.  \end{equation}

To understand the analytic continuability of the wave kernel, we
first consider Euclidean $\R^n$ and its wave kernel  $U(t, x, y) =
\int_{\R^n} e^{i t |\xi|} e^{i \langle \xi, x - y \rangle} d\xi$
which  analytically continues  to $t + i \tau, \zeta = x + i p \in
\C_+ \times \C^n$ as the integral
$$U_{\C} (t + i \tau , x + i p , y) = \int_{\R^n} e^{i (t + i \tau)  |\xi|} e^{i \langle \xi, x + i p - y
\rangle} d\xi. $$ The integral clearly converges absolutely for
$|p| < \tau.$

Exact formulae of this kind exist for $S^m$ and $\H^m$. For a
general real analytic Riemannian manifold, there exists an
oscillatry integral expression for the wave kernel of the form,
\begin{equation} \label{PARAONE} U (t, x, y) = \int_{T^*_y M} e^{i
t |\xi|_{g_y} } e^{i \langle \xi, \exp_y^{-1} (x) \rangle} A(t, x,
y, \xi) d\xi
\end{equation} where $A(t, x, y, \xi)$ is a polyhomogeneous amplitude of
order $0$. For background,  we refer to \cite{Be,D.G, T}.  The
holomorphic extension of (\ref{PARAONE}) to the Grauert tube
$|\zeta| < \tau$ in $x$ at time $t = i \tau$ then has the form
\begin{equation} \label{CXPARAONE} U_{\C} (i \tau,
\zeta, y) = \int_{T^*_y} e^{- \tau  |\xi|_{g_y} } e^{i \langle
\xi, \exp_y^{-1} (\zeta) \rangle} A(t, \zeta, y, \xi) d\xi
\;\;\;(\zeta = x + i p).
\end{equation}

Since
\begin{equation} U_{\C} (i \tau) \phi_{\lambda} = e^{- \tau \lambda}
\phi_{\lambda}^{\C}, \end{equation} the analytic continuability of
the Poisson operator to $M_{\tau}$  implies that  every
eigenfunction analytically continues to the same Grauert tube. It follows that
the analytic continuation operator to $M_{\tau}$ is given by  \begin{equation} \label{ACO} A_{\C}(\tau) =
U_{\C}(i \tau) \circ e^{\tau \sqrt{\Delta}}. \end{equation}
Thus, a function   $f \in C^{\infty}(M)$ has a holomorphic extension to the closed  tube
$\sqrt{\rho}(\zeta) \leq \tau$ if and only if $f \in Dom(e^{\tau
\sqrt{\Delta}}), $ where $e^{\tau \sqrt{\Delta}}$ is the backwards
`heat operator' generated by $\sqrt{\Delta}$ (rather than $\Delta$).
 That is, $f = \sum_{n = 0}^{\infty}
a_n \phi_{\lambda_n}$ admits an analytic continuation to the open
Grauert tube $M_{\tau}$ if and only if $f$ is in the domain of
$e^{\tau \sqrt{\Delta}}$, i.e. if $\sum_n |a_n|^2  e^{2 \tau
\lambda_n} < \infty $. Indeed, the analytic continuation is
$U_{\C}(i \tau) e^{\tau \sqrt{\Delta}} f$. The subtlety is in the
nature of the restriction to the boundary of the maximal  Grauert
tube.

This result generalizes one of the classical Paley-Wiener theorems to real analytic
Riemannian manifolds \cite{Bou,GS2}.
In the simplest case
of $M = S^1$, $f \sim \sum_{n \in \Z} a_n e^{in \theta} \in
C^{\omega}(S^1)$ is the restriction of a holomorphic function $F
\sim \sum_{n \in \Z} a_n z^n $ on the annulus $S^1_{\tau} =
\{|\log |z| | < \tau \}$ and with $F \in L^2(\partial S^1_{\tau})$
if and only if $\sum_n |\hat{f}(n)|^2 \; e^{2 |n| \tau} < \infty$.
The case of $\R^m$ is more complicated since it is non-compact. We
are mainly concerned with compact manifolds and so the
complications are not very relevant here. But we recall that one
of the classical Paley-Wiener theorems states that a
 real analytic function $f$ on $\R^n$ is the restriction of a holomorphic
 function on the closed  tube $|\Im \zeta| \leq \tau$ which satisfies
$\int_{\R^m} |F(x + i \xi)|^2 dx \leq C$ for $\xi \leq  \tau$ if and
only if $\hat{f} e^{\tau |\Im \zeta|} \in L^2(\R^n)$.

Let us consider examples of holomorphic continuations of
eigenfunctions:

\begin{itemize}

\item On  the flat torus $\R^m/\Z^m,$   the real eigenfunctions
are $\cos \langle k, x \rangle, \sin \langle k, x \rangle$ with $k
\in 2 \pi \Z^m.$ The complexified torus is $\C^m/\Z^m$ and the
complexified eigenfunctions are $\cos \langle k, \zeta \rangle,
\sin \langle k, \zeta \rangle$ with $\zeta  = x + i \xi.$

\item On the unit sphere $S^m$, eigenfunctions are restrictions of
homogeneous harmonic functions on $\R^{m + 1}$. The latter extend
holomorphically to holomorphic harmonic polynomials on $\C^{m +
1}$ and restrict to holomorphic function on $S^m_{\C}$.

\item On $\H^m$, one may use the hyperbolic plane waves $e^{ (i
\lambda + 1) \langle z, b \rangle}$, where  $\langle z, b \rangle$
is the (signed) hyperbolic distance of the horocycle passing
through $z$ and $b$ to $0$. They  may be holomorphically extended
to the maximal tube of radius $\pi/4$.

\item
 On  compact hyperbolic quotients ${\bf
H}^m/\Gamma$, eigenfunctions can be then represented by Helgason's
generalized Poisson integral formula \cite{H},
$$\phi_{\lambda}(z) = \int_B e^{(i \lambda + 1)\langle z, b
\rangle } dT_{\lambda}(b). $$ Here, $z \in D$ (the unit disc), $B
=
\partial D$, and $dT_{\lambda}
\in \dcal'(B)$ is the boundary value of $\phi_{\lambda}$, taken in
a weak sense along circles centered at the origin $0$.
 To  analytically continue $\phi_{\lambda}$ it suffices  to
analytically continue $\langle z, b\rangle. $ Writing the latter
as  $\langle \zeta, b \rangle,  $ we have:
\begin{equation} \label{HEL} \phi_{\lambda}^{\C} (\zeta) = \int_B e^{(i \lambda +
1)\langle \zeta, b \rangle } dT_{\lambda}(b). \end{equation}

\end{itemize}

The results of Sarnak \cite{Sar} on the exponential decay of
integrals $\int_{M} \phi_{\lambda} \phi_{\mu}^2 dV_g$  and
subsequent results of Miller-Schmidt \cite{MS}  and
Bernstein-Reznikov \cite{BR} are closely related to the
Paley-Wiener theory and this analytic continuation formula for
eigenfunctions on hyperbolic quotients.

\subsection{Maximal plurisubharmonic functions and growth of $\phi_{\lambda}^{\C}$}

There are natural analogues in the setting of Gruaert tubes for the basic notions of pluripotential theory
on domains in $\C^m$.  One may view the Grauert tube function $\sqrt{\rho}$ as the analogue of the pluri-complex
Green's function or Siciak maximal PSH (pluri-subharmonic) function.

In the case of domains $\Omega \subset \C^m$, we
recall that the maximal PSH function (or pluri-complex Green's
function) relative to a subset $E \subset \Omega$ is defined by
$$V_{E}(\zeta) = \sup\{u(z): u \in PSH(\Omega), u|_{E}  \leq 0, u
|_{\partial \Omega} \leq 1\}. $$ This maximal function controls the Bernstein constant
$B(f, E, \Omega) =  \frac{\max_{\Omega} |f|}{\max_E |f|}$ for any holomorphic $f$.

An alternative construction of the  maximal PSH function due to  Siciak is defined
by taking the supremum only with respect to polynomials $p$. We
denote by $\pcal^N$ the space of all complex analytic polynomials
of degree $N$ and put  $\pcal_K^N = \{p \in \pcal^N: ||p||_K \leq
1, \;\; ||p||_{\Omega} \leq e\}. $
Then define
$$\log \Phi_E^N(\zeta) = \sup\{ \frac{1}{N} \log |p_N(\zeta)|: p \in \pcal_E^N \}, \;\;\; \log \Phi_E = \limsup_{N \to \infty} \log \Phi_E^N. $$
Here, $||f||_K = \sup_{z \in K} |f(z)|$.
 Siciak proved  that $\log \Phi_E = V_E$ (see \cite{K}, Theorem
5.1.7).  Intuitively, there are enough polynomials that one can
obtain the sup by restricting to polynomials.

On a real analytic Riemannian manifold, the natural analogue of $\pcal^N$ is the space
$$\hcal^{\lambda} = \{p =  \sum_{j: \lambda_j \leq \lambda} a_j
\phi_{\lambda_j}, \;\; a_1, \dots, a_{N(\lambda)} \in \R  \} $$
spanned by eigenfunctions with frequencies $\leq \lambda$. Rather than using the sup norm,
it is convenient  to work with $L^2$ based norms than sup
norms, and so we define
$$ \hcal^{\lambda}_M = \{p =  \sum_{j: \lambda_j \leq \lambda} a_j
\phi_{\lambda_j}, \;\;||p||_{L^2(M)} =  \sum_{j = 1}^{N(\lambda)}
|a_j|^2 = 1 \}. $$ We
define the   $\lambda$-Siciak extremal function by
 $$ \Phi_M^{\lambda} (z) = \sup \{|\psi(z)|^{1/\lambda} \colon
\psi \in \hcal_{\lambda};   \|\psi \|_M \le 1 \},  $$ and the extremal function by
$$\Phi_M(z) = \sup_{\lambda} \Phi_M^{\lambda}(z). $$

We may also define a natural analogue of the pluri-complex Green's function by putting
$E = M$ and $\Omega = M_{\tau}$ and defining
$$V_g(\zeta; \tau) =  \sup\{u(z): u \in PSH(M_{\tau}), u|_{M}  \leq 0, u
|_{\partial M_{\tau}} \leq \tau\}. $$
Although it does not seem to have been proved at this time, it is
easy to guess that  $V_g = \sqrt{\rho}$ since the latter  solves the homogeneous Monge-Amp\`ere equation
$(\ddbar \sqrt{\rho})^m = 0$ on $M_{\tau} \backslash M$.
Moreover, it is not hard to prove that \begin{equation} \label{SICIAK} \Phi_M = V_g, \end{equation}  generalizing the so-called Siciak-Zaharjuta theorem
in the special case where the boundary conditions are placed on all of $M$.
To see this, we consider the analytic continuation of the spectral projections kernels $\Pi_{[0, \lambda]}(x,y) =
\sum_{j: \lambda_j \in [0, \lambda]} \phi_j(x) \phi_j(y)$. Its
complexification evaluated on the anti-diagonal equals,
\begin{equation}\label{CXSP} \Pi_{[0, \lambda]}(\zeta, \bar{\zeta}) =
 \sum_{j: \lambda_j \in 0, \lambda]}
|\phi_j^{\C}(\zeta)|^2.  \end{equation} By using a Bernstein-Walsh inequality
$$\frac{1}{N(\lambda)} \leq \frac{\Pi_{[0, \lambda]}(\zeta,
\bar{\zeta})}{\Phi_M^{\lambda}(\zeta)^2} \leq C N(\lambda)\;
e^{\epsilon N(\lambda)}, $$ it is not hard to show that
\begin{equation} \Phi_M(z) = \lim_{\lambda \to \infty} \frac{1}{\lambda} \log  \Pi_{[0,
\lambda}(\zeta, \bar{\zeta}). \end{equation}
 To evaluate the logarithm, one can show that the kernel is essentially $e^{\lambda \sqrt{\rho}}$ times
the temperate projection defined by the Poisson operator,
\begin{equation}\label{CXDSPa}  P_{[0, \lambda]}(\zeta, \bar{\zeta}) =
 \sum_{j: \lambda_j \in [0, \lambda]}  e^{- 2 \sqrt{\rho}(\zeta)\lambda_j}
|\phi_j^{\C}(\zeta)|^2.
\end{equation}
The equality (\ref{SICIAK}) follows from the fact that
 $\lim_{\lambda \to \infty} \frac{1}{\lambda} \log   P_{[0, \lambda]}(\zeta, \bar{\zeta}) = 0$.

\subsection{Analytic continuation and nodal hypersurfaces}

We now survey some methods and  results on nodal hypersurfaces of
eigenfunctions on real analytic compact Riemannian manifolds $(M,
g)$ of dimension $m$. The principal result on volumes is due to
 Donnelly-Fefferman \cite{DF}:
\begin{theo} \label{DFNODAL} \cite{DF} (see also \cite{Lin})  Let $(M, g)$ be a compact real analytic  Riemannian
manifold, with or without boundary. Then there exist $c_1, C_2$ depending only on $(M, g)$ such that
\begin{equation} \label{DF} c_1 \lambda \leq {\mathcal
H}^{m-1}(Z_{\phi_{\lambda}}) \leq C_2 \lambda, \;\;\;\;\;\;(\Delta
\phi_{\lambda} = \lambda^2 \phi_{\lambda}; c_1, C_2 > 0).
\end{equation}
\end{theo}

A very readable exposition of the proof is contained in \cite{H}. The upper and lower bounds
require rather different arguments. The upper bound is simpler and may be sketched as follows:
By a local Crofton's formula,
the real volume of the nodal hypersurface equals   the mean number of intersections it has with a random line
in a coordinate chart, i.e. by the number of zeros of $\phi_{\lambda}$ along each line. This number is obviously
bounded above by the number of complex zeros of $\phi_{\lambda}^{\C}$ on the complexification of the line.
Hence, the upper bound is reduced to bounding the number of complex zeros of a family of complex analytic
functions of one variable in a disc.  For each disc, the
number of zeros is in principle  estimated by   Jensen's formula, which bounds the number of zeros in a disc
of a holomorphic function by  the logarithm of the  modulus of the holomorphic function. One can then use
the frequency function estimates, Carleman estimates or Bernstein-Walsh type inequalities to obtain
 the  doubling estimates in Theorem \ref{DOUBLE}, and they bound the growth of the log modulus
of $\phi_{\lambda}^{\C}$
 by $C \lambda$. Jensen's formula
does not directly apply, since it concerns the growth of a single holomorphic function
in an expanding family of domains, while we are interested in a family of holomorphic functions
$\phi_{\lambda}^{\C}$ in a single domain $M_{\tau}$. Donnelly-Fefferman find a suitable replacement using Blashcke product
factorizations. Lin gives an alternative bound on zeros in terms of growth using the frequency function
for functions of one complex variable.  Thus, the upper bound reflects the growth rate
of $\phi_{\lambda}$ combined with the upper bound on the number of complex zeros by the growth rate.

The lower bound is of a somewhat different nature.  A general analytic function of exponential growth (e.g. $e^z$) need not have any
zeros, i.e. zero might be an exceptional value. Jensen's formula equates the growth of the log modulus to the sum
of the number of zeros and the `proximity to zero', i.e. in our setting  to values where $|\phi_{\lambda}(\zeta)| \leq e^{- \lambda \epsilon}$.
It is necessary to rule out the possibility that zero is an exceptional value of complexified eigenfunctions.  In effect, this is possible because
real eigenfunction must have a zero in any ball of radius $\frac{C}{\lambda}$ by Theorem \ref{COURANT}. But then one needs
a lower bound on the hypersurface volume of the nodal set in such a small ball.  More precisely, in \cite{DF}
 a lower bound for the hypersurface volume  is proved for a certain $\lambda$-independent proportion  of a covering by  small balls of
radius $\frac{C}{\lambda_j}$.  The global result then follows by summing the volume in the these small balls. Besides \cite{DF},
we refer to \cite{HL} for a detailed discussion.

\subsection{Nodal hypersurfaces in the case of ergodic geodesic flow}

We now consider global results when hypotheses are made on the
dynamics of the geodesic flow.
 Use of the global wave operator brings into
play the relation between the geodesic flow and the complexified
eigenfunctions, and this allows one to prove gobal
results on nodal hypersurfaces that reflect the dynamics of the geodesic flow.
In some cases, one can determine
not just the volume, but the limit distribution of complex nodal
hypersurfaces.

The complex nodal hypersurface of an eigenfunction is   defined by
\begin{equation} Z_{\phi_{\lambda}^{\C}} = \{\zeta \in
B^*_{\epsilon_0} M: \phi_{\lambda}^{\C}(\zeta) = 0 \}.
\end{equation}
There exists  a natural current of integration over the nodal
hypersurface in any ball bundle $B^*_{\epsilon} M$ with $\epsilon
< \epsilon_0$ , given by
\begin{equation}\label{ZDEF}  \langle [Z_{\phi_{\lambda}^{\C}}] , \phi \rangle =  \frac{i}{2 \pi} \int_{B^*_{\epsilon} M} \ddbar \log
|\phi_{\lambda}^{\C}|^2 \wedge \phi =
\int_{Z_{\phi_{\lambda}^{\C}} } \phi,\;\;\; \phi \in \dcal^{ (m-1,
m-1)} (B^*_{\epsilon} M). \end{equation} In the second equality we
used the Poincar\'e-Lelong formula. The notation $\dcal^{ (m-1,
m-1)} (B^*_{\epsilon} M)$ stands for smooth test $(m-1,
m-1)$-forms with support in $B^*_{\epsilon} M.$

The nodal hypersurface $Z_{\phi_{\lambda}^{\C}}$ also carries a
natural volume form $|Z_{\phi_{\lambda}^{\C}}|$ as a complex
hypersurface in a K\"ahler manifold. By Wirtinger's formula, it
equals the restriction of $\frac{\omega_g^{m-1}}{(m - 1)!}$ to
$Z_{\phi_{\lambda}^{\C}}$. Hence, one can regard
$Z_{\phi_{\lambda}^{\C}}$ as defining  the  measure
\begin{equation} \langle |Z_{\phi_{\lambda}^{\C}}| , \phi \rangle
= \int_{Z_{\phi_{\lambda}^{\C}} } \phi \frac{\omega_g^{m-1}}{(m -
1)!},\;\;\; \phi \in C(B^*_{\epsilon} M).
\end{equation}
 We prefer to state results in terms of the
current $[Z_{\phi_{\lambda}^{\C}}]$ since it carries more
information.

We will say that a sequence $\{\phi_{j_k}\}$ of $L^2$-normalized
eigenfunctions  is {\it quantum ergodic} if
\begin{equation} \label{QEDEF} \langle A \phi_{j_k}, \phi_{j_k} \rangle \to
\frac{1}{\mu(S^*M)} \int_{S^*M} \sigma_A d\mu,\;\;\; \forall A \in
\Psi^0(M). \end{equation} Here, $\Psi^s(M)$ denotes the space of
pseudodifferential operators of order $s$, and $d \mu$ denotes
Liouville measure on the unit cosphere bundle $S^*M$ of $(M, g)$.
More generally, we denote by $d \mu_{r}$ the (surface) Liouville
measure on $\partial B^*_{r} M$, defined by
\begin{equation} \label{LIOUVILLE} d \mu_r = \frac{\omega^m}{d |\xi|_g} \;\; \mbox{on}\;\; \partial B^*_r
M.
\end{equation}
We also denote by $\alpha$ the canonical action $1$-form of
$T^*M$.

\begin{theo}\label{ZERO}  Let $(M, g)$ be  real analytic, and let $\{\phi_{j_k}\}$ denote a quantum ergodic sequence
of eigenfunctions of its Laplacian $\Delta$.  Let
$(B^*_{\epsilon_0} M, J)$ be the maximal Grauert tube around $M$
with complex structure $J_g$ adapted to $g$. Let $\epsilon <
\epsilon_0$. Then:
$$\frac{1}{\lambda_{j_k}} [Z_{\phi_{j_k}^{\C}}] \to  \frac{i}{ \pi} \ddbar \sqrt{\rho}\;\;
\mbox{weakly in }\;\;  \dcal^{' (1,1)} (B^*_{\epsilon} M), $$
in the sense that,   for any continuous test form $\psi \in \dcal^{
(m-1, m-1)}(B^*_{\epsilon} M)$, we have
$$\frac{1}{\lambda_{j_k}} \int_{Z_{\phi_{j_k}^{\C}}} \psi \to
 \frac{i}{ \pi} \int_{B^*_{\epsilon} M} \psi \wedge \ddbar
\sqrt{\rho}. $$ Equivalently, for any  $\phi \in C(B^*_{\epsilon} M)$,
$$\frac{1}{\lambda_{j_k}} \int_{Z_{\phi_{j_k}^{\C}}} \phi \frac{\omega_g^{m-1}}{(m -
1)!}  \to
 \frac{i}{ \pi} \int_{B^*_{\epsilon} M} \phi  \ddbar
\sqrt{\rho}  \wedge \frac{\omega_g^{m-1}}{(m - 1)!} . $$
\end{theo}

\begin{cor}\label{ZEROCOR}  Let $(M, g)$ be a real analytic with ergodic  geodesic
flow.  Let $\{\phi_{j_k}\}$ denote a full density ergodic
sequence. Then for all $\epsilon < \epsilon_0$,
$$\frac{1}{\lambda_{j_k}} [Z_{\phi_{j_k}^{\C}}] \to  \frac{i}{ \pi} \ddbar \sqrt{\rho},\;\;
 \mbox{weakly in}\;\; \dcal^{' (1,1)} (B^*_{\epsilon} M). $$
\end{cor}

The proof consists of three ingredients:

\begin{enumerate}

\item By the Poincar\'e-Lelong formula, $[Z_{\phi_{\lambda}^{\C}}] = i \ddbar \log |\phi_{\lambda}^{\C}|. $
This reduces the theorem to determining the limit of $\frac{1}{\lambda} \log |\phi_{\lambda}^{\C}|$.

\item $\frac{1}{\lambda} \log |\phi_{\lambda}^{\C}|$ is a sequence of PSH functions which are uniformly bounded
above by $\sqrt{\rho}$. By a standard compactness theorem, the sequence is pre-compact in $L^1$: every sequence
from the family has an $L^1$ convergent subsequence.

\item $|\phi_{\lambda}^{\C}|^2$, when properly $L^2$ normalized on each $\partial M_{\tau}$ is a quantum
ergodic sequence on $\partial M_{\tau}$. This property implies that the $L^2$ norm of $|\phi_{\lambda}^{\C}|^2$
  on $\partial \Omega$ is asymtotically $\sqrt{\rho}$.

  \item Ergodicity and the calculation of the $L^2$ norm imply that  the only possible $L^1$ limit
of $\frac{1}{\lambda} \log |\phi_{\lambda}^{\C}|$. This concludes the proof.

\end{enumerate}

We note that the first two steps are valid on any real analytic $(M, g)$. The difference is that
the $L^2$ norms of $\phi_{\lambda}^{\C}$ may depend on the subsequence and can often not equal
$\sqrt{\rho}$. That is,  $\frac{1}{\lambda} |\phi_{\lambda}^{\C}|$ behaves like the maximal PSH function
in the ergodic case, but not in general. For instance, on a flat torus, the complex zero sets of ladders
of eigenfunctions concentrate on a real hypersurface in $M_{\C}$. This may be seen from the complexified  real
eigenfunctions $ \sin \langle k, x + i \xi \rangle$, which vanish if and only if $\langle k, x \rangle  \in 2 \pi \Z$
and $\langle k, \xi \rangle = 0$. Here, $k \in \N^m$ is a lattice point. The exact limit distribution
depends on which ray or ladder of lattice points one takes in the limit.  The result reflects the quantum integrability
of the flat torus, and a similar (but more complicated) description of the zeros exists in all quantum integrable cases.
The fact that $\frac{1}{\lambda} \log |\phi_{\lambda}^{\C}|$ is pre-compact on a Grauert tube of any real analytic
Riemannian manifold  confirms the upper bound on complex  nodal hypersurface volumes.

\subsection{Analytic domains with boundary}

Many of the basic estimates on nodal hypersurfaces, such as Theorem \ref{DF}, apply to real analytic manifolds with
boundary as well as boundaryless manifolds. We are concentrating on manifolds without boundary, but mention one result
on nodal lines where the boundary effects are central. Namely, in the case of real analytic plane domains, it is visible from
computer graphics (see \cite{FGS} for references) that only a small proportion of the nodal components touch the boundary.
Most are small nodal loops in the interior.

We thus consider  Neumann (resp. Dirichlet) eigenfunctions
$\phi_{\lambda_j}$ on (piecewise) real analytic plane domains $\Omega
\subset \R^2$, i.e. solutions of (\ref{EP}) with boundary conditions
\begin{equation}
\partial_{\nu} \phi_{\lambda_j}
 = 0 \;\;(\mbox{resp.} \; \phi_{\lambda_j} = 0) \; \mbox{on} \;\; \partial \Omega,
 \end{equation}

\begin{theo} \label{COR} \cite{TZ3} Let $\Omega$ be a piecewise analytic domain
and  let  $n_{\partial \Omega}(\lambda_j)$ be the number of
components of the nodal set of the $j$th Neumann or Dirichlet
eigenfunction which intersect $\partial \Omega$.  Then there
exists  $C_{\Omega}$   such that $n_{\partial \Omega}(\lambda_j)
\leq C_{\Omega} \lambda_j.$
\end{theo}

In the Dirichlet case, we remove the boundary before counting components.
 For q generic piecewise analytic
plane domains, zero is a regular value of all  eigenfunctions
$\phi_{\lambda_j}$, i.e. $\nabla \phi_{\lambda_j} \not= 0$ on
$Z_{\phi_{\lambda_j}}$ \cite{U}, and the nodal set is  a disjoint union
of connected components which are homeomorphic either to circles
contained in the interior $\Omega^o$ of $\Omega$ or to intervals
intersecting the boundary in two points. We call the former
`closed nodal loops' and the latter `open nodal lines'. Thus, the theorem states that the  number of open nodal lines is $O(\lambda_j)$.
 As mentioned above, the Courant nodal domain theorem implies that the number
of nodal componants is of order $O(\lambda_j^2)$.
 When the upper bound is achieved,  the number
of open nodal lines in dimension $2$ is of one lower order in
$\lambda_j$
 than the  number of closed nodal loops. This effect is known from
 numerical experiments of eigenfunctions and random waves \cite{FGS}.

The proof of  Theorem \ref{COR} is to complexify the Cauchy data  (i.e. $\phi_{\lambda} |_{\partial \Omega}$ in
the Neumann case or $\partial_{\nu} \phi_{\lambda} |_{\partial \Omega}$ in the Dirichlet case) as holomorphic
functions on the complexification of the boundary,  and to
count the number of its complex zeros. This number is bounded by the growth rate of the log modulus, which can
be directly estimated in terms of the eigenvalue. Thus, the proof is of a similar form to the upper bound half of \cite{DF}.
On the other hand, no comparable lower bound exists:  in the case of the disc (see \S \ref{DISC}), the boundary values are of the form $\sin m \theta$
or $\cos m \theta$, where $m$ is the angular momentum rather than the frequency $\lambda$. So the number of boundary zeros
remains bounded for some sequences of eigenfunctions as the eigenvalue tends to infinity.

As previously mentioned, I. Polterovich \cite{Po} used Theorem \ref{COR} to resolve an old conjecture of Pleijel,
to the effect that Pleijel's bound  on the number of nodal domains (see Theorem \ref{PLEIJEL}) is valid for Neumann
as well as Dirichlet boundary conditions. Polterovich observes that Pleijel's argument applies to any nodal domain
that does not touch the boundary. By Theorem \ref{COR}, the nodal domains which do touch the boundary are of lower order
than the bound, concluding the proof.

\section{\label{RWONB} Riemannian random waves}

We have mentioned that the random wave model provides a kind of
guideline for what to conjecture about eigenfunctions of quantum
chaotic system.  In this final section, we briefly discuss random
wave models and what they predict.

To define Riemannian random waves, we first consider the standard
sphere. We choose an orthonormal basis $\{\phi_{N j}\}_{j =
1}^{d_N}$ for $\hcal_N$. We endow the real vector space
$\mathcal{H}_N$ with the Gaussian probability measure
 $\gamma_N$ defined by
\begin{equation}\label{gaussian}\gamma_N(s)=
\left(\frac{d_N}{\pi }\right)^{d_{N}/2}e^ {-d_{N}|c|^2}dc\,,\qquad
\psi=\sum_{j=1}^{d_{N}}c_j \phi_{N j}, \,\;\; d_{N} = \dim
\hcal_{N}.
\end{equation} Here,  $dc$
is $d_{N}$-dimensional real Lebesgue measure. The normalization is
chosen so that $\E_{\gamma_N}\; \langle \psi, \psi \rangle=1$,
where $\E_{\gamma_N}$ is the expected value with respect to
$\gamma_N$. Equivalently,  the $d_N$ real variables $\ c_j$
($j=1,\dots,d_{N}$) are independent identically distributed
(i.i.d.) random variables with mean 0 and variance
$\frac{1}{2d_{N}}$; i.e.,
$$\E_{\gamma_N} c_j = 0,\;\; \quad  \E_{\gamma_N} c_j c_k =
\frac{1}{ d_N}\de_{jk}\,.$$ We note  that the Gaussian ensemble is
equivalent to picking $\psi_N \in \hcal_{N}$ at random from the
unit sphere in $\hcal_{N}$ with respect to the $L^2$ inner
product. The latter description is more intuitive but it is
technically more convenient to work with Gaussian measures. In the
cutoff ensemble, we put the  product  Gaussian measure $\Pi_{n =
1}^N \gamma_n$ on $\bigoplus_{n = 1}^N \hcal_n$.

On a general compact Riemannian manifold $(M, g)$ of dimension
$m$,  the analogue of the space $\hcal_N$ of spherical harmonics
of degree $N$ is played by the space $\hcal_{I_{N}}$ of linear
combinations of eigenfunctions  with frequencies in an interval
$I_N : = [N, N + 1]$. The precise decomposition of $\R$ into
intervals is not canonical on a generic Riemannian manifold and
the results do not depend on the choice. We choose $I_N = [N, N +
1]$ only for notational simplicity.  In the special case of Zoll
manifolds (all of whose geodesics are closed), there is a
canonical choice where the intervals are centered in the middle of the
eigenvalue clusters, i.e. the points $\frac{2 \pi}{L} N + \frac{\beta}{4}$
where $L$ (resp. $\beta$) is the common length (resp. Morse index) of the
geodesics.   Henceforth we abbreviate $\hcal_N =
\hcal_{I_N}$ on general Riemannian manifolds. We continue to
denote by $\{\phi_{N j} \}_{j = 1}^{d_N}$ an orthonormal basis of
$\hcal_N$  where $d_N = \dim \hcal_N$.   We equip it with the
Gaussian measure (\ref{gaussian}) and again denote the expected
value with respect to $(\hcal_N, \gamma_N)$ by $\E_{\gamma_N}$.

\subsection{Levy concentration of measure}  Levy concentration of measure occurs
when Lipschitz  continuous functions $f$ on a metric probability
space $(X, d, \mu)$ of large dimension $d$
 are highly concentrated around their median values $\mcal_f$.  In the
fundamental  case
 where $X$ is the unit $N$-sphere $S^N$ with the usual
distance function, and $\mu$ is the SO$(N +1)$-invariant
probability measure,
  the concentration of measure inequality says that
\begin{equation} Prob \left\{ x \in S^{N} : |f(x) - \mcal_f| \geq r \right\} \leq
\exp \left(-\frac  {(N - 1) r^2 }{2 \|f\|_{Lip}^2}\right),
\label{Levy}\end{equation} where $$\|f\|_{Lip} = \sup_{d(x, y) >
0} \frac{|f(x) - f(y)|}{d(x, y)|}
$$ is the Lipschitz norm. (See, e.g.
\cite{L}.)

A key point in the proof is the following well-known and simple fact that the mass
the sphere concentrates in a small tube around any `equatorial' sphere.

\begin{lem}\label{calculus} Let $A\in S^{2d-1}\subset\C^d$, and give
$S^{2d-1}$ Haar probability measure.  Then
$$Prob\left\{P\in S^{2d-1}: |\langle
P,A\rangle|>\la\right\} \le e^{-(d-1)\la^2}\;.$$
\end{lem}

Another basic fact about concentration of Gaussian measures is the following Levy
concentration of Gaussian measure result due to Sudakov-Tsirelson and Borell:

\begin{lem}\label{ST} \cite{ST}  Let $F \subset \hcal_N$ and let $d_N = \dim \hcal_N$. Let  $F_{+ \rho}$ be the $\rho$-tube around $F$.
Then $$\PP(F_{+ \rho}) \leq \frac{3}{4} \implies \PP(F) \leq 2 e^{- C \;\rho \; d_N}. $$
\end{lem}

\subsection{Concentration of measure and $L^p$ norms}

The purpose of this section  is to illustrate the use of   concentration of
measure inequalities in one of its simplest applications to random waves:  to determine the asymptotic behavior  of
$\lcal^p$ norms of random spherical harmonics of degree $N$. The
main functionals we consider are the   norms on $S\hcal_N$:
$${\mathcal L}^p (\psi) = \| \psi\|_{p} \quad (2\le p\le\infty).$$
 We only consider the case $p = \infty$ in detail; for $p < \infty$ see \cite{SZ}.

\begin{theo}\label{Largesupnorm}  For each of the above complex ensembles,
there exist constants $C > 0$ such that:

 $$  \nu_N\left\{s_N\in S{\mathcal H}_N: \sup_{S^m}|\psi_N|>C\sqrt{\log
N}\right\} < O\left(\frac{1}{N^2}\right)\,.  $$

In fact, for any $k>0$, we can bound the probabilities by
$O(N^{-k})$ by choosing $C$ to be sufficiently large.

\end{theo}

As a corollary we obtain  almost sure
 bounds on the growth of  $\lcal^\infty$  norms  for
 independent random sequences of $\lcal^2$-normalized spherical
 harmonics.
To state the result, we introduce the probability sequence space
${\mathcal S} = \prod_{N=1}^{\infty} S{\mathcal H}_N$ with the
measure $\nu = \prod_{N=1}^{\infty} \nu_N$. The estimate of
Theorem  \ref{Largesupnorm} immediately implies that
$$\limsup_{N\to\infty}\frac{\sup_X|\psi_N|}{\sqrt{\log N}} \le C\qquad
\mbox{\rm almost surely}\;.$$ Hence we have:

\begin{cor} \label{CX} Sequences of sections  $\psi_N \in S{\mathcal H}_N$ satisfy:
$$\|\psi_N\|_{\infty}= O(\sqrt{\log
N})\;\;\; \mbox{ almost surely}.$$

\end{cor}

Results of the latter  type  were first proved by Salem-Zygmund in
the case of random trigonometric polynomials on the circle , and
by Kahane for random trigonometric polynomials on tori. Vanderkam
\cite{Van} generalized the results to   the case of random
spherical harmonics by a geometric method
 that seems special to the sphere.  Neuheisel
\cite{Neu} adapted a  method of Nonnenmacher-Voros on holomorphic
sections to simplify the proof of the sup norm estimates of
\cite{Van}. Here we follow an approach of \cite{SZ} to derive them
again from Levy concentration of measure. It also gives
concentration results on $L^p$ norms:

\begin{theo}\label{largepnorm}  Let   $2\le p < \infty$.  Then  the median values of the ${\mathcal L}^p$
norm on the unit spheres $S\hcal_N$ are bounded by a constant
$\al=\al(p,m)$, and
$$  \nu_N \{\psi_N\in S{\mathcal H}_N: {\mathcal L}^p(\psi_N) >  r +\al \} \leq
\exp( - Cr^2 N^{ 2 m /p})\;,
$$ for some constant $C>0$. Hence,
sequences of sections $\psi_N \in S{\mathcal H}_N$ satisfy
$\|\psi_N\|_p = O(1)$ almost surely.
\end{theo}

\subsection{$\lcal^\infty$ norms: Proof of Theorem
\ref{Largesupnorm}}

\begin{proof} We refer to \S \ref{SPHERE} for notation, but drop the dimensional subscript.
Throughout this section we assume that $\|\psi_N\|_{\lcal^2} = 1$.

We define $\Phi_N(x) =  (\phi_{N, 1}, \dots, \phi_{N,
d_N}): S^m \to \R^{d_N}$. It is well known (and easy to prove)  that this eigenmap is an isometric minimal embedding of $(S^m, g_0)$
as a subsphere of a sphere of $(\R^{d_N}, ds^2)$ where $ds^2$ is the standard Euclidean
metric. In fact,
\begin{equation} \label{PULLBACK}  \Phi_N^* ds^2 =  \sum_{j = 1}^{d_{N}} d \phi^N_j (x) \otimes  d\phi^N_j(x) = \frac{\lambda_{N}^2 d_N}{m Vol(S^m)} g_0 \sim C_m N^{m + 1} g_0. \end{equation}
 Indeed, by $SO(m + 1)$ invariance, the metric (\ref{PULLBACK}) is a constant multiple of the standard metric $g_0$.
 The constant can be calculated using the fact that  $\Pi_N(x,x) $
is a constant function. By  integrating over $S^m$ one sees that $\Pi_N(x,x) = \frac{d_N}{Vol(S^m)} \sim \frac{1}{Vol(S^m)} N^{m-1}$, and by taking $\Delta$ of this
formula one can determine the constants in (\ref{PULLBACK})
  (see e.g. \cite{Ch} for background). Then
\begin{equation}\label{coherentstate} \Pi_N(x,y)=\sum_{j=1}^{d_N}
\phi^N_j(x)\overline{\phi^N_j(y)}=\langle
\Phi_N(x),\Phi_N(y)\rangle\,.
\end{equation}
Let $\psi_N=\sum_{j=1}^{d_N}c_j\phi^N_j\ $ ($\sum|c_j|^2=1$)
denote a random element of $S\hcal_N$, and write
$c=(c_1,\dots,c_{d_N})$. Recall that
\begin{equation}\label{sNx}\psi _N(x)=\int_X\Pi_N(x,y)s_N(y)dy=
\sum_{j=1}^{d_N}c_j  \phi^N_j(x)= c\cdot\Phi_N(x)\,.\end{equation}
Thus
\begin{equation}\label{cos}|\psi_N(x)|=\|\Phi_N(x)\|\cos \theta_x\,,\quad
\mbox{where\ } \cos \theta_x = \frac{\left| c\cdot
\Phi_N(x)\right|}{\| \Phi_N(x)\|} \,.\end{equation}
The angle  $\theta_x$ can be interpreted as the distance in
$S^{d_N-1}$ between $[\bar c]$ and $\tilde\Phi_N(x)$. Also,
$\| \Phi_N(x)\| = \sqrt{ \Pi_N(x,x)} \sim C_m N^{(m - 1)/2}$.

Now fix a point $x\in S^m$.  By Lemma \ref{calculus},
\begin{eqnarray}\label{prob1a} \nu_N\left\{\psi_N:\cos\theta_x\ge C
N^{-(m - 1)/2}\sqrt{\log N}\right\} &\le &
\exp\left(-(d_N-1)\frac{C^2\log N}{N^{m - 1}}\right)\nonumber \\ &=&\
N^{-C^2 N^{-(m - 1)}(d_N-1)}\;.\end{eqnarray}

We can cover $S^m$ by a collection of $k_N$ balls $B(z^j)$ of radius
\begin{equation}\label{RN} R_N:=\frac{1}{N^m}\end{equation}
 centered at points
$z^1,\dots,z^{k_N}$, where
$$k_N \le O(R^{-m})\le
O(N^{m^2})\,.$$ By  (\ref{prob1a}), we have
\begin{equation} \label{centers} \nu_N \left\{\psi_N:\max_{j} \cos\theta_{x^j} \ge C N^{-(m - 1)/2}\sqrt{\log N} \right\}\le
k_N N^{-C^2 N^{-(m - 1)}(d_N-1)}\;,\end{equation} where $x^j$ denotes a
point in $X$ lying above $z^j$.

Equation (\ref{centers}) together with
 (\ref{cos}) implies  that the desired sup-norm
estimate holds at the centers of the small balls with high
probability. To extend (\ref{centers}) to points within the balls,
 we consider an arbitrary point $w^j\in B(z^j, N^{-m})$.  To
estimate the distance  $\de_{N}^j$, between
$\Phi_N(z^j)$ and $\Phi_N(w^j)$ in $S^{d_N-1}$ we let $\ga$
denote the geodesic in $S^m$ from $z^j$ to $w^j$ and use (\ref{PULLBACK})
to get
\begin{eqnarray}\de_{N}^j &\le & \int_{\Phi_{N*}\ga} ds
\ =\ \int_\ga \sqrt{\Phi_{N}^*ds} \ = \ C_m N^{\frac{m+1}{2}} \int_\ga
 ds\nonumber \\
& \le&   C_m N^{\frac{m+1}{2}} N^{-m}  = C_m N^{- (m-1)/2}\,.
\label{delta} \end{eqnarray}

By the triangle inequality in $S^{d_N-1}$, we have
$|\theta_{x^j}-\theta_{y^j}|\le\de_N^j$. Therefore by
(\ref{delta}),
\begin{equation}\label{triangle}\cos\theta_{x^j} \ge
\cos\theta_{y^j}-\de_N^j \ge \cos\theta_{y^j}-
 C_m N^{- (m-1)/2}\,.\end{equation} By (\ref{triangle}),
$$\cos\theta_{y^j} \ge \frac{(C+1)\sqrt{\log N}}{N^{(m - 1)/2}}\Rightarrow
\cos\theta_{x^j} \ge \frac{(C+1)\sqrt{\log N} -C_m}{N^{(m - 1)/2}}
\ge  \frac{C\sqrt{\log N}}{N^{(m - 1)/2}}$$ and thus
$$\begin{array}{l}\left\{\psi_N\in
S\hcal_N:\sup \cos\theta \ge (C+1) N^{-(m -1)/2}\sqrt{\log N}
\right\}\\[8pt] \qquad\qquad\qquad \subset \left\{s_N\in
S\hcal_N:\max_{j} \cos\theta_{x^j} \ge C N^{-(m - 1)/2}\sqrt{\log
N} \right\}\,.\end{array}$$

Hence by  (\ref{centers}),
\begin{equation} \label{cos*} \nu_N \left\{\psi_N\in
S\hcal_N:\sup \cos\theta \ge (C+1) N^{-(m - 1)/2}\sqrt{\log N}
\right\}\le k_N N^{-C^2 N^{-(m - 1)}(d_N-1)}\;.\end{equation}

There exists $A_M > 0$ so that $d_N \sim A_m N^{m -1}$, so one has
\begin{eqnarray*}\nu_N \left\{\psi_N\in S\hcal_N:\sup_M |\psi_N|\ge  (C+2)
\sqrt{\log N}
\right\}\hspace{1.25in}\\
\le k_N N^{-C^2 N^{-(m - 1)}(d_N-1)}\le O\left( N^{m^2 - C^2 A_m}\right)\;.\end{eqnarray*} Choosing $C $
so that $C^2 A_m - m^2 > 0$ is sufficiently large concludes the proof.

\end{proof}

\subsection{Sup norms on small balls}
We use the same method to prove a claim in \cite{NS} on the sup norm
in small balls on $S^2$. This corresponds to $m  = 1$

\begin{cor}[Estimate of the maximum]\label{claim11}
Given $\rho>0$, there exists $A$ such that, for any $x_0\in\Ss$,
$\displaystyle \PP \big\{ \max_{\mathcal B(x_0, \rho/N)} |f|   \ge
A\big\} \le \tfrac13$.
\end{cor}

Again by  Lemma \ref{calculus}, for any $x \in  B(x_0, \rho/N)$,
\begin{eqnarray}\label{prob1} \nu_N\left\{\psi_N:\cos\theta_x\ge
N^{-1} A\right\} &\le &
\exp\left(-(2N - 2)\frac{ A^2 }{N}\right)\nonumber \;.\end{eqnarray}

We can cover $B(x_0, \rho/N)$ by a collection of  $k_0(\epsilon)$ of balls $B(x^j, \frac{\epsilon}{N})$.
By  (\ref{prob1}), we have
\begin{equation} \label{centersA} \nu_N \left\{\psi_N:\max_{j} \cos\theta_{x^j} \ge C N^{-1/2} A \right\}\le
k_0(\epsilon)  e^{-A^2}\;. \end{equation}

Again, we need that (\ref{centersA}) implies   that the desired sup-norm
estimate holds on all of $B(x_0, \frac{1}{N})$, so we
consider an arbitrary point $y^j\in B(z^j)$ and
estimate the distance  $\de_{N}^j$, between
$\Phi_N(x^j)$ and $\Phi_N(y^j)$ in $S^{d_N-1}$. As before, we have
$$\cos\theta_{y^j} \ge \frac{A}{N}\Rightarrow
\cos\theta_{x^j}
\ge  \frac{A}{N}$$ and thus
$$\begin{array}{l}\left\{\psi_N\in
S\hcal_N:\sup_{y \in (B(x, \frac{1}{N})}  \cos\theta_y \ge A/N
\right\}\\[8pt] \qquad\qquad\qquad \subset \left\{\psi_N\in
S\hcal_N):\max_{j} \cos\theta_{x^j} \ge A/N \right\}\,\end{array}$$
or
$$ \nu_N \left\{\psi_N\in
S\hcal_N: \sup_{y \in B(x, \frac{1}{N})} \theta_y  \ge A/N \right\}\le k_0(\epsilon)  e^{- A^2} \;.$$

It follows that
\begin{eqnarray*}\nu_N \left\{\psi_N\in S\hcal_N:\sup_{B(x_0, 1/N)}  |\psi_N|\ge  A
\right\}\hspace{1.25in}
\le k_0(\epsilon)  e^{- A^2} \;.\end{eqnarray*} Choosing $A$ sufficiently large will
make the right side $\leq \frac{1}{3}.$

\subsection{Relation to Levy concentration}

The estimate in Theorem  \ref{Largesupnorm} is very closely
related to Levy's estimate.  The proof shows that

\begin{enumerate}

\item [(i)] $\lcal^\infty_N$ is Lipschitz continuous with norm
$\frac{ N^{(m - 1)/2}}{\sqrt{\log N}} \leq \|\lcal^\infty_N\|_{Lip} \leq
N^{(m - 1)/2}$. ;

\item [(ii)] The median of $\lcal^\infty_N$ satisfies:
$\mcal_{\lcal^\infty_N} \leq C_m  \sqrt{\log N}$ for sufficiently
large $N$.

\end{enumerate}

Indeed, Lipschitz continuity follows from  equivalence of norms on
finite dimensional vector spaces. To estimate the Lipschitz norm,
we recall the well-known fact that the $\lcal^2$-normalized
`coherent states' $\Phi_N^w(z) = \frac{\Pi_N(z,
w)}{\sqrt{\Pi_N(w,w)}}$ are the global maxima of $\lcal^\infty_N$
on $SH^0(M, L^N)$, as follows from the Schwartz inequality applied
to the reproducing identity $s(z) = \int_{M} \Pi_N(z,w) {s(w)}
dV(w).$ Moreover, $\|\Phi_N^w(z)\|_{\lcal^\infty} =
\sqrt{\Pi_N(w,w)} \sim N^{(m - 1)/2}$. It follows that
$$\big| \| s_1 + s_2\|_{\infty} - \|s_1\|_{\infty} \big| \leq 3 N^{(m - 1)/2}.$$
Now let $s_1 $ have $\lcal^\infty$ norm $\leq C \sqrt{\log N}$ and
let $s_1 = \Phi_N^w$ for some $w$. Then we see that
$$\big| \| s_1 + s_2\|_{\infty} - \|s_1\|_{\infty} \big| \geq
\frac{ N^{(m - 1)/2}}{\sqrt{\log N}}.$$

 It obviously follows from (i)--(ii)   combined with the  Levy
estimate (\ref{Levy}) that (for any $C > 0$)
\begin{equation} \nu_N \{ \psi \in S \hcal_N : \lcal_N^{\infty}(\psi) \geq C \sqrt{\log N}  \} \leq
 \exp( - C (d_N - 1) \log N /2 N^{m-1}).
\end{equation}
Since $d_N \sim N^{m-1}$, this  is essentially the same estimate as in
Theorem \ref{Largesupnorm}.

\subsection{\label{NS} Nazarov-Sodin Theorem on the mean number of nodal domains of random spherical harmonics}

In this section we review a recent  theorem of Nazarov-Sodin \cite{NS} on the mean number
of nodal domains of random spherical harmonics on $S^2$.
We  closely follow their exposition.

We recall that a nodal domain of an eigenfunction $f$  is a connected component of $M \backslash Z_f$ where
$Z_f$ is the zero set of $f$, i.e. the nodal hypersurface. The number of nodal domains of $f$ is denoted
$N(f)$ The main result of \cite{NS} is:

\begin{theo}\label{thm.main}
There exists a constant $a>0$ such that, for every $\e>0$, we have
$$
\PP\left\{\left|\frac{N(f)}{n^2}-a\right|>\e\right\}\le
C(\e)e^{-c(\e)n}
$$
where $c(\e)$ and $C(\e)$ are some positive constants depending on
$\e$ only.
\end{theo}

\begin{rem}
 As noted above,  the total length of the
nodal line $Z(f)$ of any spherical harmonic $f\in\HH_n$ does not exceed $C n$ for some $C > 0$ \cite{DF}. Since the typical spherical
harmonic has $\sim a n^2$ nodal domains, the perimeter  of most of  its nodal domains
must be of order $\frac{1}{n}$ and  the   diameters of the nodal domains must be  comparable to
$\frac{1}{n}$.
\end{rem}

The proof of Theorem~\ref{thm.main} consists of the following steps:

\begin{enumerate}

\item Proof that  the lower bound $\E N(f) \ge
{\rm const}\, n^2$.

\item Proof of  exponential concentration of the
random variable $N(f)/n^2$ around its median. The proof uses the uniform lower continuity of the
functional $f\mapsto N(f)$ with respect to the $L^2$-norm outside
of an exceptional set $E\subset \HH$ of exponentially small
measure and  Levy's concentration of measure principle.

\item Proof of the
existence of the limit $\displaystyle \lim_{n\to\infty} \E
N(f)/n^2$. In this part, we use existence of the scaling limit for
the covariance function $\E \big\{ f(x)f(y) \big\}$.

\end{enumerate}

\subsubsection{Lower bound for $\E N(f)$}\label{section_LB}

The first step is to show that  $\E N(f) \gtrsim n^2$. By Corollary \ref{claim11},
for any $\rho>0$, there exists $A > 0$ such that, for all  $x\in\Ss$,
$$\displaystyle \PP \big\{ \max_{\mathcal D(x, \rho/n)} |f|   \ge
A\big\} \le \tfrac13. $$

Furthermore, the zonal spherical harmonic $b_x^N$ with pole at $x$, i.e. $b_x = \frac{\Pi_N(\cdot, x)}{\sqrt{\Pi_N(x, x)}}$,
has the following property: there exists $\rho$ and $c_1$ such
that, for each sufficiently large $N$
$$
\| b_x \| = 1\,, \qquad b_x(x) \ge c_1 \sqrt{N}\,, \quad
\text{and} \quad b_x\big|_{\partial \mathcal D(x, \rho/N)} \le
-c_1 \sqrt{N}\,.
$$

\bigskip\par\noindent{\em Proof of the lower bound for $N(f)$: }
Fix $x\in \Ss$. We have $f=\xi_0 b_x + f_x$ where $\xi_0$ is a
Gaussian random variable with $\E \xi_0^2 = \tfrac1{2N+1}$, and
$f_x$ is a Gaussian spherical harmonic built over the orthogonal
complement to $b_x$ in $\HH_N$ and normalized by $\E \| f_x \|^2 =
\tfrac{2N}{2N+1}$. We choose a Gaussian random variable
$\widetilde\xi_0 $ independent of $\xi_0$ and of $f_x$ with $\E
\widetilde\xi_0^2 = \E \xi_0^2 = \frac1{2N+1}$, and set $ f_{\pm}
= \pm\widetilde\xi_0 b_x + f_x$. These are Gaussian spherical
harmonics having the same distribution as $f$. Note that
\[
f = \xi_0 b_x + \frac12 \left( f_+ + f_- \right)\,,
\]
and that by Corollary ~\ref{claim11}
\[
\PP \big( \{ \max_{\mathcal D(\rho, x)} |f_+| \le A \} \cap \{
\max_{\mathcal D(\rho, x)} |f_-| \le A \} \big) \ge 1 - (
\frac13 + \frac13) = \frac13\,.
\]

Now, consider the event $\Omega_x$ that $f(x)\ge A$ and
$f\big|_{\partial \mathcal D(x; \rho/N) }\le - A$. Clearly, $\Omega_x$ must contain
a nodal component of $f$.  The event
$\Omega_x$ happens provided that
\[
\xi_0 \sqrt{N} \ge 2 c_1^{-1}A \qquad {\rm and} \qquad
\max_{\mathcal D(\rho, x)} |f_\pm| \le A\,.
\]
Since the  variance of the Gaussian random variable $\xi_0 \sqrt{N}$ is of
constant size, there exists $\kappa > 0$ such that
\begin{multline*}
\PP (\Omega_x) \ge \PP ( \xi_0 \sqrt{N} \ge 2c_1^{-1} A )\, \cdot
 \PP \big( \{ \max_{\mathcal D(\rho, x)} |f_+| \le A \}
\cap \{ \max_{\mathcal D(\rho, x)} |f_-| \le A\} \big)  \ge
\kappa > 0\,.
\end{multline*}

Now pack $S^2$ with  $\simeq N^2$ disjoint disks $D_j$   of
radius $2\rho/N$ centered at points $x_j$. Each of them contains a component of $Z(f)$ with
probability at least $\kappa$, and since the discs are disjoint,  $N(f) \geq \sum_j N(f; D_j)$ where
$N(f, D_j)$ is the number of nodal components in the disc $D_j$. Hence,
 $$\E N(f) \geq \sum_j \E N(f, D_j) \geq \sum_j \PP(\Omega_{x_j}) \gtrsim C_{\kappa, \rho} \; N^2. $$
\hfill $\Box$

\subsubsection{Exponential concentration near the median}
\label{section_concentration}

This exponential concentration follows from Levy (-Sudakov-Tsirelson) concnentration
of measure plus the delicate fact that $N(f)$ is stable under  slight perturbations of $f$
for all but an exponentially rare set  of $f$, called `unstable' spherical harmonics.

\begin{lem}\label{lemma}
For every $\e>0$, there exists $\rho>0$ and an exceptional set
$E\subset \HH_N$ of probability $\PP(E)\le C(\e)e^{-c(\e)N}$ such
that for all $f\in \HH_N\setminus E$ and for all $g\in \HH_N$
satisfying $\|g\| \le \rho$, we have $N(f+g)\ge N(f)-\e N^2$.
\end{lem}

We refer to \cite{NS} for the proof.
 Lemma~\ref{lemma} in conjunction with the Sudakov-Tsirelson concentration of measure theorem (Lemma \ref{ST})
 implies  an exponential
concentration of $\dfrac{N(f)}{N^2}$ near its median $a_N$.  Recall that  $A_{+ \rho}$ is the $\rho$-tube around $A$.

\begin{cor} $\PP \{f\in\HH_N \colon |N(f) - a_N| > \e N^2\} \le 2e^{-c\rho^2 N}. $
 \end{cor}

\begin{proof}
Let $F=\{f\in\HH_N \colon N(f)>(a_N+\e) N^2\}$.
Then for $ f\in (F\setminus E)_{+\rho}$, we have $N(f)>a_N\;  N^2$,
and therefore, $\PP( (F\setminus E)_{+\rho} )\le \frac 12$. By the Sudakov-Tsirelson inequality (Lemma \ref{ST}),
$\PP(F\setminus E)\le 2e^{-c\rho^2 N}$, hence
$$
\PP(F)\le 2 e^{-c\rho^2 N }+C(\e)e^{-c(\e) N }\le C(\e)e^{-c(\e)N}\,.
$$

On the other hand, let   $G = \{f\in\HH_N \colon N(f)<(a_N -\e)N^2\}$.
Then
\[
G_{+\rho} \subset \{f\in \HH_N \,:\,N(f)<a_N^2\}\cup E
\]
and so
\[
\PP( G_{+\rho})\le \frac 12+ C(\e)e^{-c(\e)N}<\frac 34
\]
for large $N$. It follows  again that $\PP( G ) \le 2e^{-c\rho^2 N}$
for large $N$.

\end{proof}

Further, it is proved in \cite{NS} that
\begin{prop} The  sequence of medians $\big\{a_N
\big\}$ converges. \end{prop}

 Since the random variable $N(f)/N^2$ exponentially
concentrates near its median $a_N$ and is uniformly bounded, it
suffices to show that the sequence of means $  \E N(f_N)/N^2$ converges;  the sequence of medians $\big\{a_N\big\}$
converges to the same limit. The authors of \cite{NS} show that
$\big\{ \E N(f_n)/n^2 \big\}$ is a Cauchy sequence.

\section{\label{APPENDIX} Appendix on Tauberian Theorems}

We record here the statements of the Tauberian theorems that we
use in the article. Our main reference is \cite{SV}, Appendix B
and we follow their notation.

We denote by $F_+$ the class of real-valued, monotone
nondecreasing functions $N(\lambda)$  of polynomial growth
supported on $\R_+$. The following  Tauberian theorem uses only
the singularity at $t = 0$ of $\widehat{dN}$ to obtain a one term
asymptotic of $N(\lambda)$ as $\lambda \to \infty$:
\begin{theo} \label{ET} Let $N \in F_+$ and let $\psi \in
\scal (\R)$ satisfy the conditions:  $\psi$ is  even,
$\psi(\lambda)
> 0$ for all $\lambda \in \R$,   $\hat{\psi} \in
C_0^{\infty}$, and $\hat{\psi}(0) = 1$. Then,
$$\psi * dN(\lambda) \leq A \lambda^{\nu} \implies |N(\lambda) - N *
\psi(\lambda)| \leq C A \lambda^{\nu}, $$ where $C$ is independent
of $A, \lambda$.
\end{theo}

To obtain a two-term asymptotic formula, one needs to take into
account the other singularities of $\widehat{dN}$. We let $\psi$
be as above, and also introduce a second test function $\gamma \in
\scal$ with $\hat{\gamma} \in C_0^{\infty}$ and with the supp
$\hat{\gamma} \subset (0, \infty)$.

\begin{theo} \label{TTT} Let $N_1, N_2 \in F_+$ and assume:
\begin{enumerate}

\item $N_j * \psi (\lambda) = O(\lambda^{\nu}), (j = 1,2);$

\item $ N_2 * \psi (\lambda) = N_1 * \psi (\lambda) +
o(\lambda^{\nu}); $

\item $\gamma * d N_2 (\lambda) = \gamma * d N_1(\lambda) +
o(\lambda^{\nu}$.

\end{enumerate}

Then,

$$N_1(\lambda - o(1)) - o(\lambda^{\nu}) \leq N_2(\lambda) \leq
N_1(\lambda + o(1)) + o (\lambda^{\nu}). $$

\end{theo}

This Tauberian theorem is useful when the non-zero singularities
of $\widehat{dN_2}$ are as strong as the singularity at $t = 0$
and $N_2$ does not have two term polynomial asymptotics.

Alternatively, we may use the Tauberian lemma of \cite{HoI-IV}:

\begin{lem}\label{Tauber1}  Suppose that $\mu$ is a non-decreasing temperate function satisfying
$\mu(0)=0$ and that $\nu$ is a function of locally bounded
variation such that $\nu(0)=0$.  Suppose also that $m\ge 1$ and
that $\phi\in {\scal}({\Bbb R})$ is a fixed positive function
satisfying $\int \phi(\lambda)d\lambda =1$ and $\hat \phi(t)=0$,
$t\notin [-1,1]$.  If
$\phi_\sigma(\lambda)=\sigma^{-1}\phi(\lambda/\sigma)$,
$0<\sigma\le\sigma_0$, assume that for $\lambda\in {\Bbb R}$
\begin{equation}\label{A}
|d\nu(\lambda)|\le
\bigl(A_0(1+|\lambda|)^{m}+A_1(1+|\lambda|)^{m-1}\bigr) \,
d\lambda,
\end{equation}
and that
\begin{equation}\label{B}
|((d\mu-d\nu)*\phi_\sigma)(\lambda)|\le B(1+|\lambda|)^{-2}.
\end{equation}
Then
\begin{equation}\label{C}
|\mu(\lambda)-\nu(\lambda)|\le C_m\bigl( \,
A_0\sigma(1+|\lambda|)^m+A_1\sigma(1+|\lambda|)^{m-1}+B\, \bigr),
%C_mA\sigma(1+|\lambda|)^{m}+CB,
\end{equation}
where $C_m$ is a uniform constant depending only on $\sigma_0$ and
our $m\ge1$.
\end{lem}


\begin{thebibliography}{MMMM}

\bibitem[A]{A} N. Anantharaman,  Entropy and the localization of eigenfunctions, to appear in Ann. of Math. (2), preprint, 2004. 446

\bibitem[ANK]{ANK}N. Anantharaman, H. Koch,  and S. Nonnenmacher,  Entropy of eigenfunctions,
 arXiv:0704.1564.

\bibitem[AN]{AN} N. Anantharaman and S. Nonnenmacher,
Half-delocalization of eigenfunctions for the Laplacian on an
Anosov manifold, volume 57, number 6 (2007), 2465-2523  (arXiv:math-ph/0610019).




 \bibitem[AZ]{AZ} N. Anantharaman and S. Zelditch,
 Patterson-Sullivan distributions and quantum ergodicity. Ann. Henri Poincaré 8 (2007), no. 2, 361--426.

\bibitem[Ar]{Ar} V.I. Arnol' d,
Modes and quasimodes. Funkcional. Anal. i Prilo\v zen. 6 (1972),
no. 2, 12--20.

\bibitem[Ar]{Ar} N.A. Aronsajn,
A unique continuation theorem for solution of elliptic partial
differential equations or inequalities of second order, J. Math.
Pure Appl.36 (1957), 235–249.

\bibitem[B.B]{BB} V.M.Babic, V.S. Buldyrev: {\it Short-Wavelength Diffraction Theory},
Springer Series on Wave Phenomena 4, Springer-Verlag, New York
(1991).

\bibitem[B]{B}  V. M. Babic,
Eigenfunctions which are concentrated in the neighborhood of a
closed geodesic. (Russian) Zap. Nau\v cn. Sem. Leningrad. Otdel.
Mat. Inst. Steklov. (LOMI) 9 1968 15--63.

\bibitem[BQ]{BQ} D. Bakry and Z.  Qian,
 Some new results on eigenvectors via dimension, diameter, and Ricci curvature. Adv. Math. 155 (2000), no. 1, 98--153.

\bibitem[Bar]{Bar} A. Barnett, Asymptotic rate of quantum
ergodicity in chaotic Euclidean billiards, Comm. Pure Appl. Math.
59 (2006), no. 10, 1457--1488.

\bibitem[Bae]{Bae} C. B\"ar,
On nodal sets for Dirac and Laplace operators.
Comm. Math. Phys. 188 (1997), no. 3, 709--721.


\bibitem[Be]{Be} P. B\'erard, On the wave equation without
conjugate points, Math.\ Zeit.\ {\bf 155} (1977), 249--276.

\bibitem[BR]{BR} J. Bernstein and A.  Reznikov,
Analytic continuation of representations and estimates of automorphic forms. Ann. of Math. (2) 150 (1999), no. 1, 329--352.

\bibitem[Ber]{Ber}  M. Berry,  Regular and irregular semiclassical wavefunctions, J. Phys. A 10 (1977), 2083--2091.

\bibitem[BHO]{BHO} M. Berry, J. Hannay, and A. Ozorio de Almeida, Intensity moments of
semiclassical wavefunctions, Phys. D 8:1-2 (1983), 229--242.



\bibitem[Bers]{Bers} L. Bers,
Local behavior of solutions of general linear elliptic equations.
Comm. Pure Appl. Math. 8 (1955), 473--496.

\bibitem[Bes]{Bes} G. Besson,
Propri\'et\'e\'es g\'en\'eriques des fonctions propres et
multiplicit\'e. Comment. Math. Helv. 64 (1989), no. 4, 542--588.

\bibitem[BGS]{BGS} G. Blum, S. Gnutzmann and U. Smilansky, Nodal domain statistics: A Criterion
for quantum chaos,  Phys. Rev. Lett. 88, 114101 (2002).

\bibitem[BS]{BS} E. Bogomolny and C. Schmit, Percolation model
for nodal domains of chaotic wave functions, Phys. Rev. Letters 88
(18) (2002), 114102-114102-4.

\bibitem[Bou]{Bou} L. Boutet de Monvel,
Convergence dans le domaine complexe des s\'eries de fonctions
propres.  C. R. Acad.\ Sci.\ Paris S\'er. A-B 287 (1978), no.\ 13,
A855--A856.

\bibitem[Bru1]{Bru1} A.  Brudnyi, On local behavior of analytic functions. J. Funct. Anal. 169 (1999), no. 2, 481--493.

\bibitem[Bru2]{Bru2}  ---------,  A Bernstein-type inequality for algebraic functions. Indiana Univ. Math. J. 46 (1997), no. 1, 93--116.


\bibitem[Br]{Br} J. Br\"uning, \"Uber Knoten von Eigenfunktionen des Laplace-Beltrami Operators", Math.
 Z. 158 (1978), 15--21.

 \bibitem[BI]{BI} D. Burago and S. Ivanov,  Riemannian tori without conjugate points are flat. Geom. Funct. Anal. 4 (1994), no. 3, 259--269.

  \bibitem[BD]{BD} K. Burns and V. J.  Donnay,
 Embedded surfaces with ergodic geodesic flows. Internat. J. Bifur. Chaos Appl. Sci. Engrg. 7 (1997), no. 7, 1509--1527.

\bibitem[BGT]{BGT} N. Burq, P.  G\'erard, and N. Tzvetkov,
 Restrictions of the Laplace-Beltrami eigenfunctions to submanifolds. Duke Math. J. 138 (2007), no. 3, 445--486





\bibitem[Bu]{Bu} N. Burq,
Quantum ergodicity of boundary values of eigenfunctions: A control
theory approach, to appear in Canadian Math. Bull.
(math.AP/0301349).

\bibitem[BZ]{BZ} N. Burq and M. Zworski,
 Geometric control in the presence
of a black box. J. Amer. Math. Soc. 17 (2004), no. 2, 443--471.

\bibitem[BZ2]{BZ2} -----------,
 Bouncing ball modes and quantum chaos. SIAM Rev. 47 (2005), no. 1,
 43--49.



 \bibitem[Cab]{Cab} X. Cabr\'e,
On the Alexandroff-Bakelman-Pucci estimate and the reversed Hölder inequality for solutions of elliptic and parabolic equations.
Comm. Pure Appl. Math. 48 (1995), no. 5, 539--570.

\bibitem[Ch]{Ch} I. Chavel, {\it  Eigenvalues in Riemannian geometry.} Including a chapter by Burton Randol. With an appendix by Jozef Dodziuk. Pure and Applied Mathematics, 115. Academic Press, Inc., Orlando, FL, 1984

\bibitem[C]{C} G. Chen, P.J.  Morris, and J. Zhou,
Visualization of special eigenmode shapes of a vibrating
elliptical membrane, SIAM Rev. 36 (1994), no. 3, 453--469.

\bibitem[Ch1]{Ch1} S. Y. Cheng,
 Eigenfunctions and eigenvalues of Laplacian. Differential geometry
 (Proc. Sympos. Pure Math., Vol. XXVII, Stanford Univ., Stanford, Calif., 1973), Part 2, pp. 185--193.
 Amer. Math. Soc., Providence, R.I., 1975.

 \bibitem[Ch2]{Ch2} ---------,
  Eigenvalue comparison theorems and its geometric applications. Math. Z. 143 (1975), no. 3, 289--297.

  \bibitem[Ch3]{Ch3} ---------,
Eigenfunctions and nodal sets. Comment. Math. Helv. 51 (1976), no.
1, 43--55.

\bibitem[C]{C} E. Chladni, {\it  Entdeckungen \"uber die Theorie des
Klanges} (1787).

\bibitem[C2]{C2} -----------,  Die Akustik (Breitkopf und Härtel, Leipzig, 1802).

\bibitem[Chr]{Chr} H. Christianson, Semiclassical non-concentration near hyperbolic orbits. J. Funct. Anal. 246 (2007), no. 2, 145--195.

\bibitem[CM]{CM} T. Colding and W. P. Minicozzi II,
Volumes for eigensections. Geom. Dedicata 102 (2003), 19--24.

\bibitem[CV]{CV} Y.Colin de Verdi\`ere, Ergodicit\'e et fonctions propres du
Laplacien, Comm.Math.Phys. 102 (1985), 497-502.

\bibitem[CV2]{CV2} ---------,Quasi-modes sur les varietes
Riemanniennes compactes, Invent.Math. 43 (1977), 15-52.

\bibitem[CV3]{CV3} ---------,
 Spectre conjoint d'opérateurs
pseudo-différentiels qui commutent. II. Le cas intégrable.
Math. Z. 171 (1980), no. 1, 51--73.

\bibitem[CV4]{CV4}, ---------, Semi-classical measures and entropy, S\'eminaire Bourbaki 59\`eme ann\'e\'e, 2006-7, 2006-7 (to appear).

\bibitem[CVP]{CVP} Y. Colin de Verdi\`ere and B. Parisse, \'Equilibre instable en r\'egime semi-classique. I. Concentration microlocale, Comm. in PDE 19 (1994), no. 9-10, 1535--1563.


\bibitem[DSj]{DSj} M. Dimassi and J.  Sj\"ostrand, {\it Spectral asymptotics in the semi-classical limit}.
 London Mathematical Society Lecture Note Series, 268. Cambridge University Press, Cambridge, 1999.

\bibitem[Dir]{Dir} P.A. M. Dirac, {\it  The Principles of Quantum Mechanics.} 3d ed. Oxford, at the Clarendon Press, 1947.


 \bibitem[Don1]{Don1} H. Donnelly, Exceptional sequences of eigenfunctions for hyperbolic manifolds. Proc. Amer. Math. Soc. 135 (2007), no. 5, 1551--1555.

 \bibitem[Don2]{Don2} H. Donnelly,  Quantum unique ergodicity. Proc. Amer. Math. Soc. 131 (2003), no. 9, 2945--2951.


\bibitem[Dong]{Dong} R-T. Dong,
 Nodal sets of eigenfunctions on Riemann surfaces. J. Differential Geom. 36 (1992), no. 2, 493--506.

 \bibitem[DF]{DF} H. Donnelly and C. Fefferman, Nodal sets of eigenfunctions on
Riemannian manifolds, Invent. Math. 93 (1988), 161-183.

\bibitem[DF2]{DF2} H. Donnelly and C. Fefferman,
Nodal sets of eigenfunctions: Riemannian manifolds with boundary.
Analysis, et cetera, 251--262, Academic Press, Boston, MA, 1990.


\bibitem[DF]{DF} H. Donnelly and C. Fefferman, Nodal sets of eigenfunctions on
Riemannian manifolds, Invent. Math. 93 (1988), 161-183.

\bibitem[DF2]{DF2} -----------,
Nodal sets of eigenfunctions: Riemannian manifolds with boundary.
Analysis, et cetera, 251--262, Academic Press, Boston, MA, 1990.

\bibitem[DF3]{DF3} -----------,
 Growth and geometry of eigenfunctions of the Laplacian. Analysis
  and partial differential equations, 635--655, Lecture Notes in Pure and Appl. Math., 122, Dekker, New York,
  1990.

\bibitem[DF4]{DF4}  -----------,  Nodal sets for eigenfunctions of the Laplacian on surfaces. J. Amer. Math. Soc. 3 (1990), no. 2, 333--353.

\bibitem[EZ]{EZ} L. Evans and M. Zworski, Introduction to semiclassical
 analysis, UC Berkeley lecture notes (2006).  http://math.berkeley.edu/~zworski.




\bibitem[D.G]{D.G} J.J.Duistermaat and V.Guillemin, The spectrum of positive
elliptic operators and periodic bicharacteristics, Inv.Math. 24
(1975), 39-80.

\bibitem[FNB]{FNB} F. Faure, S.  Nonnenmacher, and S. De Bi\`evre,  Scarred eigenstates for quantum cat maps of minimal periods. Comm. Math. Phys. 239 (2003), no. 3, 449--492.

\bibitem[Fed]{Fed} H. Federer, {\it Geometric measure theory.}
 Die Grundlehren der mathematischen Wissenschaften, Band 153 Springer-Verlag New York Inc., New York
1969.

\bibitem[FP]{FP} M. Feingold and A.  Peres,
Distribution of matrix elements of chaotic systems.
Phys. Rev. A (3) 34 (1986), no. 1, 591--595.

\bibitem[F]{F} G. B.  Folland,
{\it Harmonic analysis in phase space.}  Annals of Mathematics
Studies, 122. Princeton University Press, Princeton, NJ, 1989.


\bibitem[FGS]{FGS} G. Foltin, S.  Gnutzmann, and U.  Smilansky,
 The morphology of nodal lines---random waves versus percolation. J. Phys. A 37 (2004), no. 47, 11363--11371.

\bibitem[GaL]{GaL} N. Garofalo and F. H. Lin,
Monotonicity properties of variational integrals, $A\sb p$ weights
and unique continuation. Indiana Univ. Math. J. 35 (1986), no. 2,
245--268.

\bibitem[GaL2]{GaL2} -----------,
Unique continuation for elliptic operators: a
geometric-variational approach. Comm. Pure Appl. Math. 40 (1987),
no. 3, 347--366

\bibitem[Ge]{Ge} P. G\'erard,
Microlocal defect measures. Comm. Partial Differential Equations
16 (1991), no. 11, 1761--1794.

\bibitem[GL]{GL} P.G\'erard and E.Leichtnam, Ergodic properties of
eigenfunctions for the Dirichlet problem, Duke Math J. 71 (1993),
559-607.

\bibitem[GiTr]{GiTr} D. Gilbarg and N. S. Trudinger, {\it  Elliptic partial differential equations of second order.}
 Classics in Mathematics. Springer-Verlag, Berlin, 2001.

 \bibitem[GLS]{GLS} F. Golse, E. Leichtnam, and M. Stenzel,
 Intrinsic microlocal analysis and inversion formulae for
  the heat equation on compact real-analytic Riemannian manifolds.
  Ann. Sci.\ \'Ecole Norm.\ Sup. (4) 29 (1996), no.\ 6, 669--736.

  \bibitem[GO]{GO} S. Grellier and J. P. Otal, Bounded eigenfunctions in the real hyperbolic space. Int. Math. Res. Not. 2005, no. 62,
  3867--3897.


\bibitem[GSj]{GSj} A. Grigis and J. Sj\"ostrand, {\it Microlocal analysis for
 differential operators}, London Math. Soc. Lecture Notes 196 (1994).

\bibitem[GS1]{GS1}  V. Guillemin and M. Stenzel, Grauert tubes and the homogeneous Monge-Ampère equation. J. Differential Geom. 34 (1991), no. 2, 561--570.


\bibitem[GS2]{GS2}  -----------,   Grauert tubes and the homogeneous Monge-Ampère equation. II. J. Differential Geom. 35 (1992), no. 3, 627--641.



\bibitem[GM]{GM} V.Guillemin and R.B.Melrose, The Poisson summation formula for
manifolds with boundary, Adv.in Math. 32 (1979), 204 - 232.

\bibitem[GW]{GW} V. Guillemin and A.  Weinstein,
Eigenvalues associated with a closed geodesic. Bull. Amer. Math.
Soc. 82 (1976), no. 1, 92--94.







\bibitem[H2]{H2} Q. Han, Nodal sets of harmonic functions, Pure
and Applied Mathematics Quarterly 3 (3) (2007), 647-688.


\bibitem[HHL]{HHL} Q. Han, R. Hardt, and F. H.  Lin,
 Geometric measure of singular sets of elliptic equations. Comm. Pure Appl. Math. 51 (1998), no. 11-12, 1425--1443.

\bibitem[H]{H} Q. Han and F.H. Lin  {\it Nodal sets of solutions of elliptic
differential equations}, book in preparation (2007).

\bibitem[HL]{HL} -----------,   {\it Elliptic partial differential equations.}
 Courant Lecture Notes in Mathematics, 1. New York University, Courant Institute of Mathematical Sciences, New York; American Mathematical Society, Providence, RI, 1997.

\bibitem[HNOO]{HNOO} R. Hardt, M. Hoffmann-Ostenhof, T.  Hoffmann-Ostenhof and N.  Nadirashvili,
 Critical sets of solutions to elliptic equations. J. Differential Geom. 51 (1999), no. 2,
 359--373.

 \bibitem[HS]{HS} R. Hardt and L. Simon,  Nodal sets for solutions of elliptic equations, J. Differential Geom. 30, 1989, pp. 505--522.

\bibitem[HW]{HW} P. Hartman and A.  Wintner,  On the local behavior of solutions of non-parabolic partial differential equations. Amer. J. Math. 75, (1953). 449--476.

    \bibitem[HW2]{HW2} P. Hartman and W. Wintner,
On the local behavior of solutions of non-parabolic partial differential equations. III. Approximations by spherical harmonics.
Amer. J. Math. 77 (1955), 453--474.

\bibitem[HT]{HT} A. Hassell and
T. Tao,   Upper and lower bounds for normal derivatives of Dirichlet eigenfunctions. Math. Res. Lett. 9 (2002), no. 2-3, 289--305.


\bibitem[HZ]{HZ} A. Hassell and S. Zelditch,  Quantum ergodicity of boundary values of eigenfunctions. Comm. Math. Phys. 248 (2004), no. 1, 119--168.

\bibitem[Hel]{Hel} S. Helgason, {\it Topics in harmonic analysis on homogeneous spaces.}
 Progress in Mathematics, 13. Birkhäuser, Boston, Mass., 1981.



\bibitem[H]{H} E.J. Heller,
Bound-state eigenfunctions of classically chaotic Hamiltonian
systems: scars of periodic orbits. Phys. Rev. Lett. 53 (1984), no.
16, 1515--1518.

\bibitem[HC]{HC}   D. Hilbert and R. Courant, {\it Methods of mathematical physics},  Vol. I and   Vol.
II:. Interscience Publishers
 (John Wiley $\&$ Sons), New York-Lon don 1962.




\bibitem[HoI-IV]{HoI-IV}  L. H\"ormander, {\it Theory of
Linear Partial Differential Operators I-IV}, Springer-Verlag, New
York (1985).



 \bibitem[I]{I} V. Isakov,
 Carleman estimates and applications to inverse problems. Milan J. Math. 72 (2004),
 249--271.

\bibitem[Isi]{Isi} M. B. Isichenko,
 Percolation, statistical topography, and transport in random media. Rev. Modern Phys. 64 (1992), no. 4, 961--1043.

\bibitem[IS]{IS} H. Iwaniec and P.  Sarnak,
 $L\sp \infty$ norms of eigenfunctions of arithmetic surfaces. Ann. of Math. (2) 141 (1995), no. 2, 301--320.

 \bibitem[JN]{JN} D. Jakobson and N. Nadirashvili,  Eigenfunctions with few critical points. J. Differential Geom. 53 (1999), no. 1, 177--182.

\bibitem[JN2]{JN2} -----------,
Quasi-symmetry of $L\sp p$ norms of eigenfunctions.  Comm. Anal.
Geom. 10 (2002), no. 2, 397--408.

\bibitem[JZ]{JZ} D. Jakobson and S. Zelditch,
Classical limits of eigenfunctions for some completely integrable systems. (English summary) Emerging applications of number theory (Minneapolis, MN, 1996), 329--354,
IMA Vol. Math. Appl., 109, Springer, New York, 1999.

\bibitem[JL]{JL} D. Jerison and G. Lebeau,
Nodal sets of sums of eigenfunctions. Harmonic analysis and
partial differential equations (Chicago, IL, 1996), 223--239,
Chicago Lectures in Math., Univ. Chicago Press, Chicago, IL, 1999.



\bibitem[KP]{KP} M. Karplus and
R.N. Porter, {\it Atoms $\&$ Molecules: An Introduction for
Students of Physical Chemistry}, The Benjamin/Cummings Pub. Co.
(1970).

\bibitem[K]{K} J. B. Keller,
 Corrected Bohr-Sommerfeld quantum conditions for
nonseparable systems. Ann. Physics 4 1958 180--188.

\bibitem[KTZ]{KTZ} H. Koch, D.  Tataru, and M.  Zworski,
Semiclassical $L\sp p$ estimates, Ann. Henri Poincaré 8 (2007), no. 5, 885--916.

%\bibitem[KKL]{KKL}
%Katsiaryna Krupchyk, Yaroslav Kurylev, Matti Lassas,
% Inverse spectral problems
%on a closed manifold, arXiv:0709.2171.

\bibitem[Kl]{Kl} W. Klingenberg, {\it Lectures on Closed
Geodesics}, Grundlehren der.\ math.\ W. {\bf 230}, Springer-Verlag
(1978).

\bibitem[Kr]{Kr} P.  Kr\"oger,  On the spectral gap for
compact manifolds. J. Differential Geom. 36 (1992), no. 2, 315--330.

\bibitem[Kr2]{Kr2} -----------,
 On the ranges of eigenfunctions on compact manifolds. Bull. London Math. Soc. 30 (1998), no. 6, 651--655.

\bibitem[Kua]{Kua} I. Kukavica, Nodal volumes for eigenfunctions of analytic regular elliptic problems.
 J. Anal. Math. 67 (1995), 269--280.

 \bibitem[Ku]{Ku} -----------,
Quantitative uniqueness for second-order elliptic operators. Duke
Math. J. 91 (1998), no. 2, 225--240.

\bibitem[L]{L} V. F. Lazutkin,
{\it KAM theory and semiclassical approximations to
eigenfunctions.} With an addendum by A. I. Shnirelman. Ergebnisse
der Mathematik und ihrer Grenzgebiete (3), 24. Springer-Verlag, Berlin, 1993.

\bibitem[L1]{L1} -----------,  Construction of an asymptotic series of
eigenfunctions of the ``bouncing ball'' type, Proc.Steklov
Inst.Math. 95 (1968), 125- 140.


\bibitem[LR]{LR} G. Lebeau and L.  Robbiano,
 Controle exact de l'\'equation de la chaleur.
  Comm. Partial Differential Equations 20 (1995), no. 1-2, 335--356.



\bibitem[LS1]{LS1} L. Lempert and R.  Sz\"oke,
 Global solutions of the homogeneous complex Monge-Ampère equation and
 complex structures on the tangent bundle of Riemannian manifolds. Math. Ann. 290 (1991), no. 4, 689--712.

\bibitem[LS2]{LS2} -----------,
The tangent bundle of an almost complex manifold, Canad. Math.
Bull. 44 (2001), no. 1, 70--79.

\bibitem[Lin]{Lin} F.H.  Lin, Nodal sets of solutions of elliptic
and parabolic equations. Comm. Pure Appl. Math. 44 (1991), no. 3,
287--308.

\bibitem[LIND]{LIND} E. Lindenstrauss, Invariant measures and
arithmetic quantum ergodicity, Annals of Math. (2) 163 (2006), no. 1, 165--219.

\bibitem[LS]{LS} W.Luo and P.Sarnak, Quantum ergodicity of eigenfunctions on
${\rm PSL}_2(\Z) \backslash {\bf H}^2$, IHES Publ. 81 (1995),
207-237.

\bibitem[L.S.2]{L.S.2} -----------,  Quantum variance for
Hecke eigenforms, Annales Scient. de l'\'Ecole Norm. Sup. 37
(2004),  p. 769-799.

\bibitem[M]{M}  E. Malinnikova,
Propagation of smallness for solutions of generalized
Cauchy-Riemann systems.  Proc. Edinb. Math. Soc. (2) 47 (2004),
no. 1, 191--204.

\bibitem[MOZ]{MOZ} J. Marklof, S. O'Keefe,  Weyl's law and quantum ergodicity for maps with divided phase space. With an appendix "Converse quantum ergodicity" by Steve Zelditch. Nonlinearity 18 (2005), no. 1, 277--304.

    \bibitem[Mi]{Mi} L. Miller, Escape function conditions for the observation, control, and stabilization of the wave equation. SIAM J. Control Optim. 41 (2002), no. 5, 1554--1566.

\bibitem[MS]{MS} S. D. Miller and W.  Schmid,
 The highly oscillatory behavior of automorphic distributions for $\rm SL(2)$. Lett. Math. Phys. 69 (2004), 265--286.

\bibitem[Mil]{Mil} J. Milnor, {\it Morse theory.}
Annals of Mathematics Studies, No. 51 Princeton University Press,
Princeton, N.J. (1963)

\bibitem[NJT]{NJT} N. Nadirashvili, D. Jakobson, and J.A. Toth,  Geometric properties of eigenfunctions. (Russian) Uspekhi Mat. Nauk 56 (2001), no. 6(342), 67--88; translation in Russian Math. Surveys 56 (2001), no. 6, 1085--1105


\bibitem[NPS]{NPS} F. Nazarov, L.  Polterovich and M.  Sodin,
 Sign and area in nodal geometry of Laplace eigenfunctions. Amer. J. Math. 127 (2005), no. 4, 879--910.

\bibitem[NS]{NS} F. Nazarov and  M. Sodin,  On the Number of Nodal Domains of Random Spherical
Harmonics (arXiv:0706.2409)

\bibitem[Neu]{Neu} J. Neuheisel, PhD Thesis, Johns Hopkins University (2000).

\bibitem[O]{O} S. Ozawa, Asymptotic property of eigenfunction of the Laplacian at the boundary, Osaka J. Math. 30 (1993), no. 2, 303--314.

\bibitem[P]{P} A. Pleijel, "Remarks on Courant's nodal line theorem," Comm. Pure Appl. Math., 9, 543–550 (1956).

\bibitem[Po]{Po} I. Polterovich, Pleijel's nodal domain theorem for free membranes, to appear in Proc. AMS   (arXiv:0805.1553).

\bibitem[PS]{PS}
 L. Polterovich and M.  Sodin, Nodal inequalities on surfaces. Math. Proc. Cambridge Philos. Soc. 143 (2007), no. 2, 459--467
 (arXiv:math/0604493).

\bibitem[Pop]{Pop} G. Popov,  Invariant tori, effective stability, and quasimodes
with exponentially small error terms. I. Birkhoff normal forms.
Ann. Henri Poincaré 1 (2000), no. 2, 223--248.

\bibitem[Ra]{Ra} J.  Ralston,
Gaussian beams and the propagation of singularities. {\it Studies
in partial differential equations,} 206--248, MAA Stud. Math., 23,
Math. Assoc. America, Washington, DC, 1982.

\bibitem[Ra2]{Ra2} -----------,
Approximate eigenfunctions of the Laplacian. J. Differential
Geometry 12 (1977), no. 1, 87--100.

\bibitem[R]{R} J. W. Rayleigh, {\it The Theory of Sound.} 2d ed. Dover Publications, New York, N. Y., 1945.

\bibitem[RY]{RY} N. Roytwarf and Y.  Yomdin, Y.
Bernstein classes, Ann. Inst. Fourier (Grenoble) 47 (1997), no. 3, 825--858.

\bibitem[RS]{RS} Z. Rudnick and P.  Sarnak,  The behaviour of eigenstates of arithmetic hyperbolic manifolds. Comm. Math. Phys. 161 (1994), no. 1, 195--213.

\bibitem[SV]{SV} Yu.  Safarov and D.  Vassiliev, {\it The asymptotic distribution of eigenvalues of partial differential operators}
. Translated from the Russian manuscript by the authors.
Translations of Mathematical Monographs, 155.
 American Mathematical Society, Providence, RI, 1997.

 \bibitem[Sar]{Sar} P. Sarnak,  Integrals of products of eigenfunctions. Internat. Math. Res. Notices no.
 6(1997),
 251- 261.

 \bibitem[Sar2]{Sar2} -----------,   Arithmetic quantum chaos. The Schur lectures (1992) (Tel Aviv), 183--236, Israel Math. Conf. Proc., 8, Bar-Ilan Univ., Ramat Gan, 1995.

\bibitem[Sch]{Sch} E. Schr\"odinger, Quantisierung als Eigenwertproblem,  Annalen der Physik,
1926.

\bibitem[Schu]{Schu} R. Schubert,  On the rate of quantum ergodicity for quantised maps (preprint).

\bibitem[Schu2]{Schu2}. ----------,  Upper bounds on the rate of quantum ergodicity. Ann. Henri Poincaré 7 (2006), no. 6, 1085--1098.

\bibitem[SZ]{SZ} B. Shiffman and S.  Zelditch,
Random polynomials of high degree and Levy concentration of
measure, Volume in honor of Y. T. Siu, Asian J. Math. 7 (2003),
no. 4, 627--646.


\bibitem[Sj]{Sj} J.Sj\"ostrand, Semi-excited states in nondegenerate potential wells,
Asym.An. 6(1992) 29-43.

\bibitem[Sh.1]{Sh.1} A.I.Shnirelman, Ergodic properties of eigenfunctions,
Usp.Math.Nauk. 29/6 (1974), 181-182.

\bibitem[Sh.2]{Sh.2} -----------,  On the asymptotic properties of
eigenfunctions in the region of chaotic motion, addendum to
V.F.Lazutkin, {\it KAM theory and semiclassical approximations to
eigenfunctions}, Springer (1993).



\bibitem[SS]{SS} U. Smilansky and H.-J. St\"ockmann, Nodal Patterns in Physics and Mathematics,
 The European Physical Journal Special Topics
Vol. 145 (June 2007).













\bibitem[Sog]{Sog} C. D. Sogge, Concerning the $L^p$ norm of spectral clusters for
second-order elliptic operators on compact manifolds, J. Funct.
Anal. 77 (1988), 123--138.

\bibitem[SoZ]{SoZ} C. Sogge and S.  Zelditch,
Riemannian manifolds with maximal eigenfunction growth.  Duke
Math. J. 114 (2002), no. 3, 387--437

\bibitem[SHM]{SHM} RM Stratt, NC Handy, WH Miller,
 On the quantum mechanical implications of classical ergodicity,
 - The Journal of Chemical Physics,
1979.

\bibitem[ST]{ST} V.N. Sudakov and B.S. Tsirelson, B. S.
Extremal properties of half-spaces for spherically invariant measures. (Russian)
Problems in the theory of probability distributions, II.
Zap. Nau\v cn. Sem. Leningrad. Otdel. Mat. Inst. Steklov. (LOMI) 41 (1974), 14--24, 165.

\bibitem[Su]{Su} T. Sunada,  Quantum ergodicity. Progress in inverse spectral geometry, 175--196, Trends Math., Birkhäuser, Basel, 1997.

\bibitem[Su2]{Su2} ---------, Trace formula and heat equation
asymptotics for a non-positively curved manifold, Am.\ J. Math.,
vol.\ 104 (1982), 795--812.

\bibitem[Taa]{Taa} D. Tataru, On the regularity of boundary traces for the wave equation, Ann. Scuola Norm. Sup. Pisa Cl. Sci. (4) 26 (1998), 185 -- 206.

\bibitem[Ta]{Ta} ---------, Carleman estimates and unique continuation for solutions to  boundary value problems.  J. Math. Pures Appl. (9), 75 (1996), no. 4, 367-408.

\bibitem[Ta2]{Ta2} ---------,  Unique continuation problems for partial differential equations.  IMA vol. 137, Springer-Verlag (2004).

\bibitem[T]{T} M. E. Taylor,{\it Noncommutative harmonic analysis. }
Mathematical Surveys and Monographs, 22. American Mathematical
Society, Providence, RI, 1986.

\bibitem[T2]{T2} -----------,  {\it Partial Differential Equations II}, Applied Math.Sci. 116, Springer-Verlag (1996).

\bibitem[To1]{To1} J. A. Toth, Eigenfunction decay estimates in the quantum integrable case. Duke Math. J. 93 (1998), no. 2, 231--255.

\bibitem[To2]{To2}  ---------,  $ L^{2}$-restriction bounds for eigenfunctions along curves in the quantum completely integrable case,
arXiv:0803.0978.

\bibitem[TZ]{TZ} J. A. Toth and S. Zelditch,
 Riemannian manifolds with uniformly bounded eigenfunctions. Duke Math. J. 111 (2002), no. 1, 97--132.

 \bibitem[TZ2]{TZ2} -----------,
  $L\sp p$ norms of eigenfunctions in the completely integrable case. Ann. Henri Poincaré 4 (2003), no. 2,
343--368.

\bibitem[TZ3]{TZ3} ------------,
Counting Nodal Lines Which Touch the Boundary of an Analytic Domain, to appear in Jour. Diff. Geom. (arXiv:0710.0101.)

\bibitem[U]{U} K. Uhlenbeck,
Generic properties of eigenfunctions.
Amer. J. Math. 98 (1976), no. 4, 1059--1078.

\bibitem[Van]{Van} J. M. VanderKam,  $L\sp \infty$ norms and quantum ergodicity on the sphere. Internat. Math. Res. Notices 1997, no. 7, 329--347.

\bibitem[V]{V} S. Vessella,
Quantitative continuation from a measurable set of solutions of
elliptic equations.
 Proc. Roy. Soc. Edinburgh Sect. A 130 (2000),
no. 4, 909--923.


\bibitem[W]{W} A. Weinstein,
On Maslov's quantization condition. Fourier integral operators and
partial differential equations (Colloq. Internat., Univ. Nice,
Nice, 1974), pp. 341--372. Lecture Notes in Math., Vol. 459,
Springer, Berlin, 1975.


 \bibitem[Y1]{Y1} S.T. Yau, Survey on partial differential equations in differential geometry.
 {\it Seminar on Differential Geometry}, pp. 3--71, Ann. of Math. Stud., 102, Princeton Univ. Press, Princeton, N.J., 1982.




\bibitem[Y2]{Y2} -----------,   Open problems in geometry.
{\it  Differential geometry: partial differential equations on
manifolds} (Los Angeles, CA, 1990), 1--28, Proc. Sympos. Pure
Math., 54, Part 1, Amer. Math. Soc., Providence, RI, 1993.

\bibitem[Y3]{Y3} -----------,  A note on the distribution of critical points of eigenfunctions, Tsing Hua Lectures in Geometry and Analysis 315--317, Internat. Press, 1997.



\bibitem[Z1]{Z1} S. Zelditch,  The inverse spectral problem. With an appendix by Johannes Sjöstrand
and Maciej Zworski.{\it  Surv. Differ. Geom., IX, Surveys in
differential geometry. Vol. IX}, 401--467, Int. Press, Somerville,
MA, 2004.
\bibitem[Z2]{Z2} ---------,
Uniform distribution of eigenfunctions on compact hyperbolic
surfaces. Duke Math. J. 55 (1987), no. 4, 919--941

 \bibitem[Z3]{Z3} ---------, Quantum ergodicity and mixing,
{\it  Encyclopedia of Mathematical Physics} Vol. 4, Ed.
J.P. Françoise,
G. Naber,
S.T.  Tsou (2007), 183-196.



  \bibitem[Z4]{Z4} ---------, "Lectures on wave invariants", London Math. Soc. Lecture Note Ser. 273 (1999), 284--328.

  \bibitem[Z5]{Z5}---------,  Kuznecov sum formulae and
Szego limit formulae on manifolds, Comm.\ PDE {\bf 17} (1\&2)
(1992), 221--260.

\bibitem[Z6]{Z6}  -----------, On the rate of quantum ergodicity. I. Upper bounds. Comm. Math. Phys. 160 (1994), no. 1, 81--92

\bibitem[Z7]{Z7} ---------,  Complex zeros of real ergodic eigenfunctions. Invent. Math. 167 (2007), no. 2, 419--443.


\bibitem[Z8]{Z8} ---------, Wave invariants at elliptic closed geodesics.
Geom. Funct. Anal. 7 (1997), no. 1, 145--213.
\bibitem[Z9]{Z9} ---------,  Quantum ergodicity of $C\sp *$ dynamical systems. Comm. Math. Phys. 177 (1996), no. 2, 507--528.






\bibitem[Z9]{Z9} ---------,
Note on quantum unique ergodicity,
Proc. Amer. Math. Soc. 132 (2004), no. 6, 1869--1872.

\bibitem[ZZw]{ZZw} S.Zelditch and M.Zworski,  Ergodicity of eigenfunctions
for ergodic billiards,  Comm.Math. Phys. 175 (1996), 673-682.


\bibitem[ZRS]{ZRS} Ya. B. Zel'dovich, A.A.  Ruzmaikin and D.D.
Sokoloff, {\it The Almighty Chance.} World Scientific Lecture
Notes in Physics, 20. World Scientific Publishing Co., Inc., River
Edge, NJ, 1990.

\end{thebibliography}
\end{document}